\documentclass[preprint]{imsart}

\RequirePackage{amsthm,amsmath,amsfonts,amssymb,mathrsfs,mathdots,mathdots,mathtools}
\RequirePackage[numbers]{natbib}
\RequirePackage[colorlinks,citecolor=blue,urlcolor=blue]{hyperref}

\usepackage{hyperref}
\usepackage{enumitem,float}
\usepackage{graphicx}
\usepackage{dsfont}

\DeclarePairedDelimiter\floor{\lfloor}{\rfloor}

\newlist{inparaenum}{enumerate}{2}% allow two levels of nesting in an enumerate-like environment
\setlist[inparaenum]{nosep}% compact spacing for all nesting levels
\setlist[inparaenum,1]{label=(\roman*)}% labels for top level
\setlist[inparaenum,2]{label=(\roman{inparaenumi}\emph{\alph*})}% labels for 

\startlocaldefs
\newtheorem{theorem}{Theorem}[section]
\newtheorem{lemma}[theorem]{Lemma}
\newtheorem{prop}[theorem]{Proposition}
\newtheorem{cor}[theorem]{Corollary}
\theoremstyle{remark}
\newtheorem{definition}[theorem]{Definition}
\newtheorem{condition}[theorem]{Condition}
\newtheorem{example}[theorem]{Example}
\newtheorem{rem}[theorem]{Remark}

\endlocaldefs

\def\thick#1{\hbox{\rlap{$#1$}\kern0.25pt\rlap{$#1$}\kern0.25pt$#1$}}

\newcommand{\prodcop}{{C_0^{\scalebox{0.65}{$(nm_n)$}}}}

\newcommand{\prodFm}{{F_0^{\scalebox{0.65}{\emph{($n m_n$)}}}}}
\newcommand{\prodFmalt}{{F_0^{\scalebox{0.65}{{$(n m_n)$}}}}}

\newcommand{\postobs}{{\Pi_n^{\scalebox{0.65}{(o)}}}}

\newcommand{\postobsdens}{{\tilde{\Pi}_n^{\scalebox{0.65}{(o)}}}}

\newcommand{\altpostobsdens}{{\tilde{\Pi}_n^{\scalebox{0.65}{\emph{(o)}}}}}

\newcommand{\altpostobs}{{\Pi_n^{\scalebox{0.65}{\emph{(o)}}}}}

\newcommand{\prodG}{{G_\bone^{\scalebox{0.65}{($n$)}}}}
\newcommand{\prodF}{{F^{\scalebox{0.65}{($n$)}}}}
\newcommand{\iprodF}{{F^{\scalebox{0.65}{($\infty$)}}}}

\newcommand{\iprodFtruealt}{{F_0^{\scalebox{0.65}{\emph{($\infty$)}}}}}

\newcommand{\pisc}{{\pi_{\text{sc}}}}

\newcommand{\piloc}{{\pi_{\text{loc}}}}

\newcommand{\Yset}{{\mathcal{Y}}}

\newcommand{\iprodGH}{{G_\bone^{\scalebox{0.65}{\emph{($\infty$)}}}(\cdot|H_0)}}

\newcommand{\iprodGtH}{{G_{\bvartheta_0}^{\scalebox{0.65}{($\infty$)}}(\cdot|H_0)}}

\newcommand{\prodGtzerE}{{G_{\bvartheta_0}^{\scalebox{0.65}{($n$)}}(E_n|\theta_0)}}

\newcommand{\prodGtH}{{G_{\bvartheta}^{\scalebox{0.65}{($n$)}}(\cdot|H)}}

\newcommand{\tests}{{\boldsymbol{\tau}_n}}

\newcommand{\cover}{{\mathcal{N}}}

\newcommand{\prodGtHtruealt}{
		{
			G_{
				\scalebox{0.65}{$\bvartheta_0$}
			}^{
			\scalebox{0.65}{\emph{($n$)}}
			}
		}
	}

\newcommand{\prodGtHtrue}{{G_{\bvartheta_0}^{\scalebox{0.65}{($n$)}}}}

\newcommand{\iprodGtHtruealt}{{G_{\bvartheta_0}^{\scalebox{0.65}{\emph{($\infty$)}}}}}

\newcommand{\prodGtHtruealtjj}{{G_{\bvartheta_{0,j}}^{\scalebox{0.65}{{($n$)}}}}}

\newcommand{\prodGtHaltjj}{{G_{\bvartheta_{j}}^{\scalebox{0.65}{($n$)}}}}

\newcommand{\iprodGrho}{{G_{\brho_0,\bsigma_0}^{\scalebox{0.65}{\emph{($\infty$)}}}(\cdot|H_0)}}

\newcommand{\prodGrhooneProb}{{G_{\brho_0,\bone}^{\scalebox{0.65}{($n$)}}}}

\newcommand{\prodGProbaltgenn}{{G^{\scalebox{0.65}{\emph{$(n)$}}}}}

\newcommand{\iprodGzexpand}{{G_{\bsigma_0,\bmu_0}^{\scalebox{0.65}{\emph{($\infty$)}}}(\cdot|H_0)}}

\newcommand{\iprodGomnew}{{G_{\bomega_0,\bsigma_0,\bmu_0}^{\scalebox{0.65}{\emph{($\infty$)}}}(\cdot|H_0)}}

\newcommand{\iprodG}{{G_\bone^{\scalebox{0.65}{\emph{($\infty$)}}}(\cdot|\theta_0)}}

\newcommand{\iprodGalt}{{G_\bone^{\scalebox{0.65}{($\infty$)}}(\cdot|\theta_0)}}

\newcommand{\prodGt}{{G_\bvartheta^{\scalebox{0.65}{($n$)}}}}

\newcommand{\bmu}{{\boldsymbol{\mu}}}
\newcommand{\bsigma}{{\boldsymbol{\sigma}}}

\newcommand{\evc}{{C_{\scalebox{0.65}{EV}}}}
\newcommand{\ba}{{\boldsymbol{a}}}
\newcommand{\bb}{{\boldsymbol{b}}}
\newcommand{\bc}{{\boldsymbol{c}}}

\newcommand{\be}{{\boldsymbol{e}}}

\newcommand{\bu}{{\boldsymbol{u}}}
\newcommand{\bv}{{\boldsymbol{v}}}
\newcommand{\bw}{{\boldsymbol{w}}}
\newcommand{\bx}{{\boldsymbol{x}}}
\newcommand{\by}{{\boldsymbol{y}}}
\newcommand{\bt}{{\boldsymbol{t}}}
\newcommand{\bz}{{\boldsymbol{z}}}
\newcommand{\bphi}{{\boldsymbol{\varphi}}}

\newcommand{\bvarTheta}{\boldsymbol{\varTheta}}
\newcommand{\Gset}{{\mathcal{G}}}
\newcommand{\bvartheta}{{\boldsymbol{\vartheta}}}
\newcommand{\sigBor}{{\mathscr{B}}}
\newcommand{\bvarphi}{{\boldsymbol{\varphi}}}
\newcommand{\resimp}{{\mathcal{R}}}
\newcommand{\bomega}{{\boldsymbol{\omega}}}
\newcommand{\brho}{{\boldsymbol{\rho}}}

\newcommand{\bM}{{\boldsymbol{M}}}

\newcommand{\bX}{{\boldsymbol{X}}}
\newcommand{\bT}{{\boldsymbol{T}}}
\newcommand{\bU}{{\boldsymbol{U}}}

\newcommand{\bY}{{\boldsymbol{Y}}}
\newcommand{\bZ}{{\boldsymbol{Z}}}
\newcommand{\bW}{{\boldsymbol{W}}}

\newcommand{\bfbeta}{{\boldsymbol{\beta}}}
\newcommand{\balpha}{{\boldsymbol{\alpha}}}

\newcommand{\bkappa}{{\boldsymbol{\kappa}}}
\newcommand{\bxi}{{\boldsymbol{\xi}}}

\def\bzero{{\bf 0}}
\def\binf{{\boldsymbol \infty}}
\def\bone{{\bf 1}}
\def\indic{\mathds{1}}%{{\rm 1\!I}}

\def\real{{\mathbb R}}
\def\reald{{\mathbb R}^d}

\def\prob{{\mathbb{P}}}

\def\nat{{\mathbb N}}

\def\Bset{\mathcal{B}}

\def\simp{\mathcal{S}}
\def\simpint{\mathring{\simp}}

\def\sreal{\mathcal{X}}

\def\allpart{\mathscr{P}}
\def\part{\mathcal{P}}
\def\dist{\mathscr{D}}
\def\kulb{\mathscr{K}}

\def\betaf{\text{Be}}
\def\dirf{\text{Dir}}

\newcommand{\R}{\mathbb{R}}

\newcommand{\diff}{\mathrm{d}}

\def\MDA{\mathcal{D}}

\newcommand{\BQN}{\begin{eqnarray}}
\newcommand{\EQN}{\end{eqnarray}}
\newcommand{\BQNY}{\begin{eqnarray*}}
\newcommand{\EQNY}{\end{eqnarray*}}

\newcommand{\BS}{\begin{sat}}
\newcommand{\ES}{\end{sat}}
\newcommand{\BT}{\begin{theo}}
\newcommand{\ET}{\end{theo}}
\newcommand{\BK}{\begin{korr}}
\newcommand{\EK}{\end{korr}}

\newcommand{\BD}{\begin{de}}
\newcommand{\ED}{\end{de}}

\newcommand{\BIT}{\begin{itemize}}
\newcommand{\EIT}{\end{itemize}}
\newcommand{\BDI}{\begin{description}}
\newcommand{\EDI}{\end{description}}

\newcommand{\BRM}{\begin{remarks}}
\newcommand{\ERM}{\end{remarks}}

\newcommand{\BEL}{\begin{lem}}
\newcommand{\EEL}{\end{lem}}

\def\Aset{\mathcal{A}}
\def\Hset{\mathcal{H}}

\newcommand*\xbar[1]{%
   \hbox{%
     \vbox{%
       \hrule height 0.5pt % The actual bar
       \kern0.5ex%         % Distance between bar and symbol
       \hbox{%
         \kern-0.1em%      % Shortening on the left side
         \ensuremath{#1}%
         \kern-0.1em%      % Shortening on the right side
       }%
     }%
   }%
} 
\begin{document}

\begin{frontmatter}
\title{Consistency of Bayesian Inference for Multivariate Max-Stable Distributions}
%\title{A sample article title with some additional note\thanksref{t1}}
\runtitle{Consistency of Bayesian Inference for Max-stable Distributions}
\runauthor{Padoan and Rizzelli}
%\thankstext{T1}{A sample additional note to the title.}

\begin{aug}
	%%%%%%%%%%%%%%%%%%%%%%%%%%%%%%%%%%%%%%%%%%%%%%
	%%Only one address is permitted per author. %%
	%%Only division, organization and e-mail is %%
	%%included in the address.                  %%
	%%Additional information can be included in %%
	%%the Acknowledgments section if necessary. %%
	%%%%%%%%%%%%%%%%%%%%%%%%%%%%%%%%%%%%%%%%%%%%%%
	\author[A]{\fnms{Simone A.} \snm{Padoan}\ead[label=e1]{simone.padoan@unibocconi.it}}
	\and
	\author[B]{\fnms{Stefano} \snm{Rizzelli}\ead[label=e2]{stefano.rizzelli@epfl.ch}}
	%%%%%%%%%%%%%%%%%%%%%%%%%%%%%%%%%%%%%%%%%%%%%%
	%% Addresses                                %%
	%%%%%%%%%%%%%%%%%%%%%%%%%%%%%%%%%%%%%%%%%%%%%%
	\address[A]{Department of Decision Sciences,Bocconi University, via Roentgen 1, 20136 Milan, Italy,
		\printead{e1}}
%	{mailto:simone.padoan@unibocconi.it}{simone.padoan@unibocconi.it}
	
	\address[B]{Chair of Statistics STAT, EPFL, CH-1015 Lausanne, Switzerland,
		\printead{e2}}
\end{aug}

\begin{abstract}
Predicting extreme events is important in many applications in risk analysis. 
The extreme-value theory suggests modelling extremes by max-stable distributions.
The Bayesian approach provides a natural framework for statistical prediction.
Although various Bayesian inferential procedures 
have been proposed in the literature of univariate extremes and some for multivariate extremes, the study of their asymptotic properties has been left largely untouched. 
In this paper we focus on a semiparatric Bayesian method for estimating max-stable distributions in arbitrary dimension. We establish consistency of the pertaining posterior distributions for fairly general, well-specified max-stable models, whose margins can be short-, light- or heavy-tailed.
We then extend our consistency results to the case where the data come from a distribution lying in a neighbourhood of a max-stable one, which represents the most realistic inferential setting.
\end{abstract}



\end{frontmatter}
%%%%%%%%%%%%%%%%%%%%%%%%%%%%%%%%%%%%%%%%%%%%%%
%% Please use \tableofcontents for articles %%
%% with 50 pages and more                   %%
%%%%%%%%%%%%%%%%%%%%%%%%%%%%%%%%%%%%%%%%%%%%%%
%\tableofcontents

%%%%%%%%%%%%%%%%%%%%%%%%%%%%%%%%%%%%%%%%%%%%%%%%%%%%%%%%%%%%%%
\section{Introduction}\label{sec:intro}
%%%%%%%%%%%%%%%%%%%%%%%%%%%%%%%%%%%%%%%%%%%%%%%%%%%%%%%%%%%%%%

%
Predicting the extremes of multiple variables is important in many applied fields for risk management. For instance, when designing bridges in civil engineering it is crucial to quantify what forces they must sustain in the future, e.g. the maximum wind speed, maximum river level, etc. \citep[e.g.][Ch. 9.3]{r1}.
In finance, the solubility of an investment is influenced by extreme fluctuations in the financial market affecting multiple assets, such as share prices, market indexes, currency values, etc. \citep[e.g.,][]{r3}.
The extreme-value theory encompasses several approaches for modelling multivariate extremes \citep[e.g.][]{r2}.
In this paper we focus on the family of max-stable models, which arises as a class of asymptotic distributions for linearly normalised componentwise maxima of random vectors \citep[][Ch. 4]{r2}.
Max-stable models have been successfully applied in several  areas, e.g. in meteorological, environmental, medical and actuarial studies for analysing heavy rainfall, extreme temperatures, air pollution, clinical trials, insurance claims, etc. \citep[e.g. ][]{r4,r5}, in addition to those previously mentioned.
In recent years, the popularity of some max-stable models is due to max-stable processes, which have been widely used in spatial applications \citep[e.g.][]{r8,r9,r7}.

The Bayesian approach provides a natural framework for statistical prediction.
The study of asymptotic properties like the consistency of the posterior distribution of the parameter of interest is informative for the robustness of the underlying Bayesian procedure \citep{r10}. In the last two decades the asymptotic theory of infinite-dimensional Bayesian statistics has been a very active research area. Applications to several interesting non- and semiparametric statistical problems have been proposed;
among the most recent  works, see \cite{r10,r12,r14,r11} and the references therein. 
To the best of our knowledge, to date there is no such a rigorous study concerning problems in extreme value analysis, both in the univariate and multivariate context. This article fills such a gap, providing the first results on posterior consistency of  non- and semiparametric Bayesian inference for multivariate max-stable models.

The Bayesian literature for univariate extremes includes several methodological and applied contributions  \citep[e.g. see][]{r17, r15, r5, r16}, while only few works address the analysis of multivariate extremes. There are two main reasons for the slow progress in the multivariate context.
The first motivation is that multivariate max-stable distributions define a complex, infinite-dimensional model class. Their (extreme-value) copula cannot be fully characterised using parametric families \citep[e.g.][Ch. 9.2]{r17} but rather depends on an infinite-dimensional parameter, the  so-called angular probability measure, which is subject to specific mean constraints. A special transform (reparametrisation) of such a probability measure yields the well-known Pickands dependence function, complying in turn with specific shape restrictions. It is also commonly used to summarise the dependence level, as it is simple to interpret \citep[e.g.][]{r19}. 
Proper estimation of such functions, accounting for their specificities, is not a simple task. Quite sophisticated non- and semiparametric estimation methods based on polynomials and splines have been proposed for inferring the dependence structure under both parametrisations \citep[e.g.,][]{r22, r39, r21,  r23, r216}.
In particular, \citep{r24} propose a fully nonparametric Bayesian estimation method for bivariate
max-stable distributions, where both dependence parametrisations are simultneously dealt with by means of  Bernstein polynomial representations. 

The second motivation is that the expression of the likelihood function is complicated and computationally burdensome to calculate in practice \citep[e.g.][]{r25}.
Accordingly, in high dimensions the statistical inference is often performed using a composite-likelihood approach 
\cite[see][]{r26} and the development of efficient full-likelihood estimation methods still represents an active research area \citep[e.g.,][]{r25, r28, r30}.
Notably, \cite{r252}
have been able to derive a Bayesian inferential method based on the full-likelihood for fitting max-stable distributions to the data of arbitrary dimensions (greater than two). They also establish asymptotic normality of the posterior distribution, yet under the rather restrictive assumptions that the  extreme-value copula belongs to a known parametric model and the margins are all unit-Fr\'echet.

The present paper provides several contributions. We establish the strong posterior consistency of non- and semiparametric Bayesian inferential procedures 
for max-stable distributions in arbitrary dimensions, where the a prior on the dependence structure is specified through a Bernstein polynomial representation of the angular probability measure. This estimation framework is more flexible  than that considered in \cite{r252}. In the bivariate case, we show that our strong consistency results can be extended to priors specified on the Pickands dependence function, also using an alternative representation of the latter via B-splines.

Our asymptotic results are initially derived assuming that the observable dataset is sampled from a max-stable distribution with known unit-Fr\'echet margins, as in \cite{r252}.
In practice, max-stable distributions are typically used for modelling the so-called block maxima, i.e. a vector of maxima obtained componentwise on a series of multidimensional observations of a certain length (block), e.g. yearly maxima. In this case, the use of max-stable models is only asymptotically justified, under regularity conditions, for increasingly large block sizes.
Each univariate sequence of maxima (suitably normalised) must approximatively follow one of the following three types of distributions: (reverse) Weibull (short-tailed), Gumbel (ligh-tailed) or Fr\'chet (heavy-tailed). Accordingly, as a first step we extend our posterior consistency results to well-specified max-stable statistical models whose margins are all short- or light- or heavy-tailed distributions.
Typically, the assumption of marginal distributions with such a homogeous tail behaviour entails no significant practical restriction.
%
%The assumption that the marginal distributions are of the same type is not restriction in practice. 
For example, in several environmental applications physical phenomena are well described by short- or light-tailed distributions,  due 
to natural constraints \citep[e.g.][]{r1, r4}. Instead, heavy-tailed distributions 
are found to represent quite well the tail structure of many actuarial and financial data examples \citep[e.g.][]{r17, r4, r700}.

However, block maxima only approximatively follow a max-stable distribution and formally addressing this point requires to go beyond the paradigm of a bonafide Bayesian procedure, where data are exactly sampled from a distribution in the considered model class. 
Thus, in a second step, we provide conditions under which a  pseudo Bayesian procedure using data that come from a distribution in a neighbourhood of a max-stable one is mathematically justified, guaranteeing the consistent estimation of the true data generating distribution. In doing this we define a new hybrid-Bayesian approach, where data-dependent priors are specified in an empirical Bayes fashion \cite[e.g.,][]{r501}. Their use turns out to be essential to adapt classical asymptotic arguments on posterior consistency to the present nonstandard framework. Our asymptotic results are derived by leveraging on the recent theory of remote contiguity  \citep{r12, r32}, which draws a link between the limiting observational model and the actual joint probability law of the data sample. The techical tools developed in this work %that is a novelty in the Bayesian domain which 
can be also of independent interest, beyond the extreme values context,  and adapted to other statistical methods affected by a model convergence bias. 

The reminder of this paper is organised as follows. We start providing the necessary background (Section \ref{sec:MGEV}) through: the introduction of notations used throughout the paper (Section \ref{sec:notation}), a brief review of the theory on max-stable distributions (Sections \ref{sec:general_theory}, \ref{sec:extreme_dep}), a concise description of their dependence structure's representation via Bernstein polynomials (Section \ref{sec:extreme_dep_poly}). The basic asymptotic theory is developed in Section \ref{sec:bayesian_inference}. After a short introduction on the Bayesian paradigm for well-specified max-stable models (Section \ref{sec:bayesian_paradigm}) consistency results are firstly established for the class of so-called simple max-stable distributions (Section \ref{sec:binf_simple_max}). These are then extended to more general families of  max-stable distributions (Section \ref{sec:binf_general_max}). Finally, we refine our asymptotic theory to account for more realistic sampling schemes, where the data come from a distribution lying in a neighbourhood of a max-stable one (Section \ref{sec:binf_sample_max}). We conclude with a discussion (Section \ref{sec:discussion}). 
The Supplementary Material
document provided with this article offers additional theoretical findings, along with a series of auxiliary lemmas, techical details on the presented examples and all the proofs of our main results.

%%%%%%%%%%%%%%%%%%%%%%%%%%%%%%%%%%%%%%%%%%%%%%%%%%%%%%%%%%%%%%
\section{Background}\label{sec:MGEV}
%%%%%%%%%%%%%%%%%%%%%%%%%%%%%%%%%%%%%%%%%%%%%%%%%%%%%%%%%%%%%%

In this section, we report general notations used throughout the paper and review some results on max-stable models. The latter  provide the mathematical and probabilistic background for our main theoretical findings, presented in the next two sections. 

%%%%%%%%%%%%%%%%%%%%%%%%%%%%%%%%%%%%%%%%%%%%%%%%%%%%%%%%%%%%%%
\subsection{Notation}\label{sec:notation}
%%%%%%%%%%%%%%%%%%%%%%%%%%%%%%%%%%%%%%%%%%%%%%%%%%%%%%%%%%%%%%

Given $\sreal\subset \real^d$, with  $d\in \mathbb{N}_+=\{1,2,\ldots \}$, 
and $f:\sreal \rightarrow \real$, let $\|f\|_{\infty}=\sup_{\bx \in \sreal} |f(\bx)|$ and
$\|f\|_1=\int_{\sreal}|f(\bx)|\diff \bx$.
Let $I=(i_1, \ldots,i_k) \subset\{1, \ldots,d\}^d$ with $1 \leq k \leq d$ and $\bx_I:=(x_i, i \in I)$. For a differentiable function $f$ at the point $\bx_0\in\sreal$, we denote by $f_{I}(\bx_0)$ its partial derivative of order $k$ with respect to $\bx_I$.
If $d=1$ and $f$ is nondecreasing, we also denote by $f^\leftarrow(t)=\inf\{ x \in \sreal: f(x) \geq t \}$ the left generalised inverse, with $t \in \real$.

Le $F$ and $G$ be two probability measures (pm's) on a generic measurable space $\mathbb{X}:=(\mathcal{X}, \sigma(\sreal))$. When $F$ and $G$ are absolutely continuous with respect to a measure $\nu$ on $\mathbb{X}$, with
density functions $f$ and $g$, $\kulb(f,g)= \int_\sreal (\log f/g) f \diff \nu $ and $\dist_{H}^2(f,g)=\int_\sreal (\sqrt{f}-\sqrt{g})^2 \diff \nu$ are the Kullback-Leibler divergence and 
the squared Hellinger distance, respectively. When $\sreal$ is a separable metric space and $\sigma(\sreal)$ is its Borel $\sigma$-algebra, 
$\dist_{T}(F,G)=\sup_{B \in \sigma(\sreal)}|F(B)-G(B)|$ is the total variation distance, while  $\dist_W(F,G)$ denotes a metric between pm's that metrizes the topology of weak convergence \citep[e.g.,][p. 508]{r10}. When $\sreal$ is an ordered subset of $\reald$, for simplicity, given a pm $F$ we also denote with $F$ the corresponding (cumulative) distribution function.
We then denote by $\dist_{KS}(F,G)=\sup_{\bx \in \sreal}|F(\bx)-G(\bx)|$ the Kolmogorov-Smirnov distance \citep[e.g.,][Ch 3.1.1]{r10}.
Moreover, we denote by $\prodF$ and $\iprodF$ the $n$-fold and infinite-fold product pm's. Finally, $\delta_\bx(\cdot)$ is the Dyrac-delta pm at $\bx \in \sreal$.

%%%%%%%%%%%%%%%%%%%%%%%%%%%%%%%%%%%%%%%%%%%%%%%%%%%%%%%%%%%%%%
\subsection{Max-stable distributions}\label{sec:general_theory}
%%%%%%%%%%%%%%%%%%%%%%%%%%%%%%%%%%%%%%%%%%%%%%%%%%%%%%%%%%%%%%

Let $\bZ=(Z_1,\ldots,Z_d)$ be a random vector (rv) with joint distribution $F$ and $\bZ_1,\bZ_2,\ldots$ be independent and identically distributed (iid) copies of it. Hereafter, operations between vectors are meant componentwise.
We say that $F$ is in the max-domain of attraction of $G$, in symbols $F\in\MDA(G)$, if there exist norming sequences $\ba_m>\bzero=(0,\ldots,0)$ and $\bb_m\in\real^d$, $m=1,2,\ldots$, such that
\begin{equation*}\label{eq:doa}
\lim_{m\to\infty} F^m(\ba_m \bx + \bb_m)=G(\bx), \quad \bx\in\real^d.
\end{equation*}
The distribution $G$ satisfies the max-stability property: $G(\bx)=G^k(\balpha_k\bx+\bfbeta_k)$ for some $\balpha_k>\bzero$, $\bfbeta_k\in\real^d$ and all $k=1,2,\ldots$ \citep[e.g.,][p. 143]{r2}. Its form is
\begin{equation*}\label{eq:bevd}
G(\bx) = \evc \left( G_{1}(x_1), \ldots, G_{d}(x_d) \right), \qquad \bx \in \real^d,
\end{equation*}
where $\evc$ is an \textit{extreme-value copula}, allowing the representation
\begin{equation}\label{eq:evcopula}
\evc(\bu) = \exp \left[ - L\left\{  (- \ln u_1), \ldots, (- \ln u_d)\right\} \right], 
\quad \bu \in (0, 1]^d,
\end{equation}
where $L: [0,\infty)^d \mapsto [0,\infty)$ is a homogeneous function of order $1$ named {\it stable-tail} dependence function. 
We refer to Section \ref{sec:extreme_dep} for additional details on the dependence structure.
For any $j\in\{1,\ldots,d\}$, the $j$-th margin of $G$ is of one of the following three types:
\begin{equation}\label{eq:gev}
G_j(x)=\begin{cases}
\exp(-x^{-\rho_j}), \hspace{2.8em} x>0, \, \rho_j>0,\\
\exp(-\exp(-x)), \quad x \in \R,\\
\exp(-(-x)^{\omega_j}), \hspace{1.8em} x <0, \, \omega_j>0,
\end{cases}
\end{equation}
known as the $\rho_j$-Fr\'echet, the Gumbel  and the (reverse) $\omega_j$-Weibull distribution function, respectively. Their location/scale version is readily obtained as: $G_{\rho_j, \sigma_j}(x):=\exp(-(x/\sigma_j)^{-\rho_j})$, for $\sigma_j>0$; $G_{\mu_j,\sigma_j}(x):=\exp(-\exp(-(x-\mu_j)/\sigma_j))$ for $\mu_j\in\R$, $\sigma_j>0$; $G_{\omega_j,\sigma_j,\mu_j}(x):=\exp(-(-(x-\mu_j)/\sigma_j)^{\omega_j})$ for $x<\mu_j$, $\mu_j\in\R$, $\sigma_j>0$.
In the sequel, we focus on classes of absolutely continuous multivariate max-stable distributions having margins all of the same type, denoted generically by $\{G_\bvartheta, \, \bvartheta \in \bvarTheta\}$. We refer to the specific classes as $\brho$-Fr\'echet ($\bvartheta=(\brho,\bsigma)\in\bvarTheta=(0, \infty)^{2d}$), $\bomega$-Weibull ($\bvartheta=(\bomega,\bsigma,\bmu)\in\bvarTheta\subset (0, \infty)^{2d}\times \reald$) and Gumbel $(\bvartheta=(\bsigma, \bmu)\in\bvarTheta=(0,\infty)^d\times \reald$) multivariate max-stable distributions. We denote by $\bvartheta_j$ the parameters of the $j$-th margin $G_{\bvartheta_j}$ of $G_\bvartheta$, with parametric space $\bvarTheta_j$.

The density functions of the three classes are related as follows. Let $\bX=(X_1,\ldots,X_d)$ be a rv with distribution $G_\bvartheta$ and
define $Y_j=U_{\bvartheta_j}(X_j)$, $j\in\{1,\ldots,d\}$, where
\begin{equation*}
U_{\bvartheta_j}(x_j)=-1/\log G_{\bvartheta_j}(x_j),
\end{equation*}
with $x_j \in \text{supp}(G_{\bvartheta_j})$. Then, $\bY=(Y_1,\ldots,Y_d)$ follows 
a max-stable distribution with common $1$-Fr\'{e}chet margins. We refer to the latter as multivariate $\bone$-Fr\'echet-max-stable or \textit{simple} max-stable distribution and denote it by $G_{\bone}$.
Specifically, for $\by>\bzero$,
$G_{\bone}(\by)=\exp(-V(\by))$, where $V(\by)=L(1/\by)$ is known as the {\it exponent} function and is hereafter assumed to have partial mixed derivatives up to order $d$ at almost every $\by>\bzero$.
Thus, the multivariate simple max-stable density function is given by the Fa\'{a} di Bruno's formula 
\begin{equation}\label{eq:density_unit_fre}
g_{\bone}(\by)
 = \sum_{\part\in\allpart_d}G_{\bone}(\by) \prod_{i=1}^{m}\{-V_{I_i}(\by)\},
\end{equation}
where $\allpart_d$ is the set of all the partitions $\part=\{I_1,\ldots,I_m\}$ of $\{1,\ldots,d\}$, with
$m=|\part|$. 
As a result, the multivariate max-stable density density of $G_\bvartheta$ is
\begin{equation}\label{eq:gen_dens}
g_{\bvartheta}(\bx)=  U_\bvartheta'(\bx)
g_{\bone}(U_\bvartheta(\bx)),
\end{equation}
where $U_\bvartheta(\bx)=(U_{\bvartheta_1}(x_1), \ldots, U_{\bvartheta_d}(x_d))$ and $U_\bvartheta'(\bx)=\prod_{j=1}^d U_{\bvartheta_j}'(x_j)$.

%%%%%%%%%%%%%%%%%%%%%%%%%%%%%%%%%%%%%%%%%%%%%%%%%%%%%%%%%%%%%%
\subsection{Extremal dependence}\label{sec:extreme_dep}
%%%%%%%%%%%%%%%%%%%%%%%%%%%%%%%%%%%%%%%%%%%%%%%%%%%%%%%%%%%%%%

The extreme-value copula is fully characterised by the stable-tail dependence function, which in turn is by the  {\it Pickands} dependence function, defined on $\resimp:=\{\bt\in[0,1]^{d-1}: \Vert \bt \Vert_1 \leq1\}$ as $A(\bt)=L(1-t_1-\dots-t_{d-1}, t_1, \ldots, t_{d-1})$. Thus,
%
%via the relation $L(1-\Vert \bt \Vert_1, \bt)=A(\bt)$. Moreover,
%{\color{magenta} which in turn is by its restriction on the $d$-dimensional unit simplex $\simp=\{\bw\geq\bzero: \Vert \bw \Vert_1=1\}$, known as {\it Pickands} dependence function. The latter is defined on $\resimp:=\{\bt\in[0,1]^{d-1}: \Vert \bt \Vert_1 \leq1\}$ and allows the following representation}
%
\begin{equation}\label{eq:picklands}
A(\bt)=d\int_{\simp} \max\{(1-t_1-\cdots-t_{d-1})w_1, \ldots,  t_{d-1}w_d\} \diff H( \bw),
\end{equation}
where $H$ is a pm on the $d$-dimensional unit simplex  $\simp:=\{\bw\geq\bzero: \Vert \bw \Vert_1=1\}$, named \textit{angular pm}, and satisfies the mean constraints 
%$\int_{\simp_d} w_j H(\diff \bw)=1/d$ for all $j=1,\ldots,d$.
%
%
\begin{enumerate}
\item[(C1)] $\int_{\simp} w_j H(\diff \bw)=1/d,\; \forall\;j\in\{1,\ldots,d\}$.
\end{enumerate}
As a result of \eqref{eq:picklands}, the Pickands dependence function satisfies the convexity and boundary constraints
\begin{enumerate}
\item[(C2)] $A(a\bt_1+(1-a)\bt_2)\leq aA(\bt_1)+(1-a)A(\bt_2), \,a\in[0,1], \,\forall\,\bt_1,\bt_2\in\resimp$,
\item[(C3)] $1/d\leq \max\left(t_1,\ldots,t_{d-1},1-t_1-\cdots-t_{d-1} \right) \leq A(\bt) \leq 1, \,\forall\,\bt\in\resimp$.
\end{enumerate}
These are necessary and sufficient conditions to characterise the class of valid Pickands dependence functions in the case $d=2$, while they are only necessary but not sufficient when $d>2$,  see e.g. \citep[][p. 257]{r17} for a counter example.

In general, $H$ can place mass on all the $2^d-1$ subspaces of $\simp$ of the form
$\simp_{I}=\{\bv\in\simp: v_j>0 \text{ if } j\in I; v_j=0 \text{ if } j \notin I\}$, with $I$ a non-empty subset of $\{1,\ldots,d\}$ \citep[][Ch. 7]{r17}. Notice that $\simp_{\{j\}}=\{\be_j\}$, $j=1, \ldots,d$, where $\be_j$ is the $j$-th canonical basis vector, and $\simp_{\{1,\ldots,d\}}$ corresponds to the interior of $\simp$.
In the sequel, we focus on the subset of all the possible angular pm's given in Definition \ref{cond_angular}. Such a class makes statistical inference not too complicated and, at the same time, is sufficiently rich for applications.
\begin{definition}\label{cond_angular}
Let $\Hset$ denote the class of pm's on the Borel sets of $\simp$ satisfying (C.1) and having null mass outside the subset
$\tilde{\simp}:=\simpint \cup \{\be_1\}\cup\cdots \cup\{\be_d\}$, where $\simpint=\{\bv\in\simp:0<v_j<1, j=1,\ldots,d\}$. 
For any $H \in\Hset$ there are point masses
$p_j\in[0,1/d]$,  $j=1,\ldots,d$, and a Lebesgue integrable function
$h: \mathring{\resimp}\mapsto[0, \infty)$,  $\mathring{\resimp}:=\{\bv\in(0,1)^{d-1}: \Vert \bv \Vert_1 <1\}$, named \textit{angular density}, such that for all Borel subsets $B \subset \simp$ we have
\begin{equation}\label{eq:angpm}
H(B)=\sum_{j=1}^d p_j \delta_{\be_j}(B)+\int_{\pi_{\resimp}(B\cap \simpint)} h(\bv)\diff \bv,  
\end{equation}
\begin{equation*}\label{eq:Bset}
\text{with}\quad \pi_{\resimp}:\simp \mapsto \resimp: (w_1, \ldots, w_{d-1},w_d)\mapsto(w_1, \ldots, w_{d-1}).
\end{equation*}
$H(\bt)$, $\bt\in\resimp$, is  the distribution function pertaining to  the pm $H\circ \pi_{\resimp}$, called \textit{angular distribution}.
Let $\Aset$ be the set of functions on $\resimp$ 
defined through the representation in \eqref{eq:picklands}, with $H \in \Hset$. Hereafter, we refer to $\Hset$ and $\Aset$ as the spaces of valid angular pm's and Pickands functions, respectively. 
\end{definition}
%
%For any $H \in\Hset$ there are point masses
%$p_j\in[0,1/d]$,  $j=1,\ldots,d$, and a Lebesgue integrable function
%$h: \mathring{\resimp}\mapsto[0, \infty)$,  $\mathring{\resimp}:=\{\bv\in(0,1)^{d-1}: \Vert \bv \Vert_1 <1\}$, named \textit{angular density}, such that for all Borel subsets $B \subset \simp$ we have
%
%\begin{equation}\label{eq:angpm}
%H(B)=\sum_{j=1}^d p_j \delta_{\be_j}(B)+\int_{\pi_{\resimp}(B\cap \simpint)} h(\bv)\diff \bv,  
%\end{equation}
% 
%\begin{equation*}\label{eq:Bset}
%\text{with}\quad \pi_{\resimp}:\simp \mapsto \resimp: (w_1, \ldots, w_{d-1},w_d)\mapsto(w_1, \ldots, w_{d-1}).
%\end{equation*}
%
%We also denote by $H(\bt)$, $\bt\in\resimp$, the distribution function pertaining to  the pm $H\circ \pi_{\resimp}$
%with $\pi_\resimp$ as in \eqref{eq:angpm}, 
%and refer to it as the \textit{angular distribution}.
%
The dependence level among the components of a max-stable rv can be described by means of a geometric interpretation of  the angular pm.  The more the mass of $H$ concentrates around $(1/d, \ldots, 1/d)$ (the barycenter of the simplex) the more the variables are dependent on each other. On the contrary, the more the mass of $H$ accumulates close to the vertices of the simplex, the less dependent the variables are. Alternatively, the dependence level can be described via the Pickands dependence function $A$, as it satisfies the inequality in (C3),
% $\max(t_1,\ldots,t_{d-1},1-\cdots-t_{d-1})\leq A(\bt)\leq 1$, for all $\bt\in\resimp$,
  where the lower and upper bounds represent the cases of complete dependence and independence, respectively.

%%%%%%%%%%%%%%%%%%%%%%%%%%%%%%%%%%%%%%%%%%%%%%%%%%%%%%%%%%%%%%
\subsection{Polynomial representation of the extremal dependence}\label{sec:extreme_dep_poly}
%%%%%%%%%%%%%%%%%%%%%%%%%%%%%%%%%%%%%%%%%%%%%%%%%%%%%%%%%%%%%%

In recent years, different polynomial functions have been used to model the extremal dependence more flexibly than using specific parametric models. For example, polynomials in Bernstein form have been used to model the univariate angular distribution and Pickands dependence functions \citep[e.g.,][]{r24,r101}. Piecewise polynomials as linear combinations of B-spline have been used to model the univariate Pickands dependence function in a regression setting \citep{r22}.
Alternative polynomial representations have been considered, e.g., in \cite{r21} and \cite{r36}.
%and alternative basis functions can be legitimately used.
%
%However, Bernstein polynomials have shown to be very tractable from a theoretical and a computational viewpoint when performing Bayesian nonparametric inference \citep[e.g.][]{r213,r48,r24,r39}. 
%
In higher dimensions, the Pickands dependence function is less tractable \citep{r23}, while multivariate pm's on the simplex as the angular pm can be conveniently modelled through density functions and point masses \citep[e.g.,][]{ r211, r212, r39}.
%
%Thus, to provide a general and concise discussion, in the remainder we focus on angular pm's.
%
Bernstein polynomials have shown to be very tractable from a theoretical and a computational viewpoint when performing Bayesian nonparametric inference \citep[e.g.][]{r213,r48,r24,r39}
and allow for a representation of the multivariate angular density in terms of Dirichlet mixtures, one of the most popular model classes for angular pm's \cite{r40, r216}.
Hereafter, we focus on such an approach, aiming to a general and concise discussion.
%
%To provide a general and concise discussion, we hereafter focus on such an approach to modeling of the angular pm.
%
%In this section we review the Bernstein polynomials representation of the univariate angular distribution and Pickands dependence functions and the multivariate angular density. 
%	
In Section \ref{app:review} of the supplement, we also provide a novel characterisation of univariate angular distribution and Pickands dependence functions via B-splines, allowing for simple prior specification on the extremal dependence, along with additional details on Bernstein polynomial modelling. 
These representations are useful to construct prior distributions yielding consistent Bayesian semiparametric  procedures.

We now briefly describe angular pm modelling via a mixture of polynomial densities and point masses.
%
%{\color{magenta}We next illustrate the approach for modelling the angular pm via a mixture of polynomial densities and point masses}. 
For an integer $k>d$, let $\Gamma_k$ be the set of multi-indices $\balpha \in \{1, \ldots, k-d+1\}^d$
such that $\alpha_1 + \cdots + \alpha_{d-1} \le k-1$ and $\alpha_d=k- \alpha_1 + \cdots + \alpha_{d-1} $, whose cardinality is \citep[e.g,][]{r23}
\begin{equation*}\label{eq:p}
|\Gamma_{k}| = \binom{k-1}{d-1}.
\end{equation*}
For each $\balpha \in \Gamma_k$, the Bernstein polynomial basis function of index $\balpha-\bone$ and degree $k-d$ 
$$
b_{\balpha-\bone}(\bt;k-d)=\frac{(k-d)!}{(\alpha_1-1)! \cdots (\alpha_d-1)!} (w_1^{\alpha_1-1}\cdots w_d^{\alpha_d-1}), \quad \bt \in \mathring{\resimp},
$$
can be rewritten as 
$b_{\balpha-\bone}(\bt;k-d)=\dirf(\bt;\balpha)(k-d)!/(k-1)!$ for all $\bt \in \mathring{\resimp}$, where $\dirf(\bt;\balpha)$ denotes the 
Dirichlet probability density with parameters $\balpha$ \citep[e.g.,][]{r230}. 
Therefore, a $(k-d)$-th degree Bernstein polynomial representation of the angular density is given by
\begin{equation}\label{eq:bpoly_density}
h_{k-d}(\bt)=\sum_{\balpha\in \Gamma_{k}}\varphi_{\balpha}\dirf(\bt;\balpha),\quad \bt\in\mathring{\resimp},
\end{equation}
where $\varphi_{\balpha} \in [0,1]$ for any $\balpha\in\Gamma_k$.
Set $\bkappa_j=k\be_j$, $j\in\{1,\ldots,d\}$.
%where $\be_j$.
% is the $j$-th canonical unit vector.
According to \cite{r39}, the following pm on the Borel subsets $B \subset \simp$
\begin{equation}\label{eq: BPoly measure}
H_k(B):=\sum_{ j =1}^d \delta_{\be_j}(B)\varphi_{\bkappa_j}+\int_{\pi_{\resimp}(B\cap \simpint)}h_{k-d}(\bt)\diff \bt,
\end{equation}
is a valid angular pm if and only if the non-negative coefficients 
$$
\bvarphi^{(k)}=(\varphi_{\bkappa_1},\ldots,\varphi_{\bkappa_d}, \varphi_{\balpha}, \, \balpha\in\Gamma_k)
$$ 
satisfy the restrictions:
\begin{enumerate}
\item[(R1)] $\sum_{\balpha\in \Gamma_{k}} \varphi_{\balpha} = 1-\varphi_{\bkappa_1}-\cdots-\varphi_{\bkappa_d};$
\item[(R2)] $\sum_{l=1}^{k-1}\frac{l}{k} \sum_{\balpha \in \Gamma_k; \alpha_j=l}\varphi_\balpha=\frac{1}{d}-\varphi_{\bkappa_j}$, $ \forall\,j=1, \ldots,d$.
\end{enumerate}
\noindent
Thus, we define the classes of angular pm's with density in Bernstein Polynomial (BP) form via
\begin{equation*}
\mathcal{H}_k:=\{
H_k\in \Hset: H(B)=\sum_{j=1}^d \delta_{\be_j}(B)\varphi_{\bkappa_j}+\int_{\pi_{\resimp}(B\cap \simpint)}h_{k-d}(\bt)\diff \bt, \,\text{(R1)-(R2) hold true}
\},
\end{equation*}
for each integer $k>d$.
We have the following approximation property, see Section \ref{appsec:proofpoly} of the supplement for a proof.
\begin{prop}\label{prop: top supp}
For every $H \in \Hset$ and any $\epsilon>0$, there exists $k >d$ and $H_k \in \Hset_k$
such that $\Vert h-h_{k-d} \Vert_1 < \epsilon$.
\end{prop}
As a result, $\cup_{k= d+1}^\infty\mathcal{H}_k$ is a dense subset of  $(\Hset, \dist_{T})$ and, therefore, of $(\Hset, \dist_{W})$ and $(\Hset,\dist_{KS})$, where $\dist_{KS}$ metrizes the space of angular pm's through the uniform distance (over $\mathcal{R}$) between their angular distribution functions.
This guarantees
that a prior on the angular pm suitably specified via the Bernstein polynomial
representation in \eqref{eq:bpoly_density} has full support.
Proposition \ref{prop: top supp} further gives the mathematical ground for devising full support priors on $d$-variate max-stable densities and
consistent Bayesian predictive methods. 

%
%%%%%%%%%%%%%%%%%%%%%%%%%%%%%%%
\section{Bayesian inference}\label{sec:bayesian_inference}
%%%%%%%%%%%%%%%%%%%%%%%%%%%%%%%
%

%
%%%%%%%%%%%%%%%%%%%%%%%%%%%%%%%
\subsection{Bayesian paradigm for max-stable models}\label{sec:bayesian_paradigm}
%%%%%%%%%%%%%%%%%%%%%%%%%%%%%%%
%

%
In this section, we firstly consider a simple max-stable observational model and a Bayesian statistical setting that can be described in the following general terms.
We assume that the observables are iid rv's $\bY_1,\ldots,\bY_n$ in $(0, \infty)^d$ following a simple max-stable distribution $G_{\bone}\equiv G_{\bone}(\cdot|\theta_0)$.
In particular, $\theta_0$ is an unknown infinite-dimensional parameter corresponding either to the true Pickands dependence function $A_0$ or to the true angular pm $H_0$. Observe that, due to the $1$-to-$1$ relation between $A_0$ and $H_0$, these two yield equivalent representations of the dependence structure.
We thus consider a statistical model of the form $\{\prodG (\cdot|\theta): \, \theta \in \Theta\}$, where $\Theta$ equals either $\Aset$ or $\Hset$.
We endow $\Theta$ with a metric $\dist$ that makes it separable and such that, equipping $\Theta$ with the associated borel $\sigma$-field $\sigBor_\Theta$, $(\Theta, \sigBor_{\Theta})$ is a  standard Borel measurable space \cite[e.g.][p. 96]{r121}. Thus, for every Borel set $B$, $\Pi(B)=\prob(\theta\in B)$ is a prior distribution. 
%
%A prior distribution  $\Pi(B)=\prob(\theta\in B)$ is thus specified on the Borel sets $B$. 
%
%By Theorem V.58 in \cite{r45}, there exists a version of the probability density $g_\bone(\by|\theta)$, which is jointly measurable in the parameter and the observation, and which equals the expression in \eqref{eq:density_unit_fre} for almost every $\by \in (0, \infty)^d$.
%
By Theorem V.58 in \cite{r45}, there exists a version of the probability density $g_\bone(\by|\theta)$ (see \eqref{eq:density_unit_fre}), which is jointly measurable in the parameter and the observation for almost every $\by \in (0, \infty)^d$.
Thus,  a version of the posterior distribution is given by Bayes' theorem 
\begin{equation}\label{eq:post}
\Pi_n(B):=\Pi_{\Theta}(B|\bY_{1:n})=\frac{\int_B \prod_{i=1}^n g_\bone(\bY_i|\theta) \diff\Pi_{\Theta }(\theta)}{\int_\Theta \prod_{i=1}^n g_\bone(\bY_i|\theta) \diff\Pi_{\Theta }( \theta)},
\end{equation}
where $\bY_{1:n}=(\bY_1,\ldots,\bY_n)$. In particular, such a version of the conditional distribution of $\theta$ given $\bY_{1:n}$ is regular: i.e., a Markov kernel from the sample space of data into $(\Theta,\sigBor_\Theta)$. Moreover, defining 
\begin{equation}\label{eq:dens_space}
\Gset_{\bone}:=\{g_{\bone}(\cdot|\theta):\, \theta \in \Theta \},
\end{equation}
the map 
$
\phi_\Theta:(\Theta, \dist)\mapsto (\Gset_{\bone}, \dist_H):\theta \mapsto g_\bone(\cdot|\theta)
$
is Borel. Therefore,
the prior $\Pi_\Theta$ on the dependence parameter induces a prior $\Pi_{\Gset_{\bone}}=\Pi_\Theta\circ \phi_\Theta^{-1}$ on the Borel sets of $(\Gset_{\bone}, \dist_H)$, whose posterior distribution is the random pm $\tilde{\Pi}_n=\Pi_n\circ\phi_\Theta^{-1}$. 
In this setup, almost sure (as) convergence of the posterior distributions $\Pi_n$ and $\tilde{\Pi}_n$ to the Dyrac pm's $\delta_{\theta_0}$ and $\delta_{g_\bone(\cdot|\theta_0)}$ is equivalent to the following fact \citep[e.g.,][Ch 6]{r10}.
\begin{definition}\label{def:cons}
The posterior distributions $\Pi_n(\cdot)$ and $\tilde{\Pi}_n$ are strongly consistent at $\theta_0$ and $g_\bone(\cdot|\theta_0)$, respectively, if for all the neighbourhoods $\theta_0 \in\mathcal{U}\subset \Theta$ and $g_\bone(\cdot|\theta_0) \in \tilde{\mathcal{U}}\subset \Gset_{\bone}$
\[
\lim_{n\to\infty}\Pi_n(\mathcal{U}^\complement)\to 0, \quad \lim_{n\to\infty}\tilde{\Pi}_n(\tilde{\mathcal{U}}^\complement)\to 0, \quad \iprodGalt-as.
\]
\end{definition}
These notions have analogous extensions to semiparametric statistical models $\{\prodGt (\cdot|\theta): \, (\theta, \bvartheta) \in \Theta \times \bvarTheta\}$, corresponding to the three types of max-stable distribution classes introduced in Section \ref{sec:general_theory}. 
%\cite[see also][]{r400}. The latter are analysed in Section \ref{sec:binf_general_max}, where we also emphasize the implications of Hellinger consistency at the density level for the construction of consistent predictive methods.

%
%%%%%%%%%%%%%%%%%%%%%%%%%%%%%%%%%%%%%%%%%%%%%%%%%%%%%%%%%%%%%%
\subsection{Simple max-stable distributions}\label{sec:binf_simple_max}
%%%%%%%%%%%%%%%%%%%%%%%%%%%%%%%%%%%%%%%%%%%%%%%%%%%%%%%%%%%%%%

%
In this subsection, we establish strong consistency of the posterior distributions $\Pi_n$ and $\tilde{\Pi}_n$, previously introduced for data following a simple max-stable model. 
This first step allows to derive most of the mathematical ground used to establish consistency also for non-simple max-stable models, which are more realistic for applications, see Section \ref{sec:binf_general_max}.
%	
%The consistency results presented in the sequel are given for prior distributions $\Pi_\Theta$ on the extremal dependence constructed via the representations in Section \ref{sec:extreme_dep_poly}.
%
The consistency results presented in the sequel concern prior distributions $\Pi_\Theta$ on the extremal dependence constructed via the representation in Section \ref{sec:extreme_dep_poly}. 
Nevertheless, they are based on a fairly general theory of Borel pm's on the space $\Hset$ and
can be adapted to alternative prior specifications.
%
%Though, they rely on a fairly general theory of Borel pm's on metric spaces of angular pm's, which also suits alternative prior specifications.
%
A similar approach tailored to Pickands dependence function in the case $d=2$, encompassing prior constructions via B-splines, is presented in the supplement.
%
%Results for the bivariate case are presented separately from those for higher dimensions. In the former setup, consistency on the extremal dependence is obtained for stronger metrics and under less restrictive conditions by capitalizing on the relations \eqref{eq:pick}-\eqref{eq:pick_2dev} in Section \ref{app:technical} of the supplement, whose high-dimensional counterparts are intractable.
%
A plethora of extremal dependence functionals used by practitioners \citep[e.g.,][Ch 2.2]{r5} are defined for pairs of variables and are readily obtainable from the Pickands dependence function. Thus, posterior consistency for those can be straightforwardly deduced from our results in Sections \ref{app:KL_support} and \ref{sec:posterior_consistency_2D} of the supplement. 

Consistency of the posterior distribution $\tilde{\Pi}_n$ on simple max-stable densities is obtained via an extended version of Schwartz{'}s theorem \citep[Theorem 6.23]{r10}. The latter yields consistency via sequential partitioning of the density space into a set whose entropy grows at most linearly in the sample size and a set of exponentially decaying prior probability. A key requirement is the Kullback-Leibler property of prior distributions defined below.
\begin{definition}\label{defi:KL}
{\upshape
We say that the prior $\Pi_\Theta$  possess the Kullback-Leibler property at $\theta_0$ (equivalently, $g_\bone(\cdot|\theta_0)$ belongs to the Kullback-Leibler support of $\Pi_{\Gset_\bone}$)
if 
\[
\Pi_{\Theta}(\theta \in \Theta: g_{\bone}(\cdot|\theta) \in \mathcal{K}_\epsilon )=\Pi_{\Gset_\bone}( \mathcal{K}_\epsilon )>0,
\]
for all $\mathcal{K}_\epsilon:=\{g \in \Gset_\bone: \, \kulb(g_\bone(\cdot|\theta_0),g)<\epsilon \}$, $\epsilon>0$.
}
\end{definition}
In Section \ref{sec:KL_support} we give sufficient conditions under which a generic Borel prior on the angular pm possesses the Kullback-Leibler property.
As for consistency of its posterior $\Pi_n$, results under different metrics can be deduced from Hellinger consistency at the density level. As argued in the next subsection, we  consider metrizations
%these spaces
which adhere to the framework of Section \ref{sec:bayesian_paradigm}, therefore guaranteeing the existence of regular versions of the posterior distributions $\Pi_n$ and $\tilde{\Pi}_n$. This is a necessary requirement to legitimately study their consistency \citep[e.g.,][Ch 1.3]{r10}. Analogous considerations extend to the case where a prior distristibution is assigned to the Pickands dependence function (Sections \ref{app:KL_support}--\ref{sec:posterior_consistency_2D} of the supplement).

%%%%%%%%%%%%%%%%%%%%%%%%%%%%%%%%%%%%%%%%%%%%%%%%%%%%%%%%%%%%%%
\subsubsection{Kullback-Leibler theory for priors on the extremal dependence}\label{sec:KL_support}
%%%%%%%%%%%%%%%%%%%%%%%%%%%%%%%%%%%%%%%%%%%%%%%%%%%%%%%%%%%%%%

%We focus on prior distributions on the angular pm and precisely, we deal with prior distributions
%$\Pi_{\Hset}(B)=\prob(H\in B)$  on the Borel sets $B$ of $(\Hset, \dist_{W})$. 

We focus on prior distributions specified on the angular pm. Specifically, we consider 
prior distributions $\Pi_{\Hset}(B)=\prob(H\in B)$ on the Borel sets $B$ of $(\Hset, \dist_{W})$.

\begin{prop}\label{prop:Polish_d>2}
In arbitrary dimension $d\geq2$, the space of Borel pm's on $\tilde{\simp}$, endowed with the topology of weak convergence of measures, is separable and completely metrizable. 
Moreover, $\Hset$, endowed with the associated subspace topology, is a standard Borel space, 
whose Borel $\sigma$-field coincides with the one induced by $\dist_{KS}$. 
\end{prop}
\begin{cor}\label{cor:Polish_d>2} 
In dimension $d\geq2$, there exists a version of the simple max-stable density such that the map $(H,\by)\mapsto g_\bone(\by|H)$ is jointly measurable and $\phi_\Hset:(\Hset, \dist_{W})\mapsto (\mathcal{G}_\bone,\dist_H):H \mapsto g_\bone(\cdot|H)$ is a Borel map, where $\Gset_\bone$ is defined as in \eqref{eq:dens_space}, with $\Theta=\Hset$. Moreover, for all $\epsilon>0$ and $\mathcal{K}_\epsilon$ as in Definition \ref{defi:KL}, $\phi_\Hset^{-1}(\mathcal{K}_\epsilon)$ is a Borel set of $(\Hset, \dist_W)$. 
\end{cor}

The proofs of the above results are provided in Section \ref{appsec:proofsKLtheo} of the supplement.
Observe that $\Pi_\Hset$ induces a prior $\Pi_{\Gset_\bone}$ on the Borel sets of $(\Gset_\bone,\dist_H)$.
To inspect its Kullback-Leibler support, we introduce the following notion.
\begin{definition}\label{def:d_inf_prop}
Let $H_*\in\Hset$ and $\Vert h_* \Vert_{ \infty}<\infty$. The prior $\Pi_\Hset$ is said to posses the $\dist_{\infty}$-property at $H_*$ if it has positive inner probability on the sets $\{H\in\Hset:\dist_{\infty}(h,h_*)\leq \epsilon\}$, for all $\epsilon>0$.
\end{definition}
%
%We assume that the true angular pm has  continuous angular density on 
We consider true angular pm's with continuous angular density on
$\mathring{\resimp}$ and nonnegative point masses on the vertices of $\simp$ (Definition 
\ref{cond:mvt_angular}\ref{cond:true_ang_set}). 
%{\color{magenta}
%Such a class of admissible true angular pm's (Definition 
%\ref{cond:mvt_angular}\ref{cond:true_ang_set}) is less rich than the one we consider in the bivariate case (see Section \ref{app:KL_support} of the supplement).} 
Only in the bivariate case, we allow for %true 
angular densities %which 
that diverge at the vertices of $\resimp$.
Unbounded angular densities are less tractable in high dimensions, while nonessential for practical statical purposes when point masses are admitted. Popular models such as the Asymmetric Logistic, Dirichlet and Symmetric Logistic and Hulser-Reiss for suitable values of the parameters \citep[e.g.,][]{r19} are comprehended by the family in Definition \ref{cond:mvt_angular}\ref{cond: finite_new}. 
We establish that a sufficient condition for $\Pi_\Hset$ to possess the Kullback-Leibler property at true admissible pm's is that it possesses the
$\dist_\infty$-property at 
%the angular pm's in 
a subfamily of those (see Section \ref{appsec:proofKLmulti} of the supplement).
\begin{definition}\label{cond:mvt_angular}
Let $\Hset$ be as in Definition \ref{cond_angular} and:
\begin{inparaenum}
\item \label{cond:true_ang_set} Let $\Hset_0\subset \Hset$ be the class of angular pm's whose angular density functions $h$ are continuous on $\mathring{\resimp}$ and satisfy one of the following:
\begin{inparaenum}
	\item \label{cond: finite_new}
	$h$ admits a continuous extension to $\resimp$; 
	\item \label{cond: infinite_new}
	 (only if $d=2$) 
	 $\inf_{w\in(0,1)}h(w)>0$ and $\lim_{w\downarrow0}h(w)=\lim_{w\uparrow1}h(w) =+\infty$.
\end{inparaenum}
\item \label{cond:prior_ang_set} Let $\Hset'\subset \Hset_0$ be the set of angular pm's such that $H(\{\be_j\})=p_j>0$, for $j=1,\ldots,d$, with angular density satisfying the property in \ref{cond: finite_new}  and $\inf_{\bt \in \mathring{\resimp}}h(\bt)>0$.
\end{inparaenum}
\end{definition}  
\begin{theorem}\label{theo:KL_prior_multi}
Let $H_0\in\Hset_0$ be the true angular pm. Assume that for any $H\in\Hset'$, the prior $\Pi_\Hset$  posses the $\dist_{
	\infty}$-property at $H$. Then, for all $\epsilon>0$ 
\[
\Pi_{\Hset}(\phi_\Hset^{-1}(\mathcal{K}_\epsilon))=\Pi_{\Gset_\bone}(\mathcal{K}_\epsilon)>0.
\]
\end{theorem}
Theorem \ref{theo:KL_prior_multi}  is leveraged in the proof of
Theorem \ref{theo:post_consistency_mvt} to establish posterior consistency for the specific case of priors on the angular pm constructed via the representation in Section \ref{sec:extreme_dep_poly}. See also Example \ref{ex:mult_BP} for the outline of a prior specification complying with the assumptions of Theorem \ref{theo:KL_prior_multi}.
%

%%%%%%%%%%%%%%%%%%%%%%%%%%%%%%%%%%%%%%%%%%%%%%%%%%%%%%%%%%%%%%
\subsubsection{Posterior consistency}\label{sec:posterior_consistency}
%%%%%%%%%%%%%%%%%%%%%%%%%%%%%%%%%%%%%%%%%%%%%%%%%%%%%%%%%%%%%%

For each $H$ in the class $\Hset$ given in Definition \ref{cond_angular}, $g_{\bone}(\cdot|H)$ equals almost everywhere the expression in \eqref{eq:density_unit_fre} and is explicitly linked to $H$ via the relation
\begin{equation}\label{eq:kernel_density_ang}
\begin{split}
&-V_{I_i}(\by|H)\\
&\quad=
\begin{cases}
d\,p_i\,y^{-2}_{I_i}+d\left(\int_{(\bzero,\by_{I_i^\complement})}
\|\bz\|_1^{- d-1} h\circ\pi_{\resimp}(\bz/\|\bz\|_1)
\Big|_{\bz_{I_i}=\by_{I_i}}
\diff \bz_I\right), \quad \text{if } |I_i|=1\\
d\left(\int_{(\bzero,\by_{I_i^\complement})}
\|\bz\|_1^{- d-1} h\circ\pi_{\resimp}(\bz/\|\bz\|_1)
\Big|_{\bz_{I_i}=\by_{I_i}}
\diff \bz_I\right), \hspace{5.5em} \text{otherwise}
\end{cases}
\end{split}
\end{equation}
where $\by_{I_i}$ and $\by_{I_i^\complement}$ are the restrictions of $\by$ to $I_i$ and 
$I_i^\complement=\{1,\dots,d\}\setminus I_i$, and  $\pi_{\resimp}$ is the projection map in \eqref{eq:Bset}. See also \cite{r33,r25}
% and Section \ref{appsec:notation} of the supplementary material 
and Section \ref{appsec:spectral} of the supplement for details.
We consider priors on the Borel sets of $(\Hset, \dist_W)$ 
% $\Pi_\Hset$ on the Borel sets of $(\Hset, \dist_W)$ 
constructed via the representation in Section \ref{sec:extreme_dep_poly} as follows,  see also Section \ref{appsec:priorH} of the supplement for additional set-theoretical details. 
\begin{condition}\label{cond:angularprior} 	
		The prior $\Pi_\Hset$ is induced by a prior $\Pi$
		on the disjoint union space $\cup_{k\geq k_*}(\{k\}\times \Phi_k)$, constructed via direct sum of coefficient spaces 
		$$
		\{(\Phi_k,\Sigma_k, \lambda(k)\nu_k), \, k \geq k_* \},
		$$  
		for some $k_*\in \nat_+ \setminus \{1, \ldots,d\}$, where
		 $\Phi_k:=\{\bvarphi^{(k)} :\, \text{(R1)-(R2) hold true}\}$ and 
		 \begin{inparaenum}
		 	\item \label{cond: mvt th1} 
		 	 $\nu_k$ is a probability distribution with full support on $\Phi_k$, equiped with Borel $\sigma$-field $\Sigma_k$;
		 	\item \label{cond: mvt th2} 
			$\lambda(\cdot)$ is a positive probability mass function on
			$\{ k_*, k_*+1, \ldots\}$  and, for some $q>0$,
			$$\sum_{i\geq k}\lambda(i)\lesssim \exp(-qk^{d-1}), \quad (k \to \infty).
			$$ 
		 \end{inparaenum}
\end{condition}
\begin{theorem}\label{theo:post_consistency_mvt}
Let $\bY_1,\ldots,\bY_n$ be iid rv's with distribution $G_{\bone}(\cdot|H_0)$, where $H_0\in\Hset_0$ and $\Hset_0$ is as in Definition \ref{cond:mvt_angular}\ref{cond:true_ang_set}. Assume 
$\Pi_\Hset$ satisfies
\ref{cond:angularprior}.
%For some $k_*\in \N\setminus\{1,\ldots,d\}$, let $\Pi_{\Hset}$ be a prior distribution on $H$ induced by a
%prior distribution $\Pi$ on the Borel sets of the space $\cup_{k\geq k_*}(\{k\}\times \Phi_k))$, which satisfies:
%%
%\begin{inparaenum}
%%
%\item \label{cond: mvt th1} $\lambda(k)>0$ and $\nu_k$ has full support on $\Phi_k$, for all $k\geq k_*$;
%%
%\item \label{cond: mvt th2} $\sum_{i\geq k}\lambda(i)\lesssim \exp(-qk^{d-1})$, for some $q>0$;
%%
%\end{inparaenum}
%%
%where $\Phi_k:=\{\bvarphi_k: \text{(R11)-(R12) hold true}\}$.
%
Then, $\Pi_{\mathcal{G}_\bone}$ has full Hellinger support and, $\iprodGH-\text{as}$,
\begin{itemize}
\item[(a)] $\lim_{n \to \infty}\tilde{\Pi}_n(\tilde{\mathcal{U}}^\complement)=0$, for every $\dist_H$-neighbourhood $\tilde{\mathcal{U}}$ of $g_{\bone}(\cdot|H_0)$;
\item[(b)] $\lim_{n \to \infty}\Pi_n(\mathcal{U}^\complement)=0$,  for every $\dist_W$-neighborhood (if $d\geq 2$) or $\dist_{KS}$-neighborhood $\mathcal{U}_1$  of $H_0$ (if d=2).
\end{itemize}
%{\color{red} where $\Pi_n(\cdot)=\Pi_\Hset(\cdot|\bY_{1:n})$, $\tilde{\Pi}_n(\cdot)=\Pi_{\mathcal{G}_\bone}(\cdot|\bY_{1:n})$ and $\Pi_{\mathcal{G}_\bone}$ is the prior induced on $g_{\bone}(\cdot|H)$  by $\Pi_\Hset$. In particular, }
\end{theorem}
The proof of Theorem \ref{theo:post_consistency_mvt} is provided in Section \ref{appendix: cons mvt} of the supplement.
Consistency of the posterior on the angular pm  is obtained with a stronger metric in the bivariate case. 
%
%
%(see Section \ref{app:technical} of the supplement). 
Hellinger consistency of $\tilde{\Pi}_n$ 
%still allows to retrieve consistency on the Pickands dependence function 
entails that $\Pi_n$ concentrates on sets of angular pm's whose Pickands dependence functions lie in a neighbourhood  of $A_0$
under a Sobolev-type metric, induced by the norm $\Vert A\Vert_\infty+\sum_{1\leq j \leq d-1} \Vert A_{\{j\}}  \Vert_\infty$. The relations
\begin{equation*}
A_{\{j\}}(\bt)= d\int_\simp \{w_{j+1}\indic_{(w_{j+1}^*(\bt),1]}(w_{j+1})+w_{1}\indic_{(w_{1}^*(\bt),1]}(w_1)\}\diff H(\bw),
\end{equation*}
for $j=1, \ldots,d-1$, where $\bw^*(\bt)=\by^*(\bt)/\Vert \by^*(\bt) \Vert_1$ with $\bt \in \mathring{\resimp}$ and $\by^*(\bt)=(1/(1-\Vert \bt \Vert_1), 1/t_1,\ldots, 1/t_{d-1})$,
allow to turn such a concentration result into consistency on the space $\Hset$ with metric $\dist_{KS}$, only when $d=2$. 
An exception is the case where the true angular pm $H_0$ has no point masses. If so, Polya's theorem \citep[e.g.,][Proposition A.11]{r10} guarantees that every open $\dist_{KS}$-neighbourhood of $H_0$ contains a $\dist_W$-neighbourhood, thus consistency under the weaker metric extends to consistency under the stronger one.
We next sketch a prior construction for $\Pi$ that satisfies Condition \ref{cond:angularprior}.
\begin{example}\label{ex:mult_BP}
A prior $\Pi$ that exploits the representation of the angular pm in \eqref{eq: BPoly measure} can be obtained by choosing, for each $k \geq d+1$, a truncated Dirichlet prior $\nu_k(\bvarphi_k)$ on $(|\Gamma_k|+d)$-dimensional weight coefficients, with truncation outside $\Phi_k$ \citep[Section 3]{r39}. Prior specification can thus be completed by choosing $\lambda(k)$ as the probability mass function of a truncated discrete Weibull distribution, with shape parameter $d-1$.
\end{example}

%%%%%%%%%%%%%%%%%%%%%%%%%%%%%%%%%%%%%%%%%%%%%%%%%%%%%%%%%%%%%%
\subsection{$\brho$-Fr\'echet-, $\bomega$-Weibull- and Gumbel-max-stable distributions}\label{sec:binf_general_max}
%%%%%%%%%%%%%%%%%%%%%%%%%%%%%%%%%%%%%%%%%%%%%%%%%%%%%%%%%%%%%%
%
In this section we extend the consistency results of Section \ref{sec:binf_simple_max} to more general statical models $\{\prodGtH :\, (H,\bvartheta) \in \Hset\times \bvarTheta  \}$, all the proofs can be found in Section \ref{appendix:semi} of the supplement.
%
%{\color{magenta}To provide a general and concise discussion, we focus on the case where a prior distribution $\Pi_\Hset$ on $H$ is specified as in Condition \ref{cond:angularprior}}. 	
%
Uncertainty on the finite dimensional parameter $\bvartheta$ is expressed by means of a Borel prior $\Pi_{\bvarTheta}$
%also
on the space $\bvarTheta$, endowed with the $L_1$ metric.
% and consider a Borel prior for $\bvartheta$, denoted by $\Pi_{\bvarTheta}$.
% and specified independently from $H$, to which a Borel prior $\Pi_\Hset$ is assigned as above.
%
For $\bomega$-Weibull models with known true locations $\bmu_0$, the prior distribution on $\bmu$ reduces to $\delta_{\bmu_0}$ and the spaces $(0,\infty)^{2d}\times \{\bmu_0\}$ and
$(0,\infty)^{2d}$ are homeomorphic. Thus, with a little abuse of notation we also denote $\bvarTheta=(0,\infty)^{2d}$, $\bvartheta_0=(\bomega_0,\bsigma_0)$, $\bvartheta=(\bomega,\bsigma)$ and use the symbol $\Pi_{\bvarTheta}$ for the prior on shape and scale parameters.
%
%Similarly, 
%we also denote $\bvarTheta=(0,\infty)^{2d}$, $\bvartheta_0=(\bomega_0,\bsigma_0)$, $\bvartheta=(\bomega,\bsigma)$ and use the symbol $\Pi_{\bvarTheta}$ for the prior on shape and scale parameters.
%
We compactly report our assumptions on semiparametric prior specifications and true parameter values under different max-stable models.
\begin{condition}\label{cond:genprior} 
	According to the specific semiparametric model $\{G_\bvartheta(\cdot|H):  \, H \in \Hset, \bvartheta \in \bvarTheta\}$ under study, we assume that:
	\begin{inparaenum}
		\item\label{cond:indep} 
		the prior distributions $\Pi_{\bvarTheta}$ and $\Pi_{\Hset}$ on $\bvartheta$ and $H$ are are specified independently;
		\item\label{cond:angularpmprior} $\Pi_\Hset$ satisies Condition \ref{cond:angularprior} and $H_0   \in \Hset_0$;
		\item $\Pi_{\bvarTheta}$ complies with one of the following:
		\begin{inparaenum}
		\item \label{cond:shapescale}
		$\Pi_{\bvarTheta}$ has full support on $(0,\infty)^{2d}$;
		\item 
		\label{cond:compactprior}  $\Pi_{\bvarTheta}$ has marginal distribution  $\Pi_{\bvarTheta_j}$ on $\bvartheta_j$ supported on compact subsets $K_j$ of $(1,\infty)\times (0,\infty)\times \real$, for $j=1, \ldots,d$, and assigns positive mass to every neighbourhood of the true parameter $\bvartheta_0=(\bomega_0,\bsigma_{0},\bmu_0)$;
		\item 
		\label{cond:scaleloc}
		$\Pi_{\bvarTheta}$ has full support on $(0,\infty)^{d} \times \reald$.
		\end{inparaenum}	
	\end{inparaenum}
\end{condition}
The joint prior distribution $\Pi_{\Hset \times \bvarTheta}=\Pi_\Hset\times \Pi_{\bvarTheta}$ on the Borel subsets of $\Hset\times\bvarTheta$ gives rise to the posterior distribution
$$
\Pi_n(B):=\Pi_{\Hset \times \bvarTheta}(B|\bX_{1:n})=\frac{\int_B \prod_{i=1}^n g_\bvartheta(\bX_i|\theta) \Pi_{\Theta \times \bvarTheta}(\diff \theta, \diff \bvartheta)}{\int_\Theta \prod_{i=1}^n g_\bvartheta(\bX_i|\theta) \Pi_{\Theta \times \bvarTheta}(\diff \theta, \diff \bvartheta)},
$$
where $\bX_{1:n}=(\bX_1,\ldots,\bX_n)$ are iid rv's with distribution $G_{\bvartheta_0}(\cdot|H_0)$. In particular, $ (H,\bvartheta) \mapsto
g_\bvartheta(\cdot|H)=U_\bvartheta'(\cdot)
g_{\bone}(U_\bvartheta(\cdot)|H)
$, with $U_\bvartheta$ and $U_\bvartheta'$
as in \eqref{eq:gen_dens}, is a Borel measurable map between $\Hset\times\bvarTheta$ and the class of max-stable densities 
$\Gset_{\bvarTheta}:=\{ g_\bvartheta(\cdot|H),(H,\bvartheta)\in (\Hset,\bvarTheta) \}$ 
equiped with the metric $\dist_H$. Thus, $\Pi_{\Hset\times\bvarTheta}$ also induces a prior $\Pi_{\Gset_{\bvarTheta}}$ on the density $g_{\bvartheta}(\cdot|H)$ with full Hellinger support. In the sequel, strong  
consistency of the posterior distributions $\tilde{\Pi}_n(\cdot)=\Pi_{\Gset_{\bvarTheta}}(\cdot|\bX_{1:n})$ and $\Pi_n$ is obtained by adapting arguments from Schwartz' theorem to the semiparametric model under study. This is done by verifying the following general sufficient condition.
\begin{condition}\label{cond: newcond}
	The Kullback-Leibler support of $\Pi_{\Gset_{\bvarTheta}}$ contains $g_{\bvartheta_0}(\cdot|\theta_0)$,
%	$\bvartheta_0$ is an interior point of $\text{supp}(\Pi_{\bvarTheta})$
	and for every $\dist_H$-ball  $\tilde{\mathcal{U}}_\epsilon:=\{ g \in \Gset_{\bvarTheta}: 
	\dist_H(g, g_{\bvartheta_0}(\cdot|\theta_0)) \leq 4\epsilon
	\}$, $\epsilon>0$, and every $\delta \in(0, \delta_*)$, for some $\delta_*>0$, there exist a sequence of measurable partitions $\{\Gset_{\bvarTheta,n}, \Gset_{\bvarTheta,n}^\complement  \}$ of the sumbodel
	$$
	\tilde{\mathcal{U}}_\epsilon^\complement
	\cap
	\{
	g_{\bvartheta}(\cdot|H): 
	\, (H,\bvartheta) \in \Hset\times \bvarTheta , \, \Vert \bvartheta  - \bvartheta_0\Vert_{1} \leq \delta
	\}
	$$
	and test functionals $\tests(\bX_{1:n})=(s_n(\bX_{1:n}),t_{1,n}(\bX_{1:n}), \ldots, t_{n,d}(\bX_{1:n}))$ with values in $[0,1]$ such that $\Pi_{\Gset_{\bvarTheta}}(\Gset_{\bvarTheta,n}^\complement) \lesssim e^{-rn}$ as $n \to \infty$, for some $r>0$ and: 
	\begin{inparaenum} 
		\item \label{dentest}
		$s_n$ satisfies  $\int s_n(\bx_{1:n})\diff \prodGtHtrue (\bx_{1:n}|H_0) \leq e^{-n\epsilon^2}$ and 
		$$
		\sup_{G: \, g \in \Gset_{\bvarTheta,n}}
		\int \{1-s_n \}(\bx_{1:n})\}\diff
		\prodGProbaltgenn(\bx_{1:n}) \leq e^{-2n\epsilon^2};
		$$
		%	for testing $g_{\bvartheta_0}(\cdot|\theta_0)$ [...]
		\item\label{margintest} for each $j=1, \ldots,d$, $t_{n,j}$ satisfies $\int t_{n,j}(x_{j,1:n})\diff \prodGtHtruealtjj (x_{j,1:n}) \lesssim e^{-nc_j(\delta)}$ and
		$$
		\sup_{\bvartheta_j \in \text{supp}(\Pi_{\bvarTheta_j}): \,  \Vert \bvartheta_j  - \bvartheta_{0,j}  \Vert_\infty>{\delta/d} }
		\int \{
		1-t_{n,j}(x_{j,1:n})\}
		\diff \prodGtHaltjj (x_{j,1:n}) \lesssim  e^{-nc_j(\delta)},
		$$
		as $n \to \infty$, where $c_j(\delta)$ is a positive constant and $\Pi_{\bvarTheta_j}$ is the marginal prior on $\bvartheta_j$.
	\end{inparaenum}
\end{condition}
Under the above condition, for all neighbourhoods $\tilde{\mathcal{U}}$, $\mathcal{U}_1$ and $\mathcal{U}_2$ of $g_{\bvartheta_0}(\cdot|\theta_0)$, $H_0$ and $\bvartheta_0$, respectively,  the posterior distributions $\tilde{\Pi}_n$ and $\Pi_n$ satisfy an inequality of the form
\begin{equation}\label{eq:mainbound}
\begin{split}
&\max
\left\lbrace
\tilde{\Pi}_n(\tilde{\mathcal{U}}^\complement),
\Pi_n((\mathcal{U}_1\times\mathcal{U}_2)^\complement)
\right\rbrace \\
 &\qquad
 \lesssim\Vert \tests(\bX_{1:n}) \Vert_1+ \frac{\Xi_n(\bX_{1:n}, \tests, \Pi_{\Hset \times \bvarTheta})}{\int \prod_{i=1}^n \{g(\bX_i)/g_{\bvartheta_0}(\bX_i|\theta_0)\}\diff\Pi_{\mathcal{G}_{\bvarTheta}}(g)}\\
&\qquad 
\lesssim \Vert \tests(\bX_{1:n}) \Vert_1+ e^{cn}\Xi_n(\bX_{1:n}, \tests, \Pi_{\Hset \times \bvarTheta})
\end{split}	
\end{equation}
eventually almost surely as $n \to \infty$, for a positive constant $c>0$ and a functional $\Xi_n$ which depend on 
%$\tilde{\mathcal{U}}, \, \mathcal{U}_1, \, \mathcal{U}_2$ 
the neighbourhoods' choice and  comply with
\begin{equation}\label{eq:expobound}
	\prodGtHtrue \left(
	\Vert \tests(\bX_{1:n}) \Vert_1+ e^{cn}\Xi_n(\bX_{1:n}, \tests, \Pi_{\Hset \times \bvarTheta})> \varepsilon \big{|}H_0
	\right) \lesssim e^{-c'n}, \quad \forall \varepsilon>0,
\end{equation} 
for some positive constant $c'$. Therefore, strong consistency can be deduced by applying Borel-Cantelli lemma, see Section \ref{sec:gencons} of the supplement for details. 
Strong Hellinger consistency of $\tilde{\Pi}_n$
guarantees that the Bayesian estimator of $g_{\bvartheta_0}(\cdot|H_0)$ under quadratic loss, i.e. the predictive density $\hat{g}_n(\bx)=\int_{\Gset_{\bvarTheta}}g(\bx) \, \diff \tilde{\Pi}_n(g)$, is strongly Hellinger consistent too, that is
$$
\lim_{n \to \infty}
\dist_H(\hat{g}_n, g_{\bvartheta_0}(\cdot|H_0))=0, 
$$
$\iprodGtH-\text{as}$ \citep[e.g.,][Chapter 6.8.3]{r10}. The predictive aspect is particularly relevant as, in most of the real data applications concerning extreme events, % e.g. for spatial extremes, 
the prediction of future extremes is often the main inferential goal.
%
%Our assumptions on the true angular pm $H_0$ are compactly recalled next.
%{\color{magenta}
%\begin{condition}\label{con:angular}
%	If $d=2$, assume that $H_0$ corresponds to a true Pickands dependence function $A_0 \in \Aset_0$, with $\Aset_0$ as in Definition \ref{cond_density}. If $d>2$, assume $H_0 \in \Hset_0$, where $\Hset_0$ is given in Definition \ref{cond:mvt_angular}. 
%\end{condition} 
%}
%

We now start extending consistency to the case of $\brho$-Fr\'{e}chet max-stable models, with unknown shape and scale parameters $\brho=(\rho_1,\ldots,\rho_d)$ and $\bsigma=(\sigma_1,\ldots,\sigma_d)$, respectively. 
Their density function is 
\begin{equation}\label{eq:dens_a_frec}
g_{\brho, \bsigma}(\bx|H)=\prod_{j=1}^d \rho_j\sigma_j^{-\rho_j} x_j^{\rho_j-1}g_{\bone}\left((x_1/\sigma_1)^{\rho_1}, \ldots, (x_d/\sigma_d)^{\rho_d}|H\right), \quad \bx > \bzero,
\end{equation}
with $g_{\bone}(\cdot|\theta)$ almost everywhere as in \eqref{eq:density_unit_fre}. 
%almost everywhere. 
%
In this case, a joint prior $\Pi_{\Hset\times\bvarTheta}$ is assigned to $H$ and $(\brho,\bsigma)$, inducing a prior $\Pi_{\mathcal{G}_{\bvarTheta}}$ on $\Gset_{\bvarTheta}:=\{ g_{\brho,\bsigma}(\cdot|H):\,H\in \Hset, (\brho,\bsigma)\in (0,\infty)^{2d} \}$.
%
%\begin{condition}\label{cond:prior a}
%$\Pi_{\bvarTheta}$ has full support on $
%%\bvarTheta=
%(0,\infty)^{2d}$. 
%%
%\end{condition}
%

\begin{theorem}\label{th:alpha_frec}
Let $\bX_1,\ldots,\bX_n$ be iid rv with distribution $G_{\brho_0,\bsigma_0}(\cdot|H_0)$, 
%where 
%$H_0$ satisfies Condition \ref{con:angular}
%$H_0 \in \Hset_0$
% 
where $(\brho_0,\bsigma_0)  \in (0, \infty)^{2d}$. 
%
%Let $\Pi_{\Hset \times {\bvarTheta}}:=\Pi_\Hset \times \Pi_{\bvarTheta}$, where $\Pi_\Hset$ and $\Pi_{\bvarTheta}$ satisfy Conditions \ref{cond:angularprior} and \ref{cond:prior a}, respectively. 
%
Then, under Conditions \ref{cond:genprior}\ref{cond:indep}--\ref{cond:angularpmprior} and \ref{cond:genprior}\ref{cond:shapescale}, $\iprodGrho-\text{as}$
\begin{itemize}
	\item[(a)] $\lim_{n \to \infty}\tilde{\Pi}_n(\tilde{\mathcal{U}}^\complement)=0$, for every $\dist_H$-neighbourhood $\tilde{\mathcal{U}}$ of $g_{\brho_0,\bsigma_0}(\cdot|H_0)$;
	\item[(b)]  $\lim_{n\to \infty}\dist_H(\hat{g}_n,g_{\brho_0,\bsigma_{0}}(\cdot|H_0))=0$; 
	\item[(c)] $\lim_{n\to\infty}\Pi_n((\mathcal{U}_1\times\mathcal{U}_2)^\complement)=0
	$, for every $\dist_W$-neighborhood (if $d\geq 2$) or $\dist_{KS}$-neighborhood $\mathcal{U}_1$  of $H_0$ (if d=2) and $L_1$-neighborhood $\mathcal{U}_2$ of $(\brho_0,\bsigma_0)$.  
\end{itemize}
%
%{\color{red}where $\Pi_n(\cdot)=\Pi_{\Hset \times \bvarTheta}(\cdot|\bX_{1:n} )$, $\tilde{\Pi}_n(\cdot)=\Pi_{\Gset_{\bvarTheta}}(\cdot|\bX_{1:n})$.} 
%and  $\hat{g}_n(\bx)=\int_{\Gset_{\bvarTheta}}g(\bx) \, \diff \tilde{\Pi}_n(g)$.
%
\end{theorem} 
%
%As a by-product, we also obtain consistency for
We next consider $\bomega$-Weibull max-stable models, with unknown shape and scale parameters $\bomega=(\omega_1,\ldots,\omega_d)$ and $\bsigma=(\sigma_1,\ldots,\sigma_d)$, respectively, and location parameters $\bmu=(\mu_{1},\ldots,\mu_{d})$
which may be either known or 
uknown. 
Their density function is
\begin{equation}\label{eq: weibdens}
\begin{split}
&g_{\bomega,\bsigma,\bmu}(\bx|H)\\&\quad=\prod_{j=1}^d \frac{\omega_j}{\sigma_j}\left(\frac{\mu_j-x_j}{\sigma_j}\right)^{-\omega_j-1}g_{\bone}\left(\left(\frac{\mu_1-x_1}{\sigma_1}\right)^{-\omega_1}, \ldots, \left(\frac{\mu_d-x_d}{\sigma_d}\right)^{-\omega_d}\bigg{|}H\right)
\end{split}
\end{equation}
for $\bx < \bmu$.
When $\bmu\equiv \bmu_0$ is known, the latter can be interpreted as the joint density of $\bmu_0-\bY^{-1}$, where $\bY$ is a rv with distribution $G_{\brho,\bsigma^{-1}}(\cdot|H)$ and $\bomega=\brho$. 
Accordingly, 
%by a simple change of variables,
a prior $\Pi_{\bvarTheta}$ on $(\bomega, \bsigma)$ with full support on $(0,\infty)^{2d}$ along with $\Pi_\Hset$ and the induced prior $\Pi_{\Gset_{\bvarTheta}}$ on $\Gset_{\bvarTheta}=\{g_{\bomega,\bsigma, \bmu_0}: H \in \Hset, (\bomega,\bsigma)\in (0,\infty)^{2d} \}$
can be turned into priors on the parameters and the density function
of a multivariate $\brho$-Fr\'echet model, respectively.
Analogous considerations apply to posterior distributions, hence consistency obtains as a byproduct of Theorem \ref{th:alpha_frec}.  
\begin{cor}\label{cor: cons_weibull}
	Let $\bX_1,\ldots,\bX_n$ be iid rv with distribution $G_{\bomega_0,\bsigma_0,\bmu_0}(\cdot|H_0)$, where 
%	$H_0 \in \Hset_0$,
%	satisfies Condition \ref{con:angular},
	% 
	 $(\bomega_0,\bsigma_0) \in (0, \infty)^{2d}$ and $\bmu_0 \in \reald$ is known. 
%	 Denote $\bvarTheta=\bvarTheta'\times\{\bmu_0\}$, $\bvarTheta'=(0,\infty)^{2d}$ and let $\Pi_{\Hset \times {\bvarTheta}}:=\Pi_\Hset \times \Pi_{\bvarTheta}$, where $\Pi_\Hset$ satisfy Condition \ref{cond:angularprior}, $\Pi_{\bvarTheta}=\Pi_{\bvarTheta'}\times \delta_{\bmu_0}$  and $\Pi_{\bvarTheta'}$ satisfies an analogue of Condition \ref{cond:prior a}. 
	Then, under Conditions \ref{cond:genprior}\ref{cond:indep}--\ref{cond:angularpmprior} and \ref{cond:genprior}\ref{cond:shapescale}, 
	$\iprodGomnew-\text{as}$
	\begin{itemize}
		\item[(a)]  $\lim_{n \to \infty}\tilde{\Pi}_n(\tilde{\mathcal{U}}^\complement)=0$, for every $\dist_H$-neighbourhood $\tilde{\mathcal{U}}$ of $g_{\bomega_0,\bsigma_0,\bmu_0}(\cdot|H_0)$;
		\item[(b)] 	$
		\lim_{n\to \infty}\dist_H(\hat{g}_n,g_{\bomega_0,\bsigma_0,\bmu_0}(\cdot|H_0))=0$;
		\item[(c)] $\lim_{n\to\infty}\Pi_n((\mathcal{U}_1\times\mathcal{U}_2)^\complement)=0$, for every $\dist_W$-neighborhood (if $d\geq 2$) or $\dist_{KS}$-neighborhood $\mathcal{U}_1$ of $H_0$ (if d=2) and $L_1$-neighborhood $\mathcal{U}_2$ of $(\bomega_0,\bsigma_0)$.
	\end{itemize}
%{\color{red} where $\Pi_n(\cdot)=\Pi_{\Hset \times \bvarTheta}(\cdot|\bX_{1:n} )$,  $\tilde{\Pi}_n(\cdot)=\Pi_{\mathcal{G}_{\bvarTheta}}(\cdot|\bX_{1:n})$.}
%	, $\Pi_{\mathcal{G}_{\bvarTheta}}$ is the prior induced on $\Gset_{\bvarTheta}:=\{ g_{\bomega, \bsigma, \bmu_0}(\cdot|H): \, (H,\bomega, \bsigma)\in (\Hset,\bvarTheta) \}$.
%	 and  $\hat{g}_n(\bx)=\int_{\Gset_{\bvarTheta}}g(\bx) \, \diff \tilde{\Pi}_n(g)$.
\end{cor}
In several applications the location parameters are unknown and a diffuse prior distribution has to specified also on those. Notably, the study of such a case gives the mathematical ground for addressing the problem discussed in Section \ref{sec:weibulldom}, where a $\bomega$-Weibull distribution is fitted to sample maxima. For technical convenience, we restrict the parameter space to $\bvarTheta=(1,\infty)^d\times (0,\infty)^d \times \reald$
and consider a prior $\Pi_{\bvarTheta}$ on $(\bomega,\bsigma,\bmu)$ with compactly supported marginal priors on $(\omega_j,\sigma_j, \mu_j)$, $j=1, \ldots,d$, yielding a prior $\Pi_{\Gset_{\bvarTheta}}$ supported on a subset of the
model class $\Gset_{\bvarTheta}=\{g_{\bomega,\bsigma,\bmu}: \, H \in \Hset,
(\bomega,\bsigma,\bmu)\in \bvarTheta
\}$.
% we use the following assumption.
%

%\begin{condition}\label{cond:weib}
%	Let $\bvarTheta=(1,\infty)^d\times(0, \infty)^d \times \real$ and
%	assume the following:
%%	is an arbitrarily large compact subset of $(1,\infty)^d\times (0,\infty)^d\times \reald$ and
%	\begin{inparaenum}
%		\item\label{weib1} 
%		the prior $\Pi_{\bvarTheta}$ on $(\bomega,\bsigma,\bmu)$ is supported on a compact subset $K$ of the interior of $\bvarTheta$;
%		\item \label{weib2} all neighbourhoods of the true parameter values  $(\bomega_0,\bsigma_0,\bmu_0)$ have positive $\Pi_{\bvarTheta}$-mass.
%	\end{inparaenum} 
%\end{condition}
%
%
\begin{theorem}\label{theo: cons_weibull}
	Let $\bX_1,\ldots,\bX_n$ be iid rv with distribution $G_{\bomega_0,\bsigma_0,\bmu_0}(\cdot|H_0)$, where 
%	$H_0\in \Hset_0$ 
%	and 
$(\bomega_0,\bsigma_0,\bmu_0 ) \in (1,\infty)^d\times(0, \infty)^d \times\reald$. 
%	Let $\Pi_{\Hset \times {\bvarTheta}}:=\Pi_\Hset \times \Pi_{\bvarTheta}$, where $\Pi_\Hset$ and $\Pi_{\bvarTheta}$  satisfy Conditions \ref{cond:angularprior} and \ref{cond:weib}\ref{weib2}, respectively. 
	Then, 
	 under Conditions \ref{cond:genprior}\ref{cond:indep}--\ref{cond:angularpmprior} and \ref{cond:genprior}\ref{cond:compactprior},
	$\iprodGomnew-\text{as}$
	\begin{itemize}
		\item[(a)]  $\lim_{n \to \infty}\tilde{\Pi}_n(\tilde{\mathcal{U}}^\complement)=0$, for every $\dist_H$-neighbourhood $\tilde{\mathcal{U}}$ of $g_{\bomega_0,\bsigma_0,\bmu_0}(\cdot|H_0)$;
		\item[(b)] 	$
		\lim_{n\to \infty}\dist_H(\hat{g}_n,g_{\bomega_0,\bsigma_0,\bmu_0}(\cdot|H_0))=0$;
		\item[(c)] $\lim_{n\to\infty}\Pi_n((\mathcal{U}_1\times\mathcal{U}_2)^\complement)=0$, for every $\dist_W$-neighborhood (if $d\geq 2$) or $\dist_{KS}$-neighborhood $\mathcal{U}_1$ of $H_0$ (if d=2) and $L_1$-neighborhood $\mathcal{U}_2$ of $(\bomega_0,\bsigma_0,\bmu_0)$.
	\end{itemize}
%	{\color{red} where $\Pi_n(\cdot)=\Pi_{\Hset \times \bvarTheta}(\cdot|\bX_{1:n} )$,  $\tilde{\Pi}_n(\cdot)=\Pi_{\mathcal{G}_\bone}(\cdot|\bX_{1:n})$.} 
%and  $\hat{g}_n(\bx)=\int_{\Gset_{\bvarTheta}}g(\bx) \, \diff \tilde{\Pi}_n(g)$.
%	 $\Pi_{\mathcal{G}_{\bvarTheta}}$ is the prior induced on $\Gset_{\bvarTheta}:=\{ g_{\bomega, \bsigma, \bmu}(\cdot|H): \, (H,\bomega, \bsigma, \bmu)\in (\Hset,\bvarTheta) \}$.
\end{theorem}
\begin{rem}
%Alike the other consistency results presented in this section, 
Analogously to the other consistency theorems of this section, Theorem \ref{theo: cons_weibull} is proved by verifying Condition \ref{cond: newcond}.
The restriction $\bomega>\bone$ imposed to shape parameters rules out some $\bomega$-Weibull max-stable densities for which the construction of Kullback-Leibler neighbourhoods is intractable.
%
%Restricting to shape parameters $\bomega>\bone$ rules out some $\bomega$-Weibull max-stable densities with undesirable integrability properties, which hamper establishing the Kullback-Leibler property. 
%
%Such a restriction is also a common practice  in extreme-value applications 
We point out that such a restriction is a quite common practice in extreme-value applications
\citep[e.g.,][page 725]{r1001}.
The three-parmeter Weibull distribution is an iregular model and testing $(\omega_{0,j},\sigma_{0,j},\mu_{0,j})$, for $j\in\{1,\ldots,d\}$, against the complement of a large compact neighbourhood with exponentially bounded errors remains an open problem. Herein, this issue is circumvented using a prior $\Pi_{\bvarTheta}$ with compactly supported margins. 
Exponentially consistent tests on the marginal parameters (Condition \ref{cond: newcond}\ref{margintest}) can be constructed over compact parametric subspaces starting from preliminary uniformly consistent estimators  
\cite[Lemmas 10.3 and 10.6]{r999}. 
In practice, confining prior specification on shape, scale and location parameters to an arbitrarily large compact set is hardly a restriction, as in applied sciences the physically reasonable ranges for the parameters are often finite. See also  \cite[page 845]{r800} for similar considerations concerning maximum likelihood inference for the univariate generalised extreme-value distribution, where  maximization is restrained over compact parametric subspaces.
%
%As for exponentially consistent tests on the marginal parameters
%(Condition \ref{cond: newcond}\ref{margintest}),  they can be conveniently constructed over compact parametric subspaces starting from preliminary uniformly consistent estimators  
%\cite[Lemmas 10.3 and 10.6]{r999}. 
%
%Due to the poorly regular nature of three-parmeter Weibull marginal models, testing $(\omega_{0,j},\sigma_{0,j},\mu_{0,j})$ against the complement of a large compact neighbourhood with exponentially bounded errors remains an open problem, herein circumvented using a prior $\Pi_{\bvarTheta}$ with compactly supported margins.	
%
%In practice, confining inference on shape, scale and location parameters
%to an arbitrarily large compact set is hardly a restriction, as in applied sciences physically reasonable parameter ranges are typically finite. See also 
%\cite[page 845]{r800} for further remarks in the context of maximum likelihood inference for the univariate generalised extreme-value distribution with maximization over compact parametric subspaces.
\end{rem}
Finally, we consider the case where the data come from Gumbel max-stable distribution, with unknown scale and location parameters $\bsigma=(\sigma_1,\ldots,\sigma_d)$ and $\bmu=(\mu_1,\ldots,\mu_d)$, respectively. The pertaining  density function is
\begin{equation}\label{eq:densgumb}
g_{\bsigma,\bmu}(\bx|H)=\prod_{j=1}^d \frac{e^{(x_j-\mu_j)/\sigma_j}}{\sigma_j}g_{\bone}\left(e^{(x_1-\mu_1)/\sigma_1}, \ldots, e^{(x_d-\mu_d)/\sigma_d}\bigg{|}H\right), \quad \bx \in \mathbb{R}^d.
\end{equation} 
A prior $\Pi_{\bvarTheta}$ is now specified on $(\bsigma,\bmu)$, yielding a prior $\Pi_{\Gset_{\bvarTheta}}$ on the max-stable class $\Gset_{\bvarTheta}=\{g_{\bsigma,\bmu}(\cdot|H):\, H \in \Hset, (\bsigma,\bmu)\in (0,\infty)^d \times \reald \}$.
\begin{theorem}\label{cor:Gumbel_cons}
		Let $\bX_1,\ldots,\bX_n$ be iid rv's with distribution $G_{\bsigma_0,\bmu_0}(\cdot|H_0)$, where 
%		$H_0 \in \Hset_0$ 
%		satisfies Condition \ref{con:angular} 
%		and 
		$(\bsigma_0,\bmu_0)\in(0,\infty)^d \times \reald$.
		%
%		Let $\Pi_{\Hset\times\bvarTheta}=\Pi_\Hset\times\Pi_{\bvarTheta}$, where $\Pi_\Hset$ satisfy Condition \ref{cond:angularprior} and $\Pi_{\bvarTheta}$ is a Borel prior on $(\bsigma,\bmu)$ with full support on $\bvarTheta$. 
	Then, under Conditions \ref{cond:genprior}\ref{cond:indep}--\ref{cond:angularpmprior} and \ref{cond:genprior}\ref{cond:scaleloc},  $\iprodGzexpand-\text{as}$
	\begin{itemize}
		\item[(a)] $\lim_{n \to \infty}\tilde{\Pi}_n(\tilde{\mathcal{U}}^\complement)=0$, for every $\dist_H$-neighbourhood $\tilde{\mathcal{U}}$ of $g_{\bsigma_0,\bmu_0}(\cdot|H_0)$;
		\item[(b)] $
		\lim_{n\to \infty}\dist_H(\hat{g}_n,g_{\bsigma_0,\bmu_0}(\cdot|H_0))=0$;
		\item[(c)] $\lim_{n\to\infty}\Pi_n((\mathcal{U}_1\times\mathcal{U}_2)^\complement)=0$, for every $\dist_W$-neighborhood (if $d\geq 2$) or $\dist_{KS}$-neighborhood $\mathcal{U}_1$ of $H_0$ (if d=2) and $L_1$-neighborhood $\mathcal{U}_2$ of $(\bsigma_0,\bmu_0)$.
		\end{itemize}
		%

%		  $\Pi_{\mathcal{G}_\bzero}$ is the prior induced on $\Gset_{\bvarTheta}=\{g_{\bzero,\bsigma,\bmu}(\bx|H),\, H \in \Hset, (\bsigma,\bmu)\in \bvarTheta\}$  by the prior $\Pi_{\Hset\times\bvarTheta}$ 
%		and $\hat{g}_n(\bx)=\int_{\Gset_{\bvarTheta}}g(\bx) \, \diff \tilde{\Pi}_n(g)$.
	\end{theorem}
%

%%%%%%%%%%%%%%%%%%%%%%%%%%%%%%%%%%%%%%%%%%%%%%%%%%%%%%%%%%%%%%%%%%%%%%%%%%%%%%%%%%%
\section{Max-stable distributions' neighbourhoods}\label{sec:binf_sample_max}
%%%%%%%%%%%%%%%%%%%%%%%%%%%%%%%%%%%%%%%%%%%%%%%%%%%%%%%%%%%%%%%%%%%%%%%%%%%%%%%%%%%
%The large samples theory of a Bayesian procedure is often based on the paradigm that a sample is drawn from a specific model.
Consistency of a Bayesian procedure
requires the posterior distribution to allow increasingly accurate inferences on the true parameter
under study, as the sample size grows larger.
In infnite-dimensional settings, the related asymptotic theory is established mostly in the case of well specified statistical models, where the notion of posterior consistency is formalised as done, for example, in Definition \ref{def:cons} for simple max-stable data. 
%
%{\color{magenta}In presence of model misspecification, few general results are available under the assumption that data are obtained by iid sampling from a fixed distribution \cite{r77}.}
%
In practice, statistical applications of max-stable models are concerned with data samples of maxima, whose actual distribution only lies in a neighbourhood of a max-stable one. 
%and {\color{magenta}depends on the number of underlying rv's over which maxima are computed}.
%
%To the best of our knowledge, there is no result offering a rigorous mathematical justification for the use of Bayesian inference on multivariate max-stable models via the analysis of sample maxima.
%
To the best of our knowledge, there is no study offering a rigorous mathematical justification to Bayesian inference on max-stable models when the data sample consists of maxima following only approximately a max-stable distribution. 
%
%The main goal of this section is to provide conditions under which such an approach is asymptotically validated.
%
The goal of this section is to provide conditions under which the 
Bayesian 
%inference is asymptotically {\color{magenta}well founded}, establishing {\color{magenta}hence} the consistency of the posterior distribution.
approach produces posterior distributions  which concentrate near the appropriate max-stable density and
leads to  consistent estimation of the true probability density of maxima.
%
%We advocate the use of a data-dependent prior, elaborating on the existent literature on frequentist inference for max-stable distributions and proposing new guidelines for a hybrid-Bayesian approach.
%
We resort to the use of a data-dependent prior, capitalising on the existing literature on frequentist inference on max-stable distributions, and propose guidelines for a new hybrid Bayesian approach.
%
%The asymptotic techniques through remote contiguity, devised to bridge the limiting observational model and the actual data generating distribution, are new to the Bayesian domain and of independent interest for furhter applications to misspecified models, also beyond extremes.
%
We use asymptotic techniques stemming from remote contiguity, devised to bridge the limiting observational model and the actual data generating distribution, which are new to the Bayesian domain. These are of independent interest for further applications to misspecified models, also beyond the extreme values context.
	
%%%%%%%%%%%%%%%%%%%%%%%%%%%%%%%%%%%%%%%%%%%%%%%%%%%%%%%%%%%%%%%%%%%%%%%%%%%%%%
\subsection{Empirical Bayes analysis of maxima}\label{eq: EB}
%%%%%%%%%%%%%%%%%%%%%%%%%%%%%%%%%%%%%%%%%%%%%%%%%%%%%%%%%%%%%%%%%%%%%%%%%%%%%%

The posterior distributions $\tilde{\Pi}_n$ and  $\Pi_n$ considered so far are conditional pm's depending on a dataset coming from max-stable model with true unknown marginal and dependence parameters.
We now assume that the data sample consists of $n$ rv's of componentwise maxima
$$
\bM_{m_n,i}=\max(\bZ_{(i-1)m_n+1}, \ldots,\bZ_{im_n}),\quad i=1,\ldots,n, 
$$   
obtained by dividing a sample $(\bZ_1, \ldots, \bZ_{nm_n})$ of $nm_n$ iid rv's with distrution $F_0$ into blocks of size $m_n$. Consequently, $\bM_{m_n,1},\ldots,\bM_{m_n,n}$ are iid rv's with distribution $F_0^{m_n}$. We assume that $F_0$ is in the \textit{variational max-domain} of $G_{\bvartheta_0}   (\cdot|H_0)$, that is
\begin{equation}\label{eq:strongconv}
\lim_{n\to \infty} \dist_T\left(F_0^{m_n}(\ba_{m_n}\cdot+\bb_{m_n}), G_{\bvartheta_0}   (\cdot|H_0)\right)=0
\end{equation}
for suitable norming sequences $\ba_{m_n}$ and $\bb_{m_n}$, where $G_{\bvartheta_0}(\cdot|H_0)$ is known to be either a $\brho$-Fr\'{e}chet or $\bomega$-Weibull or Gumbel  multivariate max-stable distribution, with unit scale and null location parameters. 
Precisely, 
we assume $m_n \to \infty$ as $n \to \infty$, which avoids considerations about double limits,
and
rely on the following set of conditions.
\begin{condition}\label{cond:strong}
$F_0 \in \mathcal{D}(G_{\bvartheta_0}(\cdot|H_0))$ and $F_0(\bz)=C_0(F_{0,1}(z_1), \ldots, F_{0,d}(z_d))$, where
	\begin{inparaenum} 
	\item\label{cond:copdiff} $C_0$ is a $d$-times contiuously on $(0,1)^d$ differentiable copula,  which satisfies
			$$
			\lim_{n \to \infty}\frac{\partial^{|I|}}{(\partial x_i, i \in I)} m_n \left( 
			C_0\left( \bone -\frac{\bx}{m_n} \right)-1
			\right)=-\frac{\partial^{|I|}}{(\partial x_i, i \in I)} L(\bx|\theta_0), \,\bx>0,
			$$ 
			for all $I \subset \{1, \ldots,d\}$, and the stable tail-dependence function pertaining to $G_{\bvartheta_0}(\cdot|H_0)$,  $L(\cdot|H_0)$, is assumed $d$-times continuously differentiable on $(0,\infty)^d$;
	\item \label{margins}$F_{0,j}$ with $j\in\{1,\ldots,d\}$ are continuously differentiable and satisfy one of the following	
	\begin{eqnarray*}
	&&\lim_{z \to \infty}\frac{z F_{0,j}'(z)}{1-F_{0,j}(z)}=\rho_{0,j}, 
	\hspace{9.8em} \text{if  } \bvartheta_0=(\brho_0, \bone),\\
	&&\lim_{z \uparrow z_{0,j}}\frac{(z_{0,j}-z)F_{0,j}'(z)}{1-F_{0,j}(z)}=-\omega_{0,j}, \hspace{6.6em}  \text{if  } \bvartheta_0=(\bomega_0, \bone, \bzero),\\
	&&\lim_{z \uparrow z_{0,j}}\frac{F_{0,j}'(z)}{\left( 1-F_{0,j}(z)\right)^2}\int_z^{z_{0,j}}(1-F_{0,j}(t))\diff t=1,
	\quad  \text{if  } \bvartheta_0= (\bone, \bzero),\\
	\end{eqnarray*}	
	with $z_{0,j}:=\sup\{z \in \mathbb{R}: F_{0,j}(z)<1\}$ and $\brho_0, \bomega_0> \bzero$. If $F_0 \in \mathcal{D}(G_{\brho_0, \bone}(\cdot|H_0))$, without loss of generality, we also assume $\text{supp}(F_{0,j})\subset (0,\infty)$, $j=1, \ldots,d$.
%	\item $H_0 \in \Hset_0$. {\color{blue}If $d=2$,  $\bvartheta_0=(\brho_0, \bone)$ or $\bvartheta_0=(\bzero, \bone, \bzero)$ and $h_0$ behaves as in Definition \ref{cond:mvt_angular}\ref{cond: infinite_new}, further assume there exists  $s>0$ such that $h_0^{1+s}$ is integrable on $(0,1)$}. 
		\end{inparaenum}	
\end{condition}
Condition \ref{cond:strong} implies the convergence result in \eqref{eq:strongconv} \cite[Corollary 3.1]{r32}. The latter represents a stronger form of convergence than \eqref{eq:doa}. Under the above reguilarity conditions, it guarantees that the probability density of the rv's $$
\overline{\bM}_{m_n,i}:=(\bM_{m_n,i}-\bb_{m_n})/\ba_{m_n}, \quad i=1, \ldots, n,
$$ 
is in 
%a $\dist_H$-neighbourhood 
a Hellinger-neighborhood of $g_{\bvartheta_0}(\cdot|H_0)$, for all sufficiently large block sizes $m_n$. 
Without loss of generality, in the sequel we  consider the valid choices of $\ba_{m_n}$ and $\bb_{m_n}$ given by 
\begin{equation}\label{eq:norming}
\begin{split}
a_{m_n,j}, \, b_{m_n,j} &= \begin{cases}
F_{0,j}^{\leftarrow}(1-1/m_n), \hspace{5.3em}  0 \hspace{6.95em}  \text{if  } \bvartheta_0=(\brho_0, \bone)\\
z_{0,j}-F^\leftarrow_{0,j}(1-1/m_n), \hspace{2.3em} z_{0,j} \hspace{6.2em} \text{if  } \bvartheta_0=(\bomega_0, \bone, \bzero) 	\\
	m_n \int_{b_{m_n,j}}^{z_{0,j}}(1-F_{0,j}(z))\diff z, 	\hspace{1em} F_{0,j}^{\leftarrow}(1-1/m_n)	
 \hspace{1em}  \text{if  } \bvartheta_0= ( \bone, \bzero)
\end{cases}\\
\end{split}
\end{equation}
for $j=1,\ldots,d$. These are generally unkown, thus the sequence 
$
%\overline{\bM}_{m_n,1:n}\equiv
\overline{\bM}_{m_n,1:n}\equiv
(\overline{\bM}_{m_n,i})_{i=1}^n
%:=((\bM_{m_n,i}-\bb_{m_n})/\ba_{m_n})_{i=1}^n
$ 
is not directly available for approximate Bayesian inference on the limiting max-stable model. A common practice in extreme-value analysis is to fit a max-stable model directly to 
%$(\bM_{m_n,i})_{i=1}^n$ 
unnormalised maxima
ecompassing scale and location parameters, whose estimates ultimately absorbe $\ba_{m_n}$ and $\bb_{m_n}$  \citep[e.g,][]{r81}.
Following this approach, we consider the case where a misspecified semiparametric max-stable model as in Section \ref{sec:binf_general_max} is fitted to sample maxima, but replace the prior $\Pi_{\bvarTheta}$ on the finite dimensional model component with a data dependent prior sequence $\Psi_{n}$ of the following general form:
\begin{equation}\label{eq: psin}
\begin{split}
&\diff \Psi_n(\bvartheta)\propto\\
&\quad \begin{cases}
\diff\Pi_{\text{sh}}(\brho) \times \prod_{j=1}^d \pi_{\text{sc}} \left(\frac{\sigma_j}{\widehat{\sigma}_{n,j}}\right)\frac{\diff \sigma_j}{\widehat{\sigma}_{n,j}}, \hspace{12.8em} \text{if  } \bvartheta=(\brho, \bsigma)\\
\diff\Pi_{\text{sh}}(\bomega) \times \prod_{j=1}^d \pi_{\text{sc}} \left(\frac{\sigma_{n,j}}{\widehat{\sigma}_{n,j}}\right)\frac{\diff \sigma_j}{\widehat{\sigma}_{n,j}}
\times \prod_{j=1}^d  \pi_{\text{loc}}\left(
\frac{\mu_j-\widehat{\mu}_{n,j}}{\widehat{\sigma}_{n,j}}\right)\frac{\diff \mu_j}{\widehat{\sigma}_{n,j}}
, \hspace{0.85em} \text{if  } \bvartheta=(\bomega, \bsigma, \bmu) \\
\prod_{j=1}^d \pi_{\text{sc}} \left(\frac{\sigma_j}{\widehat{\sigma}_{n,j}}\right)\frac{\diff \sigma_j}{\widehat{\sigma}_{n,j}}
\times \prod_{j=1}^d  \pi_{\text{loc}}\left(
\frac{\mu_j-\widehat{\mu}_{n,j}}{\widehat{\sigma}_{n,j}}\right)\frac{\diff \mu_j}{\widehat{\sigma}_{n,j}}, \hspace{5.7em} \text{if  } \bvartheta=(\bsigma, \bmu) 
\end{cases}
\end{split}
\end{equation}
where $\widehat{\bsigma}=(\widehat{\sigma}_{n,1}, \ldots, \widehat{\sigma}_{n,d})$ and $\widehat{\bmu}=(\widehat{\mu}_{n,1}, \ldots, \widehat{\mu}_{n,d})$ are estimators of $\ba_{m_n}$ and $\bb_{m_n}$, $\Pi_{\text{sh}}$ is a pm with full support on a suitable subset of $(0, \infty)^d$ and $\pi_{\text{sc}}$ and $\pi_{\text{loc}}$ are Lebesgue probability densities whose properties are made precise in the following shares. Since priors on scale and location parameters should now incorporate information on the norming sequences, which is typically not available a priori, a genuinely subjective specification is hardly viable. In such a case, a data driven prior selection, also known as \textit{empirical Bayes}, is a popular approach in Bayesian analysis \citep[e.g.,][Sections 1 and 3]{r501}.
In this setup, we are interested in establishing 
%$Q_{n}$-almost sure 
asymptotic concentration properties of the pseudo-posterior distribution defined via
\begin{equation}\label{eq:pseudopost}
\Pi_n^{\scalebox{0.65}{\text{(o)}}}(B):=\frac{\int_B \prod_{i=1}^n g_\bvartheta(\bM_{m_n,i}|H)\diff(\Pi_{\Hset}\times\Psi_n)(H,\bvartheta)}{\int_{\Hset\times\bvarTheta} \prod_{i=1}^n g_\bvartheta(\bM_{m_n,i}|H)\diff(\Pi_{\Hset}\times\Psi_n)(H,\bvartheta)},
\end{equation}
for all $\Pi_{\Hset}\times \Psi_n$-measurable sets $B$, and of its counterpart $\tilde{\Pi}_n^{\scalebox{0.65}{\text{(o)}}}$ on the corresponding class of max-stable densities $\Gset_{\bvarTheta}$, where the superscript $\text{(o)}$ denotes dependence on observables. Moreover, we aim at establishing Hellinger consistency of the pseudo-predictive density $\hat{g}_n^{\scalebox{0.65}{(o)}}(\bx)=\int_{\Gset_{\bvarTheta}}g(\bx) \, \diff \postobsdens(g)$ as an estimator of the true probability density of unnormalised maxima  $
f_{m_n}^{\scalebox{0.65}{(o)}}(\bx)=(\partial^d/\partial \bx)F_0^{m_n}(\bx).
$
We point out that, under Condition \ref{cond:strong}, $f_{m_n}^{\scalebox{0.65}{(o)}}$ becomes topologically undistinguishable from the density of $G_{\bvartheta_0}((\cdot-\bb_{m_n})/\ba_{m_n}|H_0)$ as $n \to \infty$. Thus, we ultimately aim at showing that $\tilde{\Pi}_n^{\scalebox{0.65}{\text{(o)}}}$ cumulates an increasingly large fraction of its total mass near the latter.
%

%%%%%%%%%%%%%%%%%%%%%%%%%%%%%%%%%%%%%%%%%%%%%%%%%%%%%%%%%%%%%%%%%%%%%%%%%%%%%%
\subsection{Reparametrisation, remote contiguity}\label{sec:repar}
%%%%%%%%%%%%%%%%%%%%%%%%%%%%%%%%%%%%%%%%%%%%%%%%%%%%%%%%%%%%%%%%%%%%%%%%%%%%%%

To accomplish the objective above, we firstly provide an alternative representation of the posterior distributions under study. A change of variables in the integrals in \eqref{eq:pseudopost}
\begin{equation}\label{eq:reparam}
(H,\bvartheta) \mapsto \psi_n(\bvartheta)= 
\begin{cases}
(H,\brho, \bsigma/\ba_{m_n}), \hspace{8.4em} \text{if  } \bvartheta=(\brho, \bsigma),\\
(H,\bomega,  \bsigma/\ba_{m_n}, \{\bmu- \bb_{m_n}\}/\ba_{m_n}),\quad \text{if  } \bvartheta=(\bomega, \bsigma, \bmu),\\
(H,  \bsigma/\ba_{m_n}, \{\bmu- \bb_{m_n}\}/\ba_{m_n}),\hspace{2.2em} \text{if  } \bvartheta=( \bsigma, \bmu),\\\end{cases}
\end{equation}
corresponding to a change of parametrisation, yields the equality $\Pi_n=\postobs\circ \psi_n^{-1}$, where $\Pi_n$ is the pseudo-posterior defined via
\begin{equation}\label{eq:pseudopost}
\Pi_n(B):=\frac{\int_B \prod_{i=1}^n g_\bvartheta(\overline{\bM}_{m_n,i}|H)\diff(\Pi_{\Hset}\times\overline{\Psi}_n)(H,\bvartheta)}{\int_{\Hset\times\bvarTheta} \prod_{i=1}^n g_\bvartheta(\overline{\bM}_{m_n,i}|H)\diff(\Pi_{\Hset}\times\overline{\Psi}_n)(H,\bvartheta)},
\end{equation}
for every $\Pi_{\Hset}\times \overline{\Psi}_n$-measurable set $B$, where $\overline{\Psi}_n=\Psi_n\circ \psi_n^{-1}$. The latter also induces a pseudo posterior $\tilde{\Pi}_n$ on a corresponding class of max-sable densities which is in turn linked to $\postobsdens$ via the relation $\tilde{\Pi}_n=\postobsdens\circ \tilde{\psi}_n^{-1}$, where
\begin{equation}\label{eq:densrepar}
\tilde{\psi}_n \left( g_\bvartheta(\cdot|H) \right)= g_{\psi_n(\bvartheta)}(\cdot|H).
\end{equation}
Consequently, the required asymptotic analysis boils down to establishing consistency of $\Pi_n$ and $\tilde{\Pi}_n$ almost surely (or in probability) at $(H_0, \bvartheta_0)$ and $g_{\bvartheta_0}(\cdot|H_0)$, respectively, with respect to $Q_n$, the joint pm of $\overline{\bM}_{m_n,i:n}$. 
We stress that herein $\Pi_n$ and $\tilde{\Pi}_n$ are mathematical devices introduced to enable the study of the asymptotic behaviour of $\postobs$ and $\postobsdens$, but are not practical for statistical inference, as they depend on unobservables.
Under appropriate assumptions on the data dependent prior, $\Pi_n$ and $\tilde{\Pi}_n$ satisfy an inequality like \eqref{eq:mainbound}. Therefore, consistency can be obtained by establishing a variant of \eqref{eq:expobound} where the sample of normalised maxima $\overline{\bM}_{m_n,1:n}$ replaces the sample $\bX_{1:n}$ and the joint distribution of the former, $Q_n$, replaces the joint distribution of the latter, $\prodGtHtrue(\cdot|H_0)$. To do this, we resort to a specific form of remote-contiguity.

Recently, \cite{r12} has introduced a generalised form of contiguity.
% allowing for weaker {\color{magenta}forms} than the classical one.
%
The classical notion of contiguity can be successfully exploited to establish weak consistency (i.e. in probability) of the pseudo-posterior distribution with parametric limiting models \citep[e.g.,][]{r80}, while is unsuitable to obtain strong consistency and less accessible for nonparametric models.
See \cite[Section 3]{r12} for a comprehensive account.

\begin{definition}
Consider two sequences $\upsilon_n ,\tau_n>0$ such that as $n\to\infty$, $\upsilon_n,\tau_n \to 0$. 
As $n \to \infty$, $Q_{n}$ is said to be:
\begin{enumerate}
\item[i.] contiguous with respect to $\prodGtHtrue(\cdot|H_0)$ if, for a sequence of measurable events $E_n$, $\prodGtzerE=o(1) \implies Q_n(E_n)=o(1)$; 
\item[ii.] $\upsilon_n$-remotely contiguous with respect to $\prodGtHtrue(\cdot|H_0)$, if $\prodGtzerE=o(\upsilon_n)\implies$ $Q_n(E_n)=o(1)$;
\item[iii.] $\upsilon_n$-to-$\tau_n$-remotely contiguous with respect to $\prodGtHtrue(\cdot|H_0)$, if $\prodGtzerE=o(\upsilon_n) \implies Q_n(E_n)=o(\tau_n)$,
\end{enumerate}
where  ``$\implies$" denotes the usual implication symbol.
\end{definition}
Observe that remote contiguity prescribes specific forms of agreement between the joint law $Q_n$ of the pseudo-sample $\overline{\bM}_{m_n,1:n}$ and the joint law of a random sample of size $n$ from the limiting max-stable model. On the other hand, Condition \ref{cond:strong} only provides asymptotic guarantees on the behaviour of a single rv $\overline{\bM}_{m_n,i}$. 
%$\bM_{m_n,i}$, once standardised using the unkown theoretical sequences $\ba_{m_n},\bb_{m_n}$. 
We thus assume the following additional condition, allowing to lift the approximation in \eqref{eq:strongconv} to remote contiguity results for $Q_n$.
\begin{condition}\label{cond:densratio}
Let $f_{m_n}$ denote the density of $F_0^{m_n}(\ba_{m_n}\cdot+\bb_{m_n})$, then assume that there exists $n_0\in \nat_+$ and $J_0 \in(0,\infty)$ such that $\sup_{n \geq n_0}\Vert f_{m_n}/g_{\bvartheta_0}(\cdot|H_0)\Vert_\infty<J_0$.
\end{condition}
Conditions \ref{cond:strong} and \ref{cond:densratio} imply that the Kullback-Leibler divergence from the true limiting max-stable density
%from $g_{\bvartheta_0}(\by|\theta_0)$ 
to 
%$f_{m_n}$ 
the density of normalised maxima
is asymptotically null and the positive Kullback-Leibler variations between the two up to the fourth order \citep[e.g.,][equation (B.2)]{r10} are bounded uniformly over all large enough 
$n$. This yields polynomial tail bounds for the associated normalised log likelihood ratios, which underpin the general remote contiguity result established below (the proof is reported in Section \ref{sec:RCproof} of the supplement).
\begin{prop}\label{prop:RC}
	Let $F_0$ and $G_{\bvartheta_0}(\cdot|H_0)$ satisfy Conditions \ref{cond:strong} and \ref{cond:densratio}.
	Then $Q_{n}$ is $e^{-cn}$-to-$n^{-1-c'}$-remotely contiguous with respect to $\prodGtHtruealt(\cdot|H_0)$, $\forall c>0$, $\forall c'\in (0,1)$.
\end{prop}

\subsection{Fr\'echet domain of attraction}\label{sec:frec}
We start by analysing consistency in case where maxima are obtained from rv's whose distribution has heavy-tailed margins, which are therefore in the max-domain of attraction of a Fr\'echet distribution.
%
%We first deal with the case of heavy-tailed random variables, whose maxima are known to have distributions in the Fr\'echet domain of attraction. 
%
We recall that in this case the upper end-points of each 
marginal distribution is infinity. Without loss of generality, we restrict our attention to non-negative random variables (Condition \ref{cond:strong}\ref{margins}). Popular heavy-tailed distributions defined on the real line have symmetric tails \citep[Ch. 1--2]{{r202}}, such as Cauchy, Student-$t$, etc., and can be casted in our framework by considering their folded version. 
Then, we consider the scenario where a max-stable density as in \eqref{eq:dens_a_frec} is fitted to $\bM_{m_{n},1:n}$ and a data-dependent prior on $(\brho, \bsigma)$ is specified as in \eqref{eq: psin}. Our results  rely on the following conditions.
\begin{condition}\label{cond:frecextend}
Assume $\Pi_{\text{sh}}$ has full support on $(0, \infty)^d$ and:
\begin{inparaenum}	
\item\label{cond:pisc} The probability density function $\pisc$ satisfies $\{x \in \real: \pisc(x)>0 \}\subset(0,\infty)$; moreover, there exist $\eta \in (0,1)$ and a Lebesgue integrable continuous function  $u_{\text{sc}}:(0,\infty)\to (0,\infty)$ such that:
	\begin{inparaenum}
	\item\label{posit} $\pisc$ is continuous on $[1\pm \eta]$ and $\inf_{x \in [1\pm \eta]}\pisc(x)>0$;
	\item\label{piscbound} $\sup_{ t \in (1\pm \eta)}\pisc(x/t) \leq u_{\text{sc}}(x)$, for all $ x>0$.
	\end{inparaenum}
\item\label{cond:estscale}	$
\lim_{n\to 0}\widehat{\sigma}_{n,j}/a_{m_n,j}=1, \quad j=1,\ldots,d, \quad \iprodFtruealt-\text{as}$;
\item\label{cond:strongertruedens} if $d=2$ and the true angular density $h_0$ behaves as in Definition \ref{cond:mvt_angular}\ref{cond: infinite_new}, further assume there exists  $s>0$ such that $h_0^{1+s}$ is integrable on $(0,1)$. 
\end{inparaenum}
\end{condition}

Condition \ref{cond:frecextend}\ref{cond:pisc} is very mild. It is satisfied by regular density functions on the positive half-line (or suitable subsets), such as gamma, Pareto, half-Cauchy, etc. The proof of the following theorem is deferred to Section \ref{appsec:proof_theo_frecdom} of the supplement.
%In this setting, {\color{magenta}[...]}
% {\color{blue}the pseudo-posterior pm's
%%$ \bvarTheta=(0,\infty)^{2d}$, 
%$\Pi_n$ on $(0,\infty)^{2d}$ and $\tilde{\Pi}_n$ on $\{g_{\brho,\bsigma}(\cdot|H): \, H \in\Hset, (\brho, \bsigma)\in (0,\infty)^{2d}\}$, induced via the transformations $\psi_n$, in 
%the first line of 	\eqref{eq:reparam}, and $\tilde{\psi}_n$, in \eqref{eq:densrepar}, exhibit the following concentration properties.}
%%
\begin{theorem}\label{th:rem_cont_frec}
Let $\bM_{m_{n},1}, \ldots,\bM_{m_{n},n}$ be iid according to $F_0^{m_n}(\ba_{m_n} \cdot +\bb_{m_n})$. Assume $F_0$ and $G_{\brho_0,\bone}(\cdot|H_0)$ satisfy Conditions \ref{cond:strong} and \ref{cond:densratio}.  
%Let $\Psi_n$ be a Borel pm complying with \eqref{eq: psin}.
%
%, with $\Pi_{\emph{sh}}$ having full support on $(0, \infty)^d$, $\pisc$ as in Condition \ref{cond:pisc} and esitmators $\widehat{\bsigma}_n$ such that
%\begin{equation}\label{eq:estscale}
%	\lim_{n\to 0}\widehat{\sigma}_{n,j}/a_{m_n,j}=1, \quad j=1,\ldots,d, \quad \iprodFtruealt-\text{as}.
%\end{equation}
%Assume that $\Pi_\Hset$ satisfies Condition \ref{cond:angularprior}. 
%
Then, under Conditions \ref{cond:genprior}\ref{cond:angularpmprior} and \ref{cond:frecextend}, $\tilde{\Pi}_n$ and $\Pi_{n}$ satisfy the properties at points (a) and (c) of Theorem \ref{th:alpha_frec}, $\iprodFtruealt-\text{as}$ and: 
\begin{itemize}
	\item[(a$'$)]  $\lim_{n \to \infty}\altpostobsdens(\tilde{\mathcal{U}}_n^\complement)=0$, for every sequence of $\dist_H$-balls $\tilde{\mathcal{U}}_n$ centered at $g_{\brho_0,\ba_{m_n}}(\cdot|H_0)$ of positive fixed radius;
	\item[(b$'$)] 	$
	\lim_{n\to \infty}\dist_H(\hat{g}_n^{\scalebox{0.65}{\emph{(o)}}},g_{\brho_0,\ba_{m_n}}(\cdot|H_0))=0$;
	\item[(c$'$)] $\lim_{n\to\infty}\altpostobs((\mathcal{U}_{1}\times\mathcal{U}_{n,2})^\complement)=0$, for every $\dist_W$-neighborhood (if $d\geq 2$) or $\dist_{KS}$-neighborhood $\mathcal{U}_1$  of $H_0$ (if d=2) and every sequence of rectangles
	$\mathcal{U}_{2,n}=(\brho_0\pm \bone \epsilon)\times(\ba_{m_n}(1-\epsilon),\ba_{m_n}(1+\epsilon))$, $\epsilon>0$.
\end{itemize}
\end{theorem}
\begin{rem}\label{rem: strong}
The weak convergence result in \eqref{eq:doa} is often used as a justification for the following somewhat informal approximation
for $i=1, \ldots,n,$ as $n\to \infty$,
$$
\prob(\bM_{m_n,i}\leq \bx) \approx G_{\brho_0,\bone}(\bx/\ba_{m_n}|H_0)=G_{\brho_0, \ba_{m_n}}(\bx|H_0),
$$
and, hence, for using a family of max-stable distributions as an approximate statistical model for (unnormalised) sample maxima \citep[e.g.,][p. 48]{r4}. From a Bayesian perspective, Theorem \ref{th:rem_cont_frec}(a$'$) provides a rigorous mathematical justification for such a practice
under the strong domain of attraction  \eqref{eq:strongconv}, 
as it ensures that the pseudo-posterior distribution $\postobsdens$ asymptotically concentrates on a 
%$\dist_H$-neighborhood of
Hellinger-neighbourhood of
$g_{\brho_0, \ba_{m_n}}(\cdot|H_0)$, the density of $G_{\brho_0, \ba_{m_n}}(\cdot|H_0)$.
%
%pertaining to the distribution on the right-hand side of the above display.%	
%
Furthermore, Theorem \ref{th:rem_cont_frec}(b$'$)
%denoting the probability density of unnormalised maxima by $f_{m_n}^{\scalebox{0.65}{(o)}}(\bx)=(\partial^d/\partial \bx)F_0^{m_n}(\bx)$, 
guarantees that, almost surely as $n\to \infty$,
$$
\dist_H(\hat{g}_n^{\scalebox{0.65}{(o)}},f_{m_n}^{\scalebox{0.65}{(o)}}) \leq \dist_H(g_{\brho_0,\bone}(\cdot|H_0),f_{m_n})+
\dist_H(\hat{g}_n^{\scalebox{0.65}{(o)}},g_{\brho_0,\ba_{m_n}}(\cdot|H_0))=o(1),
$$
Thus, for large sample sizes $n$, the pseudo-predictive density $\hat{g}_n^{\scalebox{0.65}{(o)}}$ accurately estimates the true data generating density $f_{m_n}^{\scalebox{0.65}{(o)}}$.
As for the finite dimensional model components, Theorem \ref{th:rem_cont_frec}(c$'$) establishes that the marginal pseudo-posterior of the shape and scale parameters concentrates on a set of $(\brho, \bsigma)$ such that $|\rho_j - \rho_{0,j}|$ is small and $\sigma_{j}/a_{m_n,j}$ is close to one for any $j\in\{1, \ldots,d$\}, reminiscing the behaviour of the MLE in the frequentist approach \citep[Theorem 2]{r81}.
\end{rem}	
\begin{rem}\label{rem: eb}
Under Condition \ref{cond:frecextend}, the data dependent prior $\overline{\Psi}_{n}$, obtained via the reparametrisation in \eqref{eq:reparam}, is positive on a neighbourhood of the true limiting parameter $(H_0,\ \brho_0, \bone)$. This is a crucial property to obtain consistency of the pseudo-poseriors $\tilde{\Pi}_n$ and $\Pi_n$, which may fail in absence of the empirical Bayes rescaling in \eqref{eq: psin} (first line), through  estimators $\widehat{\boldsymbol{\sigma}}_n$
complying with Condition \ref{cond:frecextend}\ref{cond:estscale}.
Indeed, the norming sequence $a_{m_n,j}$, $j=1, \ldots,d$, in \eqref{eq:norming} diverges to infinity as $n \to \infty$. 
Therefore, a prior density $\pisc$ specified on each original scale parameter with no scale correction is after reparametrisation asymptotically null at the true value, since the integrability of $\pisc$ requires $\pisc(a_{m_{n,j}}1)a_{m_n,j}\to 0$.
%
%A crucial role in the consistency of the pseudo-poseriors $\tilde{\Pi}_n$ and $\Pi_n$ is played by the fact that under Condition \ref{cond:frecextend} the prior $\overline{\Psi}_{n}$, obtained after the reparametrisation in \eqref{eq:reparam}, is positive on a neighbourhood of the true limiting parameter $(H_0,\ \brho_0, \bone)$. 
%
%Recall from \eqref{eq:norming} that, as $n \to \infty$, the norming sequence $a_{m_n,j}$, $j=1, \ldots,d$, diverges to infinity.
%
%Thus, if a prior density $\pisc$ is specified on each original scale parameter with no data-dependent scale correction as in \eqref{eq: psin}, nothing prevents that after reparametrisation the prior density at the true value of the limiting scale parameter, i.e.
%$\pisc(a_{m_{n,j}}1)a_{m_n,j}$, is asymptotically null.
%the prior density after reparametrisation is $\pisc(a_{m_{n,j}} \, \cdot \, )a_{m_n,j}$. Given that $a_{m_n,j} \to \infty$ as $n \to \infty$, nothing prevents that at the at true value of the limiting scale parameters, i.e. $1$,  $\pisc(a_{m_{n,j}}1)a_{m_n,j}$ is asymptotically null.
%	
%This hihglights the importance of the empirical Bayes construction \eqref{eq: psin}, through  estimators $\widehat{\boldsymbol{\sigma}}_n$
%complying with Condition \ref{cond:frecextend}\ref{cond:estscale}, an example of which is provided next.
\end{rem}
\begin{example}\label{ex:est_frec}
Let $(\bZ_{i})_{i=1}^{n m_n}$ be an observable sample. For each $j=1, \ldots,d$, choose 
%If the sample $(\bZ_{i})_{i=1}^{n m_n}$ is observable, then choosing for each $j=1, \ldots,d$
%
$
\widehat{\sigma}_{n,j}=\widehat{F}_{n,j}^\leftarrow\left(1- 1/{m_n}\right),
$
where $\widehat{F}_{n,j}(x)=n^{-1}\sum_{i=1}^{nm_n}\indic(Z_{i,j}\leq x)$. Then, we have that for all $\epsilon>0$, as $n \to \infty$,
$$
\prodFmalt(\vert \widehat{\sigma}_{n,j}/a_{m_n,j} -1  \vert > \epsilon)\leq 4 e^{-\tau_j \sqrt{n+1}}, \quad j=1, \ldots,d,
$$
where $\tau_j \equiv \tau_j(\epsilon)$ are positive constants. A proof of the above inequality is provided in Section \ref{sec:ex_frec_est_proof} of the supplementary material. Thus, by Borel-Cantelli lemma, $\widehat{\bsigma}_n$ complies with Condition \ref{cond:frecextend}\ref{cond:estscale}. 
\end{example}

\begin{rem}\label{rem: weak}
	If the limiting relations in Condition \ref{cond:frecextend}\ref{cond:estscale} hold true in probability rather than almost surely,  the asymptotic results in Theorem \ref{th:rem_cont_frec} obtain in probability. Examples of estimators $\widehat{\boldsymbol{\sigma}}_n$ complying with such a weaker requirement can be obtained from those in \cite[pp. 130--131]{r202}
%	\cite{r81} 
	and \cite{r83}.
\end{rem}
Finally, we provide two examples of
bivariate models complying with Condition \ref{cond:densratio}.
Full derivations are provided in Sections \ref{sec: first_ex_frec}-\ref{sec:frecsecondex} of the supplement.
\begin{example}\label{ex:exp_dep_frec}
Consider 
$F_0(\bx)=1-1/x_1^{\rho_{0,1}}-1/x_2^{\rho_{0,2}}+1/\left(
x_1^{\rho_{0,1}}+x_2^{\rho_{0,2}}
 -1\right)$, for $\bx > \bone$, which is a slightly more general version of the distribution examined in \cite[p. 289]{r200}.
The distribution $F_0(\bx)$ has Pareto margins, i.e. $F_{0,j}(x_j)=1-1/x_j^{\rho_{0,j}}$, $j=1,2$, and belongs to the variational max-domain of 
 $$
 G_{\brho_0, \bone}(\bx|H_0)=\exp\left\lbrace
 -x_1^{-\rho_{0,1}}-x_2^{-\rho_{0,1}}
 +\left(
 x_1^{\rho_{0,1}}+x_2^{\rho_{0,2}}
 \right)^{-1}
 \right\rbrace, \quad \bx > \bzero,
 $$ 
whose angular distribution $H_0$ is the uniform on $[0,1]$.
\end{example}

\begin{example}\label{ex:log_dep_frec}
Consider
$$
F_0(\bx)=1-\left\lbrace x_1^{-3\rho_{0,1}}-x_2^{-3\rho_{0,2}}+
\left(
x_1^{\rho_{0,1}}x_2^{\rho_{0,2}}
\right)^{-3}\right\rbrace^{1/3},\quad \bx > \bone.
$$ 
This distribution has Pareto margins 
%$F_{0,j}(x_j)=1-1/x_j^{\rho_{0,j}}$, $j=1,2$, 
%
and Joe/B5 copula %$C_0$ 
\citep[e.g.,][p. 216]{r201}; it belongs to the variational max-domain of 
$$
G_{\brho_0, \bone}(\bx|H_0)=\exp\left\lbrace
-\left(x_1^{-3\rho_{0,1}}+x_2^{-3\rho_{0,1}} \right)^{1/3}
\right\rbrace, \quad \bx > \bzero,
$$ 
whose extreme-value copula is a member of the so-called logistic family \citep[e.g.,][p. 172]{r201}.
%$C_{EV}(\cdot|H_0)$ belongs to the logistic family.
%
\end{example}

\subsection{Gumbel domain of attraction}\label{sec:gumb}
%
%An asymptotic concentration result which parallels Theorem \ref{th:rem_cont_frec} can also be established when the marginal distributions are known to lie all in the Gumbel max-domain, in which case the max-stable density in \eqref{eq:densgumb} provides an approximate statistical model for $\bM_{m_{n},1:n}$. 
In this subsection we establish results similar to those in Theorem \ref{th:rem_cont_frec} for the case where the marginal distributions of $F_0$ belong to the Gumbel max-domain of attraction.
In this setup the max-stable density in \eqref{eq:densgumb} provides an approximate statistical model for $\bM_{m_{n},1:n}$. 
We assume that a data-dependent prior is assigned to $(\bsigma,\bmu)$, with prior densities $\pisc$
% satisfying Condition \ref{cond:pisc} 
and $\piloc$ and estimators $\widehat{\bsigma}_n$ and $\widehat{\bmu}_n$ complying with the following assumptions.
\begin{condition}\label{cond:gumbextend}
		$\pisc$ and $\widehat{\bsigma}_n$ satisfy Conditions \ref{cond:frecextend}\ref{cond:pisc} and \ref{cond:frecextend}\ref{cond:estscale}, respectively, and:
		\begin{inparaenum}
		\item
		\label{cond:piloc}
		there exist $\eta \in (0,1)$ and a Lebesgue integrable continuous function  $u_{\text{loc}}:\mathbb{R}\to (0,\infty)$ such that:
		\begin{inparaenum}
			\item\label{posit1} $\piloc$ is continuous on $[-\eta, +\eta]$ and $\inf_{x \in [-\eta, +\eta]}\piloc(x)>0$;
			\item\label{piscbound1} $\sup_{ t_1 \in (1\pm\eta), t_2 \in (-\eta, +\eta)}\piloc((x-t_1)/t_2) \leq u_{\text{loc}}(x)$, for all $ x>0$.
		\end{inparaenum}
		\item	\label{eq:estlocscale}
		$
		\lim_{n\to 0}(\widehat{\mu}_{n,j}-b_{m_n,j})/a_{m_n,j}=0,
		\quad j=1,\ldots,d, \quad \iprodFtruealt-\text{as}$.
		\end{inparaenum}
\end{condition}
%	We obtain the following concentration properties  for the pseudo-posterior distributions $\tilde{\Pi}_n$ and $\postobsdens$ on gumbel max-stable densities (linked by the transform $\tilde{\psi}_n$ in \eqref{eq:densrepar}) and
%	$\Pi_n$ and $\postobs$ on angular pm and finite-dimensional parameters (connected by the change of parametrisation in the third line of \eqref{eq:reparam}).
%
\begin{theorem}\label{th:rem_cont_gumb}
Let $\bM_{m_{n},1}, \ldots,\bM_{m_{n},n}$ be iid according to $F_0^{m_n}(\ba_{m_n} \cdot +\bb_{m_n})$, let $F_0$ and $G_{\bone, \bzero}(\cdot|H_0)$ satisfy Conditions \ref{cond:strong} and \ref{cond:densratio}. 
%Let $\Psi_n$ be a Borel pm complying with \eqref{eq: psin}.
% with $\pisc$ as in Condition \ref{cond:pisc}.
% $\piloc$ as in Condition \ref{cond:piloc} and esitmators $\widehat{\bsigma}_n$ and $\widehat{\bmu}_n$ such that
%\begin{equation}\label{eq:estlocscale}
%\lim_{n\to 0}\widehat{\sigma}_{n,j}/a_{m_n,j}=1, 
%\quad
%\lim_{n\to 0}(\widehat{\mu}_{n,j}-b_{m_n,j})/a_{m_n,j}=0,
%\quad j=1,\ldots,d, \quad \iprodFtruealt-\text{as}.
%\end{equation}
%Assume that $\Pi_\Hset$ satisfies Condition \ref{cond:angularprior}. 
%
Then, 
under Conditions \ref{cond:genprior}\ref{cond:angularpmprior}, \ref{cond:frecextend}\ref{cond:strongertruedens} and \ref{cond:gumbextend},  
$\tilde{\Pi}_n$ and $\Pi_{n}$ satisfy the properties at points (a) and (c) of Theorem \ref{cor:Gumbel_cons} $\iprodFtruealt-\text{as}$ and:
\begin{itemize}
	\item[(a$'$)]  $\lim_{n \to \infty}\altpostobsdens(\tilde{\mathcal{U}}_n^\complement)=0$, for every sequence of $\dist_H$-balls $\tilde{\mathcal{U}}_n$ centered at $g_{\ba_{m_n}, \bb_{m_n}}(\cdot|H_0)$ of positive fixed radius;
	\item[(b$'$)] 	$
	\lim_{n\to \infty}\dist_H(\hat{g}_n^{\scalebox{0.65}{\emph{(o)}}},g_{\ba_{m_n}, \bb_{m_n}}(\cdot|H_0))=0$;
	\item[(c$'$)] $\lim_{n\to\infty}\altpostobs((\mathcal{U}_{1}\times\mathcal{U}_{n,2})^\complement)=0$, for every $\dist_W$-neighborhood (if $d\geq 2$) or $\dist_{KS}$-neighborhood $\mathcal{U}_1$ of $H_0$ (if d=2) and every sequence of rectangles $\mathcal{U}_{2,n}=(\ba_{m_n}(1-\epsilon),\ba_{m_n}(1+\epsilon))\times (\bb_{m_n} \pm \epsilon \, \ba_{m_n})$, $\epsilon>0$.
\end{itemize}
\end{theorem}
The proof of Theorem \ref{th:rem_cont_gumb} is similar to that of Theorem \ref{th:rem_cont_frec} and is therefore omitted.
The main points of Remarks \ref{rem: strong}, \ref{rem: eb} and \ref{rem: weak} carry over to the present case and 
the asymptotic conditions imposed to $\widehat{\bsigma}_n$ and $\widehat{\bmu}_n$ are satisfied by empirical estimators, as illustrated by the following example.
\begin{example}\label{ex:gumb_est}
Define $\widehat{F}_{nm_n,j}$ as in Example \ref{ex:est_frec}, for $j\in\{1, \ldots,d\}$ and set
	\begin{equation*}
	\widehat{\mu}_{n,j}=\widehat{F}_{nm_n,j}^\leftarrow\left(1- \frac{1}{m_n}\right), \quad \widehat{\sigma}_{n,j}=m_n\int_{\widehat{\mu}_{n,j}}^\infty (1-\widehat{F}_{nm_n,j}(x)) \diff x.
	\end{equation*}
To satisfy Conditions \ref{cond:frecextend}\ref{cond:estscale} and \ref{cond:gumbextend}\ref{eq:estlocscale} it is sufficient that 
the following requirements are met: there exists $(\gamma_j,\alpha_j)\in(0, \infty)^2$ for every $j=1, \ldots,d$ such that
\begin{equation}\label{eq: exp_moment}
\int \exp(\gamma_j |z_j|^{\alpha_j})F_{0,j}(\diff z)< \infty
\end{equation}	
and for some $s \in (0,\min\{1,\alpha_1, \ldots ,\alpha_d\})$, as $n \to \infty$,
\begin{equation}\label{eq:rate}
		  {\log n}=
		  \begin{cases}
		  	  o\left(n a_{m_n,j}^2 /m_n\right), \hspace{7.3em} \text{if } \alpha_j\geq 1,\\
		  	  o\left(
		  	  \min\{n a_{m_n,j}^2 /m_n,
		  	  a_{m_n,j}^s n^s \}\right), \quad \text{if } \alpha_j<1.
		  \end{cases}
		\end{equation}
A proof is provided in Section \ref{sec:suppgumb} of the supplement, where we exploit the concentration inequalities for the 1-Wasserstein distances
$$
\int |F_{0,j}(z)-\widehat{F}_{nm_n,j}(z)|\diff z, \quad j=1, \ldots,d,
$$ 
recently derived by \cite{r699} under the exponential moment condition \eqref{eq: exp_moment}. The latter is satisfied by all the most common univariate distributions in the Gumbel max-domain of attraction, such as the exponential, Gaussian, Beta and log-normal distributions. Morover, for each $j\in\{1,\ldots,d\}$, the map $m_{n}\mapsto a_{m_n,j}$ in the third line of \eqref{eq:norming} is slowly varying at infinity, then \eqref{eq:rate} is easily obtained, for example, by requiring that $m_n \sim n^t$ as $n \to \infty$, with $t \in (0,1)$. 
\end{example}

We complete this section with an example of a bivariate distribution belonging to the Gumbel max-domain of attraction  
%
%$F_0$ and $G_{\bzero, \bone,\bzero}(\cdot|H_0)$ 
that satisfies Condition \ref{cond:densratio}. This is formally verified in Section \ref{appsec:gumbexamp} of the supplement.
\begin{example}\label{ex:mod_gumb}
Consider the bivariate exponential distribution \citep[e.g.,][]{r880}, i.e.
$$
		F_0(\bx)=1-e^{-x_1}-e^{-x_2}+\left(
		e^{x_1}+e^{x_2}
		-1\right)^{-1}, \quad \bx > \bzero,	
$$ 
with exponential margins $F_{0,j}(x_j)=1-e^{-x_j}$, $j=1,2$.
This distribution belongs to the variational max-domain of 
		$$
		G_{\bone, \bzero}(\bx|H_0)=\exp\left\lbrace
		-e^{-x_1}-e^{-x_2}
		+\left(
		e^{x_1}+ e^{x_2}
		\right)^{-1}
		\right\rbrace, \quad \bx \in \reald.
		$$ 
Herein, the distributions $F_0$ and $G_{ \bone, \bzero}(\bx|H_0)$ have the  the same copulas %$C_0$ and $C_{EV}(\cdot|H_0)$ 
of the distributions in Example \ref{ex:exp_dep_frec}.
\end{example}
%

%%%%%%%%%%%%%%%%%%%%%%%%%%%%%%%%%%%%%%%%%%%%%%%%%%%%%%%%%%%%%%%%%%%%%%%%%%%%%%
\subsection{Reverse Weibull domain of attraction}\label{sec:weibulldom}
%%%%%%%%%%%%%%%%%%%%%%%%%%%%%%%%%%%%%%%%%%%%%%%%%%%%%%%%%%%%%%%%%%%%%%%%%%%%%%

We finally study consistency in the case where a sample of maxima is obtained from random variables with distribution in the Weibull max-domain of attraction.
%
% tackle the case where the limiting distribution of normalised maxima is $\bomega$-Weibull. 
%
%For technical convenience, 
We consider an approximate statistical model  given by the max-stable density  \eqref{eq: weibdens} with $(\bomega, \bsigma, \bmu) \in (1,\infty)^d\times (0, \infty)^d \times \reald =: \bvarTheta$. 
For technical convenience, we assume the following conditions on the support of the prior distributions in \eqref{eq: psin}.
%
%we assume that an approximate statistical model as in \eqref{eq: weibdens} is adopted for sample maxima with $(\bomega, \bsigma, \bmu) \in (1,\infty)^d\times (0, \infty)^d \times \reald =: \bvarTheta$.
%, where $K_{\text{sh}}$ is a compact subset of $(1,\infty)^d$. 
%
%and we aditionaly assume on the support of the prior distributions used in \eqref{eq: psin}. Specifically, we assume the following.
%
\begin{condition}\label{cond:extendweib}
	$\pisc$ and $\piloc$ satisfy Conditions \ref{cond:frecextend}\ref{cond:pisc} and \ref{cond:gumbextend}\ref{cond:piloc}, respectively, and:
	\begin{inparaenum}
		\item\label{cond:compsupp1} there exist closed and bounded intervals $I_{\text{sc}} \subset (0, \infty)$ and $I_{\text{loc}} \subset \real$ such that $\{x \in \real: \pisc(x)>0 \}= I_{\text{sc}}$ and $\{x \in \real: \piloc(x)>0 \}= I_{\text{loc}}$;
		\item\label{cond:compsupp2} there exists a closed and bounded rectangle $K_{\text{sh}} \subset (1,\infty)^d$ such that $\text{supp}(\Pi_{\text{sh}})=K_{\text{sh}}$ and $\Pi_{\text{sh}}$ assigns positive mass to every neighbourhood of $\bomega_0$;
		\item \label{eq:estlocscaleweib}
		in $\prodFmalt$-probability as $n \to \infty$, 
		$\widehat{\sigma}_{n,j}/a_{m_n,j}\to1$, 
		and
		$(\widehat{\mu}_{n,j}-b_{m_n,j})/a_{m_n,j} \to 0$,
		for $j=1,\ldots,d.$
	\end{inparaenum}
\end{condition}
\begin{theorem}\label{theo:rem_weib}
Let $\bM_{m_{n},1}, \ldots,\bM_{m_{n},n}$ be iid according to $F_0^{m_n}(\ba_{m_n} \cdot +\bb_{m_n})$. Let $F_0$ and $G_{\bomega_0,\bone, \bzero}(\cdot|H_0)$ satisfy Conditions \ref{cond:strong} and \ref{cond:densratio}. 
%
%Let $\Psi_n$ be a Borel pm satisfying \eqref{eq: psin}.
% with $\Pi_{\emph{\text{sh}}}$ having full support on a compact set $K_{\emph{sh}} \subset(1, \infty)^d$ such that $\bomega_0 \in\mathring{K}_{\emph{sh}}$, $\pisc$ and $\piloc$  satisfying Conditions \ref{cond:pisc}, \ref{cond:piloc} and \ref{cond:compsupp}, and esitmators $\widehat{\bsigma}_n$ and $\widehat{\bmu}_n$ such that in $\prodFm$-probability as $n \to \infty$
%\begin{equation}\label{eq:estlocscaleweib}
%\widehat{\sigma}_{n,j}/a_{m_n,j}\to1, 
%\quad
%(\widehat{\mu}_{n,j}-b_{m_n,j})/a_{m_n,j} \to 0,
%\quad j=1,\ldots,d.
%\end{equation}
%Assume that $\Pi_\Hset$ satisfies Condition \ref{cond:angularprior}. 
%
Then, under Conditions \ref{cond:genprior}\ref{cond:angularpmprior}, \ref{cond:frecextend}\ref{cond:strongertruedens} and \ref{cond:extendweib},
$\tilde{\Pi}_n$ and $\Pi_{n}$ satisfy the properties at points (a) and (c) of Theorem \ref{cor:Gumbel_cons} in $\prodFm$-probability as $n \to \infty$ and:
\begin{itemize}
	\item[(a$'$)]  $\lim_{n \to \infty}\altpostobsdens(\tilde{\mathcal{U}}_n^\complement)=0$, for every sequence of $\dist_H$-balls $\tilde{\mathcal{U}}_n$ centered at $g_{\bomega_0,\ba_{m_n}, \bb_{m_n}}(\cdot|H_0)$ of positive fixed radius;
	\item[(b$'$)] 	$
	\lim_{n\to \infty}\dist_H(\hat{g}_n^{\scalebox{0.65}{\emph{(o)}}},g_{\bomega_0,\ba_{m_n}, \bb_{m_n}}(\cdot|H_0))=0$;
	\item[(c$'$)] $\lim_{n\to\infty}\altpostobs((\mathcal{U}_{1}\times\mathcal{U}_{n,2})^\complement)=0$, for every $\dist_W$-neighborhood (if $d\geq 2$) or $\dist_{KS}$-neighborhood $\mathcal{U}_1$  of $H_0$ (if d=2) and every sequence of rectangles 	$\mathcal{U}_{2,n}=(\ba_{m_n}(1-\epsilon),\ba_{m_n}(1+\epsilon))\times (\bb_{m_n} \pm \epsilon \, \ba_{m_n})$, $\epsilon>0$.
\end{itemize}
%
%{\color{red}{where $\altpostobs$ is the pm in \eqref{eq:pseudopost}, $\altpostobsdens$ is the pm induced by the latter on
%$\Gset_{\bvarTheta}:=\{ g_{\bomega,\bsigma,\bmu}(\cdot|H):\,H\in \Hset, (\bomega,\bsigma, \bmu)\in \bvarTheta \}$, $\bvarTheta=(1,\infty)^d\times (0, \infty)^d \times \reald $ and $\tilde{\Pi}_n=\altpostobsdens\circ \tilde{\psi}_n^{-1}$, with $\psi_n$ as in 
%the second line of 	\eqref{eq:reparam} and $\tilde{\psi}_n$ as in \eqref{eq:densrepar}}}.
\end{theorem}	
Theorem \ref{theo:rem_weib} provides convergence results which are valid in probability (see Section \ref{appsec:proof_weibdom} of the supplement for a proof). In order to obtain stronger forms of convergence as in Theorems \ref{th:rem_cont_frec} and \ref{th:rem_cont_gumb}, the moments of the positive part of the pseudo log likelihood ratios 
$$
 \log\{g_{\bomega_0, \bone, \bzero}(\bM_{m_{n},i}|H)/g_{\bomega, \bsigma, \bmu}(\bM_{m_{n},i}|H)\},
$$ 
$i=1, \ldots,n$, have to be bounded uniformly over a $(H,\bomega,\bsigma,\bmu)$-set of positive prior mass. For multivariate $\brho$-Fr\'echet and Gumbel models, the tractability of such quantities is guaranteed by Conditions \ref{cond:genprior}\ref{cond:angularpmprior} and \ref{cond:frecextend}\ref{cond:strongertruedens}. 
In the $\bomega$-Weibull model, the latter only guarantee control over the first moments, while the
higher order ones
are less tractable unless further restrictions are imposed on the shape parameters. See Section \ref{sec:preli_weib} and Lemma \ref{lem:pseudokulbweib} of the supplement for details. Accordingly, we only focus on estimators $\widehat{\bsigma}_n$, $\widehat{\bmu}_n$ satisfying the convergence properties in Condition \ref{cond:extendweib}\ref{cond:scaleloc} in probability, an example of which is provided next.
\begin{example}\label{ex: weibest}
Consider an iid observable sample $\bZ_1,\ldots, \bZ_{nm_n}$ from $F_0$. According to \cite[Ch. 4.2 and Ch. 4.5]{r202}, estimators $\widehat{\bsigma}_n$, $\widehat{\bmu}_n$ complying with Condition \ref{cond:extendweib}\ref{cond:scaleloc} can be constructed by selecting 
	$$
	\widehat{\sigma}_{n,j}=Z_{n(m_n-1),nm_n,j}{\xi}_{n,j,1}(1-\widehat{\gamma}_{n,j}^{-})/(-\widehat{\gamma}_{n,j}), \quad j=1,\ldots,d,
	$$
	and $\widehat{\mu}_{n,j}=Z_{n(m_n-1),nm_n,j} +\widehat{\sigma}_{n,j}$, for $j=1,\ldots,d$, where $Z_{s,j,nm_n}$ is the $s$-th order statistic of the marginal sample $Z_{1,j}, \ldots, Z_{nm_n,j}$ and 
$$
\widehat{\gamma}_{n,j}^{-}=\frac{\xi_{n,j,2}-2\xi_{n,j,1}^2}{2(\xi_{n,j,2}-\xi_{n,j,1}^2)},
$$
with
$$
	\xi_{n,j,l}=\frac{1}{n} \sum_{i=i}^{n-1}\left(
	\log Z_{nm_{n}-i,nm_n,j}-\log Z_{n(m_n-1),nm_n,j}
	\right)^l, \quad l=1,2,
	$$
	and $\widehat{\gamma}_{n,j}$ is an estimator satisfying $\widehat{\gamma}_{n,j}=-1/\omega_{0,j}+o_p(1)$ \citep[see][Ch. 3 for examples]{r202}.
%	and $\prob(\hat{\mu}_{n,j}\neq M_{m_n,i,j}, \, 1 \leq j \leq ,d,  1\leq i\leq n)=1$.
\end{example}

\section{Discussion}\label{sec:discussion}
%

%This work contributes to the development of Bayesian methods for max-stable models, by providing the first concentration results for semiparametric posterior distributions on extremal dependence and marginal max-stable distributions' parameters.
%
This work contributes to the development of Bayesian statistics, elaborating an asymptotic theory in the extreme value context. We provide the first concentration results for semiparametric posterior distributions stemming from max-stable stastical models.
Our asymptotic results cover two cases: (a) the data come exactly from a max-stable distribution; (b) the data generating distribution lies in a neighbourhood of max-stable one (misspecified max-stable model), which is the most realistic situation.
%
%These cover both the cases where the observables follow a max-stable models and where a misspecified max-stable model is fitted to sample maxima.
%
%Marginal distribution classes or attraction domains are assumed known and all of the same type, the three admissible types being Fr\'echet, Gumbel and (reverse) Weibull.
%
All margins of the well-specified or misspecified  max-stable model are assumed to correspond to only one of the  three admissible types: short-, light- or heavy-tailed. 
%
%classes or attraction domains are assumed known and all of the same type, the three admissible types being Fr\'echet, Gumbel and (reverse) Weibull.
%
%In applications, this is a nonrestrictive assumption as long as marginal tail behaviours are homogenous and known.
%
The assumption of a homogeneous tail behaviour among marginal distributions is unrestrictive for many real-world applications. Indeed, in environmental studies data are often modelled using distributions with short or light tails, justified on physical grounds. Heavy-tailed distributions are typically used in applications of financial and actuarial sciences, to account for the occurrence of few exceptional events.

Nevertheless, in some applications uncertainty on marginal tail behaviours has to be taken into account and heterogenous marginal tails may be found to provide a better fit to the data. 
%An extension to known heterogeneous marginal types is possible but left out, as it requires derivations which are more involved than the already quite technical ones we provide.
%
In this case a possible approach is to model the marginal distributions via the generalised extreme value distribution \citep[e.g.,][]{r2408}, with cumulative distribution function of the form
$$
G_{\gamma_j,\sigma_j,\mu_j}(x)=\exp\left( - \left(1+\frac{\gamma_j(x-\mu_j)}{\sigma_j}\right)^{1/\gamma_j}
\right),
$$
provided that $x>\mu_j-\sigma_j/\xi_j$, for $j=1, \ldots,d$.  
Within this setting, vectors of marginal parameters which are close in $L_1$ distance may correspond to joint max-stable distributions with mismatched supports and infinite Kullback-Leibler divergences.
This seems to hamper the construction of prior distributions with positive mass on Kullback-Leibler neighbourhoods of the true density function, rendering the most well established approaches to the study of posterior concentration hardly viable. 
Consequently, such a modelling framework is left out from the present work.
		
%
%This makes an approach to posterior concentration via Kullback-Leibler properties unfeasible.
%
%Such a problem can be circumvented using a model selection approach, where a prior on the model class of each marginal component is firstly specified, including Fr\'echet, Gumbel and Weibull as possible options, and then, conditionally on the latter, a prior is assigned to the corresponding parameters.
%
%If each type has positive membership probability a priori, Kullback-Leibler property can be established by adapting arguments in the supplement. Moreover, the possibility to capture tail behaviours induced by shape parameters of different signs within GEV modeling is retained.
%
%However, this approach may pose computational challenges in moderate-to-high dimensions and appears unsuitable to the analysis of (unnormalised) maxima, where prior specification for scale and locations parameters is more delicate.
%
 Bayesian computational methods for non- and semiparametric max-stable models in dimensions higher than two are still in their infancy.
Data augmentations schemes exploiting the spectral representation of max-stable rv's have recently been adopted in the parametric context \citep[e.g.,][]{r252}.
This seems to be promising for the derivation of efficient MCMC algorithms to sample from posterior distributions also in a Bayesian semiparametric approach.
Such a task requires a thourough investigation, as the infinite-dimensionality of the spaces to be explored make it sensibly more complex than in the parametric case.   
Therefore, a comprehensive discussion of computational aspects is deferred to future studies.
Nonetheless, our theoretical results draw mathematical guidelines for good Bayesian inferential practices, benefitting from frequentist large sample guarantees. In particular, we show that empirical Bayes prior specifications constitute an asymptotically principled approach for the prediction of extremes, allowing for adaptation to some unkown extremal features of the data.
We hope our work opens the way towards new the developments in the theory and practice of Bayesian analysis of extreme values.
\section*{Acknowledgements}
The authors are grateful to Michael Falk, Bas Kleijn and Anthony Davison for their valuable support and help. Simone Padoan is supported by the Bocconi Institute for Data Science and Analytics (BIDSA), Italy. 

%The first author was supported by NSF Grant DMS-??-??????.
%
%The second author was supported in part by NIH Grant ???????????.

%%  MR numbers will be added by VTeX.                      %%
%%                                                         %%
%%  Use \cite{...} to cite references in text.             %%
%%                                                         %%
%%%%%%%%%%%%%%%%%%%%%%%%%%%%%%%%%%%%%%%%%%%%%%%%%%%%%%%%%%%%%

%% if your bibliography is in bibtex format, uncomment commands:
%\bibliographystyle{imsart-number} % Style BST file (imsart-number.bst or imsart-nameyear.bst)
%\bibliography{bibliography}       % Bibliography file (usually '*.bib')

%% or include bibliography directly:

\newpage

%%%%%%%%%%%%%%%%%%%%%%%%%%%%%%%%%%%%%%%%%%%%%%
%% Supplementary Material, if any, should   %%
%% be provided in {supplement} environment  %%
%% with title and short description.        %%
%%%%%%%%%%%%%%%%%%%%%%%%%%%%%%%%%%%%%%%%%%%%%%
\begin{supplement}
\setcounter{page}{1}
This supplementary material document contains a discussion on additional results derived for bivariate simple max-stable distributions (see Section~\ref{app:technical}), a collection of auxiliary results (see Section~\ref{app:auxiliary}), the proofs of all theoretical findings presented both in the main paper and in this manuscript (see Section~\ref{appsec:proofs}).
\appendix 
\section{Notation}\label{appsec:notation}
In this section provide some additional notations that are used thoughout the supplement, along with those in Section \ref{sec:notation} of the main article.

Let $\sreal\subset\R$ and $f\in C^1(\sreal)$ with an absolutely continuous first derivative. We denote the latter by $f'$ and its weak derivative by $f''$ \citep[e.g.][p. 22]{r99}. 
If $f$ is twice continuously differentiable, then $(\partial^2/\partial x^2)f(x)=f''(x)$ almost everywhere; therefore, in such a case, we consider the continuous representative given by the strong (i.e. the canonical) second derivative. 
We define the functionals
$$
\|f\|_{1,\infty}:=\|f\|_\infty+ \|f'\|_\infty, \quad \|f\|_{2,\infty}:=\|f\|_{1,\infty}+ \|f''\|_\infty,
$$
and the associated metrics $\dist_{p,\infty}(f,g):=\|f-g\|_{p,\infty}$, $p=1,2$, with $g:\sreal \rightarrow \real$ sharing the same properties of $f$. 
If, instead, $f$ is only weakly differentiable, then $f'$ denotes denotes its weak derivative.
For a generic space $\sreal$, endowed with a metric $\dist$, $x_* \in \mathcal{X}$ and $\epsilon>0$ we denote
$$
B_{\epsilon, \dist}(x_*)=\{ x \in \sreal: \, \dist(x,x_* )\leq  \epsilon \};
$$
for the sake of a lighter notation, when $\dist$ is an $L_p$-distance, $p \in  [1, \infty]$, we write $B_{\epsilon, p}(\cdot)$.
With a little notational abuse, when $H_* \in \Hset$, with $\Hset$ given in Definition \ref{cond_angular} of the main article, we also use the symbol
$$
B_{\epsilon, \infty}(H_*):=\{H \in \Hset: \, \Vert
h-h_*
\Vert_\infty \leq \epsilon
\},
$$
where $h$ and $h_*$ denote the angular densities of $H$ and $H_*$, respectively.
Moreover, for any $B \subset \sreal$, we denote by $\cover(\delta, B, \dist)$ be the $\delta$-covering number of a set $B$ with respect to the metric $\dist$ \citep[e.g.,][Appendix C]{r10}. Finally, for pairs of pm's  $F$ and $G$ with density functions $f$ and $g$ with respect to some dominating measure $\nu$ on $\mathcal{X}$, we denote the higher-order positive Kullback-Leibler divergences \citep[e.g.,][Appendix B]{r10} 
%{\color{magenta}Ghosal and van der Vaart 2017}) 
by
$$
\kulb_+^{(l)}(f,g)=\int_\mathcal{X}
\left[
\log^+\{f(x)/g(x)\}
\right]^l
f(x) \diff \nu(x),
$$
where $l \in \nat_+$ and, for all $y>0$, $\log^+(y)=\max\{\log(y),0\}$.  

\section{Bivariate simple max-stable distributions}
\label{app:technical}
This section complements Sections  \ref{sec:MGEV}--\ref{sec:bayesian_inference} of the main paper, by providing additional notions and results on representations of the extremal dependence and the related Bayesian inference in the bivariate case.
Given a Pickands dependence function $A \in \Aset$ (Definition \ref{cond_angular} of the main article), when $d=2$ the pertaining simple max-stable density function (equation \eqref{eq:density_unit_fre} of the main article)
simplifies to
\begin{equation}\label{eq:bdens_fre}
g_{\bone}(\by|A)=G_{\bone}(\by|A)\left(\frac{(A(t)-tA'(t))(A(t)+(1-t)A'(t))}{(y_1y_2)^2}+\frac{A''(t)}{(y_1+y_2)^3}\right),
\end{equation}
where $\by>\bzero$ and $t\equiv t(\by)=y_1/(y_1+y_2)$.
Moreover,
conveniently setting $p_0=H(\{\be_2\})$, the pertaining angular cumulative distribution function (cdf) is of the form 
$$
H(t)=p_0+\int_0^th(v)\diff v + p_1\delta_{1}([0,t]),\quad t\in[0,1].
$$
Consequently, $A$ and its derivatives are related to $H$ and $h$ through the explicit formulas  
\begin{eqnarray}
\label{eq:pick}
A(t)&=&1+2\int_0^t H(v)\diff v - t, \quad t \in[0,1],\\
\label{eq:pickdev}
A'(t)&=&-1+2H(t), \hspace{4.7em} t \in (0,1),\\
\label{eq:pick_2dev}
A''(t)&=&2h(t),\hspace{7.5em} \text{for almost every } t \in (0,1).
\end{eqnarray}
Furthermore, $A'(0)=2p_0-1$ and $A'(1)=1-2p_1$, where, with an abuse of notation, $A'$
also denotes the continuous extension of the first derivative of $A$ on $[0,1]$. We stress that the notation $p_0$ should not be confused with a true parameter value, used to actually generate the data. In the sequel, the point masses pertaining to the true bivariate angular pm $H_0$ are denoted by $p_{0,0}$ and $p_{0,1}$, to avoid ambiguities.

The relations in \eqref{eq:pick}-\eqref{eq:pickdev} makes simultaneous polynomial modeling of the Pickands dependence and the angular distribution functions very tractable, as illustrated in the next subsection. In Section \ref{app:review_BP}, we review a Bernstein polynomial representation of the Pickands dependence and the angular distribution functions which, in the specific case $d=2$, corresponds to the Bernstein polynomial representation of the angular density in Section \ref{sec:extreme_dep_poly} of the main article.
In Section \ref{app:review_BS} we provide a novel characterization via B-splines.
Both constructions can be used for prior specifications on the extremal dependence yielding consistent posterior distributions, as shown in Section \ref{sec:posterior_consistency_2D}.
The consistency results provided therein build on the general theory on Kullback-Leibler property presented in Section \ref{app:KL_support}. 
The latter rephrases the analysis in Section \ref{sec:KL_support} of the main article from the perspective where uncertainty on the extremal dependence structure is dealt with by assigning a prior to the Pickands dependence function.	
We recall that the regularity conditions (C.2)-(C.3) in the main article univocally identify the class of valid Pickands dependence functions only when $d=2$. 
Therefore, a direct specification of a prior with full support on the space of Pickands dependence functions is essentially tractable only in the bivariate case.  
From a practical viewpoint, the Pickands dependence function provides an easy-to-interpret description of the dependence structure of a bivariate max-stable rv and allows to readily retrieve some of the most commonly used extremal dependence measures. Examples are the extremal coefficient, the coefficient of upper tail dependence and Spearman's rho for extreme-value copulas, given by $2A(1/2)$,  $2-2A(1/2)$ and
$$
12\int_0^1\frac{1}{\{1+A(t)\}^2}\diff t-3,
$$
respectively, see Ch. 8.2.7 in \cite{r17} for a comprehensive account. 
From a mathematical stance, we highlight that the case $d=2$ is the only one where the map  $(\Aset,\dist_{\infty})\mapsto(\Hset, \dist_{KS}):A \mapsto H$ is homeomorphic (see also Proposition \ref{prop: basic metric bvt}). Such a relation between metric spaces is underpinned by equation \eqref{eq:pick_2dev} and the convexity of the Pickands dependence function. 
In inferential terms, this guarantees that $\dist_{\infty}$-consistency at the true Pickands dependence function directly translates into $\dist_{KS}$-consistency at the true angular pm, while an equivalent implication fails in higher-dimensions. 
For technical convenience, the class $\Aset$ is hereafter endowed with the equivalent metric $\dist_{1,\infty}$ and consistency results are  provided with respect to the latter,
paralleling the analysis in
Section \ref{sec:posterior_consistency} of the main article.
%

%In the bivariate case a direct prior specification on $A$ is fairly tractable. {\color{red}Therefore} we focus our discussion on Pickands dependence function's prior distribution.

%
\subsection{Polynomial representation of the extremal dependence}
\label{app:review}
\subsubsection{Bernstein form}
\label{app:review_BP}
According to 
%\cite{r24}, 
\cite{r24},
for $k \geq 2$, a $k-1$-th degree Bernstein polynomial representation of the angular cdf is given by 
\begin{equation}\label{eq:bpoly_angdist}
H_{k-1}(t):=
\begin{cases}
\sum_{j=0}^{k-1} \eta_j k^{-1}\betaf(t|j+1,k-j), & \text{ if } t\in[0,1)\\
1, &  \text{ if } t=1
\end{cases}
\end{equation}
where $\betaf(\cdot|a,b)$ denotes the beta density function with shape parameters
$a,b>0$ and, for $0\leq j\leq k-1$, $k^{-1}\betaf(t|j+1,k-j)$ is the $j$-th Bernstein basis function. If the polynomial's coefficients $\eta_0, \ldots, \eta_{k-1}$ satisfy the restrictions:
\begin{enumerate}
\item[(R3)] $0\leq p_0=\eta_0\leq \eta_1\leq \ldots \leq \eta_{k-1}=1-p_1\leq 1$,
\item[(R4)] $\eta_0+\cdots+\eta_{k-1}=k/2$,
\end{enumerate}
where $0\leq p_0,p_1\leq 1/2$, the function in \eqref{eq:bpoly_angdist} is a valid angular cdf.
Similarly, the Bernstein polynomial of degree $k$, for $k=2,3,\ldots$, given by 
\begin{equation}\label{eq:bpoly_picka}
A_{k}(t) := \sum_{j=0}^k \beta_j (k+1)^{-1}\betaf(t|j+1,k-j-1),
\qquad t \in[0,1],
\end{equation}
defines a valid Pickands dependence function if 
its coefficients satisfy the restrictions:
\begin{enumerate}
\item[(R5)] $\beta_0=\beta_k=1\ge \beta_j,$ for all $j=1,\ldots,k-1$;
\item[(R6)] $\beta_1=\frac{k-1 + 2 p_0}{k}$ and $\beta_{k-1} = \frac{k-1 + 2p_1}{k}$;
\item[(R7)] $\beta_{j+2}-2\beta_{j+1}+\beta_j\geq 0$, $j=0,\ldots,k-2$.
\end{enumerate}
Thus, for $k=2,3,\ldots$,  we define the classes of  Pickands dependence functions
and angular pm's with cdf's 
\textit{in Bernstein polynomial (BP) form} via
\begin{eqnarray*}
	\Aset_k&=&\{A_k \in \Aset:\, A_k(t)=\sum_{j=0 }^k\frac{\beta_j}{k+1}\betaf(t|j+1,k-j-1),\,  \text{(R5)-(R7) hold true}\},\\
	\Hset_{k-1}&=&\{ H_{k-1}\in \Hset: H_{k-1}(t)= \sum_{j=0}^{k-1} \frac{\eta_j}{k}\betaf(t|j+1,k-j),
	\,\text{(R3)-(R4) hold true}\}.
\end{eqnarray*}
Notably, for each $A_k\in\Aset_k$ it is possible to derive a polynomial $H_{k-1}\in \Hset_{k-1}$
and vice versa, by means of precise relationships between the two polynomials' coefficients, see \cite[][Proposition 3.2]{r24} for details. Moreover, by arguments in \cite[][Propositions
3.1-3.3]{r24},   $\cup_{k=2}^{\infty}\Aset_{k}$ and $\cup_{k=2}^{\infty}\Hset_{k-1}$ are dense subsets of
the spaces $(\Aset, \dist_{1, \infty})$ and $(\Hset, \dist_{KS})$, respectively.

Concluding, we point out that, for any $t \in (0,1)$, the angular density corresponding to the distribution in \eqref{eq:bpoly_angdist} is given by
\begin{equation}\label{eq:polyrep_dens}
	h_{{k-2}}(t)=\sum_{j=1}^{k-1}(\eta_{j}-\eta_{j-1})\text{Be}(t|j,k-j).
\end{equation}
Letting $\varphi_{\bkappa_2}=\eta_0$, $\varphi_\balpha=\eta_{\alpha_1}-\eta_{\alpha_1-1}$, $\balpha = (j,k-j)$, $j=1, \ldots, k-1$,  and
$\varphi_{\bkappa_1}=1-\eta_{k-1}$, we have that if $\eta_0,\ldots,\eta_{k-1}$ satisfy (R3)-(R4), then $(\varphi
_{\bkappa_1}, \varphi
_{\bkappa_2}, \varphi_\balpha, \, \balpha \in \Gamma_k)$ satisfy
(R1)-(R2) and the Bernstein polynomial representation in Section \ref{sec:extreme_dep_poly} of the main article is retrieved.

\subsubsection{Piecewise polynomial representation with B-splines}
\label{app:review_BS}

Piecewise polynomial representations of the angular cdf and the Pickands dependence function on $[0,1]$ as linear combinations of B-splines can be obtained as follows.
For $\kappa\geq 1$ and $m\geq 2$ we define the interior knots as the sequence
$0<\tau_{m+1}<\cdots<\tau_{m+\kappa}<1$ and
the exterior knots by $\tau_1=\cdots=\tau_m=0$ and $\tau_{m+\kappa+1}=\cdots=\tau_{2m+\kappa}=1$.
Consider the B-spline basis of order $m$ defined by the recursive formula \citep{r35}
$$
\phi_{j,i}(t)=\frac{t-\tau_j}{\tau_{j+i-1}-\tau_j}\phi_{j,i-1}(t)+\frac{\tau_{j+i}- t}{\tau_{j+i}-\tau_{j+1}}\phi_{j+1,i-1}(t),\quad t\in[0,1],
$$
for $1< i\leq m$ and $1\leq j\leq 2m+\kappa-i$, starting with
$$
\phi_{j,1}(t)=
\begin{cases}
1, & \tau_j \leq t < \tau_{j+1},\\
0, & \text{otherwise},
\end{cases}
$$
for $1\leq j \leq  2m+\kappa-1$. 
Given $\kappa$ points in $(0,1)$, a spline of order $m$, with knot sequence defined as above, can be expressed
as the linear combination of $m+\kappa$ piecewise polynomials of degree $m-1$ (B-spline basis functions) and has $m-1$ continuous derivatives at each of the $\kappa$ interior points.
Let
\begin{equation}\label{eq:bspline_angdist}
H_{k-1}(t)=
\begin{cases}
\sum_{j=1}^{k-1}\eta_j\phi_{j,m-1}(t), & t\in[0,1),\\
1, & t=1,
\end{cases}
\end{equation}
be a spline of order $m-1$, whose basis consists of $k-1=m+\kappa-1$ piecewise polynomials of degree $m-2$.
Notice that taking the first derivative of \eqref{eq:bspline_angdist} with respect to $t$ we obtain
the following spline of order $m-2$
\begin{equation}\label{eq:bspline_angdens}
H_{k-1}'(t)=\sum_{j=1}^{k-2}\left(\frac{(m-2)(\eta_{j+1}-\eta_{j})}{\tau_{j+m-1}-\tau_{j+1}}\right)\phi_{j+1,m-2}(t),\quad t\in(0,1).
\end{equation}
Similarly, let
\begin{equation}\label{eq:bspline_pick}
A_k(t)=\sum_{j=1}^k\beta_j\phi_{j,m}(t), \quad t\in[0,1],
\end{equation}
be a spline of order $m$, whose basis consists of $k=m+\kappa$ piecewise polynomials of degree $m-1$.
Notice that
the first two derivatives of \eqref{eq:bspline_pick} with respect to $t$ gives the following splines of order $m-1$ and $m-2$, respectively,
\begin{equation}\label{eq:bspline_fder_pick}
A_{k}'(t)=\sum_{j=1}^{k-1}\beta_{j,1}\phi_{j+1,m-1}(t),\quad A_{k}''(t)=\sum_{j=1}^{k-2}\beta_{j,2}\phi_{j+2,m-2}(t),
\end{equation}
where $t \in (0,1)$ and
\begin{equation*}
\beta_{j,s}=\frac{(m-s)(\beta_{j+1,s-1}-\beta_{j,s-1})}{\tau_{j+m+1-s}-\tau_{j+1}}, \quad s=1,2,
\end{equation*}
with $\beta_{j,0}\equiv\beta_j$ for $j=1,2,\ldots,k-1$.

The next result provides necessary and sufficient conditions on the splines' coefficients in \eqref{eq:bspline_angdist} and \eqref{eq:bspline_pick} to define valid angular cdf's and Pickands dependence functions of order $2$ and $3$, respectively. The proof is provided in Section \ref{sec:proofPropSP1}.
We found that extending such a result to higher orders is less tractable, while nonessential for practical statistical purposes.
Hereafter, for any given integer $k>3$, we fix the sequence of internal knots $(i/(k-2), i=1,\ldots,k-3)$.
We refer to the functions \eqref{eq:bspline_angdist} and \eqref{eq:bspline_pick}, satisfying the restrictions of Proposition \ref{prop:bspline_cond}, as the angular cdf and Pickands dependence function \textit{in B-Spline (BS) form}.
\begin{prop}\label{prop:bspline_cond}
For $m=3$ and $k\geq 4$, the spline  $H_{k-1}(t)$ in \eqref{eq:bspline_angdist} 
is a valid angular cdf if and only if the spline's coefficients satisfy the restrictions:
\begin{enumerate}
\item[(R8)] $0\leq p_0=\eta_1\leq \eta_2\leq \ldots \leq \eta_{k-1}=1-p_1\leq 1$; 
\item[(R9)] $\eta_1+2(\eta_2+\cdots+\eta_{k-2})+\eta_{k-1}=(k-2)$;
\end{enumerate}
where $0\leq p_0,p_1\leq 1/2$.
Likewise, the spline $A_{k}(t)$  in \eqref{eq:bspline_pick} is a valid Pickands dependence function if and only if the spline's coefficients satisfy the restrictions:
\begin{enumerate}
\item[(R10)] $\beta_1=\beta_k=1\ge \beta_j$, for all $j=2,\ldots,k-1$;
\item[(R11)] $\beta_2=1+(p_0-1/2)/(k-2)$ and $\beta_{k-1} = 1+(p_1-1/2)/(k-2)$;
\item[(R12)] $\beta_3 -3\beta_2+2\beta_1\geq0$, $2\beta_k - 3\beta_{k-1}+\beta_{k-2}\geq 0$ and 
$\beta_{j}-2\beta_{j-1}+\beta_{j-2}\geq 0$, $j=4,\ldots,k-1$.
\end{enumerate}
\end{prop}
Similarly to the Bernstein polynomial case, the classes of angular cdf's and
Pickands dependence functions in BS form are related via their coefficients, as described in the following proposition. See Section \ref{sec:proofBS2} for its proof. 
\begin{prop}\label{prop:bspline_rel_pick_ang}
For $m=3$ and $k \geq 4$, let $H_{k-1}(t)$ be  defined as in \eqref{eq:bspline_angdist}. Define 
$$
\eta_j=\frac{1}{2}+\frac{\beta_{j+1}-\beta_{j}}{\tau_{j+2}-\tau_j},\quad j=1,\ldots,k-1,
$$
where $\beta_1,\ldots,\beta_k$ satisfy the restrictions (R10)-(R12) and
$(\tau_j, j=1,\ldots,k+1)=(0,0, 1/(k-2), \ldots, (k-3)/(k-2),1,1)$.
Then, $H_k(t)$ is a valid angular cdf.
%the pertaining pm is a valid angular pm, i.e. $H_{k-1}\in \Hset$.
Similarly, define $A_{k}(t)$ as in \eqref{eq:bspline_pick} and set
$$
\beta_1=1, \quad \beta_j=\sum_{i=1}^{j-1}(\eta_i-1/2)(\tau_{i+3}-\tau_{i+1})+1,\quad j=2,\ldots,k,
$$
where $\eta_1,\ldots,\eta_{k-1}$ satisfy the restrictions (R8)-(R9) and 
$(\tau_j, j=1,\ldots,k+3)=(0,0,0,1/(k-2), \ldots,(k-3)/(k-2), 1, 1, 1)$.
Then, $A_k(t)$ is a valid Pickands dependence function.
\end{prop}
Finally, we show that the elements of 
$\Aset$ and $\Hset$ are well approximated by Pickands dependence functions and angular pm's with cdf's in BS form, respectively. For the proof, see Section \ref{sec:proofBS3}.
\begin{prop}\label{prop:bspline_full_supp}
For $m=3$, define for each $k \geq 4$
\begin{eqnarray*}
\Aset_k&:=&\{ A_k\in \Aset: \, A_k(t)=\sum_{j=1}^k\beta_j\phi_{j,m}(t), \,  \text{(R10)-(R12) hold true}\},\\
\Hset_{k-1}&:=&\{H_{k-1} \in \Hset: \, H_{k-1}(t)= \sum_{j=1}^{k-1}\eta_j\phi_{j,m-1}(t), \,
\text{(R8)-(R9) hold true}\}.
\end{eqnarray*}
Then, $\cup_{k=4}^{\infty}\Aset_{k}$ and $\cup_{k=4}^{\infty}\Hset_{k-1}$ are dense subsets of
$(\Aset, \dist_{1, \infty})$ and  $(\Hset, \dist_{KS})$, respectively.
\end{prop}
%

%%%%%%%%%%%%%%%%%%%%%%%%%%%%%%%
\subsection{Kullback-Leibler theory for priors on the extremal dependence}\label{app:KL_support}
%%%%%%%%%%%%%%%%%%%%%%%%%%%%%%%
%
The Pickands dependence functions in $\Aset$  
are continuous and differentiable, with absolutely continuous first derivatives. Accordingly, we denote with 
$\Pi_{\Aset}(B)=\prob(A\in B)$ a prior distribution on the Borel sets $B$ of $(\Aset, \dist_{1,\infty})$. We have the following set theoretical results (the referenced definitions and equations are given in the main paper).
\begin{prop}\label{prop:Polish_d2}
Let $\mathbb{W}^{1,\infty}((0,1))$ be the Sobolev space of bounded functions with bounded weak derivative on $(0,1)$, endowed with $\dist_{1, \infty}$. Then, in dimension $d=2$,  $\Aset$ is a closed subset of $\mathbb{W}^{1,\infty}((0,1))$. In particular, $(\Aset, \dist_{1, \infty})$ is a Polish subspace.
\end{prop}
\begin{cor}\label{cor:measurable}
	In dimension $d=2$, there exists a version of the simple max-stable density such that the map $(A,\by)\mapsto g_\bone(\by|A)$ is jointly measurable and $\phi_\Aset:(\Aset, \dist_{1,\infty})\mapsto (\mathcal{G}_\bone,\dist_H):A \mapsto g_\bone(\cdot|A)$ is a Borel map, where $\Gset_\bone$ is defined as in \eqref{eq:dens_space} with $\Theta=\Aset$. Moreover, for all $\epsilon>0$ and $\mathcal{K}_\epsilon$ as in Definition \ref{defi:KL}, $\phi_\Aset^{-1}(\mathcal{K}_\epsilon)$ is a Borel set of $(\Aset, \dist_{1, \infty})$.
\end{cor}
By Corollary \ref{cor:measurable}, $\Pi_\Aset$ induces a prior $\Pi_{\Gset_\bone}$ on the Borel sets of $(\Gset_\bone,\dist_H)$. We next give conditions under which the true density is in the Kullback-Leibler support of the latter. In particular, we make use of the following notion.
\begin{definition}\label{defi:d_2inf}
	Let $A_*\in\Aset$ and $\Vert A_* \Vert_{2, \infty}<\infty$. The prior $\Pi_\Aset$ is said to posses the $\dist_{2,\infty}$ property at $A_*$ if it has positive inner probability on the sets $\{A\in\Aset:\dist_{2,\infty}(A,A_*)\leq \epsilon\}$, for all $\epsilon>0$. 
\end{definition}
To establish the Kullback-Leibler property at $A_0$ for a prior $\Pi_\Aset$, we postulate twice continuous differentiabilty of the true Pickands dependence function on $(0,1)$ and impose mild restrictions on the behaviour of the second derivative near the boundary.
% assume that {\color{red}the true Pickands dependence function is associated to an angular density which is continuous on $(0,1)$ and either:
%has finite limits at the vertices or is bounded away from zero and diverges at the vertices.} 
The resulting class $\Aset_0$ of admissible true Pickands dependence functions (Definition \ref{cond_density}\ref{cond:true_set}) is rich enough for applications, since it includes many well known parametric models, such as Symmetric Logistic, Asymmetric Logistic, Hulser-Reiss, Pairwise-Beta, Dirichlet \citep[see, e.g.,][]{r19}. 
We also point out that, by the fundamental theorem of calculus, $\Aset_0$ coincides with the class of Pickands dependence functions whose angular pm's lie in the class $\Hset_0$, given in Definition \ref{cond:mvt_angular} of the main paper.
As we claim next, it suffices to check that $\Pi_\Aset$ posseses the $\dist_{2,\infty}$-property at the Pickands dependence  functions in a subclass (Definition \ref{cond_density}\ref{cond:prior_set}).
Loosely speaking, the use of the stronger metric $\dist_{2,\infty}$ is justified by the fact that the map $f \mapsto \kulb(g_\bone(\cdot|A),g_\bone(\cdot|f))$ on $(\Aset,\dist_{1,\infty})$ is measurable for any fixed $A\in \Aset$ but not necessarily continuous.
\begin{definition}\label{cond_density}
Let $\Aset$ be as in Definition \ref{cond_angular} of the main paper and: 
%(Definition \ref{cond_angular})
%
\begin{inparaenum}
\item \label{cond:true_set} Let $\Aset_0\subset \Aset$ be the class of twice continuously on $(0,1)$  (strongly) differentiable Pickands dependence functions $A$, whose second derivative satisfies one of the following:
\begin{inparaenum}
\item \label{cond: finite}
$0 \leq \lim_{t\downarrow0}A''(t)<+\infty$ and $0\leq \lim_{t\uparrow1}A''(t) <+\infty$;
\item \label{cond: infinite} $\inf_{t\in(0,1)}A''(t)>0$ and $\lim_{t\downarrow0}h(t)=\lim_{w\uparrow1}h(t) =+\infty$.
\end{inparaenum}
\item \label{cond:prior_set} Let $\Aset'\subset\Aset_0$ be a class of Pickands dependence functions whose extended first derivative satisfies $A'(0)>-1$, $A'(1)<1$, and whose second derivative
complies with $\inf_{t\in(0,1)}A(t)>0$, together with property \ref{cond: finite}.
\end{inparaenum}
\end{definition}
\begin{theorem}\label{theo:KL_prior}
Let $A_0\in\Aset_0$ be the true Pickands dependence function.
Assume that $\Pi_\Aset$ possesses the $\dist_{2,\infty}$-property at every $A\in\Aset'$. Then, for any $\epsilon>0$
%, where $\Aset'$ is the family satisfying Definition \ref{cond_density}\ref{cond:prior_set}. Then, for all $\epsilon>0$ 
%
\begin{equation}\label{eq:KLpicka}
\Pi_{\Aset}( \phi_\Aset^{-1}(\mathcal{K}_\epsilon))=\Pi_{\Gset_\bone}(\mathcal{K}_\epsilon)>0.
\end{equation}
\end{theorem}
\begin{rem}
	As a map from $(\Aset, \dist_{1, \infty})$ to $(\Hset, \dist_{KS})$, $A\mapsto H$ is 1-to-1 and continuous (thus, Borel). Hence, the prior $\Pi_\Aset$ on the Pickands dependence function also induces a prior $\Pi_\Hset$ on the Borel sets of $(\Hset, \dist_{KS})$. 
	If
	$\Pi_\Aset$ possesses the $\dist_{2,\infty}$-property at every $A\in\Aset'$, then $\Pi_\Hset$ possesses the $\dist_\infty$-property at every $H\in\Hset'$, see Definitions \ref{def:d_inf_prop} and \ref{cond:mvt_angular}\ref{cond:prior_ang_set} of the main paper.
	Consequently, the result in 
	Theorem \ref{theo:KL_prior} redily follows from Theorem \ref{theo:KL_prior_multi}.
\end{rem}
To establish the consistency results in Section \ref{sec:posterior_consistency_2D}, Theorem \ref{theo:KL_prior} is applied to the specific cases of prior distributions constructed via the BP and BS representations of the Pickands dependence functions (Sections \ref{app:review_BP}-\ref{app:review_BS}). Explicit constructions of such priors are detailed in Examples \ref{ex:biv_BP}-\ref{ex:BS}.

%%%%%%%%%%%%%%%%%%%%%%%%%%%%%%%
\subsection{Posterior consistency}\label{sec:posterior_consistency_2D}
%%%%%%%%%%%%%%%%%%%%%%%%%%%%%%%

We consider priors on $A$ constructed through the following scheme, which relies on the representations in Sections \ref{app:review_BP}-\ref{app:review_BS}, see also Section \ref{appsec:measur_det_D2} for technical details.

\begin{condition}\label{cond:prior_D2}
	 Let $k_* \in \nat_+$ and let
	 the set sequence $(\Bset_k)_{k=k^*}^\infty$ be defined either via $\Bset_k:=\{\bfbeta_k\in[0,1]^{k+1}:\text{(R5)-(R7) hold true}\}$ or $\Bset_k=\{\bfbeta_k\in[0,1]^{k}:\text{(R10)-(R12) hold true}\}$.
	 Assume $\Pi_\Aset$ is the prior on the Borel sets of $(\Aset, \dist_{1, \infty})$ induced by $\Pi$,
	 the Borel pm on the disjoint union space $(\cup_{k\geq k_*}(\{k\}\times \Bset_k), \Sigma)$ obtained
	 via direct sum 
	 % \citep[p. 41]{r700} 
	 of the family $\{(\mathcal{B}_k,\Sigma_k, \lambda(k)\nu_k), \, k \geq k_* \}$, where, for for every $k \geq k_*$:
	 \begin{inparaenum}
	 	\item\label{cond:B10i} $\nu_k$ is a pm with full support on $\mathcal{B}_k$, equiped with its Borel $\sigma$-field $\Sigma_k$,
	 	\item\label{cond:B10ii} $\lambda(k)>0$,  with $\lambda(\cdot)$ denoting a probability mass function on $\{k_*, k_*+1, \ldots \}$ such that
	 	$$
	 	\sum_{i\geq k}\lambda(i) \lesssim e^{-qk} 	\quad (k \to \infty) 
	 	$$
	 	for some $q>0$.
	 \end{inparaenum}
\end{condition}

\begin{theorem}\label{theo:post_consistency}
Let $\bY_1,\ldots,\bY_n$ be iid rv's with distribution $G_{\bone}(\cdot|A_0)$, where $A_0\in \Aset_0$ and $\Aset_0$ is given in Definition \ref{cond_density}\ref{cond:true_set}. 
Let $\Pi_{\Aset}$ be a prior distribution on $A$ satisfying Condition \ref{cond:prior_D2}.
%%
%\begin{itemize}
%%
%\item[(i)] $\lambda(k)>0$ and $\nu_k$ has full support on $\Bset_k$,
%for all 
%%Borel sets $B\subset \Bset_k$ and 
%$k\geq k_*$;
%%
%\item[(ii)] $ \sum_{i\geq k}\lambda(i)\lesssim e^{-qk}$, for some $q>0$;
%%
%\end{itemize}
%%
%where, for each $k=k_*,k_*+1,\ldots$, $\Bset_k:=\{\bfbeta_k\in[0,1]^{k+1}:\text{(R5)-(R7) hold true}\}$ or $\Bset_k=\{\bfbeta_k\in[0,1]^{k}:\text{(R10)-(R12) hold true}\}$.
%
Then, $\iprodG-\text{as}$
\begin{itemize}
\item[(a)] $\lim_{n \to \infty}\tilde{\Pi}_n(\tilde{\mathcal{U}}^\complement)=0$, for every $\dist_H$-neighbourhood $\tilde{\mathcal{U}}$ of $g_{\bone}(\cdot|A_0)$;
\item[(b)] $\lim_{n \to \infty}\Pi_n(\mathcal{U}^\complement)=0$, for every $\dist_{1,\infty}$-neighbourhood $\mathcal{U}$ of $A_0$; 
\end{itemize}
where $\Pi_n(\cdot)=\Pi_\Aset(\cdot|\bY_{1:n})$, $\tilde{\Pi}_n(\cdot)=\Pi_{\mathcal{G}_\bone}(\cdot|\bY_{1:n})$ and $\Pi_{\mathcal{G}_\bone}$ is the prior on $g_{\bone}(\cdot|A)$  induced by $\Pi_\Aset$. 
%In partiucular, $\Pi_{\mathcal{G}_\bone}$ has full Hellinger support.
%
%
\end{theorem}
%
%
%As a map from $(\Aset, \dist_{1, \infty})$ to $(\Hset, \dist_K)$, $A\mapsto H$ is 1-to-1 and continuous (thus, Borel). Hence, the prior $\Pi_\Aset$ on the Pickands dependence function also induces a prior $\Pi_\Hset$ on the Borel sets of $(\Hset, \dist_K)$. 
%Theorem \ref{theo:post_consistency} has the following immediate consequence.
%
%\begin{cor}\label{cor:extention_consistency}
%%
%Let $\Pi_\Aset$ comply with the assumption of Theorem \ref{theo:post_consistency} and let $\Pi_\Hset$ be the induced prior on $H$. Then, for all $\dist_K$-neighbourhoods $\mathcal{U}$ of $H_0$,
%$\lim_{n\to\infty}\Pi_n(\mathcal{U}^\complement)=0$, $\iprodGH-\text{as}$,
%where $\Pi_{n}(\cdot):=\Pi_{\Hset}(\cdot|\bY_{1:n})$.  
%%
%\end{cor}
For the sake of completeness, in Section \ref{appsec:proofconsD2} a sketch of the proof is provided for both the Bernstein polynomial and the B-spline representations of $A$.
For the former, the consistency results in Theorem \ref{theo:post_consistency} also obtains as a byproduct of Theorem \ref{theo:post_consistency_mvt}, once noticing that the prior $\Pi_\Aset$ induces a prior $\Pi_\Hset$ satisfying the assumptions therein, see Remark \ref{rem:prior_identity} for technical aspects.
Concluding, we provide two examples of prior distributions $\Pi$ complying with Condition \ref{cond:angularprior}.
\begin{example}\label{ex:biv_BP}
	{\upshape
		Exploiting the BP representation of the Pickands dependence function in \eqref{eq:bpoly_picka}, a prior $\Pi$ can be constructed by specifying priors $\nu_k$ for $k+1$-dimensional coefficient vectors $\bfbeta_k=(\beta_0, \ldots,\beta_k)$ via mixtures of the kernels
		\begin{align*}
		\diff \nu_k(\bfbeta_k|p_1,p_0)=&\delta_1(\beta_0)\delta_{(k-1+2p_0)/k}(\beta_1)\delta_{(k-1+2p_1)/k}(\beta_{k-1})\delta_{1}(\beta_k)\\
		&\times \prod_{j=2}^{k-2}\text{Unif}(\beta_j;a_{k,j},b_{k,j})\left(\frac{k}{2}\right)^{k-3}\diff \beta_j
		\end{align*}
		 with respect to the mixing densities $\chi_k( p_1,  p_0)=\text{Unif}(p_1;a_{k,1}, b_{k,1}) \text{Unif}(p_0;0, 1/2)$,
		where $\text{Unif}(\cdot;a,b)$ is the uniform probability density over the interval $(a,b)$, 
		$a_{k,1}=\max(0,(k-1)p_0-k/2+1)$, $b_{k,1}=(p_0+k/2-1)/(k-1)$ and, for $j=2,\ldots,k-2$,
		\[
		a_{k,j}=\max(2\beta_{j-1}-\beta_{j-2},(k-j)\beta_{k-1}-(k-j-1)); \quad
		b_{k,j}\frac{1}{k-j}(\beta_{k-1}+(k-j-1)\beta_{j-1})
		\]
		%
%		\citep[][Corollary 3.4]{r24}. 
%	({\color{magenta} Marcon et al. 2016}, Corollary 3.4)
	\citep[Corollary 3.4]{r24}.
	Then, the prior probability mass function $\lambda$ for the polynomial degree $k$ can be chosen to be, e.g., a Poisson truncated outside $\{3,4,\ldots\}$.
	}
\end{example}
\begin{example}\label{ex:BS}
	{\upshape 
		A prior construction $\Pi$ based on the BS representation of the Pickands dependence function in \eqref{eq:bspline_pick}, with order $m=3$, is as follows.
		For $k \geq 4$, let 
		\begin{equation*}
		\begin{split}
		\bW_k=(W_1,\ldots,W_{k-2}) \sim \chi_k, \quad  J_k|\bW_k \sim \psi_k,
		\end{split}
		\end{equation*}
		where $\chi_k$ is any distribution on the $(k-2)$-dimensional simplex, with full support (e.g., Dirichlet), and $\psi_k$ is any absolutely continuous distribution on the interval $(l_k,1)$ (e.g., uniform, truncated Beta), where  $l_k =\max\{1-1/(2E_k), 1-1/(2(1-E_k))\}$ and $E_k= \sum_{1\leq j \leq k-2} W_{k,j}(2j-1)/(2(k-2))$.
		Set 
		$p_0=\frac{1}{2}-(1-J_k)(1-E_k)$,
		$p_1=\frac{1}{2}-(1-J_k)E_k$,
		and define $\mathcal{I}_{k,0}:=(0,1/(k-2))$, $\mathcal{I}_{k,j}:=[j/(k-2), (j+1)/(k-2))$, $1 \leq j \leq k-4$, $\mathcal{I}_{k,k-3}:=((k-3)/(k-2),1)$. For $j=0, \ldots,k-3$, let $\indic_{\mathcal{I}_{k,j}}(t)$, $t\in(0,1)$, be the indicator function, taking value one if $t \in \mathcal{I}_{k,j}$ and zero otherwise. Then, via equation \eqref{eq:angpm}, the rescaled random histogram $\sum_{0\leq j\leq k-3}W_{k,j+1}(1-J_k)\indic_{\mathcal{I}_{k,j}}(\cdot)$ and the point masses $p_0$ and $p_1$ at $\{\be_2\}$ and $\{\be_1\}$ induce  a valid random angular pm, whose cdf is a random spline of the form in \eqref{eq:bspline_angdist}. By Propositions \ref{prop:bspline_cond}-\ref{prop:bspline_rel_pick_ang}, choosing $\nu_k$ as the distribution of the rv $\bfbeta_k=(\beta_1,\ldots,\beta_k)$, defined via $\beta_1\equiv1$,
		\begin{equation*}
		\begin{split}
		\beta_j&=\sum_{i=1}^{j-1}\left\lbrace p_0+\sum_{1 \leq s \leq i-1}\frac{W_s(1-J_k)}{k-2}-1/2\right\rbrace(\tau_{i+3}-\tau_{i+1})+1,\quad j=2,\ldots,k-1,\\
		\beta_k&=\sum_{i=1}^{k-2}\left\lbrace p_0+\sum_{1 \leq s \leq i-1}\frac{W_s(1-J_k)}{k-2}-1/2\right\rbrace(\tau_{i+3}-\tau_{i+1})+(1/2-p_1)(\tau_k-\tau_{k-2})+1,
		\end{split}
		\end{equation*}
	$\nu_k$ has full support on $\Bset_k$, defined in Condition \ref{cond:prior_D2}. In the above display,
	$(\tau_j, j=1,\ldots,k+3)=(0,0,0,1/(k-2), \ldots,(k-3)/(k-2), 1, 1, 1)$. Moreover, by construction, $\beta_k\equiv1$.
    Prior specification can be thus completed by choosing $\lambda$ as in the previous example, with truncation outside $\{4,\ldots\}$.
	}
\end{example}

\section{Auxiliary Results}\label{app:auxiliary}
In this section we provide some auxiliary results on simple max-stable models and their dependece structure, that are useful for the proofs of the results in Sections \ref{sec:MGEV}--\ref{sec:bayesian_inference} of the main article and \ref{app:KL_support}--\ref{sec:posterior_consistency_2D} of the present manuscript.
In the sequel, we consider distributions of arbitrary dimension $d \geq 2$, unless otherwise specified.

\subsection{Spectral representation}\label{appsec:spectral}
We start by recalling some known characterisations of simple max-stable models and setting up the corresponding notation.
Any simple max-stable rv allows the following spectral representation 
%(e.g. {\color{magenta} Falk et al. 2011, Ch. 4--5})
\citep[e.g.,][Ch. 4--5]{r32} 
$$
\bY = \max (\bxi_i, i=1,2,\ldots),
$$ 
where $\bxi_i$, $i=1,2,\ldots$ are the points of a Poisson process on $E:=[0, \infty]^d\setminus\{\bzero\}$ with mean measure $\Lambda(\cdot|H)$ satisfying 
$$
\Lambda([0, \by]^\complement|H)= V(\by|H)= d \int_\simp \max_{1 \leq j \leq d}(w_j/y_j)\diff H(\bw).  
$$
For each nonempty subset $ I \subset \{1, \ldots, d\}$, define the subspace 
$$
E_I:=\{ \by \in E: \by_{I} >\bzero, y_j = 0, \, \forall j \in I^c \}.
$$
By Definition \ref{cond_angular} in the main paper, for $H \in \Hset$ the restriction $\Lambda_{\{1, \ldots,d\}}(\cdot|H)$ of $\Lambda(\cdot|H)$ on $E_{\{1,\ldots,d\}}$ has Lebesgue density  
$$
\lambda_{\{1, \ldots,d\}}(\bz|H)=d\Vert \bz \Vert_1^{-d-1}h\circ \pi_{\resimp}(\bz /\Vert \bz \Vert_1),
$$ 
while, for $j=1,\ldots,d$, the restrictions $\Lambda_{\{j\}}(\cdot|H)$ on $E_{\{j\}}$ satisfy
$$
\Lambda_{\{j\}}(B|H) = dH(\{\be_j\}) \int_{\pi_j(B \cap E_{\{j\}})} z_{j} ^{-2} \diff z_j,
$$
for all Borel sets $B \subset E$, where, for $\bx\in E$, $\pi_j(\bx)=x_j$.
Also note that $\Lambda(\cdot|H)$ has no mass on the subsets $E_I$ where $I$ is neither a singleton or the entire set $\{1, \ldots,d\}$. These properties yield the identity \eqref{eq:kernel_density_ang} in the main article.

For $j=1,\ldots,d$, define the random index $i_j^*$ by 
$
\xi_{i_j^*}=\max_{i=1,2,\ldots} \xi_{i,j}.
$
Then, the set $\{i_1^*, \ldots, i_d^*\}$ induces a random partition of $\{1, \ldots,d\}$. 
For a given angular pm $H \in \Hset$, we denote  the joint density of a simple max-stable rv and the corresponding random partition by
$$
g_\bone(\by,\part|H)=G_\bone(\by|H)\prod_{i=1}^m \{-V_{I_i}(\by|H)\}, \quad \part=\{I_1,\ldots, I_m\} \in \allpart_d,
$$ 
 for almost every $\by \in \Yset$, where $\Yset=(0, \infty)^d$ is the sample space for simple max-stable models,
see e.g. 
\citep[][pp. 4816-4818]{r252} 
%{\color{magenta}Dombry et al. (2017b, pp. 4816-4818)}
for additional details. In the sequel, we denote by $G_\bone(\cdot,\cdot|H)$ the associated pm on the subsets of $\Yset \times \allpart_d$. 

\subsection{Metric properties of simple max-stable models}
The results presented in Section \ref{sec:binf_simple_max} of the main article and \ref{app:technical} of the present manuscript involve several types of distance for simple max-stable distributions and the associated dependence functions. The following propositions describe the relationships existing among those. The first result we provide is valid in arbitrary dimension $d\geq 2$.

\begin{prop}\label{lem: basic metric}
	Let $A_1$ and $A_2$ be the Pickands dependence functions corresponding to the angular measures
	$H_1, H_2 \in \mathcal{H}$, respectively.
	The following results hold true:
	\begin{inparaenum}
		\item \label{res: basic 1}
		$
		e^{-1} \dist_\infty(A_1, A_2) \leq  \dist_H(g_\bone(\cdot|H_1), g_\bone(\cdot|{H_2}))
		\leq \Vert  g_\bone(\cdot,\cdot|H_1)-g_\bone(\cdot, \cdot|H_1)\Vert_1^{1/2}; 
		$
		\item \label{res: basic 2}
		$
		\dist_\infty(A_1, A_2) \leq   \min \left\lbrace 2 \dist_\infty(h_1,h_2)/\Gamma(d),
		2d\Vert h_1-h_2\Vert_1
		\right\rbrace;
		$
		\item \label{res: basic 4}
		$\forall \epsilon>0$, $\exists \eta>0$ such that, if $\dist_{W}(G_\bone(\cdot|H_1),G_\bone(\cdot|H_2))<\eta$, then $\dist_W(H_1,H_2)<\epsilon$;
	\end{inparaenum}
	where $\Gamma(\cdot)$ is the gamma function.
\end{prop}
\begin{proof}
	The first half of the result at point (i) can be established by observing that for all $\bt \in \resimp$, denoting $\by=(1/(1-\Vert \bt \Vert_1), 1/\bt)\in \Yset$, we have
	\begin{equation*}
	\begin{split}
	|A_1(\bt) - A_2(\bt)| &\leq e |e^{-A_1(\bt)}-e^{-A_2(\bt)}|\\
	&=e |G_\bone(\by|H_1)-G_\bone(\by|H_2)|\\
	& \leq e \dist_T(G_\bone(\cdot|H_1), G_\bone(\cdot|H_2))\\
	&\leq e \dist_H(g_\bone(\cdot|H_1), g_\bone(\cdot|H_2)),
	\end{split}
	\end{equation*}	
	see e.g. equation (B.1) and Lemma B.1(i) in
	\cite{r10}
%	 {\color{magenta}Ghosal and van der Vaart (2017)}
	%\citeN[equation (B.1) and Lemma B.1(i)]{ghosal2017} 
	for the latter inequality. 
	The second half directly follows from the inequalities
	\begin{equation*}
	\begin{split}
	\dist_H(g_\bone(\cdot|H_1), g_\bone(\cdot|{H_2}))
	&\leq \Vert g_\bone(\cdot|H_1) - g_\bone(\cdot|{H_2}) \Vert_1^{1/2}\\
	&= \left\lbrace
	2\dist_T(G_\bone(\cdot|H_1),G_\bone(\cdot|H_2))
	\right\rbrace^{1/2}\\
	&\leq \left\lbrace
	2\dist_T(G_\bone(\cdot,\cdot|H_1),G_\bone(\cdot,\cdot|H_2))
	\right\rbrace^{1/2}\\
	&=\Vert g_\bone(\cdot,\cdot|H_1) - g_\bone(\cdot,\cdot|{H_2}) \Vert_1^{1/2},
	\end{split}
	\end{equation*}
	see also Lemma B.1(ii) in \cite{r10}.
%	{\color{magenta}Ghosal and van der Vaart (2017)}. 
	%	\citeN[Lemma B.1(ii)]{ghosal2017} 
	%	and equation (2.6$^{*}$).
	%
	
	Next, observe that, using equation \eqref{eq:picklands} and Definition \ref{cond_angular}, the term  $|A_1(\bt)-A_2(\bt)|$ can be bounded from above by
	\begin{equation}\label{eq:what_we_bound}
	\begin{split}
	%	d\sum_{j=1}^d 
	%	\int_{\simp_d} t_j w_j
	%	\left| H(\diff \bw)-H'(\diff \bw))
	%	\right| &=
	&d \left[ 
	(1-\Vert \bt\Vert_1)|p_{1,1}-p_{2,1}|+
	\sum_{j=2}^d t_{j-1}|p_{1,j}-p_{2,j}|
	\right]\\
	& \quad +d\int_{\mathring{\resimp}} \max\left\lbrace
	(1-\Vert \bt \Vert_1)v_1, t_1v_2,\ldots, t_{d-2}v_{d-1},
	t_{d-1}(1-\Vert \bv \Vert_1)
	\right\rbrace
	|h_1(\bv)-h_2(\bv)|\diff \bv,
	%	&\leq d\sum_{j=1}^d t_j \int_{\simpint_d}  w_j |h(\bw)-h'(\bw)|\diff w_1 \cdots \diff w_d\\
	%	&\quad + d \int_{\simpint_d} \max_{1\leq j \leq d}(w_j t_j)|h(\bw)-h'(\bw)|\diff w_1 \cdots \diff w_d.
	\end{split}
	\end{equation}
	where, by the mean-contraints (C1),
	\begin{equation*}
	\begin{split}
	&|p_{1,j}-p_{2,j}|\leq \int_{\mathring{\resimp}} v_j|h_1(\bv)-h_2(\bv)|\diff \bv, \quad j=1, \ldots,d-1,\\
	&|p_{1,d}-p_{2,d}| \leq \int_{\mathring{\resimp}}(1-\Vert \bv \Vert_1)|h_1(\bv)-h_2(\bv)|\diff \bv.
	\end{split}
	\end{equation*}
	With few algebraic steps it can be now verified that the upper bound at point (ii) bounds from above the term in \eqref{eq:what_we_bound}.

	Finally, consider a sequence $(H_k, k=1,2,\ldots)\subset \Hset$ such that 
	$$
	\lim_{k\to\infty }\dist_{W} (G_\bone(\cdot|H_k),G_\bone(\cdot|H))=0.
	$$ 
	Then, 
	$$
	\lim_{k \to \infty}\Lambda([\bzero, \by]^\complement|H_k)=\lim_{k \to \infty}-\log G_\bone(\by|H_k) = -\log G_\bone(\by|H)=\Lambda([\bzero, \by]^\complement|H),
	$$
	pointwise for every $\by \in \Yset$.
	This is sufficient to claim that $\Lambda(\cdot|H_k)$ converges vaguely to $\Lambda(\cdot|H)$. The identities $\Lambda(B|H_k)=H_k(B)$, $\Lambda(B|H)=H(B)$, for all Borel subset of $\simp$, together with Proposition 3.12(ii) in 
	\cite{r11}
	%	\citeN[Proposition 3.12(ii)]{resnick2007}, 
	now entail that $\lim_{k\to \infty}\dist_W(H_k,H)=0$. We can conclude that the map $G_\bone(\cdot|H) \mapsto H$ is sequentially $\dist_W/\dist_W$-continuos and therefore $\dist_W/\dist_W$-continuous. The result at point (iii) follows.
\end{proof}

In the bivariate case ($d=2$), the simpler relations between the Pickands dependence and the angular distribution functions (equations \eqref{eq:pick}--\eqref{eq:pickdev}) underpin a stronger form of continuity for the map $A \mapsto H$, mapping a Pickands dependence function to its angular pm, than in the higher dimensional case ($d>2$). Moreover, the more tractable form of the simple max-stable density (equation \eqref{eq:bdens_fre}) allows to establish a general Lipschitz continuity result for the map $H \mapsto g_\bone(\cdot|H)$, mapping an angular pm to the pertaining simple max-stable density. 
%In Section {\color{blue}XXX}, similar results are obtained for arbitrary dimensions $d\geq 2$ only once restricting to specific polynomial forms of the angular density.
%
\begin{prop}\label{prop: basic metric bvt}
	In the specific case of $d=2$, the following results hold true for any $H_1, H_2 \in \mathcal{H}$, with corresponding Pickands $A_1, A_2$:
	\begin{inparaenum}
		\item \label{res: basic bvt 1}
		$\dist_H^2(g_\bone(\cdot|H_1),g_\bone(\cdot|H_2)) \leq c \Vert h_1-h_2 \Vert_1$, for a positive global constant $c$;
		\item \label{res: new pick}
		$\forall \epsilon>0$, $\exists \, \eta>0$ such that, if $\dist_\infty(A_1,A_2)<\eta$, then $\dist_{1, \infty}(A_1,A_2)\leq \epsilon$; in particular, $\dist_{KS}(H_1,H_2)\leq \epsilon$.
		%		\item \label{res: basic bvt 2} {\color{magenta}further assuming that
		%		$H_1$ has a continuous (on $\mathring{\resimp}$) angular density, then $\forall\,\epsilon>0$, $\exists\,\eta>0$ such that, if $\dist_\infty(H_1,H_2)<\eta$, then $\Vert h_1- h_2\Vert_1 \leq \epsilon$;}
	\end{inparaenum}	
	%			where $\overline{\Hset}$ is as in Proposition \ref{prop: top supp}.
\end{prop}

\begin{proof}
	Recall the general expression of the simple-max stable density in \eqref{eq:bdens_fre} and observe that, by a change of variables, the term $\Vert g_\bone(\cdot|H_1)-g_\bone(\cdot|H_2) \Vert_1$ can be re-expressed as 
	\begin{equation*}
	\begin{split}
	&\int_0^1\int_0^\infty \left| \frac{\phi_1(t)+r\chi_1(t)}{r^3t^2(1-t)^2}\exp
	\left(
	-\frac{A_1(t)}{rt(1-t)}
	\right)
	%	\right.\\
	%	&\qquad -  
	%	\left. 
	-\frac{\phi_2(t)+r\chi_2(t)}{r^3t^2(1-t)^2}\exp
	\left(
	-\frac{A_2(t)}{rt(1-t)}
	\right)
	\right|\diff r \diff t		\\
	&\quad \leq 
	\int_0^1\int_0^\infty
	\left| \frac{\phi_1(t)+r\chi_1(t)}{r^3t^2(1-t)^2}
	-\frac{\phi_2(t)+r\chi_2(t)}{r^3t^2(1-t)^2}
	\right|
	\exp
	\left(
	-\frac{A_1(t)}{rt(1-t)}
	\right)\diff r \diff t\\
	&\qquad + \int_0^1\int_0^\infty
	\left|
	\exp
	\left(
	-\frac{A_1(t)}{rt(1-t)}
	\right)
	-	\exp
	\left(
	-\frac{A_2(t)}{rt(1-t)}
	\right)
	\right|\frac{\phi_2(t)+r\chi_2(t)}{r^3t^2(1-t)^2} \diff r \diff t\\
	& \quad =: I_1+I_2,
	\end{split}
	\end{equation*}
	where, for $t \in (0,1)$, 
	$$
	\phi_\bullet(t)=\left(A_\bullet(t)-tA_\bullet' (t)\right)\left(A_\bullet(t)+(1-t)A_\bullet'(t)\right), \quad \chi_\bullet(t)=t^2(1-t)^2A_\bullet''(t). 
	$$
	On one hand, Proposition \ref{lem: basic metric}\ref{res: basic 2}, the identities in \eqref{eq:pickdev}-\eqref{eq:pick_2dev}
%	\begin{equation}\label{eq: deriv}
%	A_\bullet'(t)=2H_\bullet(t)-1, \quad A_\bullet''(t)=2h_\bullet(t), \quad t \in (0,1)
%	\end{equation}
	and few algebraic manipulations yield
	\begin{equation*}
	\begin{split}
	I_1 & \leq 
	\int_0^1\int_0^\infty
	\left(
	\frac{|\phi_1(t)-\phi_2(t)|}{r^3t^2(1-t)^2}	
	+\frac{|\chi_1(t)-\chi_2(t)|}{r^2t^2(1-t)^2}\right)
	\exp
	\left(
	-\frac{A_1(t)}{rt(1-t)}
	\right)\diff r \diff t
	\\
	& \leq \int_0^1\left| 
	\phi_1(t)-\phi_2(t)
	\right|\frac{\Gamma(2)}{(A_1(t))^2}\diff t+
	\int_0^1\left| 
	h_1(t)-h_2(t)
	\right|\frac{\Gamma(1)t(1-t)}{A_1(t)}\diff t\\
	& \leq c_1 \Vert h_1-h_2 \Vert_1,
	\end{split}
	\end{equation*}	
	for a positive global constant $c_1$. On the other, similar arguments together with the Lipschitz continuity of the exponential function on bounded sets give
	\begin{equation*}
	\begin{split}
	I_2 & \leq \int_0^1
	\int_0^\infty
	\left|
	A_1(t)-A_2(t)
	\right| 
	\frac{
		%	\exp\left\{ 
		e^{-\frac{1}{2r}t^{-1}(1-t)^{-1}}
		%	\right\}
	}{rt(1-t)}
	\frac{\phi_2(t)+r\chi_2(t)}{r^3t^2(1-t)^2} \diff r \diff t\\
	& =\int_0^1 \left|
	A_1(t)-A_2(t)
	\right| 
	\left(
	8\phi_2(t) \Gamma(3)
	+4\frac{\chi_2(t)\Gamma(2)}{t(1-t)}
	\right)\diff t\\
	& \leq c_2 \Vert h_1-h_2 \Vert_1
	,
	\end{split}
	\end{equation*}	
	where $c_2$ is a positive global constant. The result at point (i) now follows from Lemma B.1(ii) in \cite{r10} 
%	{\color{magenta}Ghosal and van der Vaart (2017)}.
	%	\citeN[Lemma B.1(ii)]{ghosal2017}.
	%
	
	To establish the result at point (ii), start by considering a sequence 
	$(A_k, k=1,2,\ldots)\subset \Aset$ such that $\lim_{k \to \infty}\dist_\infty(A_k,A_*)=0$. Observe that, by Definition \ref{cond_angular}, the identity in \eqref{eq:pickdev} with $A$ replaced by $A_*$ and $A_k$, $H$ by $H_*$ and $H_k$, respectively, entails that $A_*$ and $A_k$ are continuously differentiable on $[0,1]$. For, extend the first derivatives by considering the right and the left derivatives at $0$ and $1$, respectively. Thus, Theorem 25.7 in 
	\cite{r34}
%	{\color{magenta} Rockafellar (2015)}
	%	\citeN{rockafellar2015} 
	entails that $\lim_{k \to \infty}\dist_{\infty}(A_*',A_k')=0$. We have then established that the map $A \mapsto A'$ is sequentially $\dist_\infty/\dist_\infty$-continuous, and thus continuous, at $A_*$. The first half of the claim at point (ii) now follows. The second half is a direct consequence of the first one, together with the identity in \eqref{eq:pickdev}. The proof is now complete.
\end{proof}

\subsection{Positive Kullback-Leibler divergences for simple max-stable models}\label{sec:signed_kulb_simple}
In this section we introduce some technical results concerning the classical Kullback-Leibler and other related divergences for simple max-stable densities, which give the mathematical ground for proving the theorems of Section \ref{sec:KL_support} of the main paper and Section \ref{app:KL_support} of the present manuscript. They also represent a first step for establishing some of the lemmas in
Section \ref{appendix:semi} of the present manuscript, concerning semiparametric max-stable models. The actual interest on stronger notions of Kullback-Leibler divergences is motivated by the results in Section \ref{sec:binf_sample_max} of the main paper, whose proofs require control on higher-order moments of log likelihood ratios.

For any $\{H,H_1, H_2\}\subset \Hset$ and $l \in \mathbb{N}_+$, define the functional
\begin{equation*}
\mathscr{V}_H^{(l)}(H_1; H_2):=\left(
\int_{\Yset} \left[\log^+\left\lbrace
\frac{g_\bone(\by|H_1)}{g_\bone(\by|H_2)}
\right\rbrace\right]^l g_\bone(\by|H)\diff\by 
\right)^{1/l}.
\end{equation*}
In particular, 
$\mathscr{V}_{H_1}^{(l)}(H_1; H_2)$ equals the $l$-th order positive Kullback-Leibler divergence from $g_\bone(\cdot|H_2)$ to $g_\bone(\cdot|H_1)$  and it holds that
$$
\kulb(g_\bone(\cdot|H_1),g_\bone(\cdot|H_2)) \leq \mathscr{V}_{H_1}^{(1)}(H_1; H_2)= \kulb_+^{(1)}(g_\bone(\cdot|H_1),g_\bone(\cdot|H_2)).
$$
Before stating our first auxiliary result, we introduce the following 
class of angular pm's.
\begin{definition}
	 Let $\Hset'' \subset \Hset$ be the set of angular pm's such that $H(\{\be_j\})=p_j>0$, for $j=1,\ldots,d$, with angular density satisfying $\Vert h\Vert_\infty < \infty$  and $\inf_{\bt \in \mathring{\resimp}}h(\bt)>0$.
\end{definition}
\begin{lemma}\label{lem: ratio 1}
	Let $H \in \Hset''$. Then,
	for every $l \in \mathbb{N}_+$ and  $\epsilon >0$ there exists $\delta>0$ such that, for all $H_1\in B_{\delta,\infty}(H)$ and $H_2 \in \Hset$,
	\begin{equation}\label{eq: ratio bound}
	\mathscr{V}_{H_2}^{(l)}(H; H_1) \leq 
	\epsilon + \log(1+\epsilon).
	\end{equation}
	In particular, when $l=1$, it also holds that $\kulb(g_\bone(\cdot|H),g_\bone(\cdot|H_1))\leq\epsilon+\log(1+\epsilon)$.
\end{lemma}
\begin{proof}
	Define $\bar{\epsilon}$ via $(1+\bar{\epsilon})=(1+{\epsilon})^{1/d}$ and ${h}_{\inf}:=\inf_{\bv \in \mathring{\resimp}}h(\bv)$. 
	Let
	\begin{equation}\label{eq: constant}
	0<\delta< \min \left\lbrace 
	\frac{\epsilon}{2c  \{d\Gamma(l+1) \}^{1/l}},  \frac{\bar{\epsilon}}{1+\bar{\epsilon}}{h_{\inf}},
	\frac{\bar{\epsilon}}{1+\bar{\epsilon}}\frac{d}{c}\,
	\min_{j=1,\ldots,d} p_j
	\right\rbrace,
	\end{equation}
	with $c=1/\Gamma(d)$. Fix $H_1 \in B_{\delta, \infty}(H)$. By Minkowski's inequality, the left hand-side of \eqref{eq: ratio bound} is bounded from above by
	\begin{equation*}
	\begin{split}
	T_1+T_2 &\equiv 
	\left(
	\int_{\Yset}  \left|V(\by|H_1)-V(\by|H)\right|^l g_\bone(\by|H_2)\diff \by\right)^{1/l}\\
	&\quad+
	\left(
	\int_{\Yset}  \left[\log^+ \left\lbrace
	\frac{\sum_{\part \in \allpart_d}
		\prod_{i=1}^{m}\{-V_{I_i}(\by|H)\}}{\sum_{\part \in \allpart_d}
		\prod_{i=1}^{m}\{-V_{I_i}(\by|H_1)\}}
	\right\rbrace\right]^l
	g_\bone(\by|H_2)\diff \by \right)^{1/l}.
	\end{split}
	\end{equation*}

	Proposition \ref{lem: basic metric}\ref{res: basic 2}, the bound in \eqref{eq: constant} and the fact that 
	$g_\bone(\cdot|H_2)$ has unit Fr\'{e}chet margins allow to claim that
	%	since $ \dist_\infty (A_{H},A_{H'})\leq \dist_{T}(H,H')\leq 2c\eta<\epsilon d^{-1}$
	\begin{equation*}
	\begin{split}
	T_1& \leq 
	\left(
	\int_{\Yset} {
		\left|A_1\left(
		\frac{1/\by_{2:d}}{\Vert 1/\by \Vert_1}
		\right)-A\left(
		\frac{1/\by_{2:d}}{\Vert 1/\by \Vert_1}
		\right)\right|^l } 
	\Vert 1/\by\Vert_1^l g_\bone(\by|H_2)\diff \by\right)^{1/l}\\
	&\leq \{d\Gamma(l+1) \}^{1/l} \dist_\infty (A,A_1) 
	<  \epsilon .
	%	& \leq d D_{TV}(H,H')\\
	%	& \leq d(d+c)\delta < \epsilon.
	\end{split}
	\end{equation*}
	Moreover, defining for all $\part \in \allpart_d$ the positive weights
	$$
	w_H(\part)= \frac{\prod_{i=1}^m\{-V_{I_i}(\by|H)\}}{\sum_{\part' \in \allpart_d}\prod_{i=1}^{m'}\{-V_{I_i'}(\by|H)\}},
	$$
	Jensen's inequality yields that 
	\begin{equation}\label{eq:jensen}
	\begin{split}
	\log^+ \left\lbrace
	\frac{\sum_{\part \in \allpart_d}
		\prod_{i=1}^{m}\{-V_{I_i}(\by|H)\}}{\sum_{\part \in \allpart_d}
		\prod_{i=1}^{m}\{-V_{I_i}(\by|H_1)\}}
	\right\rbrace&= \max \left[0, -\log \left\lbrace
	\frac{\sum_{\part \in \allpart_d} 
		\prod_{i=1}^{m}\{-V_{I_i}(\by|H_1)\}}{\sum_{\part \in \allpart_d}
		\prod_{i=1}^{m}\{-V_{I_i}(\by|H)\}}	\right\rbrace
	\right]\\
	& = \max \left[0, -\log \left\lbrace
	\sum_{\part \in \allpart_d} w_H(\part)
	\prod_{i=1}^{m} \frac{\{-V_{I_i}(\by|H_1)\}}{\{-V_{I_i}(\by|H)\}}	\right\rbrace
	\right]\\
	& \leq \max \left[0, -\sum_{\part \in \allpart_d} w_H(\part)\log
	\left\lbrace
	\prod_{i=1}^{m} \frac{\{-V_{I_i}(\by|H_1)\}}{\{-V_{I_i}(\by|H)\}}
	\right\rbrace
	\right]\\
	& \leq \max \left[0, \max_{\part \in \allpart_d} \log
	\left\lbrace \prod_{i=1}^m
	\frac{-V_{I_i}(\by|H)}{-V_{I_i}(\by|H_1)}
	\right\rbrace
	\right]
	\end{split}
	\end{equation}
	Therefore, $T_2 \leq \log(1+\epsilon)$, since for all $\part \in \allpart_d$, $i \in \{1, \ldots, m\}$ and (almost) every $ \by \in \Yset$ we have
	$$
	\left\lbrace
	\frac{-V_{I_i}(\by|H)}{-V_{I_i}(\by|H_1)}
	\right\rbrace \leq 1+\bar{\epsilon}.
	$$
	The latter inequality easily follows from \eqref{eq:kernel_density_ang}  and the subsequent facts:
	\begin{inparaenum}
		\item\label{fact I}
		since $\dist_\infty(h,h_1) < \delta$, then $h-h_1 < \delta= \frac{\delta}{{h}_{\inf} - \delta}({h}_{\inf} -\delta) < \bar{\epsilon}({h}_{\inf} -\delta)< \bar{\epsilon} h_1$. As a result, 
		\begin{equation*}
		\begin{split}
		&\int_{
			(\bzero, \by_{I_i^\complement})
		}
		\frac{d h\circ\pi_{\resimp}(\bz/\Vert \bz\Vert_1)}{\Vert \bz \Vert_1^{d+1} }
		\bigg{|}_{\bz_{\tau_i}=\by_{I_i}} \diff \bz_{I_i^c }  \leq 	(1+\bar{\epsilon})\int_{
			(\bzero, \by_{I_i^\complement})
		}
		\frac{d h_1\circ\pi_{\resimp}(\bz/\Vert \bz \Vert_1)}{\Vert \bz \Vert_1^{d+1} } \bigg{|}_{\bz_{I_i}=\by_{I_i}} \diff \bz_{I_i^\complement}; 
		\end{split}
		\end{equation*}
		\item \label{fact II}
		by exploting the mean constraints in (C.1) we deduce that for $j=1,\ldots,d-1$,
		$$
		(p_{j}-p_{1,j})/p_{1,j}=\frac{\int_{\mathring{\resimp}}v_j [h_1(\bv)-h(\bv)]\diff \bv}{p_{1,j}
		} 
		\leq \frac{ \delta cd^{-1}}{p_{j}-\delta cd^{-1}} < \bar{\epsilon},
		$$
		and
		$$
		(p_{d}-p_{1,d})/p_{1,d}=\frac{\int_{\mathring{\resimp}}(1-\Vert \bv \Vert_1) [h_1(\bv)-h(\bv)]\diff \bv}{p_{1,d}
		} 
		\leq \frac{ \delta cd^{-1}}{p_{d}-\delta cd^{-1}} < \bar{\epsilon},
		$$
		therefore, for all $j \in\{1, \ldots,d\}$, $p_j\leq (1+\bar{\epsilon})p_j'$, wherefrom we conclude $p_{j} d y_{I_i }^{-2} \leq (1+\bar{\epsilon})p_{j}' d y_{I_i }^{-2}$, with $I_i=\{j\}$.
	\end{inparaenum}
	The proof is now complete.
\end{proof}

\begin{rem}
	Observe that the definition of $\Hset''$ is similar to that of $\Hset'$ in Definition \ref{cond:mvt_angular}\ref{cond:prior_ang_set} of the main article.
	Though, the pm's in $\Hset''$ are only required to have a bounded angular density, which may be discontinuous.
	As a consequence, $\Hset' \subset \Hset''$.
	The assumption $H \in \Hset'' $ allows to apply Lemma \ref{lem: ratio 1} also in the case where $H$ is an angular pm whose cdf is in BS form, with piecewise constant angular density (see Section \ref{app:review_BS} of the present manuscript).
\end{rem}

\begin{lemma}\label{prop: new_KL_prop}
	Let $H \in \Hset_0$ and assume it satisfies property \ref{cond: finite_new} in Definition \ref{cond:mvt_angular} of the main article. Then, for all $l \in \nat_+$ and $\epsilon>0$ there exist $H_* \in\Hset'$ such that
	\begin{equation*}
		\kulb_+^{(l)}(g_\bone(\cdot|H), g_\bone(\cdot|H_*)) \leq \epsilon.
	\end{equation*}
	In the particular case of $l=1$, it also holds that
	$
	\kulb(g_\bone(\cdot|H),g_\bone(\cdot|H_*))\leq\epsilon.
	$ 
\end{lemma}
\begin{proof}
	Consider the non-trivial case where $H \notin \Hset'$ and fix $\epsilon>0$. For any $H_1,H_2 \in \Hset$, Minkowski's inequality entails that
	\begin{equation}\label{eq: KL decomposition2}
	\begin{split}
	\kulb_+^{(l)}(g_\bone(\cdot|H), g_\bone(\cdot|H_2))&=\mathscr{V}_{H}^{(l)}(H; H_2)\\
	&\leq\mathscr{V}_{H}^{(l)}(H; H_1)+\mathscr{V}_{H}^{(l)}(H_1; H_2).
	\end{split}
	\end{equation}
	Let $c=1/\Gamma(d)$ and fix an aribtrarily small $\epsilon'>0$ such that
	$$
	\epsilon' + \log(1+\epsilon') < \epsilon/2.
	$$	
	Set $\bar{\epsilon}'$ via $(1+\bar{\epsilon}')=(1+\epsilon')^{1/d}$ and let $c_1, c_2$ be any two positive constants satisying
	\begin{equation*}
	\begin{split}
	c_1 < \min\left\lbrace \frac{\epsilon'}{\{d\Gamma(1+l)\}^{1/l}2c\Vert h \Vert_\infty}, \bar{\epsilon}'  \right\rbrace,\quad 
	c_2 < \min \left\lbrace
	\frac{\epsilon'}{\{d\Gamma(1+l)\}^{1/l}2c},
	\frac{1}{c}\frac{\bar{\epsilon}'}{1+\bar{\epsilon}'}
	\frac{c_1}{1+c_1}
	\right\rbrace.
	\end{split}
	\end{equation*}
	Choose $H_1$ and $H_2$ in \eqref{eq: KL decomposition2} as the angular pm's corresponding to densities $h_1={h}/(1+c_1)$ and $h_2=h_1+c_2$, respectively, and point masses
	\begin{equation*}
	\begin{split}
	p_{i,j}&=\frac{1}{d} - \int_{\mathring{\resimp}}v_j h_i(\bv) \diff \bv, \quad i=1,2, \quad j=1,\ldots,d-1,\\
	p_{i,d}&=\frac{1}{d} - \int_{\mathring{\resimp}}(1-\Vert \bv \Vert_1)h_i(\bv) \diff \bv, \quad i=1,2.
	\end{split}
	\end{equation*}
	Proceeding as in the proof of Lemma \ref{lem: ratio 1} and
	using results similar to facts \ref{fact I}-\ref{fact II} therein,
	we can show that both the terms on the right-hand side of \eqref{eq: KL decomposition2} are smaller than $\epsilon'+\log(1+\epsilon')$. Since $H_2 \in \Hset'$, the result in the statement now follows selecting $H_*=H_2$.

	In passing, note that, if $\inf_{w\in \simpint_{d}}h(w)>0$, then $H_1\in\Hset'$ and, selecting $H_*=H_1$, the result follows directly from 
	$$
	\mathscr{V}_{H}^{(l)}(H; H_1)< \epsilon'+\log(1+\epsilon').
	$$
The proof is now complete.
\end{proof}

\begin{prop}\label{cor: New_KL_cor}
	Let $H \in\Hset_0$ and satisfies property \ref{cond: finite_new} in Definition \ref{cond:mvt_angular}, then for any $l \in \nat_+$ and $\epsilon>0$ there exist $H_1 \in \Hset'$ and $\delta>0$ such that, for all $H_2 \in B_{\delta, \infty}(H_1)$, 
	\begin{equation}\label{eq:l_div_simp}	\kulb_{+}^{(l)}(g_\bone(\cdot|H), g_\bone(\cdot|H_2))\leq \epsilon.
	\end{equation}

	In the particular case of $l=1$, it also holds that
	$
	\kulb(g_\bone(\cdot|H),g_\bone(\cdot|H_2))\leq\epsilon.
	$
\end{prop}	
\begin{proof}
	The result readily obtains by bounding from above $\kulb_{+}^{(l)}(g_\bone(\cdot|H), g_\bone(\cdot|H_2))$ as in \eqref{eq: KL decomposition2}. On one hand, by Lemma \ref{prop: new_KL_prop}, the first term on the right-hand side can be made arbitrarily small via an appropriate choice of $H_1 \in\Hset'$. On the other hand, by Lemma \ref{lem: ratio 1}, also the second term can be made arbitrarily small by choosing $H_2 \in B_{\delta, \infty}(H_1)$ and $\delta>0$ small enough.
\end{proof}

In the specific case of $d=2$, we establish a result similar to the previous one for unbounded angular densities.
\begin{prop}\label{prop:newprop_KL}
	Let $d=2$ and $H \in \Hset_0$, with angular density $h$ satisfying property \ref{cond: infinite_new} in Definition \ref{cond:mvt_angular} of the main article.
	Then, for all $\epsilon>0$ there exist $H_1 \in\Hset'$ and $\delta>0$ such that
	$$
	\kulb(g_\bone(\cdot|H),g_\bone(\cdot|H_2))\leq\epsilon,
	$$
	for all $H_2 \in B_{\delta, \infty}(H_1)$. Further assuming that for some $s>0$ 
	$$
	\int_0^1 (h(t))^{1+s}\diff t<\infty, 
	$$ 
	for each $l \in \nat_+$ there exist choices of such $H_1$ and $\delta$ 
	for which the inequality  \eqref{eq:l_div_simp} is also satisfied by all $H_2 \in B_{\delta, \infty}(H_1)$.
\end{prop} 
\begin{proof}
	To establish the first part of the statement, we construct an angular pm $H_1$ as follows.
	For small constants $\epsilon_1, \epsilon_2$, there exist positive bounded functions $\gamma_{\epsilon_1},\gamma_{\epsilon_2}$ which are continuous on $[0,1]$, satisfy $\gamma_{\epsilon_1} \leq h_0(t)$, $\forall t \in (0, \epsilon_1)$, $\gamma_{\epsilon_2} \leq h_0(t)$, $\forall t \in (1-\epsilon_2, 1)$, and are such that the 
	function
	$$
	h_{1}(t):= 
	\begin{cases}
	\gamma_{\epsilon_1}(t), \quad  t \in (0, \epsilon_1)\\
	h(t), \hspace{1.7em} t \in (\epsilon_1, 1- \epsilon_2)\\
	\gamma_{\epsilon_2}(t), \quad  t \in (1-\epsilon_2,1)
	\end{cases}
	$$ 
	can be continuously extended to $[0,1]$ and is bounded from below by $\inf_{w \in (0,1)}h_0(w)>0$. Set 
	$$
 	p_{1,1}= 1/2 - \int_{0}^1 t\,h_{1}(t)\diff t,  \quad 
 	p_{1,2}= 1/2 - \int_{0}^1 (1-t)h_{1}(t)\diff t, 
	$$
	Then, the corresponding angular pm $H_{1}$ is an element of $\Hset'$. Consequently, for any $H_2 \in B_{\delta, \infty}(H_1)$, we can exploit the bound
	$$
	\kulb(g_\bone(\cdot|H),g_\bone(\cdot|H_2))\leq \kulb(g_\bone(\cdot|H),g_\bone(\cdot|H_1)) + \mathscr{V}_{H}^{(1)}(H_1; H_2)
	$$
	together with Lemma \ref{lem: ratio 1} (applied with $H, H_1, H_2$ in place  of $H_2, H, H_1$ in the statement, respectively), and conclude by showing that $\kulb(g_\bone(\cdot,\cdot|H),g_\bone(\cdot,\cdot|H_{1}))$	can be made arbitrarily small if $\epsilon_1$ and $\epsilon_2$ are chosen small enough. 

	Observe that, in the present bivariate case, $V(\by|H)=A(t_\by)(y_1^{-1}+y_2^{-2})$ and
	\begin{equation*}
	\begin{split}
	V_{\{1\}}(\by|H)V_{\{1\}}(\by|H)&=\frac{\left\{A(t_\by)-tA'(t_\by)\right\}\left\{A(t_\by)+(1-t_\by)A'(t_\by)\right\}}{(y_1y_2)^2},\\
	V_{\{1,2\}}(\by|H)&=\frac{A''(t_\by)}{(y_1+y_2)^3},
	\end{split}
	\end{equation*}
	where $t_\by=y_1/\Vert \by \Vert_1$ and $A$ is the Pickands dependence function pertaining to $H$. 
	Similar equalities holds true for $H_1$ and the pertaining Pickands dependence and exponent functions, $A_1$ and $V(\cdot|H_1)$, respectively.
	This fact and few algebraic manipulations now give 
	$$
	\kulb(g_\bone(\cdot|H),g_\bone(\cdot|H_1)) \leq \kulb(g_\bone(\cdot,\cdot|H),g_\bone(\cdot, \cdot|H_1))
	\leq T_1 + T_2 + T_3
	$$
	where
	%	, denoting by $A_{i_0}$ the Pickands corresponding to $H_{i_0}$, 
	$T_1=2\dist_\infty(A_0,A_{1})$ and
	\begin{equation*}
	\begin{split}
	T_2&= \int_{\Yset}
	\log\left(\frac{A(t_\by)-tA'(t_\by)}{A_{1}(t_\by)-t_\by A_{1}'(t_\by)}\frac{A(t)+(1-t_\by)A'(t_\by)}{A_{1}(t_\by)+(1-t_\by)A_{1}'(t_\by)}
	\right)g_\bone(\by, \part_1|H) \diff \by,\\
	T_3&=\int_{\Yset}
	\log 
	\left(
	\frac{h_0(t_\by)}{h_{1}(t_\by)}
	\right) g_\bone(\by, \part_2|H_0) \diff \by,
	\end{split}
	\end{equation*}
	with $\part_1=\{\{1\},\{2\}\}$, 
	$\part_2=\{1,2\}$ and $g_\bone(\by, \part|\bullet)$ defined as in Section \ref{appsec:spectral}.
	By Proposition \ref{lem: basic metric}\ref{res: basic 2}, we have that 
	$$
	T_1 \leq 8 \Vert h-h_{1}\Vert_1 \leq 8 \left(\int_0^{\epsilon_1}+\int_{1-\epsilon_2}^1 \, \right) h(t)\diff t,
	$$
	and the right-hand side can be made arbitrarily small by choosing $\epsilon_1,\epsilon_2$ small enough.
	As for $T_2$, observe that for every $t\in (0,1)$
	\begin{eqnarray*}
	0 <  T_2^{(1)}(t):=\frac{A(t)-tA'(t)}{A_{1}(t)-tA_{1}'(t)}&=&\frac{1-2\int_0^t\int_v^th(s)\diff s \diff v}{1-2\int_0^t\int_v^th_{1}(s)\diff s \diff v} \leq 1,\\
	0 <T_2^{(2)}(t):= \frac{A(t)+(1-t)A'(t)}{A_{1}(t)+(1-t)A_{1}'(t)} &=&
	\frac{1/2-\int_0^{1-t}\int_t^{1-v}h(s)\diff s \diff v}{1/2-\int_0^{1-t}\int_t^{1-v}h_{1}(s)\diff s \diff v} \leq 1,
	\end{eqnarray*}
	and $\lim_{t \to  0}T_2^{(1)}(t)T_2^{(2)}(t)=p_{2}/p_{1,2} \leq 1$, $\lim_{t \to  1}T_2^{(1)}(t)T_2^{(2)}(t)=p_{1}/p_{1,1}\leq 1$. Therefore, $T_2 \leq 0$. 
	As for $T_3$, due to the integrability of $h$, there
	exists a constant $\varsigma \in (0,1)$ such that 
	$$
	\lim_{t \to 0}t^\varsigma h(t)=\lim_{t \to 1}(1-t)^\varsigma h(t)=0.
	$$ 
	Consequently, for $\epsilon_1,\epsilon_2$ sufficiently small, Jensen's inequality, the inequality $\log(x)\leq x$ for all $x \geq1$ and some algebraic steps allow to deduce that $T_3/(1-p_1-p_2)$ is bounded from above by
	$$
	 \log \left( 1+ \frac{a_1 \epsilon_1^{1-\varsigma}+a_2\int
		_0^{\epsilon_1}h(t)t^{1-\varsigma}\diff t}{\inf_{0<v<1}h(v)}
	+\frac{a_3(1-\epsilon_2)^\varsigma+a_3 \int_{1-\epsilon_2}^1 h(t)(1-t)^\varsigma \diff t}{\inf_{0<v<1}h(v)}\right),
	$$
	where $a_1, \ldots,a_4$ are positive global constants. The term in the above display can be made arbitrarily small by choosing $\epsilon_1,\epsilon_2$ sufficiently small. The result now follows.

	Fix $l \in \nat_+$.
	To establish the second part of the statement, we now resort to the inequality \eqref{eq: KL decomposition2},  where, by Lemma \ref{lem: ratio 1}, the second term on the right-hand side can be made arbitrarily small by choosing $\delta$ small enough.
	Consequently, the result obtains once we show that
	$$
	\mathscr{V}^{(l)}_{H}(H;H_1)
	=\kulb_+^{(l)}(g_\bone(\cdot|H), g_\bone(\cdot|H_1))
	$$ 
	can be made arbitrarily small by choosing $\epsilon_1,\epsilon_2$ small enough.
	Minkowski's inequality and a few manipulations allow to deduce that the $l$-th order postive Kullback-Leibler divergence from $g_\bone(\cdot|H_1)$ to $g_\bone(\cdot|H)$ is bounded from above by $R_1+R_2+R_3$, where
	\begin{eqnarray*}
		R_1 &=& \dist_{\infty}(A,A_1)\Gamma^l(1+l),\\
		R_2 &=& 
		\left\lbrace
		\int_\Yset
		\left[
		\log^+\left(\frac{A(t_\by)-t_\by A'(t_\by)}{A_{1}(t_\by)-t_\by A_{1}'(t_\by)}\frac{A(t_\by)+(1-t_\by)A'(t_\by)}{A_{1}(t_\by)+(1-t_\by)A_{1}'(t_\by)}
		\right)
		\right]^{l}
		g_\bone(\by|H) \diff \by
		\right\rbrace^{1/l},\\
		R_3 &=&
		\left\lbrace
		\int_{\Yset}
		\left[
		\log^+ 
		\left(
		\frac{h_0(t_\by)}{h_{1}(t_\by)}
		\right)
		\right]^l
		 g_\bone(\by|H_0) \diff \by
		 \right\rbrace^{1/l},
	\end{eqnarray*}
and, as before, $t_\by=y_1/\Vert \by \Vert_1$. Reasoning as above, we conclude that $R_1$ can be made arbitrarily small by choosing $\epsilon_1,\epsilon_2$ small enough and $R_2=0$. The limit $\lim_{x\to \infty}\log(x)x^{-s/l}=0$ and a few algebraic manipulations also allow to deduce that, for $\epsilon_1,\epsilon_2$ sufficiently small,
\begin{equation*}
\begin{split}
R_3 \leq \frac{a_5}{{\inf_{0<v<1}h^s(v)}}
\left\lbrace \left(\int_0^{\epsilon_1}+\int_{1-\epsilon_2}^1 \, \right) (h(t))^{1+s}\diff t
\right\rbrace^{1/l}
\end{split}	
\end{equation*}
where $a_5$ is a positive global constant. Since, by hypothesis, $h^{1+s}$ is integrable, the term on the right-hand side can be made arbitrarily small by choosing $\epsilon_1,\epsilon_2$ small enough, whence the conclusion.
\end{proof}

%%%%%%%%%%%%%%%%%%%%%%%%%%%%%%%%%%%%%%%%%%%%%%%%%%%%

\subsection{Construction of priors for simple max-stable models}
In this subsection we discuss more extensively some set theoretical aspects of the prior constructions in Section \ref{sec:posterior_consistency} of the main paper and Section \ref{sec:posterior_consistency_2D} of this manuscript.
\subsubsection{Prior distributions on $A$}\label{appsec:measur_det_D2}
We first deal with the case where $d=2$ and give additional details on the prior construction in  Condition \ref{cond:prior_D2}. 
As in Sections \ref{app:KL_support}--\ref{sec:posterior_consistency_2D}, we endow $\Aset$ with the Sobolev metric $\dist_{1, \infty}$ and denote by $\mathscr{B}_\Aset$ the induced Borel $\sigma$-algebra. In what follows, for a fixed positive integer $k_*$, the subsets $\{\Aset_{k}, \, k\geq k_*\}$ denote the classes of Pickands dependence functions either in BP form or in BS form, given in
Sections \ref{app:review_BP} and \ref{app:review_BS},
respectively. Note that in both cases, for each $k \geq k_*$, $\Aset_k$ is $\dist_{1, \infty}$-close, therefore  $\Aset_k \in \mathscr{B}_\Aset$ and 
$$
\mathcal{A}_*:=\cup_{k \geq k_*}\Aset_k \in \mathscr{B}_\Aset.
$$
A pm on $(\Aset, \mathscr{B}_\Aset)$, say $\Pi_\Aset$, can be obtained from a pm $\Pi_{\mathcal{A}_*} $ on the subspace $\sigma$-algebra $\mathscr{B}_{\mathcal{A}_*}$ pertaining to $\mathcal{A}_*$ via
$$
\Pi_\Aset(B):=\Pi_{\Aset_* }(B \cap \Aset_* ), \quad \forall B \in \mathscr{B}_\Aset.
$$  
%	see e.g. \citeN[Section 3]{ershov74}. 
Clearly, since $\Aset_*$ is dense in $\Aset$, if $\Pi_{\Aset_*}$ has full support, the same is true for $\Pi_{\Aset}$. This fact is exploited in Section \ref{sec:posterior_consistency_2D} to specify a prior distribution on the Pickands dependence function.
Therein, a prior on 
$(\Aset_*,\mathscr{B}_{\Aset_*})$ is firstly induced by specifying a probability measure $\Pi$ on the Borel sets of the disjoint union topological space 
$$
\mathcal{B}_*:=\cup_{k=k_*}^\infty (\{k\}\times \mathcal{B}_k),
$$
where the $\mathcal{B}_k$'s are Euclidean subspaces of suitable linear coefficient vectors (Condition \ref{cond:prior_D2}). 
Precisely, letting $\lambda(\cdot)$ be a positive probability mass function on $\{k_*, k_*+1, \ldots\}$
%$\Pi(\{ k\})$, $k \geq k'$, be a sequence of nonegative reals summing up to one 
and $\nu_k$, $k \geq k_*$, a squence of pm's on $\mathcal{B}_k$'s endowed with the Borel $\sigma$-algebras $\Sigma_k$, $\Pi$ denotes the (probability) measure pertaining to the direct sum of the measure spaces $(\mathcal{B}_k, \Sigma_k, \Pi_k)$, $k \geq k_*$, where $\Pi_k:=\lambda(k)\nu_k(\cdot)$. Note that the $\sigma$-field of the measure space constructed via direct sum coincides with the Borel $\sigma$-algebra generated by the disjoint union topology, hereafter denoted by $\mathscr{B}_{\mathcal{B}_*}$; see e.g.
\citep[p. 41]{r700} and \citep[pp. 18, 84 and 828]{r701}.
%{\color{magenta}Fremlin 2001, p. 41; Fremlin 2003, pp. 18, 84 and 828}).
%(see \citeNP[p. 41]{fremlin2}; \citeNP[pp. 18, 84 and 828]{fremlin}). 
%
For $k \geq k_*$, define by 
$$
\phi_{\Aset_k}^{(BP)}:\mathcal{B}_k \mapsto \Aset_* : (\beta_0,  \ldots,\beta_k) \mapsto \sum_{0 \leq j \leq k} \beta_j (k+1)^{-1}\betaf(\cdot|j+1,k-j-1)
$$
and
$$
\phi_{\Aset_k}^{(BS)}:\mathcal{B}_k \mapsto \Aset_* : (\beta_1, \ldots,\beta_k) \mapsto \sum_{1 \leq j \leq k} \beta_j\phi_{j, 3}
$$
the maps transforming linear coefficients into the corresponding Pickands dependence functions in BP and BS forms, respectively. Clearly, these are 
%$\phi_{\Aset_k}^{(\bullet)}$ is 
injective and continuous with respect to the Euclidean distance on $\mathcal{B}_k$ and $\dist_{1, \infty}$ on $\Aset_*$. Thus, by the universal property of disjoint union topologies, there exists a unique continuous map
$$
\phi_{\Aset_*} ^{(\bullet)}: \mathcal{B}_*\mapsto \Aset_*
$$ 
such that,  for all $k \geq k_*$, 
$$
\phi_{\Aset_k}^{(\bullet)}(\boldsymbol{\beta})= \phi_{\Aset_*}^{(\bullet)}((k,\boldsymbol{\beta}))
$$ 
whenever $\beta \in \mathcal{B}_k$. Finally, we have that
$
\Pi_{\Aset_*}= \Pi\circ \{\phi_{\Aset_*}^{(\bullet)}\}^{-1}.
$

Next, we focus on the Bernstein polynomial representation.
For $k \geq k_*$, 
%denote by $\mathcal{E}_{k-1}$ the set of linear coefficient vectors satisfying (R3)-(R4), for Bernstein polynomials representation, and (R8)-(R9), for B-splines representation. 
let $\Phi_{k-1}$ be defined as in Condition \ref{cond:angularprior}.
Then, define the maps
$$
\phi^{(BP)}_k: \mathcal{B}_k \mapsto \Phi_{k-1}:(\beta_0, \ldots, \beta_k)\mapsto (\varphi_{\bkappa_1}, \varphi_{\bkappa_2}, \varphi_\balpha, \, \balpha \in \Gamma_k)\equiv\bphi^{(k)}
$$ 
%and
%$$
%\phi^{(BS)}_k: \mathcal{B}_k \mapsto \mathcal{E}_{k-1}:(\beta_1, \ldots, \beta_k)\mapsto
%\left(\frac{1}{2}+\frac{\beta_{2}-\beta_{1}}{\tau_{3}-\tau_1},\ldots, \frac{1}{2}+\frac{\beta_{j+1}-\beta_{j}}{\tau_{k+1}-\tau_{k-1}}\right),
%$$
transforming linear coefficients for the Pickands dependence function into the linear coefficients for the associated angular pm on the two dimensional simplex with density in Bernstein polynomial form, where
\begin{eqnarray*}
	\varphi_{\bkappa_1}&=&\frac{k \beta_{k-1}-k+1}{2}, \hspace{3.1em} 
	\varphi_{\bkappa_2} \, = \, \frac{k \beta_1-k+1}{2},
	\\
	\varphi_\balpha& =&\frac{k}{2} (\beta_{j+1}-2\beta_j +\beta_{j-1}), \quad \balpha=(j,k-j), \quad j=1, \ldots, k-1,
\end{eqnarray*}
see Proposition 3.2 in \cite{r24} 
%{\color{magenta}Marcon et al. (2016)}
%\citeN[Proposition 3.2]{marcon2016} 
and 
%Proposition \ref{prop:bspline_rel_pick_ang} 
equation \eqref{eq:polyrep_dens} of this manuscript. 
%
%For both representations, 
Consider the disjoint union topological space 
$$
\Phi_*:=\cup_{k \geq k_*} ( \{k\}\times \Phi_{k-1})
$$ 
and endow it with the corresponding Borel $\sigma$-field $\mathscr{B}_{\Phi_*}$. The maps $\phi_k^{(BP)}$ are 1-to-1 and continuous; thus, by the universal property of disjoint union topologies, there exists a unique continuous map $$\phi_*^{(BP)}:\mathcal{B}_*\mapsto \Phi_*$$ such that, for all $k \geq k_*$,
$$
\phi_k^{(BP)}(\boldsymbol{\beta})= \phi_*^{(BP)}((k,\boldsymbol{\beta}))
$$
whenever $\boldsymbol{\beta} \in \mathcal{B}_k$. Hence, a prior $\Pi$ on $(\mathcal{B}_*, \mathscr{B}_{\mathcal{B}_*})$ induces a prior $\Pi'=\Pi\circ\{\phi_*^{(\bullet)}\}^{-1}$ on $(\Phi_*,\mathscr{B}_{\Phi_*})$.
As done before for Pickands dependence funcitons, 
let $\Hset_*:=\cup_{k \geq k_*}\Hset_{k-1}$ and,
for $k \geq k_*$, define by 
\begin{equation}\label{eq:measuremap}
	\begin{split}
	&\phi_{\Hset_{k-1}}^{(BP)}:\Phi_{k-1} \mapsto  \Hset_*: \bphi^{(k)} \mapsto 
	%\sum_{0 \leq j \leq k-1} \eta_j {\color{magenta}b_{j,k-1}}
	\sum_{j=1}^2 \delta_{\be_j}(\cdot)\varphi_{\bkappa_j}+\int_{\pi_{\resimp}(\cdot \, \cap \simpint)}h_{k-2}(t)\diff t\\
	&\text{with} \quad  h_{k-2}(t)= \sum_{\balpha \in \Gamma_k} \varphi_\balpha \text{Be}(t|\alpha_1, \alpha_2)=
	\sum_{1 \leq j \leq k-1} \varphi_{(j,k-j)} \beta(t|j, k -j)\\
	& \text{and} \quad \bphi^{(k)}=(\varphi_{\bkappa_1}, \varphi_{\bkappa_2}, \varphi_\balpha, \, \balpha\in\Gamma_k)
	\end{split}
\end{equation}
%and
%$$
%\phi_{\Hset_{k-1}}^{(BS)}:\mathcal{B}_k \mapsto \Hset_*: (\eta_1, \ldots,\eta_{k-1}) \mapsto \sum_{1 \leq j \leq k-1} \beta_j\phi_{j, 2}
%$$
the map transforming linear coefficients into the corresponding angular pm with density in Bernstein polynomial form. Once more, these are 1-to-1 and continuous, hence there exists a unique continuous map 
$$
\phi^{(BP)}_\Hset : \Phi_*\mapsto \Hset_*
$$ 
such that, for all $k \geq k_*$, 
\begin{equation}
	\phi_{\Hset_{k-1}}^{(BP)}(\bphi)= \phi_{\Hset}^{(\bullet)}((k,\bphi))
\end{equation} 
whenever $\bphi \in \Phi_{k-1}$. Thus, a prior is induced on $(
\Hset_*, \mathscr{B}_{\Hset_*})
%\cup_{k \geq k'}\Hset_{k-1}
%\cup_{k \geq k'}\Hset_{k-1}
%, \mathscr{B}(\cup_{k \geq k'}\Hset_{k-1})
$ 
via
\begin{equation}
	\Pi_{\Hset_*}(B)= \Pi' \circ \{\phi_{\Hset}^{(\bullet)}\}^{-1}(B), \quad  \forall B \in \mathscr{B}_{\Hset_*},
\end{equation}
where $\mathscr{B}_{\Hset_*}$ is the Borel $\sigma$-field generated by $\dist_{KS}$. In turn, a prior on the Borel subsets $B$ of  $(\Hset,\dist_{KS})$ is defined via
\begin{equation}\label{eq:finprior}
\Pi_\Hset(B)= \Pi_{\Hset_*}(B\cap \Hset_*).
\end{equation}

%On the other hand, we have the following fact (the proof is immediate and, therefore, omitted).
%
%\begin{lemma}
%	The map $\phi_{\Aset,\Hset}:\Aset \mapsto \Hset: A \mapsto H$ is 1-to-1 and measurable  with respect to $\mathscr{B}(\Aset)$ and the Borel $\sigma$-field induced by $\dist_\infty$ on $\Hset$, respectively. Consequently, the prior $\Pi_\Aset$ on $(\Aset, \mathscr{B}(\Aset))$ induces a prior $\Pi_\Aset \circ \phi_{\Aset,\Hset}^{-1}$ on the Borel subsets of $(\Hset,\dist_\infty)$.
%\end{lemma}
%\begin{rem}\label{rem:prior_identity}

On the other hand, the map $\phi_{\Aset/\Hset}:\Aset \mapsto \Hset: A \mapsto H$ is 1-to-1 and measurable  with respect to $\mathscr{B}_\Aset$ and the Borel $\sigma$-field induced by $\dist_{KS}$ on $\Hset$, respectively. Consequently, the prior $\Pi_\Aset$ on $(\Aset, \mathscr{B}_{\Aset})$ induces a prior $\Pi_\Aset \circ \phi_{\Aset/\Hset}^{-1}$ on the Borel subsets of $(\Hset,\dist_{KS})$. We thus point out the following equivalence between pm's.
\begin{rem}\label{rem:prior_identity}
For each $B \in \mathscr{B}_{\Hset_*}$ we have the identity
	$$
	\{\phi_{\Aset}^{(\bullet)}\}^{-1}\circ \phi_{\Aset/\Hset}^{-1}(B) = \{\phi^{(\bullet)}\}^{-1} \circ \{\phi^{(\bullet)}_\Hset\}^{-1}(B),
	$$
	whence it immediately follows that $\Pi_\Hset=\Pi_\Aset \circ \phi_{\Aset/\Hset}^{-1}$. An analogous identity holds for the corresponding posterior distributions. In this view, a prior on the Pickands dependence function constructed as in Condition \ref{cond:prior_D2} induces a prior on the angular pm complying with Condition \ref{cond:angularprior}.
\end{rem}

\subsubsection{Prior on multivariate $H$}\label{appsec:priorH}

In this subsection we deal with the general case $d\geq 2$. As in Sections \ref{sec:KL_support}--\ref{sec:posterior_consistency} of the main article, we endow $\Hset$ with the a metric $\dist_W$ metrizing weak convergence
%, e.g. Prohorov metric, see
% \citeN[Ch. 11.3]{dudley_2002}
%Ch. 11.3 in {\color{magenta}Dudley (2002)}. 
and denote by $\mathscr{B}_{\Hset}$ the induced Borel $\sigma$-algebra. For $k\geq d+1$, we denote by $\Hset_k$ and $\Phi_k$ the sets of angular measures with Bernstein polynomial density of degree $k-d$ and the corresponding sets of vectors of mixture weights, introduced in Section \ref{sec:extreme_dep_poly} and Condition \ref{cond:angularprior} of the main paper. Observe that each $\Hset_k$ is a Borel set of $(\Hset,\dist_W)$, since it is the range of a continuous injective map of a Polish space, namely $\Phi_k$ endowed with the Euclidean metric;  see, e.g., Corollary A.7 in \cite{r2400}.
%(see, e.g., \citeNP[Corollary A.7]{takesaki78}). 
%{\color{magenta} \textbf{Comment}: the same holds true for the images of Borel subsets of $\Psi_k$ under the map $\psi \mapsto H_k$, we should perhaps mention this.}
Thus, 
$$
\Hset_*=\cup_{k \geq d+1} \Hset_k \in \mathscr{B}_\Hset.
$$
Arguments similar to those presented in the bivariate case show that a prior distribution $\Pi$ on the Borel sets of the disjoint union topological space
$$
\Phi_*:=\cup_{k \geq d+1}(\{k\}\times \Phi_k)
$$
induces a prior on $\Hset_*$ endowed with the subspace $\sigma$-algebra, which, in turn, allows to specify a prior $\Pi_\Hset$  on $(\Hset,\mathscr{B}_\Hset)$. 
Specifically, the construction is via equations \eqref{eq:measuremap}--\eqref{eq:finprior}, with the following adaptations: $k_*$ is set equal to $d+2$, $\Pi'$ is replaced by $\Pi$ and equation \eqref{eq:measuremap}
is generalised to
\begin{equation*}
\begin{split}
&\phi_{\Hset_{k-1}}^{(BP)}:\Phi_{k-1} \mapsto  \Hset_*: \bphi^{(k)} \mapsto 
%\sum_{0 \leq j \leq k-1} \eta_j {\color{magenta}b_{j,k-1}}
\sum_{j=1}^d \delta_{\be_j}(\cdot)\varphi_{\bkappa_j}+\int_{\pi_{\resimp}(\cdot \, \cap \simpint)}h_{k-d}(\bt)\diff \bt\\
&\text{with} \quad  h_{k-d}(\bt)= \sum_{\balpha \in \Gamma_k} \bphi_\balpha \text{Dir}(\bt|\balpha), \, \bt \in \mathring{\resimp}; \quad \bphi^{(k)}=(\varphi_{\bkappa_1}, \ldots,\varphi_{\bkappa_d}, \varphi_\balpha, \, \balpha\in\Gamma_k).
\end{split}
\end{equation*} 

\section{Proofs}\label{appsec:proofs}

\subsection{B-Splines}

%%%%%%%%%%%%%%%%%%%%%%%%%%%%%%%%%%%%%%%%%%%
\subsubsection{Proof of Proposition \ref{prop:bspline_cond}}
\label{sec:proofPropSP1}
%%%%%%%%%%%%%%%%%%%%%%%%%%%%%%%%%%%%%%%%%%%

For the spline in \eqref{eq:bspline_angdist} with $m=3$, we use knots $\tau_1=\tau_2=0$, $\tau_{j+2}=j/(k-2)$, $j=1, \ldots, k-3$, $\tau_{k}=\tau_{k+1}=1$. 
In particular, we have $H_{k-1}(0)=\eta_1$, $H_{k-1}(1^-)=\eta_{k-1}$.
Then, $H_{k-1}$ is a valid angular cdf  if and only if $\eta_1=p_0$, $\eta_{k-1}=1-p_1$, for some $0 \leq p_0,p_1 \leq 1/2$. Furthermore, for any $t_1 \leq t_2 \in[0,1)$ we have that
\begin{equation*}
\begin{split}
H_{k-1}(t_2)-H_{k-1}(t_1) & =  \int_{t_1}^{t_2}  H_{k-1}'(v)\diff v\\
&=  \int_{t_1}^{t_2}  \left\lbrace
\sum_{j=1}^{k-2} \left( \frac{\eta_{j+1} - \eta_{j}}{\tau_{j+2}-\tau_{j+1}}\right) \phi_{j+1,1}(v) \right\rbrace \diff v \\
&= \sum_{j=1}^{k-2} \left( \frac{\eta_{j+1} - \eta_{j}}{\tau_{j+2}-\tau_{j+1}}\right) \int_{t_1}^{t_2} \mathds{1}_{[\tau_{j+1},\tau_{j+2})}(v)\diff v.
\end{split}
\end{equation*}
Hence, $\eta_{j+1} \geq \eta_{j}$, $j=1, \ldots, k-2$ is a sufficient condition for $H_{k-1}$ to be monotone nondecreasing. 
In particular, in the case $\tau_{j+1} \leq t_1 \leq t_2<\tau_{j+2}$ the above equality reduces to 
$$
H_{k-1}(t_2)-H_{k-1}(t_1)=\left( \frac{\eta_{j+1} - \eta_{j}}{\tau_{j+2}-\tau_{j+1}}\right)(t_2-t_1)
$$
and therefore  $\eta_{j+1} \geq \eta_{j}$ is also a necessary condition. 
Finally, notice that the mean constraints in (C1) of the main article for $H_{k-1}$ are equivalent to
\begin{equation*}
\begin{split}
1/2= p_1+\int_{0}^{1}  v H_{k-1}'(v)\diff v = p_1 + \frac{1}{2}\sum_{j=1}^{k-2}(\eta_{j+1} - \eta_{j})(\tau_{j+2}+\tau_{j+1})
\end{split}
\end{equation*}
and
\begin{equation*}
\begin{split}
1/2=p_0+\int_{0}^{1}  (1-v)H_{k-1}'(v)\diff v=p_0+\frac{1}{2} \sum_{j=1}^{k-2}(\eta_{j+1} - \eta_{j}) \left(1-\frac{\tau_{j+1}+\tau_{j+2}}{2} \right),
\end{split}
\end{equation*}
which both reduce to (R9) after few algebraic manipulations.

Next, for the spline in \eqref{eq:bspline_pick} with $m=3$, we use knots $\tau_1=\tau_2=\tau_3=0$, $\tau_{j+3}=j/(k-2)$, $j=1, \ldots, k-3$, $\tau_{k+1}=\tau_{k+2}=\tau_{k+3}=1$. 
In particular, we have that $A_{k}(0)=\beta_1$ and $A_{k}(1)=\beta_k$, therefore $A_{k}(0)=A_{k}(1)=1$ if and only if $\beta_1=\beta_k=1$. Then, notice that for every $j=3, \ldots, k-1$
$$
A_{k}''(t)= \frac{2}{\tau_{j+1}-\tau_j}\left( 
\frac{\beta_j - \beta_{j-1}}{\tau_{j+2}-\tau_j} - 
\frac{\beta_{j-1}-\beta_{j-2}}{\tau_{j+1}-\tau_{j-1}}
\right), \quad t \in [\tau_j, \tau_{j+1}).
$$
Therefore, $A_k$ is convex if and only if the second terms in the right hand-side of the above display is nonnegative, which reduces to condition (R12).
Furthermore, we have that
$$
A_{k}'(0)=2 \frac{\beta_2 - \beta_1}{\tau_4-\tau_2}\phi_{2, 2}(0)= 2(\beta_2 - \beta_1)(k-2)
$$
and
$$
A_{k}'(1^-)=2 \frac{\beta_k - \beta_{k-1}}{\tau_{k+2}-\tau_k}\phi_{k,2}(1^-)= 2 (\beta_k -\beta_{k-1})(k-2).
$$
Then, we can continuously extend $A_k'$ from $[0,1)$ to $[0,1]$ by imposing $A_k'(1)= 2 (\beta_k -\beta_{k-1})(k-2)$. 
By a corollary in \citep[p. 131]{r35}
% {\color{magenta}de Boor (1978, p. 131)}
%\citeN[p. 131]{de1978} we have
%
$A_{k}(t) \leq \max (\beta_{j-2}, \beta_{j-1}, \beta_j)$ $ \leq 1$ for all $t \in [\tau_j, \tau_{j+1}]$.
As a consequence, under conditions (R10) and (R12), $A_{k}$ is convex, equals 1 at $\{0\}$ and $\{1\}$ and is less than equal to 1 on $(0,1)$. 
Such a function $A_{k}$ satisfies the lower bound condition $\max(t,1-t) \leq A_{k}(t)$ if and only if the inequalities
\begin{equation}\label{eq: splines bounds first deriv}
-1 \leq A_{k}'(0) \leq 0, \quad 0 \leq A_{k}'(1) \leq 1
\end{equation}
hold true, which are satisfied in turn if $1- 2^{-1}(k-2)^{-1} \leq \beta_2 , \beta_{k-1} \leq 1$. This is 
is equivalent to (R11). Consequently, (R10)-(R12) are sufficient conditions for $A_{k}$ to satisfy conditions (C2)-(C3) of the main article.
For the converse implication, assume $A_{k}$ is a valid Pickands dependence function. Then, it satisfies
\eqref{eq: splines bounds first deriv}, since it is convex and $\max(t,1-t) \leq A_{k}(t) \leq 1$, 
which is (R11), together with (R12) and $\beta_1=\beta_k=1$. Finally, the inequalities $\beta_j\leq 1$, $j=3, \ldots, k-1$ are necessary conditions by (R12) and the inequality $A_{k}(t) \leq \max (\beta_{j-2}, \beta_{j-1}, \beta_j) \leq 1$ for all $t \in [\tau_j, \tau_{j+1}]$. The proof is now complete.

%%%%%%%%%%%%%%%%%%%%%%%%%%%%%%%%%%%%%%%%%%%
\subsubsection{Proof of Proposition \ref{prop:bspline_rel_pick_ang}}\label{sec:proofBS2}
%%%%%%%%%%%%%%%%%%%%%%%%%%%%%%%%%%%%%%%%%%%

Let $A_k$ be defined via \eqref{eq:bspline_pick} with $m=3$ and coefficients satisfying restrictions (R10)-(R12) in Proposition \ref{prop:bspline_cond}. 
Then, $(1+A_k'(t))/2$ provides the expression of a valid angular cdf for $t\in [0,1)$. Therefore, the first part of the statement follows from \eqref{eq:bspline_fder_pick}, after a suitable reorganisation of the knots.

Let $H_{k-1}$ be defined via \eqref{eq:bspline_angdist} with $m=3$ and coefficients satisfying the restrictions (R8)-(R9) in Proposition \ref{prop:bspline_cond}.
Then, $1+2\int_0^tH_{k-1}(v)\diff v - t$ is a valid Pickands dependence function for $t\in [0,1]$. 
Properties (36) and (33) in 
\citep[pp. 96 and 128]{r35}
%{\color{magenta}de Boor (1978, pp.  96 and 128)}
%\citeN[pp. 96 and 128]{de1978} 
and some simple algebraic manipulations allow to express it as linear combination of B-splines of order 3, leading to the second half of the statement.
%

%%%%%%%%%%%%%%%%%%%%%%%%%%%%%%%%%%%%%%%%%%%%%%%%%%%%%%%%%
\subsubsection{Proof of Proposition \ref{prop:bspline_full_supp}}\label{sec:proofBS3}
%%%%%%%%%%%%%%%%%%%%%%%%%%%%%%%%%%%%%%%%%%%%%%%%%%%%%%%%%
Let $A \in \Aset$ and for every $k>3$ consider the sequence of knots $(\tau_j, j=1,\ldots,k+3)=(0,0,0,1/(k-2), \ldots, (k-3)/(k-2),1,1,1)$. Set
$$
\tau_j^*:=\frac{\tau_{j+1}+\tau_{j+2}}{2}, \quad j=1, \ldots, k.
$$
The spline $S_k(t;A):=\sum_{j=1}^k A(\tau_j^*)\phi_{j,3}(t)$ (known as Shoenberg's
variation diminishing approximation) is shape preserving, thus $S_k(t;A)\in\Aset$, and satisfies
$\dist_{\infty}(A, S_k(\cdot;A))\to 0$ as $k \to \infty$ \citep[][Ch. 11]{r35}.
% \citeN[Ch. 11]{de1978}
%{\color{magenta}de Boor (1978, Ch. 11)}. 
By this fact and Proposition \ref{prop:bspline_cond}, the first half of the statement follows. The second half follows from the first one and Proposition \ref{prop: basic metric bvt}\ref{res: new pick}.

%
%
%%%%%%%%%%%%%%%%%%%%%%%%%%%%%%%%%%%%%%%%%%%%
%\subsubsection{Proof of Proposition \ref{prop:bspline_full_supp}}\label{sec:proofBS3}
%%%%%%%%%%%%%%%%%%%%%%%%%%%%%%%%%%%%%%%%%%%%
%
%
%In this section we provide the proof of Proposition \ref{prop: top supp} and a series of technical results involving max-stable distributions with angular densities in BP form, which are used in the derivations of Sections \ref{sec:res_cons_simp}-\ref{appendix:semi}. The arguments presented herein are valid in dimension $d \geq 2$, unless otherwise specified.
%%
%In what follows, for a given integer $k>d$: if $d=2$, let $\varphi_{\bkappa_2}=\eta_0$, $\varphi_\balpha=\eta_{\alpha_1}-\eta_{\alpha_1-1}$, $\alpha \in \Gamma_k$, $\varphi_{\bkappa_1}=1-\eta_{k-1}$, where $(\eta_0,\ldots,\eta_{k-1})$ satisfies (R6)-(R7); if $d>2$, denote
%\begin{eqnarray}
%\bphi_{\partial}&:=&(\varphi_{\bkappa_1},\ldots,\varphi_{\bkappa_d}),\\
%\label{eq:interior_coef}
%\bphi_\circ&:=&(\varphi_{\balpha},\balpha\in\Gamma_k),\\
%\bphi&:=&(\bphi_{\partial}, \bphi_\circ),
%\end{eqnarray}
%and let $\bphi$ be a coefficients' vector satisfying (R11)-(R12). Accordingly, the symbol $\Hset_k$ denotes the set defined on pages {\color{magenta}XXX} of the main article, for the cases $d=2$ and $d>2$, respectively. 
%%
%When dealing with sequences $H_k \in \Hset_k$, $k=d+1,d+2,\ldots$, we occasionally use the symbol $\bphi^{(k)}=(\bphi_\partial^{(k)},\bphi_\circ^{(k)})$
%to denote the linear mixture weights pertaining to each $H_k$. 
%%

%%%%%%%%%%%%%%%%%%%%%%%T%%%%%%%%%%%%%%%%%%%

\subsection{Bernstein polynomials}

%%%%%%%%%%%%%%%%%%%%%%%%%%%%%%%%%%%%%%%%%%

In this subsection we provide the proof of Proposition \ref{prop: top supp} and a series of technical results involving max-stable distributions with angular densities in Bernstein polynomial form, which are used in the derivations of Sections \ref{appsec:proof_bivar}-\ref{appendix:semi}.
In what follows, for a given integer $k>d$
denote
\begin{eqnarray}
\bphi_{\partial}&:=&(\varphi_{\bkappa_1},\ldots,\varphi_{\bkappa_d}),\\
\label{eq:interior_coef}
\bphi_\circ&:=&(\varphi_{\balpha},\balpha\in\Gamma_k),\\
\bphi&:=&(\bphi_{\partial}, \bphi_\circ),
\end{eqnarray}
and let $\bphi$ be a coefficients' vector satisfying (R1)-(R2). Accordingly, the symbol $\Hset_k$ denotes the set defined in Section \ref{sec:extreme_dep_poly} of the main article. 
When dealing with sequences $H_k \in \Hset_k$, $k=d+1,d+2,\ldots$, we occasionally use the symbol $\bphi^{(k)}=\left(\bphi_\partial^{(k)},\bphi_\circ^{(k)}\right)$
to denote the mixture weights pertaining to each $H_k$. 
%

%%%%%%%%%%%%%%%%%%%%%%%%%%%%%%%%%%%%%%%%%%
\subsubsection{Proof of Proposition \ref{prop: top supp}}\label{appsec:proofpoly}
%%%%%%%%%%%%%%%%%%%%%%%%%%%%%%%%%%%%%%%%%%

\begin{lemma}\label{lem: approx}
	Let $\Hset'$ be as in Definition \ref{cond:mvt_angular}\ref{cond:prior_ang_set} and $d \geq 2$.
	For every $H \in \Hset'$ there exists a sequence $H_k \in\mathcal{H}_k$, with coefficients $\bphi^{(k)}$, such that
	\begin{equation}\label{eq: approx lem}
	\dist_\infty\left(h, 
	h_{k-d}\right)=o(1), \quad k  \to \infty.
	\end{equation}
\end{lemma}

\begin{proof}
	By assumption, the angular density $h$ admits an extension $\bar{h}$ which is bounded and continuous on $\resimp$. Then,
	defining $\balpha_{1:d-1}=(\alpha_1,\ldots,\alpha_{d-1})$ and
	\begin{equation*}
	\begin{split}
	b_{\balpha-\bone}(\bt)&=\frac{(k-d)!}{\prod_{j=1}^{d-1}(\alpha_j-1)!(k-d-\Vert \balpha_{1:d-1} -\bone\Vert_1 )!}\prod_{j=1}^{d-1}t_j^{\alpha_j-1}(1-\Vert \bt\Vert_1)^{k-d-\Vert \balpha_{1:d-1}-\bone \Vert_1}\\
	&=\frac{(k-d)!}{\prod_{j=1}^d(\alpha_j-1)!}\prod_{j=1}^{d-1}t_j^{\alpha_j} (1- \Vert \bt \Vert_1)^{\alpha_d-1}\\
	&=\frac{(k-d)!}{(k-1)!}\text{Dir}(\bt; \balpha)
	\end{split}
	\end{equation*}
	where $\alpha_d=k-\Vert \balpha_{1:d-1} \Vert_1$, it is well known that, as $k \to \infty$ 
	$$
	B_{k-d}(\bar{h}; \bt):= \sum_{\balpha \in \Gamma_k} \bar{h}\left(\frac{\balpha_{1:d-1} -\bone}{k-d}\right)b_{\balpha-\bone}(\bt)=\bar{h}(\bt)+o(1),
	$$
	where the error term is uniform over $\resimp$, see e.g. 
	\citep[p. 51]{r44}.
%	{\color{magenta}Lorentz (1986)}. 
	%	\citeN{lor86}.
	Therefore, letting $c= 1/\Gamma(d)$, for every $\epsilon \in (0, \wedge_{1 \leq j \leq d}p_j)$ there exists $k_\epsilon$ such that for all $ k \geq k_\epsilon$
	\begin{itemize}
		\item 	$
		\dist_\infty(B_{k-d}(\bar{h} ; \cdot),\bar{h}) < \epsilon
		$
		and $\int_{\mathring{\resimp}}B_{k-d}(\bar{h}; \bt)\diff \bt < \int_{\mathring{\resimp}}h(\bt)\diff \bt+\epsilon<1$;
		\item $\int_{\mathring{\resimp}}t_j B_{k-d}(\bar{h};\bt)\diff \bt <d^{-1}-p_j+c\epsilon d^{-1} < d^{-1}$, for $j=1, \ldots, d-1$, and 
		$\int_{\mathring{\resimp}}(1-\Vert \bt \Vert_1) B_{k-d}(\bar{h};\bt)\diff \bt <d^{-1}-p_d+c\epsilon d^{-1} < d^{-1} $.
	\end{itemize}
	As a consequence, setting for each $k\geq k_\epsilon$
	\begin{eqnarray*}
		\varphi_{\bkappa_j}^{(k)}&=&\frac{1}{d}-\int_{\mathring{\resimp}}t_j B_{k-d}(\bar{h};\bt)\diff \bt, \quad j=1,\ldots,d-1,\\
		\varphi_{\bkappa_d}^{(k)}&=&\frac{1}{d}-\int_{\mathring{\resimp}}(1-\Vert \bt \Vert_1) B_{k-d}(\bar{h};\bt)\diff \bt, \\
		\varphi_\balpha^{(k)}&=&
		\bar{h}\left(\frac{\balpha_{1:d-1} -\bone}{k-d}\right)
		%		\frac{1}{(k-1)\cdots(k-d+1)}, 
		\frac{(k-d)!}{(k-1)!},
		\hspace{3em} \balpha \in \Gamma_k
	\end{eqnarray*} 
	we obtain a sequence of valid angular pm's $H_{k}$ of the form \eqref{eq: BPoly measure}. In particular, the coefficients $\bphi^{(k)}$ satisfy (R11)-(R12), being the latter necessary and sufficient conditions to define valid angular pm's via multivariate Bernstein polynomials.
\end{proof}

\begin{proof}[Proof of Proposition \ref{prop: top supp}]
	Fix an arbitrary $\epsilon>0$. For any positive constant $c< \epsilon/3$, we can define a valid spectral pm $H^* \in \Hset$ 
	%	 satisfying $\dist_T(H,H')<\epsilon/3$ 
	via
	$$
	H^*(B)=\sum_{j=1}^d  p_j^* \delta_{\be_j}(B)+\int_{\pi_\resimp(B\cap \simpint)} h^*(\bv) \diff \bv,
	$$
	for all Borel subsets $B$ of $\tilde{\simp}$, where $h^*=h/(1+c)$ and $p_j^*=1/d- \int_{\mathring{\resimp}}t_j h^*(\bt)\diff \bt$, $j=1, \ldots, d-1$, $p_d^*=1/d- \int_{\mathring{\resimp}}(1-\Vert \bt \Vert_1)h^*(\bt)\diff \bt$. Note that
	$$
	\Vert h - h^* \Vert_1 \leq c/(1+c) <\epsilon/3
	$$
	and $\int_{\mathring{\resimp}}h^*(\bt)\diff \bt<1$, $\min_{j=1, \ldots, d} p_j^*>0$. There exists a nonnegative continuous function $h_K^*$ with compact support $K \subset \mathring{\resimp}$ such that 
	$$
	\Vert h^* - h_K^* \Vert_1 < \frac{1}{6} \min\left\lbrace{\epsilon}, \min_{j=1, \ldots, d} p_j^*\right\rbrace=:c',
	$$
	see e.g. \citep[Ch. 3]{r871}.
%	Chapter 3 of {\color{magenta} Rudin (1987)}.
	%\citeN[Ch. 3]{Rudin1987}. 
	Consequently, choosing $c''\in (0,c')$, setting $h^{**}=h_K^*+c''$, $p_j^{**}=1/d- \int_{\mathring{\resimp}}t_j h^{**}(\bt)\diff \bt$, $j=1, \ldots, d-1$,
	$p_d^{**}=1/d- \int_{\mathring{\resimp}}(1-\Vert \bt \Vert_1)h^{**}(\bt)\diff \bt$
	and defining $H^{**}$ via
	$$
	H^{**}(B)=\sum_{j=1}^d p_j^{**}\delta_{\be_j}(B)+\int_{\pi_\resimp(B \cap \simpint)} h^{**}(\bt) \diff \bt,
	$$
	for all Borel subsets $B$ of $\tilde{\simp}$, we have $H^{**}\in \Hset'$ and 
	$$
	\Vert h^* - h^{**}\Vert_1 \leq  \Vert h^{*} - h^*_K \Vert_1 + \Vert h^*_K - h^{**} \Vert_1 < \epsilon/3.
	$$
	Finally, by Lemma \ref{lem: approx}, there exists a sequence $H_k \in \Hset_k$ such that, for $k$ sufficiently large, $\Vert h^{**}- h_{k-d} \Vert<\epsilon/3$. The result now follows by triangular inequality.
\end{proof}

\begin{cor}\label{lem:1to1}
	The set $\{g_\bone(\cdot|H): \, H \in \cup_{k=d+1}^\infty\Hset_k \}$ is $\dist_H$-dense in $\mathcal{G}_\bone$, given in equation \eqref{eq:dens_space} of the main article. 
	%	Consequently, a prior $\Pi_\Hset$ on $(\Hset, \mathscr{B(\Hset)})$ satisfying assumption (i) of Theorem 3.11 induces a prior $\Pi_{\mathcal{G}_*}$ with full support on $(\mathcal{G}_*,\mathscr{B}(\mathcal{G}_*))$.
\end{cor}

\begin{proof}
	As a consequence of Proposition \ref{prop: top supp}, for any $H\in \Hset$ there exists a sequence $H_k \in \Hset_k $ such that $\lim_{k \to \infty}\Vert h - h_{k-d} \Vert_{1} \to 0$.	Then, there exists a subsequence $h_{k_s-d}$ that converges to $h$ 
	%	almost uniformly, i.e. for every $\epsilon'>0$ there exists a measurable set $B_{\epsilon'} \subset \simpint_d$ such that $\mathring{\nu}(B_{\epsilon'}^c)<\epsilon'$ and $h_{k_s-d}$ converges to $h$ uniformly on $B_{\epsilon'}$ 
	pointwise almost-everywhere 
	\citep[Corollary 1.5.10]{r2500}. 
%	(see, e.g., Corollary 1.5.10 in 
	%	\citeNP[Corollary 1.5.10]{tao11}	
%	{\color{magenta} Tao 2011}
%	). 
	We next show that also 
	\begin{equation}\label{eq:totalvarsub}
	\dist_{T}(G_\bone(\cdot|H_{k_s}),G_\bone(\cdot|H))\to 0
	\end{equation}
	as $k_s \to \infty$, entailing that $\dist_H(g_\bone(\cdot|H_{k_s}),g_\bone(\cdot|H)) \to 0$ and establishing the result.

	In what follows, we use the notation introduced in Section \ref{appsec:notation} of this manuscript. We denote by $\Lambda_{k_s}:=\Lambda(\cdot|H_{k_s})$ and $\Lambda:=\Lambda(\cdot|H)$ the exponent measures corresponding to $H_{k_s}$ and $H$, respectively. 
	Denote by $N^{(k_s)}$ and $N$ two Poisson random measures on $E$, with mean measures $\Lambda_{k_s}$ and $\Lambda$, respectively. For $I \subset \{1, \ldots, d\}$, denote by $\pi_I$ the projection map $\pi_I(\bx)=(x_j)_{j \in I}$, $\bx \in E$. Accordingly, let $$
	N_{j}^{(k_s)}:=N^{(k_s)} \circ \pi_{\{j\}}^{-1}, \quad  N_j:=N \circ \pi_{\{j\}}^{-1}, \quad j=1, \ldots,d,
	$$ 
	denote the marginal Poisson random measures. Moreover, let $\bY^{(k_s)}$ and $\bY$ be rv's distributed according to $G_\bone(\cdot|H_{k_s})$ and $G_\bone(\cdot|H)$, respectively. For $t>0$, define $\bY_t^{(k_s)}$ via
	$$
	\bY_t^{(k_s)}:=\left(Y_1^{(k_s)}\indic(Y_1^{(k_s)}>t), \ldots, Y_d^{(k_s)}\indic(Y_d^{(k_s)}>t)\right)
	$$
	and $\bY_t$ in an analogous fashion. For a random element $X$, let $\mathcal{L}(X)$ denote the pertaining pm. Fix a small $\epsilon>0$ and observe that
	\begin{equation}\label{eq: central_ineq}
	\begin{split}
	\dist_T((G_\bone(\cdot|H_{k_s}),G_\bone(\cdot|H))) &\leq \dist_T(\mathcal{L}(\bY^{(k_s)}),\mathcal{L}(\bY_t^{(k_s)}))+\dist_T(\mathcal{L}(\bY),\mathcal{L}(\bY_t)) \\
	& \quad+
	\dist_T(\mathcal{L}(\bY_t^{(k_s)}),\mathcal{L}(\bY_t)).
	\end{split}
	\end{equation}
	Since $G_\bone(\cdot|H_{k_s})$ has unit Fr\'echet margins, it holds that
	$$
	\dist_T(\mathcal{L}(\bY^{(k_s)}),\mathcal{L}(\bY_t^{(k_s)}))\leq \mathbb{P}(\bY^{(k_s)} \neq \bY^{(k_s)}_t)\leq d e^{-1/t}.
	$$
	Analogously, $\dist_T(\mathcal{L}(\bY),\mathcal{L}(\bY_t)) <de^{-1/t}$. Thus, for $t$ sufficiently small, the sum of the first two terms on the right-hand side of \eqref{eq: central_ineq} is smaller than $\epsilon/2$. Next, for  $j=1, \ldots, d$, denote  $$
	N_{j,t}^{(k_s)}=N^{(k_s)}_j(\cdot\cap (t, \infty]), \quad N_{j,t}=N_j(\cdot\cap (t, \infty]).
	$$
	For a Poisson random measure $N^*=\sum_{i \geq i} \delta_{X_i}$ on $(t, \infty]$, define the random functional
	$$
	\phi(N^*)=
	\begin{cases}
	0, \hspace{5em} \text{if } N^{*}((t,\infty])=0\\
	\max_{i \geq i} X_i, \quad \text{otherwise}
	\end{cases}.
	$$ 
	Then, we have that $\bY_t^{(k_s)}$ and $\bY_t$ are equivalent in distribution to $\boldsymbol{\Phi}_{t}^{(k_s)}:=\{\phi(N_{j,t}^{(k_s)})\}_{j=1}^d$ and $\boldsymbol{\Phi}_{t}:=\{\phi(N_{j,t})\}_{j=1}^d$, respectively. Corollary 1.4.2 in \cite{r2002}
	%	\citeN{reiss1993} 
	and arguments on pages 237--239 in \cite{r2003},
%	{\color{magenta} Kaufmann and Reiss (1993)},
	%	\citeN[pp. 237--239]{kauf_reiss1995}, 
	with a few adaptations, now yield
	\begin{equation}\label{eq:second_cent_ineq}
	\begin{split}
	\dist_T(\mathcal{L}(\bY_t^{(k_s)}),\mathcal{L}(\bY_t))& = \dist_T(\mathcal{L}(\boldsymbol{\Phi}_{t}^{(k_s)}),\mathcal{L}(\boldsymbol{\Phi}_{t}))\\
	&\leq c_d \sum_{I \subset\{1, \ldots,d\}} \dist_T(\Lambda_{k_s}^{(I,t)}\circ \pi_I^{-1},\Lambda_{k_s}^{(I,t)}\circ \pi_I^{-1})\\
	&\leq  c_d \sum_{I \subset\{1, \ldots,d\}} \dist_T(\Lambda_{k_s}^{(I,t)},\Lambda_{k_s}^{(I,t)})
	\end{split}
	\end{equation}
	where $c_d$ is a positive constant depending only on $d$, $E_{I,t}=\{\by \in E: y_j >t, \forall j \in I\}$, $\Lambda_{k_s}^{(I,t)}=\Lambda_{k_s}(\cdot\cap E_{I,t})$ and $\Lambda_{k_s}^{(I,t)}=\Lambda_{k_s}(\cdot\cap E_{I,t})$. Observe that each $E_{I,t}$ is relatively compact, thus $\Lambda_{k_s}^{(I,t)}$ and $\Lambda^{(I,t)}$ are finite measures. Moreover, 
	\begin{eqnarray*}
		\Lambda_{k_s}^{(I,t)}(E_{I,t}\cap E_{\{1, \ldots,d\}}) &=& d\int_{\simpint}\min_{j \in I} (w_j/t)H_{k_s}(\diff \bw),\\
		\Lambda^{(I,t)}(E_{I,t}\cap E_{\{1, \ldots,d\}})&=& d\int_{\simpint}\min_{j \in I} (w_j/t)H(\diff \bw),
	\end{eqnarray*}
	therefore it can be easily seen that $\Lambda_{k_s}^{(I,t)}(E_{I,t}\cap E_{\{1, \ldots,d\}}) \to \Lambda^{(I,t)}(E_{I,t}\cap E_{\{1, \ldots,d\}})$ as $k_s \to \infty$. Consequently, by Scheff\'e's lemma, as $k_s \to \infty$,
	\begin{equation*}
	\begin{split}
	\dist_T(\Lambda_{k_s}^{(I,t)},\Lambda_{k_s}^{(I,t)}) &\leq \frac{dt}{2} \Vert h_{k_s}-h \Vert_1\\
	&\quad+\int_{E_{I,t}\cap E_{\{1,\ldots,d\}}} d \Vert \bz \Vert_1^{-d-1}|h_{k_s}\circ \pi_\resimp(\bz/\vert \bz \vert_1)-h\circ \pi_\resimp(\bz/\vert \bz \vert_1)| \diff \bz\\
	& \to 0.
	\end{split}
	\end{equation*}
	Due to \eqref{eq:second_cent_ineq}, we now deduce that, for large enough $k_s$, also the third therm on the right-hand side of \eqref{eq: central_ineq} is smaller than $\epsilon/2$ and $\dist_T((G_\bone(\cdot|H_{k_s}),G_\bone(\cdot|H))) < \epsilon$. The result in \eqref{eq:totalvarsub} now follows and the proof is complete.
	%	The second result can be established using arguments similar to those used to prove the second statement of Lemma \ref{lem: bivar_dens}.	
\end{proof}

%%%%%%%%%%%%%%%%%%%%%%%%%%%%%%%%%%%%%%%%%%
\subsubsection{Additional technical results}
%%%%%%%%%%%%%%%%%%%%%%%%%%%%%%%%%%%%%%%%%%

%
The following lemmas are used in the proofs of Theorems
\ref{theo:post_consistency_mvt}, \ref{th:alpha_frec} and \ref{theo: cons_weibull}, provided in Sections \ref{appsec:proof_simp}--\ref{appendix:semi} of this manuscript.
In order to provide a concise account of the rather involved algebraic arguments used to establish such lemmas, we
make use of the following compact notations.
For $I \subset \{1,\ldots,d\}$, denote by
\begin{equation}\label{eq:I_min}
I^-:=I \setminus\{ \max(i: i \in I) \}
\end{equation}
the set obtained by removing from $I$ its largest element. As a convention, products over empty sets are meant as equal to $1$, e.g., $\bx \in \mathbb{R}^d$ and $I = \emptyset$ yield $\prod_{i\in I}x_i=1$. For a function $f:\mathbb{R}^d\mapsto\mathbb{R}$ and subsets $I_1, I_2 \subset \{1,\ldots,d\}$, $I_1 \cap I_2=\emptyset$, we denote by $f(\bx)|_{\bx_{I_1}=\by_{{I_1}}}^{\bx_{I_2}=\by_{I_2}}$ the function $f$ evaluated at $\bx$ with fixed components $x_i=y_i$, for $i \in I_j$, $j=1,2$.
Finally, for $\bv \in \resimp$, we define the function
\begin{equation*}
\begin{split}
&q(I, \balpha, \bv):=\\
& \begin{cases}
0, & \balpha=\bkappa_j, j \in \{1,\ldots,d\}, I \neq \{j\},\\
{d}v_{I}^{-2}, & (I,\balpha) \in\cup_{i=1}^{d-1} (\{i\}\times\{\bkappa_i\}),\\
{d}(1-\Vert \bv \Vert_1)^{-2}, & (I,\balpha) =(\{d\}, \bkappa_d),\\
\frac{{d}v_{I}^{-2}\balpha_{I}}{k} \left(1-\mathcal{I}_{v_{I}}(\balpha_{I}+1,\Vert \balpha_{I^\complement}\Vert_1)\right), & I \in \{\{1\}, \ldots, \{d-1\}\}, \balpha \in \Gamma_k,\\
\frac{{d}(1-\Vert \bv \Vert_1)^{-2}\balpha_{I}}{k} \left(1-\mathcal{I}_{1-\Vert \bv \Vert_1}(\alpha_{I}+1,\Vert \balpha_{I^\complement}\Vert_1)\right), & I =\{d\}, \balpha \in \Gamma_k,\\
\text{Dir}(\bv; \balpha),  & I =\{1, \ldots,d\}, \balpha \in \Gamma_k,\\
\frac{d\, \Vert \balpha_{I} \Vert_1}{k(1-\Vert \bv_{I^c} \Vert_1)^{|I|+1}}\times\\
\text{Dir}\left(
\frac{\bv_{I^-}}{1-\Vert \bv_{I^c} \Vert_1 }
;
\balpha_{I}
\right)
\left(1-\mathcal{I}_{1-\Vert \bv_{I^c}\Vert_1}(\Vert \balpha_{I} \Vert_1+1,\Vert \balpha_{I^\complement}\Vert_1)\right), & d \in I, 1<|I|<d, \balpha \in\Gamma_k,\\
\frac{d\, \Vert \balpha_{I} \Vert_1}{k\Vert \bw_{I} \Vert_1^{|I|+1}}\times\\
\text{Dir}\left(
\frac{\bv_{I^-}}{\Vert \bv_{I} \Vert_1}
;
\balpha_{I}
\right)
\left(1-\mathcal{I}_{\Vert \bv_{I}\Vert_1 }(\Vert \balpha_{I} \Vert_1+1,\Vert \balpha_{I^\complement}\Vert_1)\right), & d \notin I, 1<|I|<d, \balpha \in\Gamma_k.
\end{cases}
\end{split}
\end{equation*}
For any $a,b>0$ and $x \in [0,1]$, $\mathcal{I}_x(a,b)$ is the regularised incomplete Beta function, namely the cdf of a Beta distrubtion of parameters $a,b$ evaluated at $x$, while the (complete) Beta function is denoted by $B(a,b)$.

\begin{lemma}\label{lem: expression omega}
	Let $k\geq d+1$, $H_k \in \Hset_k$ and $I \subset \{1, \ldots,d\}$. Then, for every $\by \in \Yset$  
	\begin{equation}\label{eq:expr_delta}
	\begin{split}
	-V_I(\by|H_k)=\Vert \by \Vert_1^{-|I|-1}
	\sum_{\balpha}
	\varphi_\balpha q\left(I, \balpha, \frac{\by_{1:d-1}}{\Vert \by \Vert_1}\right)
	,
	\end{split}
	\end{equation}
	where $\balpha$ ranges over $\cup_{i=1}^d \{\bkappa_i\} \cup \Gamma_k$.
\end{lemma}
\begin{proof}
	We consider different types of $I$ separately.

	\noindent
	\underline{\textbf{Case 1}}\textbf{: $I=\{j\}$, for some $j \in \{1, \ldots,d\}$.}
	Some changes of variables allow to express $-V_I(\by|H_k)$ as 
	\begin{equation*}
	\begin{split}
	&= d \varphi_{\bkappa_j}y_I^{-2}+\int_{(\bzero,\by_{I^\complement})} d \Vert \bz \Vert_1^{-d-1} 
	h_{k-d}\circ \pi_\resimp \left(\bz/\Vert \bz\Vert_1 \right) \big{|}_{z_I=y_I} \diff \bz_{I^\complement}\\
	&=d \varphi_{\bkappa_j}y_I^{-2}+\int_{y_I}^{\Vert \by \Vert_1}ds^{|I^\complement|-d-2}\int_{\mathcal{W} (I^\complement;y_I/s)} h_{k-d}\circ \pi_\resimp(\bw)\big{|}_{w_I=y_I/s}^{w_l=1-\Vert \bw_{-l} \Vert_1} \diff \bw_{I^\complement \setminus\{l\}} \diff s\\
	&=d \varphi_{\bkappa_j}y_I^{-2}+\int_{y_I/\Vert \by \Vert_1}^{1}t^{d-|I^\complement|}dy_I^{|I^\complement|-d-1}
	\int_{\mathcal{W} (I^\complement;t)} h_{k-d}\circ \pi_\resimp(\bw)\big{|}_{w_I=t}^{w_l=1-\Vert \bw_{-l} \Vert_1} \diff \bw_{I^\complement \setminus\{l\}} \diff t,
	\end{split}
	\end{equation*}
	where $\mathcal{W} (I^\complement;x)=\{\bv \in(0,1)^{|I^\complement|-1}: \Vert \bv \Vert_1 \leq 1-x \}$, $\bw_{-l}=\bw_{\{1,\ldots,d\} \setminus \{l\}}$ and $l = \max\{j: j\in I^\complement\}$. 
	Recall that $h_{k-d}(\bt)=\sum_{\balpha \in \Gamma_k}\varphi_\balpha \text{Dir}(\bt;\balpha)$, $\bt \in \resimp$. Then, the aggregation property of the Dirichlet distribution and some algebraic manipulations allow to rephrase the right-hand side as
	\begin{equation*}
	\begin{split}
	&dy_I^{-2}\left(
	\varphi_{\bkappa_j}
	+\sum_{\balpha \in \Gamma_k}\varphi_\balpha
	\int_{y_I/\Vert \by \Vert_1}^{1} t \text{Dir}\left(t; \alpha_I, \Vert
	\balpha_{I^\complement} \Vert_1\right)\diff t
	\right)\\
	&\quad = dy_I^{-2}\left(
	\varphi_{\bkappa_j}
	+\sum_{\balpha \in \Gamma_k}\varphi_\balpha
	\frac{B(\alpha_I+1, \Vert \balpha_{I^\complement} \Vert_1)}{B(\alpha_I, \Vert \balpha_{I^\complement} \Vert_1)}
	\left(
	1- \mathcal{I}_{y_I/\Vert \by \Vert_1}
	(\alpha_I+1, \Vert \balpha_{I^\complement}\Vert_1)
	\right)
	\right)\\
	&\quad = dy_I^{-2}\left(
	\varphi_{\bkappa_j}
	+\sum_{\balpha \in \Gamma_k}\varphi_\balpha
	\frac{\alpha_I}{k}
	\left(
	1- \mathcal{I}_{y_I/\Vert \by \Vert_1}
	(\alpha_I+1, \Vert \balpha_{I^\complement}\Vert_1)
	\right)
	\right).
	\end{split}
	\end{equation*}
	The result now follows after some simple algebraic tweaking.

	\noindent
	\underline{\textbf{Case 2}}\textbf{: $1<|I|<d$.}
	Arguments similar to the ones above allow to rephrase $-V_I(\by|H_k)$ as follows
	\begin{equation*}
	\begin{split}
	&= \int_{(\bzero,\by_{I^c})} d \Vert \bz \Vert_1^{-d-1} h_{k-d}\circ \pi_\resimp\left(\bz/\Vert \bz\Vert_1 \right) \big{|}_{\bz_I=\by_I} \diff \bz_{I^\complement}\\
	&=\int_{\Vert \by_I \Vert_1}^{\Vert \by \Vert_1}ds^{|I^\complement|-d-2}\int_{\mathcal{W} (I^\complement;\Vert \by_I \Vert_1/s)} h_{k-d}\circ \pi_\resimp(\bw)\big{|}_{\bw_I=\by_I/s}^{w_l = 1- \Vert \bw_{-l}\Vert_1} \diff \bw_{I^\complement \setminus\{l\}}\\
	&=\int_{\Vert\by_I\Vert_1/\Vert \by \Vert_1}^{1}dt^{d-|I^\complement|}\Vert \by_I \Vert_1^{|I^\complement|-d-1}
	\int_{\mathcal{W} (I^\complement;t)} h_{k-d}\circ \pi_\resimp(\bw)\big{|}_{\bw_I=t\by_I/\Vert \by_I\Vert_1}^{w_l = 1- \Vert \bw_{-l}\Vert_1} \diff \bw_{I^\complement \setminus\{l\}}\\
	&=
	d\Vert \by_I\Vert_1^{-|I|-1}
	\sum_{\balpha \in \Gamma_k}\varphi_\balpha
	\int_{\Vert\by_I\Vert_1/\Vert \by \Vert_1}^{1} t^{|I|} \text{Dir}\left(t \frac{\by_I}{\Vert \by_I \Vert_1};\balpha_I, \Vert
	\balpha_{I^\complement} \Vert_1\right)\diff t\\
	&=
	d\Vert \by_I\Vert_1^{-|I|-1}
	\sum_{\balpha \in \Gamma_k}\varphi_\balpha \left( 
	\prod_{i \in I} \left( \frac{y_i}{\Vert \by_I\Vert_1}\right)^{\alpha_i-1}
	\right)\frac{B(\Vert\balpha_I\Vert_1+1, \Vert \balpha_{I^\complement} \Vert_1)}{B(\balpha_I, \Vert \balpha_{I^\complement} \Vert_1)}\\
	& \hspace{11em}\times
	\left(
	1- \mathcal{I}_{\Vert \by_I\Vert_1/\Vert \by \Vert_1}
	(\Vert \balpha_I \Vert_1+1, \Vert \balpha_{I^c}\Vert_1)
	\right)\\
	&=d\Vert \by_I\Vert_1^{-|I|-1}
	\sum_{\balpha \in \Gamma_k}\varphi_\balpha\frac{\Vert \balpha_I \Vert_1}{k} \text{Dir}\left(\frac{\by_{I^-}}{\Vert \by_I\Vert_1} ; \balpha_I\right)
	\left(
	1- \mathcal{I}_{\Vert \by_I\Vert_1/\Vert \by \Vert_1}
	(\Vert \balpha_I \Vert_1+1, \Vert \balpha_{I^c}\Vert_1)
	\right),
	\end{split}
	\end{equation*}
	where $I^-$ is defined as in \eqref{eq:I_min}.
	Once more, the result follows after some simple algebraic tweaking.

	\noindent
	\underline{\textbf{Case 3}}\textbf{: $I=\{1, \ldots, d\}$.} In this case, the result straightforwardly follows from
	\eqref{eq:kernel_density_ang}.
\end{proof}

\begin{lemma}
	Let $k\geq d+1$ and $H_k \in \Hset_k$. Then, conditionally on $H_k$, 
	the probability of a partition $\part$ is
	\begin{equation}\label{eq: prob partition}
	\begin{split}
	\int_{\Yset}g_\bone(\by,\part|H_k)\diff \by& 
	=\int_{\resimp}
	\frac{
		\Gamma(m)
		\prod_{j=1}^m \{-V_{I_j}(\bv, 1-\Vert \bv\Vert_1|H_k)\} }{\left(A_{k}(\bt_\bv)r_\bv\right)^{m}} \diff \bv
	\\
	&= \sum_{\balpha^{(1)}}\ldots \sum_{\balpha^{(m)}}\left(\prod_{j=1}^{m} \varphi_{\balpha^{(j)}} \right) \int_{\resimp}
	\frac{\Gamma(m)\prod_{j=1}^{m} q(I_j, \balpha^{(j)}, \bv)}{\left(A_k(\bt_\bv)r_{\bv}\right)^{m}}  \diff \bv,
	\end{split}
	\end{equation}
	where all the integer vectors $\balpha^{(j)}$ range over $\cup_{i=1}^d \{\bkappa_i\} \cup \Gamma_k$ and
	\begin{equation}\label{eq:functofv}
	\bt_\bv:=
	\left( \frac{1/v_2}{r_\bv},\ldots,\frac{1/v_{d-1}}{r_\bv}, \frac{1/(1-\Vert \bv \Vert_1)}{r_\bv}\right),
	\quad
	r_\bv:=\Vert 1/\bv \Vert_1 +1/(1-\Vert \bv \Vert_1). 
	\end{equation}
\end{lemma}
\begin{proof}
	The change of variables in \eqref{eq: change}, Lemma \ref{lem: expression omega} and few algebraic manipulations yield
	\begin{equation*}
	\begin{split}
	\int_{\Yset}& g_\bone(\by,\part|H_k) \diff \by
	\\
	&=\int_{\Yset} \prod_{j=1}^m 
	\{-V_{I_i}(\by|H_k)\}
	e^{-V(\by|H_k)} \diff \by\\
	&= \int_{\resimp}\int_0^\infty r^{d-1} \prod_{j=1}^m\{-V_{I_i}(r\bv, r(1-\Vert \bv \Vert_1)|H_k)\}e^{-V(\bv, 1-\Vert \bv \Vert_1|H_k)/r} \diff r \diff \bv \\
	&=\int_{\resimp}\int_0^\infty r^{d-1} \prod_{j=1}^m \left[
	r^{-|I_j|-1}\sum_{\balpha^{(j)}}\varphi_{\balpha^{(j)}} q(I_j, \balpha^{(j)}, \bv)
	\right]e^{-V(\bv, 1-\Vert \bv \Vert_1|H_k)/r} \diff r \diff \bv\\
	&=\int_{\resimp}
	\prod_{j=1}^m \left[
	\sum_{\balpha^{(j)}}\varphi_{\balpha^{(j)}} q(I_j, \balpha^{(j)}, \bv)
	\right]\frac{\Gamma(m)}{\left[V(\bv, 1- \Vert \bv \Vert_1|H_k)\right]^m} \diff \bv\\
	&= \sum_{\balpha^{(1)}}\ldots \sum_{\balpha^{(m)}}\left(\prod_{j=1}^{m} \varphi_{\balpha^{(j)}}\right) \int_{\resimp}
	\frac{\Gamma(m)\prod_{j=1}^{m} q(I_j, \balpha^{(j)}, \bv)}{\left(A_k(\bt_\bv)r_\bv\right)^{m}}  \diff \bv.
	\end{split}
	\end{equation*}
	The first equality of the statement follows from the fifth line and equation \eqref{eq:expr_delta}, with $\by$ replaced by $(\bv, 1-\Vert \bv \Vert_1)$.
	The second equality is given in the sixth line.
\end{proof}

\begin{lemma}\label{lem: bound for sums}
	Let $k\geq d+1$ and $\bphi,\widetilde{\bphi}\in \Phi_k$, with $\Phi_k$ as in Condition \ref{cond:angularprior} of the main paper. Let $H_k$ and $\widetilde{H}_k$ be the angular pm's corrisponding to $\bphi$ and $\widetilde{\bphi}$, respectively, then denote by $A_k$ and $\widetilde{A}_k$ the related Pickands dependence functions. Let $\part\in \allpart_d \setminus \{1, \ldots,d\}$ and, for $l \in \{1, \ldots, m\}$,
	denote	
	$$
	\sum_{-l}:=\sum_{\balpha^{(1)}}\ldots\sum_{\balpha^{(l-1)}}\sum_{\balpha^{(l+1)}}\ldots\sum_{\balpha^{(m)}}.
	$$
	Then, for any $\balpha^{(l)} \in \cup_{i=1}^d \{\bkappa_i\} \cup \Gamma_k$ it holds that
	$$
	\sum_{-l}\left( \prod_{s<l}\varphi_{\balpha^{(s)}} \right)\left( \prod_{b>l}\widetilde{\varphi}_{\balpha^{(b)}} \right)\int_{\resimp}
	\frac{\Gamma(m)\prod_{j=1}^{m} q(I_j, \balpha^{(j)}, \bv)}{\left(\widetilde{A}_{k}(\bt_\bv) r_\bv\right)^{m}}  \diff \bv \leq d^m.
	$$
\end{lemma}
\begin{proof}
	By the second line of \eqref{eq: prob partition}, we have that
	$
	\sum_{\balpha^{(1)}} \varphi_{\balpha^{(1)}} s_{\balpha^{(1)}}\leq 1
	$, where
	$$
	s_{\balpha^{(1)}}:=
	\sum_{-1}\left(\prod_{j=2}^{m} \varphi_{\balpha^{(j)}}\right) \int_{\resimp}
	\frac{\Gamma(m)\prod_{j=1}^{m} q(I_j, \balpha^{(j)}, \bv)}{\left[A_k(\bt_\bv)r_\bv\right]^{m}}  \diff \bv, \quad \balpha^{(1)}\in \cup_{i=1}^d \{\bkappa_i\} \cup \Gamma_k.
	$$	
	Consequently, we have that $s_{\balpha^{(1)}} \leq d$, for all $\balpha^{(1)}\in \cup_{i=1}^d \{\bkappa_i\} \cup \Gamma_k$. We easily prove this by contradiction. For an arbitrary choice of $\balpha^{(1)}_* $ in $\cup_{i=1}^d \{\bkappa_i\} \cup \Gamma_k$, we might have chosen $	\varphi_{\balpha^{(1)}_* }$ equal to the maximum value allowed by the mean constraint, i.e. 
	$$
	\varphi_{\balpha^{(1)}_* }=\frac{k}{d}\max_{1\leq j\leq d}\alpha_{*,j}^{(1)}.
	$$
	If $s_{\balpha^{(1)}_* }> d> d \max_{1\leq j\leq d}\alpha_{*,j}^{(1)}/k$, we would have
	$$
	1 \geq \sum_{\balpha^{(1)} \neq \balpha_{*}^{(1)}} \varphi_{\balpha^(1)} s_{\balpha^{(1)}}+\varphi_{\balpha^{(1)}_* }s_{\balpha^{(1)}_*} > 
	\sum_{\balpha^{(1)} \neq \balpha_{*}^{(1)}} \varphi_{\balpha^(1)} s_{\balpha^{(1)}}+1,
	$$
	yielding a contradiction. Next, observe that each $s_{\balpha^{(1)}}$ can be written in the form 
	$$
	s_{\balpha^{(1)}}=\sum_{\balpha^{(2)}}\varphi_{\balpha^{(2)}} t_{\balpha^{(2)}}
	$$ 
	and, once more, a contradiction argument allows to prove $t_{\balpha^{(2)}}\leq d^2$, for all $\balpha^{(2)}\in \cup_{i=1}^d \{\bkappa_i\} \cup \Gamma_k$. Proceeding in this way, we can finally prove that
	$$
	\int_{\resimp}
	\frac{\Gamma(m)\prod_{j=1}^{m} q(I_j, \balpha^{(j)}, \bv)}{\left(A_k(\bt_\bv)r_\bv\right)^{m}}  \diff \bv \leq d^m, \quad \text{for all }\balpha^{(1)}, \ldots,\balpha^{(d)}. 
	$$
	As a consequence, for all $l \in \{1, \ldots,d\}$ and $\balpha^{(l)}\in \cup_{i=1}^d \{\bkappa_i\} \cup \Gamma_k$,
	\begin{equation*}
	\begin{split}
	&\sum_{-l}\left( \prod_{s<l}\varphi_{\balpha^{(s)}} \right)\left( \prod_{b>l}\widetilde{\varphi}_{\balpha^{(b)}} \right)\int_{\resimp}
	\frac{\Gamma(m)\prod_{j=1}^{m} q(I_j, \balpha^{(j)}, \bv)}{\left(\widetilde{A}_{k}(\bt_\bv)r_\bv\right)^{m}}  \diff \bv \\
	& \quad \leq  d^m \sum_{-l}\left( \prod_{s<l}\varphi_{\balpha^{(s)}} \right)\left( \prod_{b>l}\widetilde{\varphi}_{\balpha^{(b)}} \right)\\
	&\quad \leq d^m,
	\end{split}
	\end{equation*}
	where the third line follows from the constraint (R1). The proof is now complete.
\end{proof}

\begin{lemma}\label{lem: L1}
	Let $k \geq d+1$ and $H_k,\widetilde{H}_k \in \Hset_k$, with coefficients $\bphi,\widetilde{\bphi}\in \Phi_k$. Then, there exists a positive constant $c$ (depending on $d$) such that
	\begin{equation}\label{eq: ineq for Hellinger}
	\dist_H^2(g_\bone(\cdot|H_k),g_\bone(\cdot|\widetilde{H}_k)) \leq  
	c \Vert {\bphi}_{\circ}- \widetilde{\bphi}_{\circ}\Vert_{1},
	\end{equation}
	where ${\bphi}_{\circ}$ and $\widetilde{\bphi}_{\circ}$ are defined as in  \eqref{eq:interior_coef}.
\end{lemma}
\begin{proof}
	By Proposition \ref{lem: basic metric}\ref{res: basic 1}, the left hand-side of \eqref{eq: ineq for Hellinger} is bounded from above by
	\begin{equation*}
	\begin{split}
	&\Vert g_\bone(\cdot,\cdot|H_{k})-g_\bone(\cdot, \cdot|\widetilde{H}_{k})\Vert_1\\
	&\quad= \sum_{\part \in \allpart_d} \int_{\Yset}
	\left|
	e^{-V(\by|H_k)}\prod_{j=1}^{m} \{-V_{I_j}(\by|H_k)\}
	-
	e^{-V(\by|\widetilde{H}_k)}\prod_{j=1}^{m} \{-V_{I_j}(\by|\widetilde{H}_k)\}
	\right|\diff \by.
	\end{split}
	\end{equation*}
	Thus, to establish \eqref{eq: ineq for Hellinger}, it is sufficient to show that for each $\part \in \allpart_d$
	\begin{equation}\label{eq: bound for each tau}
	\int_{\Yset}
	\left|
	e^{-V(\by|H_k)}\prod_{j=1}^{m} \{-V_{I_j}(\by|H_k)\}
	-
	e^{-V(\by|\widetilde{H}_k)}\prod_{j=1}^{m} \{-V_{I_j}(\by|\widetilde{H}_k)\}
	\right|\diff \by
	\leq 	c_d \Vert {\bphi}_{\circ}- \widetilde{\bphi}_{\circ}\Vert_{1},
	\end{equation}
	where $c
	_d$ is a positive constant. The term on the left hand-side can be bounded from above by
	\begin{equation*}
	\begin{split}
	T_1+T_2 & := \int_{\Yset}
	\left|
	e^{-V(\by|H_k)}
	-
	e^{-V(\by|\widetilde{H}_{k})}\right| \prod_{j=1}^{m} \{-V_{I_j}(\by|H_k)\}
	\diff \by\\
	&\quad +\int_{\Yset}
	e^{-V(\by|\widetilde{H}_{k})}\left|\prod_{j=1}^{m} \{-V_{I_j}(\by|H_k)\}
	-
	\prod_{j=1}^{m} \{-V_{I_j}(\by|\widetilde{H}_k)\}
	\right|\diff \by.
	\end{split}
	\end{equation*}	
	
	We now derive upper bounds for $T_1$ and $T_2$, starting from $T_1$. Assume first that $\part \neq \{1, \ldots, d\}$. 
	Consider the change of variables
	\begin{equation}\label{eq: change}
	\by \mapsto (r,\bv):= \left(\Vert \by\Vert_1, y_1/\Vert \by \Vert_1, \ldots,y_{d-1}/\Vert \by\Vert_1\right).
	\end{equation}	
	This together with a Lipschitz continuity argument, Proposition \ref{lem: basic metric}\ref{res: basic 2}, Lemma \ref{lem: expression omega} and the first line of \eqref{eq: prob partition} 
	lead to the following upper-bounds
	\begin{equation*}
	\begin{split}
	&\dist_\infty (A_{k},\widetilde{A}_{k})\int_{\resimp}  \int_0^\infty
	r^{d-1}
	\frac{r_\bv}{r} 
	e^{-r_{\bv}/dr}  
	\prod_{j=1}^{m}\{- V_{I_j} (r\bv, r(1-\Vert \bv \Vert_1)|H_k) \} 
	\diff r \diff \bv\\
	& \leq 2d \Vert {\bphi}_{\circ}- \widetilde{\bphi}_{\circ}\Vert_{1} \int_{\resimp}  \int_0^\infty
	r^{-m-2}r_\bv
	e^{-r_{\bv}/dr} 
	%	\left(
	\prod_{j=1}^{m} \{-V_{I_j}(\bv, 1- \Vert \bv \Vert_1|H_k)\}
	%	\right)
	\diff r \diff \bv\\
	& \leq 2d  \Vert {\bphi}_{\circ}- \widetilde{\bphi}_{\circ}\Vert_{1} md^{m+1}
	\int_{\resimp}
	\frac{\Gamma(m)
		{
			\prod_{j=1}^{m} \{-V_{I_j}(\bv, 1- \Vert \bv \Vert_1|H_k)\} }
	}{\left(A_{k}(\bt_\bv)r_\bv\right)^{m}} \diff \bv\\
	&= 2d \Vert {\bphi}_{\circ}- \widetilde{\bphi}_{\circ}\Vert_{1} md^{m+1}\int_{\Yset}g_\bone(\by,\part|H_{k})\diff \by,
	\end{split}
	\end{equation*}
	%}
	whence we conclude $T_1  \leq 2  d^{d+3} \Vert {\bphi}_{\circ}- \widetilde{\bphi}_{\circ}\Vert_{1}$. 
	Furthermore, we have that
	\begin{equation*}
	\begin{split}
	T_2 & \leq
	\int_{\Yset}
	e^{-V(\by|\widetilde{H}_{k})}
	\{-V_{I_1}(\by|H_k)\}
	\left|\prod_{j=2}^{m} \{-V_{I_j}(\by|H_k)\}
	-
	\prod_{j=2}^{m} \{-V_{I_j}(\by|\widetilde{H}_{k})\}
	\right|\diff \by\\
	&\quad +\int_{\Yset} e^{-V(\by|\widetilde{H}_{k})}\left|
	V_{I_1}(\by|\widetilde{H}_{k})-V_{I_1}(\by|H_k)
	\right|
	\prod_{j=2}^{m}\{-V_{I_j} (\by|\widetilde{H}_{k})\}\diff \by
	\end{split}
	\end{equation*}
	and recursively upper-bounding the first term on the right hand-side we finally obtain the bound from above
	\begin{equation*}
	\begin{split}
	\sum_{l=1}^{m}T_2^{(l)} := 
	\sum_{l=1}^{m}&  \int_{\Yset}
	e^{-V(\by|\widetilde{H}_{k})}\left|
	V_{I_l}(\by|H_k)
	-
	V_{I_l}(\by|\widetilde{H}_{k})
	\right|\\
	&\qquad
	\times\left( 
	\prod_{j=1}^{l-1}\{-V_{I_j}(\by|H_k)\}
	\right)\hspace{-.3em}
	\left( 
	\prod_{j=l+1}^{m}\{-V_{I_j}(\by|\widetilde{H}_{k})\}
	\right) \hspace{-.3em}\diff \by.
	\end{split}
	\end{equation*}
	For each $l \in \{1, \ldots, m\}$, the change of variables in \eqref{eq: change} and the homogeneity of the functions $V_{I}(\bx|H_k)$, $V_{I}(\bx|\widetilde{H}_k)$, $I \subset\{1,\ldots,d\}$, allow to re-express $T_2^{(l)}$ as
	\begin{equation*}
	\begin{split}
	&\int_{\resimp}\int_0^\infty r^{d-1} e^{-V(\bv, 1-\Vert \bv \Vert_1|\widetilde{H}_{k})/r}
	\left|
	V_{I_l}(r\bv, r(1-\Vert \bv \Vert_1)|H_k)
	-
	V_{I_l}(r\bv, r(1-\Vert \bv \Vert_1)|\widetilde{H}_k)
	\right|\\
	& \qquad \times \left( 
	\prod_{j=1}^{l-1} \{-V_{I_j}(r\bv, r(1-\Vert \bv \Vert_1)|H_k)\}
	\right)\hspace{-.3em}\left( 
	\prod_{j=l+1}^{m} 	\{-V_{I_j}(r\bv, r(1-\Vert \bv \Vert_1)|\widetilde{H}_{k})\}
	\right)\hspace{-.3em}\diff r\diff \bv \\
	& = \int_{\resimp}\int_0^\infty r^{-m-1}  e^{-V(\bv, 1-\Vert \bv \Vert_1|\widetilde{H}_{k})/r}
	\left|
	V_{I_l}(\bv,1-\Vert\bv \Vert_1|H_k)
	-
	V_{I_l}(\bv,1-\Vert\bv \Vert_1|\widetilde{H}_{k})
	\right|\\
	& \qquad \quad \times\left( 
	\prod_{j=1}^{l-1} 	\{-V_{I_j}(\bv,1-\Vert\bv \Vert_1|H_k)\}
	\right)\hspace{-.3em}\left( 
	\prod_{j=l+1}^{ m } 	V_{I_j}(\bv,1-\Vert\bv \Vert_1|\widetilde{H}_{k})
	\right) \hspace{-.3em} \diff r \diff \bv \\
	&=\int_\resimp 	\left|
	V_{I_l}(\bv,1-\Vert\bv \Vert_1|H_k)
	-
	V_{I_l}(\bv,1-\Vert\bv \Vert_1|\widetilde{H}_{k})
	\right|\frac{\Gamma(m)}{\left(\widetilde{A}_{k}(\bt_\bv)r_\bv\right)^{m}}\\
	& \quad  \qquad \times \left( 
	\prod_{j=1}^{l-1} 	\{-V_{I_j}(\bv,1-\Vert\bv \Vert_1|H_k)\}
	\right)\hspace{-.3em}\left( 
	\prod_{j=l+1}^{ m } 	V_{I_j}(\bv,1-\Vert\bv \Vert_1|\widetilde{H}_{k})
	\right) \hspace{-.3em}  \diff \bv
	\end{split}
	\end{equation*}
	where, by Lemmas \ref{lem: expression omega}-\ref{lem: bound for sums}, the term on the right-hand side is bounded from above by
	\begin{equation*}
	\begin{split}
	&  	\sum_{\balpha^{(l)}} |\varphi_{\balpha^{(l)}}-\widetilde{\varphi}_{\balpha^{(l)}}| \sum_{-l}\left( \prod_{s<l}\varphi_{\balpha^{(s)}} \right)\left( \prod_{b>l}\widetilde{\varphi}_{\balpha^{(b)}} \right)
	%	\\
	\int_{\resimp}
	\frac{\Gamma(m)
		\prod_{j=1}^{m} q(I_j, \balpha^{(j)}, \bv)
	}{\left(\widetilde{A}_{k}(\bt_\bv)r_\bv\right)^{m}}  \diff \bv \\
	&\leq d^m\Vert {\bphi}_{\circ}- \widetilde{\bphi}_{\circ}\Vert_{1}.
	\end{split}
	\end{equation*}
	Therefore, $T_2 \leq d^{d+1}\Vert {\bphi}_{\circ}- \widetilde{\bphi}_{\circ}\Vert_{1}$. 
	The inequality \eqref{eq: bound for each tau} now follows for every $\part \neq \{1,\ldots,d\}$.
	Using similar arguments, the same upper bounds can be obtained for $T_1$ and $T_2$ in the simpler case of $\part=\{1,\ldots,d\}$. Hence, inequality \eqref{eq: bound for each tau} holds true also in this instance, completing the proof.
\end{proof}

The next two results are specific to the proof of Theorem \ref{th:alpha_frec}.
We recall that for any $\by=(y_1, \ldots, y_d)>\bzero$ and $\bc=(c_1, \ldots, c_d)\in \reald$, we denote by 
$$
\by^\bc=(y_1^{c_1}, \ldots,y_d^{c_d}).
$$
We also recall that, for $(\brho, \bsigma) \in (0,\infty)^{2d}$ and $H \in \Hset$, the density $g_{\brho, \bsigma}(\cdot|H)$ is given by
$$
g_{\brho, \bsigma}(\by|H)=g_\bone(\{\by/\bsigma\}^\brho|H)\prod_{j=1}^d \rho_j y_j^{\rho_j-1} \sigma_j^{-\rho_j}, \quad \by \in (\bzero,\binf).
$$
Moreover, we make use of the following notations.
For $l \in \nat_+$, we define
\begin{equation}\label{eq:gammas}
\gamma_{(l,+)}:=\int_0^\infty x^l e^{-x}e^{-e^{-x}}\diff x, \quad
\gamma_{(l,-)}:=\int_0^\infty x^l e^{x}e^{-e^{x}}\diff x.
\end{equation}
Notice that if $X$ is a rv with standard Gumbel distribution and we denote $X^+=\max(X,0)$, $X^-=\max(-X,0)$, then $\gamma_{(l,+)}=\mathbb{E} [X^+]^l$ and $\gamma_{(l,-)}=\mathbb{E} [X^-]^l$. Of particular interest for later derivations are the cases $l=1,\ldots,4$, where
\begin{equation}\label{eq:gammas_ex}
\begin{split}
\gamma_{(1,+)}&=\gamma -\text{Ei}(-1), \quad \hspace{9.3em}\gamma_{(1,-)}=\text{Ei}(-1),\\
\gamma_{(2,+)}&=2_3F_3(1,1,1;2,2,2;-1), \hspace{5.5em} \gamma_{(2,-)}=2G^{3,0}_{2,3}\left(1\bigg{|} \begin{smallmatrix}1,1\\ 0,0,0\end{smallmatrix} \right),\\
\gamma_{(3,+)}&=6_4F_4(1,1,1,1;2,2,2,2;-1),  \hspace{3.7em}\gamma_{(3,-)}=6G^{4,0}_{3,4}\left(1\bigg{|} \begin{smallmatrix}1,1,1\\ 0,0,0,0\end{smallmatrix} \right),\\
\gamma_{(4,+)}&=24_5F_5(1,1,1,1,1;2,2,2,2,2;-1), \hspace{1.4em}
\gamma_{(4,-)}=24G^{5,0}_{4,5}\left(1\bigg{|} \begin{smallmatrix}1,1,1,1\\ 0,0,0,0,0\end{smallmatrix} \right)
\end{split}
\end{equation}
with $\gamma$ denoting the Euler-Mascheroni constant, $\text{Ei}(z)$, $z \in \mathbb{R}\setminus\{0\}$, denoting the exponential integral function, $G^{p,q}_{r,s}(z|\begin{smallmatrix}t_1,\ldots,t_p \\ v_1,\ldots,v_q\end{smallmatrix})$, $z\in \mathbb{R}\setminus\{0\}$, denoting the Meijer $G$-function and $_pF_q(a_1,\ldots, a_p; b_1, \ldots,b_q; z)$, $z>0$, denoting the generalized hypergeometric function. See the proofs of Lemmas \ref{lem: KLalpha} and \ref{lem:pseudokulb} for details. For the sake of the lighter notation, all the intervals of the form $\times_{j=1}^d(x_j,y_j)$, for $\bx, \by \in [-\infty,\infty]^d$, $\bx < \by$, are denoted by $(\bx,\by)$.

\begin{lemma}\label{aux1}
	For all $k\geq d+1$, $l \in \nat_+$, $H \in \Hset$, $H_k \in \Hset_k$, $\brho, \widetilde{\brho} \in (\bzero, \binf)$ and $\bsigma, \widetilde{\bsigma} \in  (\bzero, \binf)$ it holds that
	\begin{equation}\label{eq:term_to_be_bounded}
	\left(
	\int_{(\bzero,\binf)}
	\left[\max_{\part \in \allpart_d}\max_{I_i \in \part} \log^+
	\left\lbrace 
	\frac{-V_{I_i}(\by|H_k)}{-V_{I_i}(\by^{	\widetilde{\brho}/\brho} \{\bsigma/\widetilde{\bsigma}\}^{\widetilde{\brho}}|H_k)}
	\right\rbrace
	\right]^l g_\bone(\by|H)\diff \by
	\right)^{1/l}
	\end{equation}
	is bounded from above by 
	$$
	3k \left\lbrace
	\left(\gamma_{(+,l)}+\gamma_{(l,-)}\right)\sum_{j=1}^d  \frac{|\rho_j-\widetilde{\rho}_j|}{\rho_j}	
	+
	\sum_{s=1}^d \frac{|1-\{\sigma_s/\widetilde{\sigma}_s\}^{\widetilde{\rho}_s}|}{\min\left(1, \{\sigma_s/\widetilde{\sigma}_s\}^{\widetilde{\rho}_s}  \right)}
	\right\rbrace,
	$$
	where $\gamma_{(+,l)},\gamma_{(l,-)}$ are given in \eqref{eq:gammas}.
	\begin{proof}
		By Minkowski's inequality, the term in \eqref{eq:term_to_be_bounded} is bounded from above by $T_1+T_2$, where
		\begin{equation*}
		\begin{split}
		T_1&:=\left(
		\int_{(\bzero,\binf)}
		\left[\max_{\part \in \allpart_d}\max_{I_i \in \part} \log^+
		\left\lbrace 
		\frac{-V_{I_i}(\by|H_k)}{-V_{I_i}(\by^{	\widetilde{\brho}/\brho} |H_k)}
		\right\rbrace
		\right]^l g_\bone(\by|H)\diff \by
		\right)^{1/l},\\
		T_2&=\left(
		\int_{(\bzero,\binf)}
		\left[\max_{\part \in \allpart_d}\max_{I_i \in \part} \log^+
		\left\lbrace 
		\frac{-V_{I_i}(\by^{	\widetilde{\brho}/\brho}|H_k)}{-V_{I_i}(\by^{	\widetilde{\brho}/\brho} \{\bsigma/\widetilde{\bsigma}\}^{\widetilde{\brho}}|H_k)}
		\right\rbrace
		\right]^l g_\bone(\by|H)\diff \by
		\right)^{1/l}.
		\end{split}
		\end{equation*}
		The reminder of the proof consists of two parts: the derivation of upper bounds for the terms in curly brackets in the expression of $T_1,T_2$ and the conclusion.
		\\
		
		\noindent
		\textit{Derivation of upper bounds.}
		To derive upper bounds for the terms of the form
		$$
		\log^+
		\left\lbrace 
		\frac{-V_{I}(\by|H_k)}{-V_{I}(\by^{	\widetilde{\brho}/\brho} |H_k)}
		\right\rbrace,
		\quad \log^+
		\left\lbrace 
		\frac{-V_{I}(\by^{	\widetilde{\brho}/\brho}|H_k)}{-V_{I}(\by^{	\widetilde{\brho}/\brho} \{\bsigma/\widetilde{\bsigma}\}^{\widetilde{\brho}}|H_k)}
		\right\rbrace,
		$$
		we treat different types of $I\subset\{1,\ldots,d\}$ separately.
		
		\vspace{0.5em}
		\noindent
		\underline{\textbf{Case 1:}} \textbf{$I=\{j\}$, for some $j \in \{1, \ldots,d\}$}. By Lemma \ref{lem: expression omega}, we have that
		\begin{equation*}
		\begin{split}
		&\log^+
		\left\lbrace 
		\frac{-V_{I}(\by|H_k)}{-V_{I}(\by^{	\widetilde{\brho}/\brho} |H_k)}
		\right\rbrace	\\
		& \quad \leq  2\left|\frac{\widetilde{\rho}_{I}}{\rho_{I}}-1\right||\log y_I|
		+
		\log^+\left\lbrace
		\frac{\varphi_{\bkappa_j}
			+\sum_{\balpha \in \Gamma_k}\varphi_\balpha
			\frac{\alpha_{I}}{k}
			\left(
			1- \mathcal{I}_{y_{I}/\Vert \by \Vert_1}
			(\alpha_{I}+1, \Vert \balpha_{I^\complement}\Vert_1)
			\right)}{\varphi_{\bkappa_j}
			+\sum_{\balpha \in \Gamma_k}\varphi_\balpha
			\frac{\alpha_{I}}{k}
			\left(
			1- \mathcal{I}_{\widetilde{y}_{I}/\Vert \widetilde{\by} \Vert_1}
			(\alpha_{I}+1, \Vert \balpha_{I^\complement}\Vert_1)
			\right)}
		\right\rbrace,
		\end{split}
		\end{equation*}
		where $\widetilde{y}_{I}=y_{I}^{\widetilde{\rho}_{I}/\rho_{I}}$, $\widetilde{\by}=\by^{\widetilde{\brho}/\brho}$, $y_I=\by_I=y_j$, $\rho_I=\brho_I=\rho_j$, $\widetilde{\rho}_I=\widetilde{\brho}_I=\widetilde{\rho}_j$ and $\alpha_I=\balpha_I=\alpha_j$. 
		The second term on the right-hand side is bounded from above by
		\begin{equation*}
		\begin{split}
		&	 \log^+ \max\left\lbrace 1, 
		\frac{^{}
			1- \mathcal{I}_{y_{I}/\Vert \by \Vert_1}
			(\alpha_{I}+1, \Vert \balpha_{I^\complement}\Vert_1)
		}{
			1- \mathcal{I}_{\widetilde{y}_{I}/\Vert \widetilde{\by} \Vert_1}
			(\alpha_{I}+1, \Vert \balpha_{I^\complement}\Vert_1)
		}, \balpha \in \Gamma_k
		\right\rbrace\\
		& = \max \left\lbrace \log^+ 
		\frac{
			1- \mathcal{I}_{y_{I}/\Vert \by \Vert_1}
			(\alpha_{I}+1, \Vert \balpha_{I^\complement}\Vert_1)
		}{
			1- \mathcal{I}_{\widetilde{y}_{I}/\Vert \widetilde{\by} \Vert_1}
			(\alpha_{I}+1, \Vert \balpha_{I^\complement}\Vert_1)
		},  \balpha \in \Gamma_k 
		\right\rbrace\\
		&= \max \left\lbrace
		\log^+ 
		\frac{
			\mathcal{I}_{1-y_{I}/\Vert \by \Vert_1}
			( \Vert \balpha_{I^\complement}\Vert_1, \alpha_{I}+1)
		}{
			\mathcal{I}_{1-\widetilde{y}_{I}/\Vert \widetilde{\by} \Vert_1}
			(\Vert \balpha_{I^\complement}\Vert_1, \alpha_{I}+1)
		},  \balpha \in \Gamma_k 
		\right\rbrace
		\\
		&= \max \left\lbrace
		0, \ell_{\balpha,\by}(\bone)-\ell_{\balpha,\by}(\widetilde{\brho}/\brho)
		,  \balpha \in \Gamma_k 
		\right\rbrace,
		\end{split}
		\end{equation*}
		where, for all $\balpha \in \Gamma_k $ and $\boldsymbol{\beta}\in (\bzero, \binf)$,
		$$
		\ell_{\balpha,\by}(\boldsymbol{\beta}):=\log\left( 
		\mathcal{I}_{1-x}
		( \Vert \balpha_{I^\complement}\Vert_1, \alpha_{I}+1)\right) \Big{|}_{x=y_{I}^{\beta_{I}}/\Vert \by^{\boldsymbol{\beta}} \Vert_1}.
		$$
		By the multivariate mean-value theorem there exists $c\in (0,1)$ such that, setting $\boldsymbol{\beta}=c\bone +(1-c)\widetilde{\brho}/{\brho}$,  
		\begin{equation*}
		\begin{split}
		&\ell_{\balpha,\by}(\bone)-
		\ell_{\balpha,\by}(\widetilde{\brho}/\brho)\\
		&=\frac{\left(1-y_{I}^{\beta_{I}}/\Vert \by^{\boldsymbol{\beta}}\Vert_1\right)^{\Vert \balpha_{I^\complement}\Vert_1} \left(y_{I}^{\beta_{I}}/\Vert \by^{\boldsymbol{\beta}}\Vert_1\right)^{\alpha_{I}+1}}{B\left(1-y_{I}^{\beta_{I}}/\Vert \by^{\boldsymbol{\beta}}\Vert_1;\Vert \balpha_{I^\complement}\Vert_1, \alpha_{I}+1\right)}
		(-\log y_{I}) \left(1-\frac{\widetilde{\rho}_{I}}{\rho_{I}}\right) 
		\\
		&\quad +\sum_{s \in I^\complement}
		\frac{\left(1-y_{I}^{\beta_{I}}/\Vert \by^{\boldsymbol{\beta}}\Vert_1\right)^{\Vert \balpha_{I^\complement}\Vert_1-1} \left(y_{I}^{\beta_{I}}/\Vert \by^{\boldsymbol{\beta}}\Vert_1\right)^{\alpha_{I}+1}}{B\left(1-y_{I}^{\beta_{I}}/\Vert \by^{\boldsymbol{\beta}}\Vert_1;\Vert \balpha_{I^\complement}\Vert_1, \alpha_{I}+1\right)}
		\frac{y_s^{\beta_s}}{\Vert \by^{\boldsymbol{\beta}}\Vert_1}(\log y_{s}) \left(1-\frac{\widetilde{\rho}_{s}}{\rho_{s}}\right) 	
		\\
		&\leq  \sum_{s=1}^d\frac{\alpha_{I}B\left(\Vert \balpha_{I^\complement}\Vert_1, \alpha_{I}\right)}{B\left(\Vert \balpha_{I^\complement}\Vert_1, \alpha_{I}+1\right)}|\log y_{s}|\left|
		1-\frac{\widetilde{\rho}_{s}}{\rho_{s}}
		\right|
		\\
		&  \leq k  \sum_{s=1}^d \frac{|\rho_s-\widetilde{\rho}_s|}{\rho_s} |\log y_s|,
		\end{split}
		\end{equation*}
		where 
%		$B(q_1, q_2)$, $q_1,q_2>0$, is the Beta function and 
		$B(x;q_1, q_2)=B(q_1, q_2)\mathcal{I}_x(q_1,q_2)$, $x \in (0,1)$. Therefore, we conclude that
		\begin{equation}\label{eq:bound1}
		\log^+
		\left\lbrace 
		\frac{-V_{I}(\by|H_k)}{-V_{I}(\by^{	\widetilde{\brho}/\brho} |H_k)}
		\right\rbrace 
		\leq (2+k) \sum_{s=1}^d \frac{|\rho_s-\widetilde{\rho}_s|}{\rho_s}|\log y_s|.
		\end{equation}

		By a similar reasoning and exploiting the inequality $\log^+(x)\leq |x-1|$, $x>0$, we also obtain
		\begin{equation}\label{eq:firstbound}
		\begin{split}
		&\log^+
		\left\lbrace 
		\frac{-V_{I}(\by^{	\widetilde{\brho}/\brho}|H_k)}{-V_{I}(\by^{	\widetilde{\brho}/\brho} \{\bsigma/\widetilde{\bsigma}\}^{\widetilde{\brho}}|H_k)}
		\right\rbrace\\
		&
		\quad \leq  2\log^+\left\lbrace (\sigma_I/\widetilde{\sigma}_I)^{\widetilde{\rho}_I}\right\rbrace
		+
		\log^+\left\lbrace
		\frac{\varphi_{\bkappa_j}
			+\sum_{\balpha \in \Gamma_k}\varphi_\balpha
			\frac{\alpha_{I}}{k}
			\left(
			1- \mathcal{I}_{\widetilde{y}_{I}/\Vert \widetilde{\by} \Vert_1}
			(\alpha_{I}+1, \Vert \balpha_{I^\complement}\Vert_1)
			\right)}{\varphi_{\bkappa_j}
			+\sum_{\balpha \in \Gamma_k}\varphi_\balpha
			\frac{\alpha_{I}}{k}
			\left(
			1- \mathcal{I}_{\widehat{y}_{I}/\Vert \widehat{\by} \Vert_1}
			(\alpha_{I}+1, \Vert \balpha_{I^\complement}\Vert_1)
			\right)}, 
		\right\rbrace\\
		& \quad \leq (2+ k) \sum_{s=1}^d \frac{|1-\{\sigma_s/\widetilde{\sigma}_s\}^{\widetilde{\rho}_s}|}{\min\left(1, \{\sigma_s/\widetilde{\sigma}_s\}^{\widetilde{\rho}_s}  \right)}
		\end{split}
		\end{equation}
		where $\widehat{y}_I=\widetilde{y}_I  \{\sigma_I/\widetilde{\sigma}_I\}^{\widetilde{\rho}_I}$,  $\widehat{\by}=\widetilde{\by}  \{\bsigma/\widetilde{\bsigma}\}^{\widetilde{\brho}}$.

		\vspace{0.5em}
		\noindent
		\underline{\textbf{Case 2:}} \textbf{$1<|I|\leq d-1$.}
		By arguments similar to those used in the previous case,  we have that
		\begin{equation*}
		\begin{split}
		\log^+
		\left\lbrace 
		\frac{-V_{I}(\by|H_k)}{-V_{I}(\by^{	\widetilde{\brho}/\brho} |H_k)}
		\right\rbrace \leq S_1+S_2,
		\end{split}
		\end{equation*}
		where
		\begin{equation*}
		\begin{split}
		S_1&:= (|I|+1)\log^+ \frac{\Vert 
			\widetilde{\by}_{I}
			\Vert_1}{\Vert \by_{I}\Vert_1} \leq d\sum_{s=1}^d |\log y_s| \frac{|\widetilde{\rho_s}-\rho_s|}{\rho_s}
		\end{split}
		\end{equation*}
		and
		\begin{equation*}
		\begin{split}
		S_2&:=	\log^+\left\lbrace
		\frac{\sum_{\balpha \in \Gamma_k}\varphi_\balpha
			\frac{\Vert \balpha_{I}\Vert_1}{k}
			\text{Dir}\left( \by_{I^-}/\Vert \by_{I_i}\Vert_1 ; \balpha_{I}\right)
			\left(
			1- \mathcal{I}_{\Vert \by_{I}\Vert_1/\Vert \by \Vert_1}
			(\Vert \balpha_{I} \Vert_1+1, \Vert \balpha_{I^\complement}\Vert_1)
			\right)}{
			\sum_{\balpha \in \Gamma_k}\varphi_\balpha
			\frac{\Vert \balpha_{I}\Vert_1}{k}
			\text{Dir}\left( \widetilde{\by}_{I^-}/\Vert \widetilde{\by}_{I}\Vert_1 ; \balpha_{I}\right)
			\left(
			1- \mathcal{I}_{\Vert \widetilde{\by}_{I}\Vert_1/\Vert \widetilde{\by} \Vert_1}
			(\Vert \balpha_{I} \Vert_1+1, \Vert \balpha_{I^\complement}\Vert_1)
			\right)}
		\right\rbrace\\
		&\leq \max_{\balpha \in \Gamma_k}\left\lbrace
		\log^+ \frac{\text{Dir}\left( \by_{I^-}/\Vert \by_{I}\Vert_1 ; \balpha_{I}\right)}
		{	\text{Dir}\left( \widetilde{\by}_{I^-}/\Vert \widetilde{\by}_{I}\Vert_1 ; \balpha_{I}\right)}
		\right\rbrace + 
		\max_{\balpha\in \Gamma_k}
		\left\lbrace 
		\log^+
		\frac{1- \mathcal{I}_{\Vert \by_{I}\Vert_1/\Vert \by \Vert_1}
			(\Vert \balpha_{I} \Vert_1+1, \Vert \balpha_{I^\complement}\Vert_1)}{1- \mathcal{I}_{\Vert \widetilde{\by}_{I}\Vert_1/\Vert \widetilde{\by} \Vert_1}
			(\Vert \balpha_{I} \Vert_1+1, \Vert \balpha_{I^\complement}\Vert_1)}
		\right\rbrace\\
		& \quad  \leq 2k  \sum_{s=1}^d \frac{|\rho_s-\widetilde{\rho}_s|}{\rho_s} |\log y_s|,
		\end{split}
		\end{equation*}
		with $I^-$ is as in \eqref{eq:I_min}. Thus, we conclude that
		\begin{equation}\label{eq: res_multi}
		\begin{split}
		\log^+
		\left\lbrace 
		\frac{-V_{I}(\by|H_k)}{-V_{I}(\by^{	\widetilde{\brho}/\brho} |H_k)}
		\right\rbrace \leq 3k \sum_{s=1}^d \frac{|\rho_s-\widetilde{\rho}_s|}{\rho_s} |\log y_s|.
		\end{split}
		\end{equation}
		Moreover, we have that
		\begin{equation*}
		\begin{split}
		\log^+
		\left\lbrace 
		\frac{-V_{I}(\by^{	\widetilde{\brho}/\brho}|H_k)}{-V_{I}(\by^{	\widetilde{\brho}/\brho} \{\bsigma/\widetilde{\bsigma}\}^{\widetilde{\brho}}|H_k)}
		\right\rbrace\leq S_3+S_4,
		\end{split}
		\end{equation*}
		where
		\begin{equation*}
		\begin{split}
		S_3:=(|I|+1)\log^+ \frac{\Vert 
			\widehat{\by}_{I}
			\Vert_1}{\Vert \widetilde{\by}_{I}\Vert_1}\leq d \log^+
		\left\lbrace
		\max_{1\leq s \leq d}(\sigma_s/\widetilde{\sigma})^{\widetilde{\rho}_s}
		\right\rbrace\leq d\sum_{s=1}^d\left|1- 
		(\sigma_s/\widetilde{\sigma}_s)^{\widetilde{\rho}_s}
		\right|
		\end{split}
		\end{equation*}
		and, using the inequality $\log^+(1/x)\leq |1-x|/x$, $x>0$,
		\begin{equation*}
		\begin{split}
		S_4&:=	\log^+\left\lbrace
		\frac{\sum_{\balpha \in \Gamma_k}\varphi_\balpha
			\frac{\Vert \balpha_{I}\Vert_1}{k}
			\text{Dir}\left( \widetilde{\by}_{I^-}/\Vert \widetilde{\by}_{I_i}\Vert_1 ; \balpha_{I}\right)
			\left(
			1- \mathcal{I}_{\Vert \widetilde{\by}_{I}\Vert_1/\Vert \widetilde{\by} \Vert_1}
			(\Vert \balpha_{I} \Vert_1+1, \Vert \balpha_{I^\complement}\Vert_1)
			\right)}{
			\sum_{\balpha \in \Gamma_k}\varphi_\balpha
			\frac{\Vert \balpha_{I}\Vert_1}{k}
			\text{Dir}\left( \widehat{\by}_{I^-}/\Vert \widehat{\by}_{I}\Vert_1 ; \balpha_{I}\right)
			\left(
			1- \mathcal{I}_{\Vert \widehat{\by}_{I}\Vert_1/\Vert \widehat{\by} \Vert_1}
			(\Vert \balpha_{I} \Vert_1+1, \Vert \balpha_{I^\complement}\Vert_1)
			\right)}
		\right\rbrace\\
		&\leq \max_{\balpha \in \Gamma_k}\left\lbrace
		\log^+ \frac{\text{Dir}\left( \widetilde{\by}_{I^-}/\Vert \widetilde{\by}_{I}\Vert_1 ; \balpha_{I}\right)}
		{	\text{Dir}\left( \widehat{\by}_{I^-}/\Vert \widehat{\by}_{I}\Vert_1 ; \balpha_{I}\right)}
		\right\rbrace + 
		\max_{\balpha\in \Gamma_k}
		\left\lbrace 
		\log^+
		\frac{1- \mathcal{I}_{\Vert \widetilde{\by}_{I}\Vert_1/\Vert \widetilde{\by} \Vert_1}
			(\Vert \balpha_{I} \Vert_1+1, \Vert \balpha_{I^\complement}\Vert_1)}{1- \mathcal{I}_{\Vert \widehat{\by}_{I}\Vert_1/\Vert \widehat{\by} \Vert_1}
			(\Vert \balpha_{I} \Vert_1+1, \Vert \balpha_{I^\complement}\Vert_1)}
		\right\rbrace\\
		&  \leq  
		\max_{\balpha \in \Gamma_k}\sum_{s=1}^d(\alpha_s-1)\log^+\left\lbrace (\sigma_s/\widetilde{\sigma}_s)^{-\widetilde{\rho}_s}\right\rbrace
		+ k  \sum_{s=1}^d\frac{\left|1- 
			(\sigma_s/\widetilde{\sigma})^{\widetilde{\rho}_s}
			\right|}{\min\left( 1, 
			(\sigma_s/\widetilde{\sigma})^{\widetilde{\rho}_s}
			\right)}\\
		&\leq 2k \sum_{s=1}^d\frac{\left|1- 
			(\sigma_s/\widetilde{\sigma}_s)^{\widetilde{\rho}_s}
			\right|}{\min\left( 1, 
			(\sigma_s/\widetilde{\sigma}_s)^{\widetilde{\rho}_s}
			\right)},
		\end{split}
		\end{equation*}
		consequently
		\begin{equation}\label{eq:secondbound}
		\log^+
		\left\lbrace 
		\frac{-V_{I}(\by^{	\widetilde{\brho}/\brho}|H_k)}{-V_{I}(\by^{	\widetilde{\brho}/\brho} \{\bsigma/\widetilde{\bsigma}\}^{\widetilde{\brho}}|H_k)}
		\right\rbrace \leq 3k \sum_{s=1}^d\frac{\left|1- 
			(\sigma_s/\widetilde{\sigma})^{\widetilde{\rho}_s}
			\right|}{\min\left( 1, 
			(\sigma_s/\widetilde{\sigma})^{\widetilde{\rho}_s}
			\right)}.
		\end{equation}

		\vspace{0.5em}
		\noindent
		\underline{\textbf{Case 3:}} \textbf{$I=\{1,\ldots,d\}$.} Once again, Lemma \ref{lem: expression omega} and a few algebraic manipulations similar to those described above yield
		\begin{equation}\label{eq:bound_final}
		\begin{split}
		&\log^+
		\left\lbrace 
		\frac{-V_{I}(\by|H_k)}{-V_{I}(\by^{	\widetilde{\brho}/\brho} |H_k)}
		\right\rbrace\\
		&\quad =\log^+
		\left\lbrace 
		\frac{\Vert \by \Vert_1^{-d-1} \sum_{\balpha \in \Gamma_k}\varphi_\balpha
			\text{Dir}\left(
			\by_{1:{d-1}}/\Vert \by \Vert_1; \balpha		
			\right)}{\Vert \widetilde{\by} \Vert_1^{-d-1} \sum_{\balpha \in \Gamma_k}\varphi_\balpha
			\text{Dir}\left(
			\widetilde{\by}_{1:{d-1}}/\Vert  \widetilde{\by} \Vert_1; \balpha		
			\right)}
		\right\rbrace\\
		&\quad \leq (d+1)\log^+\frac{\Vert \widetilde{\by} \Vert_1}{\Vert \by \Vert_1}  + 
		\max\left[
		0,
		\max_{\balpha \in \Gamma_k}\log \frac{\text{Dir}\left(
			\by_{1:{d-1}}/\Vert \by \Vert_1; \balpha		
			\right)}{\text{Dir}\left(
			\widetilde{\by}_{1:{d-1}}/\Vert  \widetilde{\by} \Vert_1; \balpha		
			\right)}
		\right]\\
		& \quad  \leq (d+1)\sum_{s=1}^d \frac{|\rho_s-\widetilde{\rho}_s|}{\rho_s} |\log y_s|
		+k\sum_{s=1}^d \frac{|\rho_s-\widetilde{\rho}_s|}{\rho_s} |\log y_s|\\
		& \quad \leq  2k\sum_{s=1}^d \frac{|\rho_s-\widetilde{\rho}_s|}{\rho_s} |\log y_s|
		\end{split}
		\end{equation}
		and
		\begin{equation}\label{eq:thirdbound}
		\begin{split}
		&\log^+
		\left\lbrace 
		\frac{-V_{I}(\by^{	\widetilde{\brho}/\brho}|H_k)}{-V_{I}(\by^{	\widetilde{\brho}/\brho} \{\bsigma/\widetilde{\bsigma}\}^{\widetilde{\brho}}|H_k)}
		\right\rbrace \\
		&\quad =\log^+
		\left\lbrace 
		\frac{\Vert \widetilde{\by} \Vert_1^{-d-1} \sum_{\balpha \in \Gamma_k}\varphi_\balpha
			\text{Dir}\left(
			\widetilde{\by}_{1:{d-1}}/\Vert \widetilde{\by} \Vert_1; \balpha		
			\right)}{\Vert \widehat{\by} \Vert_1^{-d-1} \sum_{\balpha \in \Gamma_k}\varphi_\balpha
			\text{Dir}\left(
			\widehat{\by}_{1:{d-1}}/\Vert  \widehat{\by} \Vert_1; \balpha		
			\right)}
		\right\rbrace\\
		& \quad \leq (d+1)\log^+ \frac{\Vert 
			\widehat{\by}_{I}
			\Vert_1}{\Vert \widetilde{\by}_{I}\Vert_1}+
		\max\left[
		0,
		\max_{\balpha \in \Gamma_k}\log \frac{\text{Dir}\left(
			\widetilde{\by}_{1:{d-1}}/\Vert \widetilde{\by} \Vert_1; \balpha		
			\right)}{\text{Dir}\left(
			\widehat{\by}_{1:{d-1}}/\Vert  \widehat{\by} \Vert_1; \balpha		
			\right)}
		\right]\\
		& \quad \leq (d+1) \log^+
		\left\lbrace
		\max_{1\leq s \leq d}(\sigma_s/\widetilde{\sigma}_s)^{\widetilde{\rho}_s}
		\right\rbrace + \max_{\balpha \in \Gamma_k}\sum_{s=1}^d(\alpha_s-1)\log^+\left\lbrace (\sigma_s/\widetilde{\sigma}_s)^{-\widetilde{\rho}_s}\right\rbrace\\
		& \quad  \leq 2k \sum_{s=1}^d\frac{\left|1- 
			(\sigma_s/\widetilde{\sigma}_s)^{\widetilde{\rho}_s}
			\right|}{\min\left( 1, 
			(\sigma_s/\widetilde{\sigma}_s)^{\widetilde{\rho}_s}
			\right)}.
		\end{split}
		\end{equation}
		
		\vspace{0.5em}
		\noindent
		\textit{Conclusion.}
		Combining the bounds in \eqref{eq:bound1}, \eqref{eq: res_multi} and \eqref{eq:bound_final} and using Minkowski's inequality together with the fact that $g_\bone(\cdot|H)$ has unit Fr\'{e}chet margins, we obtain
		\begin{equation}
		\begin{split}
		T _1 &\leq
		\left(
		\int_{(\bzero,\binf)}
		\left[
		3k \sum_{s=1}^d \frac{|\rho_s-\widetilde{\rho}_s|}{\rho_s} |\log y_s|
		\right]^l g_\bone(\by|H)\diff \by
		\right)^{1/l}\\
		& \leq 3k\sum_{j=s}^d  \frac{|\rho_s-\widetilde{\rho}_s|}{\rho_s}	\left(
		\int_{(\bzero,\binf)}
		|\log y_s|^l g_\bone(\by|H)\diff \by
		\right)^{1/l}\\
		&=3k \left(\gamma_{(+,l)}+\gamma_{(l,-)}\right)\sum_{s=1}^d  \frac{|\rho_s-\widetilde{\rho}_s|}{\rho_s}.	
		\end{split}
		\end{equation}
		Moreover, combining the bounds in \eqref{eq:firstbound}, \eqref{eq:secondbound} and \eqref{eq:thirdbound} and using Minkowski's inequality we conclude
		\begin{equation*}
		\begin{split}
		T_2 & \leq 	\left(
		\int_{(\bzero,\binf)}
		\left[
		3k \sum_{s=1}^d \frac{|1-\{\sigma_s/\widetilde{\sigma}_s\}^{\widetilde{\rho}_s}|}{\min\left(1, \{\sigma_s/\widetilde{\sigma}_s\}^{\widetilde{\rho}_s}  \right)}
		\right]^l g_\bone(\by|H)\diff \by
		\right)^{1/l}\\
		& \leq 3k \sum_{s=1}^d \frac{|1-\{\sigma_s/\widetilde{\sigma}_s\}^{\widetilde{\rho}_s}|}{\min\left(1, \{\sigma_s/\widetilde{\sigma}_s\}^{\widetilde{\rho}_s}  \right)}.
		\end{split}
		\end{equation*}
		The result now follows.
	\end{proof}
	
\end{lemma}

\begin{lemma}\label{aux: n2}
	For all $k>d$, $H_k, \widetilde{H}_k \in \Hset_k$, $\brho_0\in (\bzero, \binf)$, $\brho,\widetilde{\brho}  \in B_{\delta_0,1}(\brho_0)$, $\bsigma_0 \in (\bzero, \binf)$,
	$\bsigma,\widetilde{\bsigma} \in B_{\delta_0,1}(\bsigma_0)$, with  
	\begin{equation}\label{eq:deltastar}
	0<\delta_0 <\frac{1}{5}\min\left(\min_{1\leq j \leq d} \rho_{0,j}, \min_{1\leq j \leq d} \sigma_{0,j}
	\right)=:\delta_*,
	\end{equation}
	it holds that
	\begin{equation*}
	\begin{split}
	\dist_H (g_{\brho,\bsigma}(\cdot|H_k), g_{\widetilde{\brho},\widetilde{\bsigma}}(\cdot|\widetilde{H}_k)) &\leq  \sqrt{ c_0 \Vert \bphi_\circ - \widetilde{\bphi}_\circ\Vert_1} + \sqrt{c_0k \Vert \brho-\widetilde{\brho} \Vert_1 + c_0k \Vert \bsigma-\widetilde{\bsigma} \Vert_1},
	\end{split}
	\end{equation*}
	where $c_0$ is a positive constant depending on $d$, $\brho_0$ and $\bsigma_0$, while $\Vert \bphi_\circ - \widetilde{\bphi}_\circ\Vert_1$ is as in Lemma \ref{lem: L1}.
\end{lemma}
\begin{proof}
	Preliminary observe that $B_{\delta_0,1}(\brho_0)\subset B_{\delta_{0},\infty}(\brho_0)$ and
	$B_{\delta_0,1}(\bsigma_0)\subset B_{\delta_{0},\infty}(\bsigma_0)$. Therefore,
	$\widetilde{\brho} \in B_{2\delta_{0},\infty}(\brho)$, as well as
	$\bsigma \in B_{2\delta_{0},\infty}(\widetilde{\bsigma})$. Moreover,
	\begin{equation}\label{eq:deltabound1}
	2 \delta_0<\min_{1\leq j \leq d}\frac{\rho_{0,j}-\delta_0}{2}
	<\frac{1}{2}\min_{1\leq j \leq d}\rho_j
	\end{equation}
	and 
	\begin{equation}\label{eq:deltabound2}
	2 \delta_0<\min_{1\leq j \leq d}\frac{\sigma_{0,j}-\delta_0}{2}
	<\frac{1}{2}\min_{1\leq j \leq d}\sigma_j.
	\end{equation}
	Therefore, $2\delta_0$ satisfies the condition in \eqref{eq:deltas}, with $\delta_1$ and $\delta_2$ replaced by $2\delta_0$ and $\varepsilon=1/2$.

	Next, observe that, 
	%	by Lemma B.1(ii) in {\color{magenta}Ghosal and van der Vaart (2017)}
	%	\citeN[Lemma B.1(ii)]{ghosal2017},
	%	$$
	%		\dist_H^2 (g_{\brho,\bsigma}(\cdot|H_k), g_{\widetilde{\brho},\widetilde{\bsigma}}(\cdot|\widetilde{H}_k)) \leq \Vert
	%		 g_{\brho,\bsigma}(\cdot|H_k)- g_{\widetilde{\brho},\widetilde{\bsigma}}(\cdot|\widetilde{H}_k)
	%	\Vert_1
	%	$$ 
	%	and
	\begin{equation*}
	\begin{split}
	%	\Vert
	%	g_{\brho,\bsigma}(\cdot|H_k)- g_{\widetilde{\brho},\widetilde{\bsigma}}(\cdot|\widetilde{H}_k)
	%	\Vert_1 
	\dist_H (g_{\brho,\bsigma}(\cdot|H_k), g_{\widetilde{\brho},\widetilde{\bsigma}}(\cdot|\widetilde{H}_k))
	& =
	%	= \Vert
	%	g_\bone(\cdot|H_k)-
	%	g_{\widehat{\brho}, 
	%	\widehat{\bsigma}
	%	}(\cdot|\widetilde{H}_k) \Vert_1
	\dist_H (
	g_\bone(\cdot|H_k),
	g_{\widehat{\brho}, \widehat{\bsigma}}(\cdot|\widetilde{H}_k)
	)
	%	g_{
	%		\widetilde{\brho}/\brho, \{\widetilde{\bsigma}/\bsigma\}^{1/\brho}
	%	}(\cdot|\widetilde{H}_k)\Vert_1
	\\
	&
	\leq 
	%	\Vert 
	\dist_H(g_\bone(\cdot|H_k), g_\bone(\cdot|\widetilde{H}_k))
	%	\Vert_1
	+
	%	\Vert 
	\dist_H(	g_\bone(\cdot|\widetilde{H}_k),
	g_{\widehat{\brho},
		\widehat{\bsigma}
	}(\cdot|\widetilde{H}_k) ),
	%	\Vert_1,
	\end{split}
	\end{equation*}
	where $\widehat{\brho}=\widetilde{\brho}/\brho$
	and $\widehat{\bsigma}=\{
	\widetilde{\bsigma}/\bsigma
	\}^{1/\brho}$. 
	By  
	%	Proposition A.2(i), for the case $d=2$, and 
	Lemma \ref{lem: L1}, 
	%	for the case $d>2$, 
	we have that the first term on the right-hand side is bounded from above by the square root of $c\Vert {\bphi}_{\circ}- \widetilde{\bphi}_{\circ}\Vert_{1}$, for some $c>0$ depending only on $d$. Moreover, by Lemma B.1(iv) in \cite{r10},
%	{\color{magenta}Ghosal and van der Vaart (2017)}, 
	the second term on the right hand side is bounded from above by the square root of 
	\begin{equation*}
	\begin{split}
	\kulb(
	g_\bone(\cdot|\widetilde{H}_k),
	g_{\widehat{\brho}, 
		\widehat{\bsigma}
	}(\cdot|\widetilde{H}_k) 
	)& = \sum_{j=1}^d \log \frac{\widehat{\sigma}_j^{\widehat{\rho}_j}}{\widehat{\rho}_j}  +\sum_{j=1}^d \left(1-\widehat{\rho}_j\right)\int_{(\bzero, \binf)}\log y_j g_\bone(\by|\widetilde{H}_k) \diff \by\\
	& \quad +  \int_{(\bzero, \binf)}[V(\{\by/\widehat{\bsigma}\}^{\widehat{\brho}}|H_k)- \widetilde{V}_k(\by)]g_\bone(\by|\widetilde{H}_k) \diff \by\\
	&\quad +  
	\int_{(\bzero, \binf)} \log
	\frac{
		\sum_{\part \in \allpart_d}
		\prod_{i=1}^m\{-V_{I_i}(\by|\widetilde{H}_k)\}}
	{\sum_{\part \in \allpart_d}
		\prod_{i=1}^m\{-V_{I_i}(\{\by/\widehat{\bsigma}\}^{\widehat{\brho}}|\widetilde{H}_k)\}}
	g_\bone(\by|\widetilde{H}_k)
	\diff \by\\
	&=: T_1+T_2+T_3+T_4.
	\end{split}
	\end{equation*}
	A few algebraic manipulations lead to show that the first two terms satisfy
	$$
	T_1 \leq \frac{5}{4\min_{1\leq j \leq d}\rho_{0,j}} \Vert \brho-\widetilde{\brho}\Vert_1
	+ \frac{75}{8\min_{1\leq j \leq d}\rho_{0,j}\sigma_{0,j}} \Vert \bsigma-\widetilde{\bsigma}\Vert_1
	$$ 
	and 
	$$
	T_2 \leq \frac{5\gamma }{4\min_{1\leq j \leq d} \rho_{0,j}}\Vert \brho-\widetilde{\brho}\Vert_1.
	$$ 
	By Lemma \ref{lem: aux} and the bounds in \eqref{eq:deltabound1}-\eqref{eq:deltabound2}, we also have that
	\begin{equation*}
	\begin{split}
	T_2 & \leq 	\frac{ \upsilon_1 + \Gamma(2)
	}{\min_{1\leq j \leq d}\rho_j}
	\Vert \brho -\widetilde{\brho} \Vert_1+
	\Gamma(3)
	\max_{1\leq j \leq d} \left\lbrace \frac{3\rho_j}{\sigma_j}
	\left(\frac{3}{2}\right)^{\frac{3}{2}\rho_j}
	\right\rbrace
	\Vert \bsigma -\widetilde{\bsigma}\Vert_1\\
	& \leq \frac{5(\upsilon_1+1)}{4\min_{1\leq j \leq d}\rho_{0,j}} \Vert \brho -\widetilde{\brho} \Vert_1
	+\Gamma(3)
	\max_{1\leq j \leq d} \left\lbrace \frac{9\rho_{0,j}}{2\sigma_{0,j}}
	\left(\frac{3}{2}\right)^{\frac{9}{5}\rho_{0,j}}
	\right\rbrace
	\Vert \bsigma -\widetilde{\bsigma}\Vert_1,
	\end{split}
	\end{equation*}
	with $\upsilon_1$ as in \eqref{eq:upsilons}. Finally, a Jensen's inequality argument like that in \eqref{eq:jensen}, Lemma \ref{aux1}, the bounds in \eqref{eq:deltabound1}-\eqref{eq:deltabound2} and a few algebraic steps analogous to those in \eqref{eq:similar} yield 
	\begin{equation*}
	\begin{split}
	T_4 &\leq 3k \left\lbrace
	\left(\gamma_{(1,+)}+\gamma_{(1,-)}\right)\sum_{j=1}^d  \frac{|\rho_j-\widetilde{\rho}_j|}{\rho_j}	
	+
	\sum_{s=1}^d \frac{|1-\{\sigma_s/\widetilde{\sigma}_s\}^{\widetilde{\rho}_s}|}{\min\left(1, \{\sigma_s/\widetilde{\sigma}_s\}^{\widetilde{\rho}_s}  \right)}
	\right\rbrace \\
	& \leq 15k \frac{\gamma_{(+,1)}+\gamma_{(1,-)}}{4 \min_{1\leq j \leq d}\rho_{0,j}} \Vert \brho -\widetilde{\brho}\Vert_1+
	\frac{27}{2}k
	\frac{\max_{1\leq j \leq d}\rho_{0,j}}{\min_{1\leq j \leq d}\sigma_{0,j}} 
	\max_{1\leq j \leq d}(2/3)^{3\rho_{0,j}}
	\Vert \bsigma -\widetilde{\bsigma}\Vert_1,
	\end{split}
	\end{equation*}
	where $\gamma_{(1, \bullet)}$ is as in \eqref{eq:gammas_ex}. The result now follows.
\end{proof}

The two final results are specific to the proof of Theorem \ref{theo: cons_weibull}. They parallel the results established in Lemmas \ref{aux1}-\ref{aux: n2} for $\brho$-Fr\'echet max-stable densities. We recall that, for $H \in \Hset$, $(\bomega, \bsigma,\bmu)\in (\bzero,\binf)\times (\bzero,\binf) \times \reald$ and $\bx \in (-\binf,\bmu)$, the $\bomega$-Weibull max-stable density is given by
$$
g_{\bomega,\bsigma,\bmu}(\bx|H)=\prod_{j=1}^d \frac{\omega_j}{\sigma_j}\left(\frac{\mu_j-x_j}{\sigma_j}\right)^{-\omega_j-1}g_{\bone}\left(\left(\frac{\mu_1-x_1}{\sigma_1}\right)^{-\omega_1}, \ldots, \left(\frac{\mu_d-x_d}{\sigma_d}\right)^{-\omega_d}\bigg{|}H\right).
$$
% We recall that for any $\by=(y_1, \ldots, y_d)>\bzero$ and $\bc=(c_1, \ldots, c_d)\in \reald$, we denote by 
%$$
%\by^\bc=(y_1^{c_1}, \ldots,y_d^{c_d}).
%$$

\begin{lemma}\label{lem:weibkulb}
	For all $k\geq d+1$, $\widetilde{H} \in \Hset$, $H_k \in \Hset_k$, $\bomega, \widetilde{\bomega} \in (\bone, \binf)$, $\bsigma, \widetilde{\bsigma} \in  (\bzero, \binf)$ and $\bmu > \widetilde{\bmu} \in \reald$ it holds that
	\begin{equation}\label{eq:tobound}
	\int_{(-\binf,\widetilde{\bmu})}
	\left[\max_{\part \in \allpart_d}\max_{I_i \in \part}
	\log^+
	\left\lbrace
	\frac{-V_{I_i}\left(
		\left\lbrace
		\frac{\widetilde{\bmu} -\bx}{\bsigma}
		\right\rbrace^{-\bomega}
		\bigg{|}H_k
		\right) }{-V_{I_i}\left(
		\left\lbrace
		\frac{\bmu -\bx}{\bsigma}
		\right\rbrace^{-\bomega}
		\bigg{|}H_k
		\right) }	\right\rbrace
	\right]
	g_{\widetilde{\bomega},\widetilde{\bsigma},\widetilde{\bmu}}(\bx|\widetilde{H}) \diff \bx
	\end{equation}
	is bounded from above by 
	$$
	4k \Vert \bmu-\widetilde{\bmu} \Vert_1 \max_{1\leq j \leq d}\frac{\omega_j}{\widetilde{\sigma}_j}\Gamma\left(
	1-\frac{1}{\widetilde{\omega}_j}
	\right).
	$$
\end{lemma}
\begin{proof}
	A change of variables allow to re-express the term in \eqref{eq:tobound} as
	$$
	\int_{(-\binf,\bzero)}
	\left[\max_{\part \in \allpart_d}\max_{I_i \in \part}
	\log^+
	\left\lbrace
	\frac{-V_{I_i}\left(
		\left\lbrace
		\frac{ -\bx}{\bsigma}
		\right\rbrace^{-\bomega}
		\bigg{|}H_k
		\right) }{-V_{I_i}\left(
		\left\lbrace
		\frac{\bmu - \widetilde{\bmu}-\bx}{\bsigma}
		\right\rbrace^{-\bomega}
		\bigg{|}H_k
		\right) }	\right\rbrace
	\right]
	g_{\widetilde{\bomega},\widetilde{\bsigma},\bzero}(\bx|\widetilde{H}) \diff \bx.
	$$
	As done in the proof of Lemma \ref{aux1}, by reasoning on a case by case basis via Lemma \ref{lem: expression omega} and the multivariate mean-value theorem we end up with the general inequality
	$$
	\log^+
	\left\lbrace
	\frac{-V_{I_i}\left(
		\left\lbrace
		\frac{ -\bx}{\bsigma}
		\right\rbrace^{-\bomega}
		\bigg{|}H_k
		\right) }{-V_{I_i}\left(
		\left\lbrace
		\frac{\bmu - \widetilde{\bmu}-\bx}{\bsigma}
		\right\rbrace^{-\bomega}
		\bigg{|}H_k
		\right) }	\right\rbrace \leq 4k\sum_{j=1}^d\omega_j\frac{\mu_j-\widetilde{\mu}_j}{-x_j}
	$$
	for all $I_i \in \part$, $\part \in \allpart$. Consequently,  the term in \eqref{eq:tobound} is bounded from above by
	\begin{equation*}
	\begin{split}
	4k \sum_{j=1}^d\omega_j\int_{-\infty}^0\frac{\mu_j-\widetilde{\mu}_j}{-x_j}g_{\widetilde{\omega}_j,\widetilde{\sigma}_j}(x_j)\diff x_j=4k\sum_{j=1}^d\frac{\omega_j}{\widetilde{\sigma}_j}\Gamma\left(
	1-\frac{1}{\widetilde{\omega}_j}
	\right)(\mu_j-\widetilde{\mu}_j),
	\end{split}
	\end{equation*}
	where $g_{\widetilde{\omega}_j,\widetilde{\sigma}_j}(x_j)$ is the probability density function defined in  \eqref{eq:sing_weib}.
	The result now follows.
\end{proof}

\begin{lemma}\label{lem:entropyweib}
	For all $k \geq d+1 $, $H_k, \widetilde{H}_k\in \Hset_k$,  $\bomega_0 \in (\bone, \binf)$, ${\bomega} ,\widetilde{\bomega} \in B_{\delta_1,1}(\bomega_0 )$, 
	$\bsigma \in (\bzero, \binf)$,
	${\bsigma},\widetilde{\bsigma} \in B_{\delta_2,1}(\bsigma_0)$,
	$\bmu_0\in \reald$, ${\bmu},\widetilde{\bmu} \in B_{\delta_3,1}(\bmu_0)$ with 
	\begin{equation}\label{eq:delta0_weib}
	\begin{split}
	\delta_0 &<\frac{1}{5}\min\left(\min_{1\leq j \leq d} \omega_{0,j}, \min_{1\leq j \leq d} \sigma_{0,j}^{-1}
	\right)\\
	\delta_1 & <  (1-p) \min_{1\leq j \leq d}\omega_{0,j},  \quad \left( 
	\min_{1\leq j \leq d}\omega_{0,j}
	\right)^{-1}<p<1,\\
	\delta_2& < \frac{16}{25}\min_{1 \leq j \leq d}\sigma_{0,j}^2{\delta_0}, \hspace{1.9em} \delta_3 < 1/4, 
	\end{split}
	\end{equation} 
	there exists a constant $c_0>0$ depending only on $\bomega_0, \bsigma_0, \bmu_0$ such that
	\begin{equation*}
	\begin{split}
	\dist_H(g_{\bomega,\bsigma,\bmu}(\cdot|H_k),
	g_{\widetilde{\bomega},\widetilde{\bsigma},\widetilde{\bmu}}(\cdot|\widetilde{H}_k))&\leq 
	\sqrt{ c_0 \Vert \bphi_\circ - \widetilde{\bphi}_\circ\Vert_1} + \sqrt{c_0 k \Vert \bomega-\widetilde{\bomega} \Vert_1 + c_0 k \Vert \bsigma-\widetilde{\bsigma} \Vert_1}\\
	&\quad +\sqrt{c_0k \Vert \bmu - \widetilde{\bmu} \Vert_1}.
	\end{split}
	\end{equation*}
\end{lemma}
\begin{proof}
	Preliminarily observe that, by 
	%	Lemma B.1(ii) in {\color{magenta}Ghosal and van der Vaart (2017)} and 
	triangular inequality, 
	\begin{equation*}
	\begin{split}
	\dist_H(g_{\bomega,\bsigma,\bmu}(\cdot|H_k),
	g_{\widetilde{\bomega},\widetilde{\bsigma},\widetilde{\bmu}}(\cdot|\widetilde{H}_k))&\leq
	\dist_H(
	g_{\bomega,\bsigma,
		\widetilde{\bmu}
	}(\cdot|H_k)
	,
	g_{\widetilde{\bomega},\widetilde{\bsigma},\widetilde{\bmu}}(\cdot|\widetilde{H}_k))
	\\
	%		\Vert 
	&\quad +
	\dist_H(
	g_{\bomega,\bsigma,\bmu}(\cdot|H_k)
	,
	g_{\bomega,\bsigma,\bmu\vee
		\widetilde{\bmu}
	}(\cdot|H_k))
	%		\Vert_1\\
	\\
	&\quad +
	%		\Vert
	\dist_H(
	g_{\bomega,\bsigma,\bmu\vee
		\widetilde{\bmu}
	}(\cdot|H_k)
	,
	g_{\bomega,\bsigma,\widetilde{\bmu}}(\cdot|H_k))
	%		\Vert_1\\
	\\
	%			&\quad +
	%		\Vert
	%		\Vert_1\\
	&=: S_1+S_2+S_3.
	\end{split}
	\end{equation*}
	We have that 
	\begin{equation}\label{eq:sigma_ineq}
	\begin{split}
	\Vert \bsigma^{-1} - \widetilde{\bsigma}^{-1} \Vert_1 &\leq \left(\min_{1 \leq j \leq d}\sigma_j \widetilde{\sigma}_j \right)^{-1}\Vert \bsigma - \widetilde{\bsigma} \Vert_1\\
	& \leq  \left(\frac{4}{5}\min_{1 \leq j \leq d} \sigma_{0,j} \right)^{-2}\Vert \bsigma - \widetilde{\bsigma} \Vert_1\\
	& < \delta_0,
	\end{split}
	\end{equation}
	therefore, denoting $\brho_0=\bomega_0$, $\brho=\bomega$, $\widetilde{\brho}=\widetilde{\bomega}$, it holds that $\brho, \widetilde{\brho} \in B_{\delta_0,1}(\brho_0)$ and $\bsigma^{-1}, \widetilde{\bsigma}^{-1} \in B_{\delta_0,1}(\bsigma_0^{-1})$. Moreover, a change of variables yields
	$$
	S_1 =	\dist_H(
	g_{\brho,\bsigma^{-1}
	}(\cdot|H_k)
	,
	g_{\widetilde{\brho},\widetilde{\bsigma}^{-1}}(\cdot|\widetilde{H}_k)).
	$$
	Consequently, by Lemma \ref{aux: n2} and \eqref{eq:sigma_ineq},
	\begin{equation*}
	\begin{split}
	S_1 &  \leq \sqrt{ c_2 \Vert \bphi_\circ - \widetilde{\bphi}_\circ\Vert_1} + \sqrt{c_2 k \Vert \bomega-\widetilde{\bomega} \Vert_1 + c_2 k \Vert \bsigma-\widetilde{\bsigma} \Vert_1},
	\end{split}
	\end{equation*}
	where $c_2$ is a positive constant depending only on $d$, $\bomega_0$ and $\bsigma_0$. Furthermore, by Lemma B.1(iv) in \cite{r10}
%	{\color{magenta}Ghosal and van der Vaart (2017)},
	\begin{equation*}
	\begin{split}
	S_2^2 &\leq \kulb(	g_{\bomega,\bsigma,\bmu}(\cdot|H_k)
	,
	g_{\bomega,\bsigma,\bmu\vee
		\widetilde{\bmu}
	}(\cdot|H_k))\\
	& \leq \int_{(-\binf,\bmu)}
	\left[ V\left(
	\left\lbrace
	\frac{\bmu -\bx}{\bsigma}
	\right\rbrace^{-\bomega}
	\bigg{|}H_k
	\right) 
	-V\left(
	\left\lbrace
	\frac{\bmu\vee
		\widetilde{\bmu} -\bx}{\bsigma}
	\right\rbrace^{-\bomega}
	\bigg{|}H_k
	\right) 
	\right]_+g_{{\bomega}, {\bsigma}, {\bmu}}(\bx|{H}_k)\diff \bx
	\\
	&\quad + \sum_{j=1}^d \int_{(-\binf,\bmu)} \log^+\left\lbrace
	\frac{\mu_j\vee\widetilde{\mu}_j-x_j}{\mu_{j}-x_j} 
	\right\rbrace
	g_{\bomega,\bsigma,\bmu}(\bx|H_k) \diff \bx\\
	&\quad +
	\int_{(-\binf,\bmu)}
	\log^+
	\left\lbrace
	\frac{\sum_{\part \in \allpart_d}
		\prod_{i=1}^{m}\left[-V_{I_i}\left(
		\left\lbrace
		\frac{\bmu \vee \widetilde{\bmu} -\bx}{\bsigma}
		\right\rbrace^{-\bomega}
		\bigg{|}H_k
		\right) \right]}{\sum_{\part \in \allpart_d}
		\prod_{i=1}^{m}\left[-V_{I_i}\left(
		\left\lbrace
		\frac{\bmu -\bx}{\bsigma}
		\right\rbrace^{-\bomega}
		\bigg{|}H_k
		\right) \right]}
	\right\rbrace
	g_{{\bomega}, {\bsigma}, {\bmu}}(\bx|{H}_k)\diff \bx
	\\
	&=:T_1+T_2+T_3.
	\end{split}
	\end{equation*}
	Since $\bmu \vee \widetilde{\bmu} \in B^+_{2\delta_3,\infty}(\bmu)$, by Lemma \ref{lem:vweib} and the bounds in \eqref{eq:sigma_ineq},	
	$$
	T_1 \leq c_3 \Vert \bmu \vee \widetilde{\bmu} - \bmu \Vert_1 \leq  c_3 \Vert  \widetilde{\bmu} - \bmu \Vert_1,
	$$
	where $c_3>0$ only depends on $\bomega_0,\bsigma_0$. Moreover, following steps similar to those in \eqref{eq:T2} and exploiting \eqref{eq:sigma_ineq}, we obtain
	\begin{equation}
	\begin{split}
	T_2 &\leq \sum_{j=1}^d \frac{\mu_j\vee \mu_j-\mu_{j}}{\sigma_{j}}\Gamma(1-1/\omega_{j})\\
	& \leq 
	\frac{5}{4}\max_{1\leq j \leq d} \frac{\Gamma(1-1/(p \omega_{0,j}))}{\sigma_{0,j}} \Vert \bmu -\widetilde{\bmu}\Vert_1.
	\end{split}
	\end{equation}
	Finally, by Lemma \ref{lem:weibkulb}
	\begin{equation*}
	\begin{split}
	T_3 &\leq 4k \Vert \widetilde{\bmu} \vee \bmu-\bmu \Vert_1 \max_{1\leq j \leq d}\frac{\omega_j}{\sigma_j}\Gamma\left(
	1-\frac{1}{\omega_j}\right)\\
	&\leq 10k \Vert \widetilde{\bmu} -\bmu \Vert_1 \max_{1\leq j \leq d}\frac{\omega_{0,j}}{\sigma_{0,j}}\Gamma\left(
	1-\frac{1}{p\omega_{0,j}}\right).
	\end{split}
	\end{equation*}
	The term $S_3$ can be bounded from above analogously, whence the conclusion.
\end{proof}

%%%%%%%%%%%%%%%%%%%%%%%%%%%%%%%%%%%%%%%%%%

\subsection{Proofs of the results in Sections \ref{app:KL_support}--\ref{sec:posterior_consistency_2D}}
\label{appsec:proof_bivar}

%%%%%%%%%%%%%%%%%%%%%%%%%%%%%%%%%%%%%%%%%%

%%%%%%%%%%%%%%%%%%%%%%%%%%%%%%%%%%%%%%%%%%%%%%%%%%%%%%%%%%%%%%%%%%

\subsubsection{Proof of Proposition \ref{prop:Polish_d2}}

%%%%%%%%%%%%%%%%%%%%%%%%%%%%%%%%%%%%%%%%%%%%%%%%%%%%%%%%%%%%%%%%%%

Let $(A_k)_{k=1}^\infty \subset \Aset$ be a sequence which converges in $\dist_{1, \infty}$-metric to some $A_* \in \mathbb{W}^{1,\infty}((0,1))$. We now show it must be that $A_* \in \Aset$, whence we conclude that $\Aset$ is a closed subset of $\mathbb{W}^{1,\infty}((0,1))$. By the uniform limit theorem, $A_*$ must be continuous. Consequently, it must also be convex \citep[Theorem  E, p. 17]{r344}.
%({\color{magenta}Roberts and Varberg 1973}, Theorem E, page 17).
%%\cite[Theorem  E, p. 17]{convfun}. 
%
Obviously, 
$$
\max(t,1-t)\leq \lim_{k \to \infty}A_k(t) \leq 1, \quad t \in [0,1].
$$
Thus, $A_*$  is a Pickands dependence function. Moreover, for any collection of nonoverlapping intervals $[a_i,b_i]$, $i=1, \ldots, s$, with $s \in \mathbb{N}$, for any $\varepsilon>0$ and $k$ large enough we have 
$$
\sum_{i=1}^s \left|A_*'(b_i)-A_*'(a_i)\right| \leq \sum_{i=1}^s \left|A_k'(b_i)-A_k'(a_i)\right|+\varepsilon/2. 
$$
By absolute continuity of $A_k'$, there exists $\delta>0$ such that, if $\sum_{i=1}^s \left|b_i-a_i\right|<\delta$, then the first term on the right is smaller than $\varepsilon/2$. Hence, $A_*'$ is absolutely continuous too and $H_*(t)=(A_*'(t)+1)/2$, $0<t<1$, is an angular cdf with absolutely continuous restriction on $(0,1)$. We can now deduce that $A_* \in \Aset$.

Since $\Aset$ is a closed subset of $\mathbb{W}^{1,\infty}((0,1))$ and the latter space is complete 
%({\color{magenta}Adams and Fournier 2002}, Theorem 3.3) 
%%\cite[Theorem 3.3]{sobolev}, 
\citep[Theorem 3.3]{r99},
we deduce that $(\Aset, \dist_{1, \infty})$ is a complete subspace. Moreover, we have that the class of polynomials in Bernstein form $\cup_{k \geq 1}\Aset_k$, given in Section \ref{app:review_BP}, is dense in $(\Aset, \dist_{1, \infty})$. Thus, the countable subclass of polynomials with rational coefficients is dense too. We now deduce that $(\Aset, \dist_{1, \infty})$ is also separable and the proof is complete.

%%%%%%%%%%%%%%%%%%%%%%%%%%%%%%%%%%%%%%%%%%%%%%%%%%%%%%%%%%%%%%%%%%%

\subsubsection{Proof of Corollary \ref{cor:measurable}}

%%%%%%%%%%%%%%%%%%%%%%%%%%%%%%%%%%%%%%%%%%%%%%%%%%%%%%%%%%%%%%%%%%%

Due to the relation between the first derivative of a Pickands dependence function, say $A\in \Aset$, and the associated angular pm, say $H \in \Hset$, the map $A \mapsto H$ is 1-to-1 and continuous with respect to the $\dist_{1,\infty}$-metric topology and the topology of weak convergence on $\Aset$ and $\Hset$, respectively, thus measurable with respect to the corresponding Borel $\sigma$-algebras. Noting that $g_\bone(\cdot|A)=g_\bone(\cdot|H)$, the conclusion now follows from Corollary \ref{cor:Polish_d>2}, by map composition.

\subsubsection{Proof of Theorem  \ref{theo:post_consistency}}\label{appsec:proofconsD2}

To prove the first statement, we resort to Theorem 6.23 in 
%{\color{magenta}Ghosal and van der Vaart (2017)} 
\cite{r10} and verify that the conditions therein are satisfied. The second one obtains as a by-product.

\vspace{0.5em}
\noindent
\textit{Kulback-Leibler property.}
For all $A \in \Aset'$ and $\delta>0$ there exists a sequence $A_k \in \Aset_k$, with linear coefficients $\boldsymbol{\beta}^{(k)}\in \mathcal{B}_k$, such that, as $k \to \infty$,
\begin{equation*}
\begin{split}
B_{\delta,k}:=\{\widetilde{A}_k \in \Aset_k: \, 8k^2\Vert \widetilde{\boldsymbol{\beta}}^{(k)} - \boldsymbol{\beta}^{(k)}\Vert_\infty < \delta \} &\subset
\{\widetilde{A}_k \in \Aset_k: \dist_{2,\infty}(\widetilde{A}_k, A_k)< \delta/2\}\\
&	 \subset
\{\widetilde{A} \in \Aset: \dist_{2,\infty}(\widetilde{A}, A)< \delta\}. 
\end{split}
\end{equation*}
This is guaranteed by Theorem 6.3.2 in 
%\citet{r42}
%\bibitem[\protect\citeauthoryear{Davis}{1975}]{r42}
%Davis, P. J. (1975). 
%\textit{Interpolation and approximation}, Applied Mathematical Sciences, 27. Springer-Verlag, New York-Berlin.
%{\color{magenta}Davis (1975)}
\cite{r77}, when the extremal dependence is represented through Bernstein polynomials, and Theorem 2.1 in 
\cite{r355}
%{\color{magenta}de Boor and Fix (1973)}
%\citet{r41} 
%\bibitem[\protect\citeauthoryear{de Boor and Fix}{1973}]{r41}
%de Boor, C.; Fix, G. J. (1973). 
%Spline approximation by quasiinterpolants.
%\textit{Approximation Theory},
%\textbf{8}, 19--45.
%\MR{0340893}
together with Proposition \ref{prop:bspline_cond}, when the extremal dependence is represented through splines. Furthermore, by Conditions \ref{cond:prior_D2}\ref{cond:B10i}--\ref{cond:B10ii}, $\Pi_\Aset(B_{\delta,k})>0$. 
We can now apply Theorem \ref{theo:KL_prior} and conclude that $\Pi_\Aset$ complies with \eqref{eq:KLpicka}, for all $\epsilon>0$.

\vspace{0.5em}
\noindent
\textit{Metric entropy.} Let $A_k$ and $\tilde{A}_k$ be two Pickands dependence functions in BP form, with degree $k$. 
Let $H_{k-1}$ and $\widetilde{H}_{k-1}$ be the corresponding angular distributions in BP form, obtained 
by applying Proposition 3.2(i) in \cite{r24}.
%{\color{magenta}Marcon et al. (2016)}. 
The associated angular densities satisfy
\begin{equation*}
\begin{split}
\Vert h_{k-2}-\widetilde{h}_{k-2} \Vert_1&= \int_0^1 \left|
\sum_{j=0}^{k-2}(\eta_{j+1}-\eta_j-\widetilde{\eta}_{j+1}+\widetilde{\eta}_j)\betaf(v|j+1, k-j-1)
\right|\diff v \\
& \leq \sum_{j=0}^{k-2}|\eta_{j+1}-\eta_j-\widetilde{\eta}_{j+1}+\widetilde{\eta}_j|.
\end{split}
\end{equation*} 
Analogously, let $A_k$ and $\tilde{A}_k$ be two valid Pickands dependence spline functions of order $3$ and
$H_{k-1}$ and $\widetilde{H}_{k-1}$ be the corresponding angular distribution spline functions, obtained by applying
Proposition \ref{prop:bspline_rel_pick_ang}.
By \eqref{eq:bspline_angdens} with $m=3$, their angular densities satisfy
\begin{equation*}
\begin{split}
\Vert h_{k-2}-\widetilde{h}_{k-2} \Vert_1&= \int_0^1 \left|
\sum_{j=1}^{k-2}\left(\frac{\eta_{j+1}-\eta_{j}-\widetilde{\eta}_{j+1}+\widetilde{\eta}_{j}}{\tau_{j+2}-\tau_{j+1}}\right)\indic_{[\tau_{j+1},\tau_{j+2})}(v)
\right|\diff v \\
& \leq \sum_{j=1}^{k-2}|\eta_{j+1}-\eta_j-\widetilde{\eta}_{j+1}+\widetilde{\eta}_j|.
\end{split}
\end{equation*} 
%	
%		Under both representations, Proposition \ref{lem: basic metric}\ref{res: basic 2} and the relation between the first derivative of the Pickands and the corresponding angular measure yield
%		$$
%		\dist_{1, \infty}(A_k, \widetilde{A}_k) \leq 6 \Vert 
%		h_{k-2}-\widetilde{h}_{k-2}
%		\Vert_1.
%		$$
%		By the above inequalities, Proposition \ref{prop: basic metric bvt}\ref{res: basic bvt 1} and \citet[Proposition C.2]{r10}
%		we have that, for any 
%%		$\widetilde{\Aset}_k \subset \Aset_k$ and 
%		set of the form $\mathcal{G}_*^{(k)}=\{g_*(\cdot|A_k): A_k \in \widetilde{\Aset}_k\}$, 
%		%
%		$$
%		N(\epsilon, \mathcal{G}_*^{(k)}, \dist_H) \leq N(c'\epsilon^2, \{\bx \in \real^{k-1}:\Vert \bx \Vert_1 \leq 1\}, L_1) \leq (3/c'\epsilon^2)^{k-1},
%		$$
%		%
%		for some $c'>0$, where, without loss of generality, we assume $c'\epsilon^2<1$. Therefore, Conditions i. and ii. in \citet[][Theorem 6.23]{r10} can be verified by using similar arguments to those in Appendix \ref{appendix: cons mvt}. 
By the above inequalities, Proposition \ref{prop: basic metric bvt}\ref{res: basic bvt 1} and Proposition C.2 in \cite{r10},
%{\color{magenta}Ghosal and van der Vaart (2017)}
we have that, for any
%		 $\widetilde{\Aset}_k \subset \Aset_k$ and 
set of the form $\mathcal{G}_\bone^{(k)}=\{g_\bone(\cdot|A_k): A_k \in {\Aset}_k\}$, 
$$
\mathcal{N}(\epsilon, \mathcal{G}_\bone^{(k)}, \dist_H) \leq \mathcal{N}(c'\epsilon^2, \{\bx \in \real^{k-2}:\Vert \bx \Vert_1 \leq 1\}, L^1) \leq (3/c'\epsilon^2)^{k-2},
$$
for some $c'>0$, where, without loss of generality, we assume $c'\epsilon^2<1$. Therefore, Conditions i. and ii. of Theorem 6.23 in \cite{r10}
%{\color{magenta}Ghosal and van der Vaart (2017)} 
can be verified by arguments similar to those in Section \ref{appendix: cons mvt}.

\vspace{0.5em}
\noindent
\textit{Conclusion.} The first statement is now proven. As for the second statement, it follows from the first one together with Proposition \ref{lem: basic metric}\ref{res: basic 1} and Proposition \ref{prop: basic metric bvt}\ref{res: new pick}.

%%%%%%%%%%%%%%%%%%%%%%%%%%%%%%%%%%%%%%%%%%

\subsection{Proofs of the results in Section \ref{sec:binf_simple_max}}\label{appsec:proof_simp}

%%%%%%%%%%%%%%%%%%%%%%%%%%%%%%%%%%%%%%%%%%

%%%%%%%%%%%%%%%%%%%%%%%%%%%%%%%%%%%%%%%%%%%%%%%%%%%%%%%%%%%%%%%%%%%

\subsubsection{Proof of Proposition \ref{prop:Polish_d>2}}
\label{appsec:proofsKLtheo}

%%%%%%%%%%%%%%%%%%%%%%%%%%%%%%%%%%%%%%%%%%%%%%%%%%%%%%%%%%%%%%%%%%%
Observe that, as a finite union of an open and $d$ closed sets, $\tilde{\simp}=\simpint \cup\{ \cup_{j=1}^d \{\be_j\}\}$ is a $G_\delta$ subset of $\simp$. Thus, $\tilde{\simp}$, endowed with the restriction of the Euclidean metric, is a Polish subspace of the simplex. As a consequence, the class of Borel pm's on the latter space, say $\mathcal{M}$, endowed with the weak topology, is metrizable by some metric $\dist_W$ and Polish 
%({\color{magenta}Ghosal and van der Vaart 2017}, Theorem A.3).
%%\cite[Theorem  A.3]{ghosal2017}. 
%
\citep[Theorem  A.3]{r10}.

To establish the result in the second part of the statement, we use arguments similar to those of Theorem 4.1 in \cite{r214}.
%{\color{magenta}Gaudard and Hadwin (1989)}.
%%\citeN[Theorem 4.1]{gaud89}. 
Let $\mathring{\nu}$ be the Lebesgue measure on $\mathring{\resimp}$ and define the class
\begin{equation*}
\begin{split}
\mathcal{I}:=\{ f \in L_1(\mathring{\nu}): \, &\int_{\mathring{\resimp}} v_jf(\bv) \mathring{\nu}(\diff \bv) \leq 1/d,\,j=1,\ldots,d-1, \\ 
&  \int_{\mathring{\resimp}} (1-\Vert \bv \Vert_1)f(\bv) \mathring{\nu}(\diff \bv) \leq 1/d, 
\, \, f \geq 0 \,\, \text{a.e.} \}
\end{split}
\end{equation*}
and the map $\phi:\mathcal{I}\mapsto \mathcal{M}$ via 
\begin{equation}\label{eq: map_def}
\begin{split}
[\phi(f)](B)=\int_{\pi_{\resimp}(B\cap \simpint)} f(\bv)\mathring{\nu}(\diff \bv) &+ \sum_{j=1}^{d-1}\left[1/d-\int_{\pi_{\resimp}( \simpint)} v_jf(\bv) \mathring{\nu}(\diff \bv)\right]\delta_{\be_j}(B)\\
&+
\left[1/d-\int_{\pi_{\resimp}( \simpint)} (1-\Vert \bv \Vert_1)f(\bv) \mathring{\nu}(\diff \bv)\right]
\delta_{\be_d}(B)
\end{split}
\end{equation}
for any Borel subset $B$ of $\tilde{\simp}$. It can be easily seen that $\mathcal{I}$ is the intersection of the $L_1$-closed sets 
$
\{ f \in L_1(\mathring{\nu}): \,  f \geq 0 \,\, \text{a.e.} \}$ and
\begin{eqnarray*}
	&&\{ f \in L_1(\mathring{\nu}): \, \int_{\mathring{\resimp}} v_jf(\bv) \mathring{\nu}(\diff \bv) \leq 1/d \}, \quad j=1,\ldots,d-1,\\
	&&\{ f \in L_1(\mathring{\nu}): \, \int_{\mathring{\resimp}} (1-\Vert \bv \Vert_1)f(\bv) \mathring{\nu}(\diff \bv) \leq 1/d \},
\end{eqnarray*}
thus $\mathcal{I}$ is closed. 
We deduce that $\mathcal{I}$ is a Borel subset of the Polish space $L_1(\mathring{\nu})$. Consequently, to establish the second result,  it is sufficient to show that $\phi$ is a Borel isomoprhism from $\mathcal{I}$ onto $\Hset$. 

Since $\mathcal{I}$, equiped with the subspace $L_1$-topology, is standard Borel (in fact, Polish) and $(\mathcal{M},\dist_W)$ is Polish, we can resort to 
Theorem 3.3.2 in \cite{r10000}
%{\color{magenta}Arveson (1976)}
%%\citeN[Theorem 3.3.2]{arveson76} 
and prove that $\phi$ is a $1-\text{to}-1$ Borel map from $\mathcal{I}$ to $\mathcal{M}$, satisfying $\phi(\mathcal{I})=\Hset$. From the definition in \eqref{eq: map_def} it immediately follows that, for each $f \in \mathcal{I}$, we have $H(\cdot)=[\phi(f)](\cdot)\in \Hset$, thus $\phi(\mathcal{I})\subset \Hset$. On the other hand, by Definition \ref{cond_angular} in the main article, each $H \in \Hset$ admits the representation in  \eqref{eq: map_def}, for some nonnegative Lebesgue-integrable function $f \in \mathcal{I}$, hence $\phi^{-1}(\Hset) \subset \mathcal{I}$, i.e. $\Hset  \subset \phi(\mathcal{I})$. By the Radon-Nikodym theorem, the function $f$ satisfying such a representation is uniquely defined (up to a $\mathring{\nu}$-null set), thus $\varphi$ is also $1-\text{to}-1$. Finally, for any $f,h \in \mathcal{I}$ and any Borel subset $B \subset \tilde{\simp}$, we have
$$
|[\phi(f)](B)-[\phi(h)](B)| \leq (d+1)\Vert f-h\Vert_1,
$$
i.e. the map $\varphi:\mathcal{I}\mapsto \mathcal{M}$ is Lipschitz continuous and, thus, a Borel map. 

We have now established that $\Hset$ endowed with the subspace topology of weak convergence of pm's, metrized by $\dist_W$, is standard Borel. 
Denote by $\sigBor_\Hset$ the pertaining 
Borel $\sigma$-field of subsets of $\Hset$. 
To prove the last claim in the statment, observe that $\dist_{KS}$ induces a stronger topology on $\Hset$ than 
%the subspace topology of weak convergence
$\dist_W$. Then, the corresponding $\sigma$-field, say $\sigBor_{KS}$,  contains $\sigBor_W$. On the other hand, for any pair of angular pm's $H_1, H_2 \in\Hset $, we also have
$$
\dist_{KS}(H_1,H_2) \leq (d+1)\Vert \phi^{-1}(H_1)-\phi^{-1}(H_2)\Vert_1=: \dist_{\mathcal{I}}(H_1, H_2)
$$
and $\mathcal{I}$, endowed with the $L_1$ distance, is Polish (as a closed subset of a Polish space). 
Consequently, denoting by $\sigBor_{\mathcal{I}}$ the Borel $\sigma$-algebra induced on $\Hset$ by the $\dist_{\mathcal{I}}$-metric topology, 
%in view of 
%the Borel isomorphism established above 
we can deduce
$$
\sigBor_{KS} \subset \sigBor_{\mathcal{I}} \subset \sigBor_\Hset.
$$
It now follows that $\sigBor_{KS} = \sigBor_\Hset$, which completes the proof.

%%%%%%%%%%%%%%%%%%%%%%%%%%%%%%%%%%%%%%%%%%%%%%%%%%%%%%%%%%%%%%%%%

\subsubsection{Proof of Corollary \ref{cor:Polish_d>2}}

%%%%%%%%%%%%%%%%%%%%%%%%%%%%%%%%%%%%%%%%%%%%%%%%%%%%%%%%%%%%%%%%%%
Preliminary notice that the map $H\mapsto g_\bone(\cdot|H)$ 1-to-1. Indeed, each simple max-stable distribution can be seen as the distribution of the componentwise maxima over the points of a Poisson process, whose mean measure is univocally identified by an angular pm on the unit simplex $\simp$; see  Section \ref{appsec:notation} and, e.g., Chapter 1.1 of \cite{r200}.
%{\color{magenta}Reiss 1993}. 
%\citeNP[Ch. 1.1]{reiss1993}). 
Hence, the map linking an angular pm to the corresponding max-stable distribution is injective.

Next, denote by $\mathscr{B}_{\Yset}$  the Borel $\sigma$-algebra of $\Yset=(0, \infty)^d$ and by $\mathscr{B}_{\Hset}$ the Borel $\sigma$-algebra induced on $\Hset$ by the metric $\dist_W$. By Theorem V.58 in 
\cite{r45},
%{\color{magenta}Dellacherie and Meyer (1982)}
%\citeN[Theorem V.58]{dellacherie82} 
there exists a positive $\mathcal{B}_{\Yset} \times \mathscr{B}_\Hset$-measurable function $(\by,H) \mapsto g_\bone(\by|H)$ such that, for all $H$, $g_\bone(\cdot|H)$ is a Lebesgue densiy of $G_\bone(\cdot|H)$. Then, arguments analogous to those in Appendix A of 
\cite{r48}
%{\color{magenta}Petrone and Wasserman (2002)}
%\citeN[Appendix A]{petJRSSB}
yield that, for any fixed $H_* \in \Hset$ and $\epsilon>0$, 
$$
\{H \in\Hset:\, \dist_H(g_\bone(\cdot|H_*),g_\bone(\cdot|H))<\epsilon\} \in \mathscr{B}_\Hset
$$ 
and 
$$
\{H \in\Hset:\, \kulb(g_\bone(\cdot|H_*),g_\bone(\cdot|H))<\epsilon\} \in \mathscr{B}_\Hset. 
$$
Since $(\Gset_\bone, \dist_H)$ is separable, the associated Borel $\sigma$-algebra, $\mathscr{B}_{\Gset_\bone}$, is generated by $\dist_H$-open balls and we can conclude 
%(e.g., {\color{magenta}Srivastava 1998}, p. 86) 
%(e.g. \citeNP[p. 86]{borel}) 
\citep[e.g.,][p. 86]{r121}
that the map $H \mapsto g_\bone(\cdot|H)$ is $\mathscr{B}_\Hset/ \mathscr{B}_{\Gset_\bone}$-measurable. 

%%%%%%%%%%%%%%%%%%%%%%%%%%%%%%%%%%%%%%%%%%%%%%%%%%%%%%%%%%%%%%%%%

\subsubsection{Proof of Theorem \ref{theo:KL_prior_multi}}\label{appsec:proofKLmulti}

%%%%%%%%%%%%%%%%%%%%%%%%%%%%%%%%%%%%%%%%%%%%%%%%%%%%%%%%%%%%%%%%%

It is sufficient to show that, for any $\epsilon>0$, there exists a measurable subset of $\Hset$ with positive $\Pi_\Hset$-mass, whose elements $H$ satisfy
\begin{equation}\label{eq: KL joint}
\kulb(g_\bone(\cdot|H_{0}), g_\bone(\cdot|H)) <\epsilon.
\end{equation}
By Proposition \ref{cor: New_KL_cor}, the  above inequality is satisfied for all $H \in\ B_{\delta,\infty}(H_*)$, for some $H_* \in \Hset'$ and $\delta>0$. Morever, by assumption, there exists a measurable set $B_\delta \subset \ B_{\delta,\infty}(H_*)$ such that $\Pi_\Hset(B_\delta)>0$. The result now follows.

%%%%%%%%%%%%%%%%%%%%%%%%%%%%%%%%%%%%%%%%%%%%%%%%%%%

\subsubsection{Proof of Theorem \ref{theo:post_consistency_mvt}}\label{appendix: cons mvt}

%%%%%%%%%%%%%%%%%%%%%%%%%%%%%%%%%%%%%%%%%%%%%%%%%%%

%
Firstly, we verify the conditions of Theorem 6.23 in 
%{\color{magenta}Ghosal and van der Vaart (2017)}
\cite{r10} 
to establish the result at point (a). Then, the one at point (b) is deduced from the latter. Finally, the claim concering the support of $\Pi_{\Gset_{\bone}}$ is proved.

\vspace{0.5em}
\noindent
\textit{Kulback-Leibler property.} By assumption, $H_0\in \Hset_0$, with $\Hset_0$ given in Definition \ref{cond:mvt_angular}\ref{cond:true_ang_set} of the main article. Also, by Lemma \ref{lem: approx}, for all $H \in\Hset'$ (see Definition \ref{cond:mvt_angular}\ref{cond:prior_ang_set} of the main article) and $\delta>0$, there exists a sequence $H_k \in \Hset_k$, with weights $\bphi^{(k)}\in \Phi_k$, such that, defining
\begin{equation*}
\begin{split}
B_{\delta,k}:=\left\lbrace \widetilde{H}_k \in \Hset_k: 2\prod_{i=1}^{d-1}(k+i)\Vert \widetilde{\bphi}^{(k)}-\bphi^{(k)} \Vert_\infty < \delta \right\rbrace,
\end{split}
\end{equation*}
we have $	B_{\delta,k} \subset B_{\delta/2,\infty}(H_k)\subset B_{\delta/2,\infty}(H_k)$, as $k\to \infty$.
By  Condition \ref{cond:angularprior},
%Theorem \ref{theo:post_consistency_mvt}, 
$\Pi_\Hset(B_{\delta,k})>0$, hence the prior $\Pi_\Hset$ posses the $\dist_{
	\infty}$-property at $H$ (Definition \ref{def:d_inf_prop} of the main article). Therefore,
by Theorem \ref{theo:KL_prior_multi}, we conclude that
%		 $g_*(\cdot|H_0)$ is in the Kulback-Leibler support of the prior. 
%		possesses the Kulback-Leibler property at $H_0$.
$\Pi_{\Gset_\bone}(\mathcal{K}_\epsilon)>0$, for all $\epsilon>0$, where $\mathcal{K}_\epsilon$ is as in Definition \ref{defi:KL} of the main paper with $\theta_0=H_0$. That is, $g_\bone(\cdot|H_0)$ is in the Kullback-Leibler support of the prior $\Pi_{\Gset_\bone}$.

\vspace{0.5em}
\noindent
\textit{Metric entroypy.} Observe that $\dist_H$ is a metric that generates convex balls and define 
\begin{eqnarray*}
	&\mathcal{G}^{(k)}_\bone&:=\{g_\bone(\cdot|H_k): \, 
	H_k \in \Hset_k; \, 
	\dist_H(g_\bone(\cdot|H_k), g_\bone(\cdot|H_0)) >4\epsilon  \},\\
	&\mathcal{G}_{\bone,n,1}&:=\cup_{k=d+1}^{\nu_n} \mathcal{G}_\bone^{(k)}, \\
	&\mathcal{G}_{\bone,n,2}&:=\mathcal{G}_{\bone,n,1}^\complement=\mathcal{G}_\bone\setminus \mathcal{G}_{\bone,n,1}
	%			\cup_{k=\nu_n+1}^{\infty} \mathcal{G}^{(k)},
\end{eqnarray*}
where $\nu_n$ is a sequence of positive integers and $\mathcal{G}_\bone:=\{g_\bone(\cdot|H):\, H \in \Hset\}$. 
Then, by Lemma \ref{lem: L1} and Proposition C.2 in
%{\color{magenta}Ghosal and van der Vaart (2017)}
\cite{r10} we have
$$
\mathcal{N}(2\epsilon, \mathcal{G}_\bone^{(k)}, \dist_H)\leq \mathcal{N}\left(c'\epsilon^2,  \{\bx \in \real^{|\Gamma_k|}:\Vert \bx \Vert_1 \leq 1\}, L_1 \right) \leq \left(3/c'\epsilon^2\right)^{\binom{k-1}{d-1}},
$$
where $c'$ is a positive global constant and, without loss of generality, we assume $c'\epsilon^2<1$. As a consequence, we also have that
\begin{equation}
\begin{split}
\mathcal{N}(2\epsilon, \mathcal{G}_{\bone,n,1}, \dist_H) \leq \sum_{k=d+1}^{\nu_n} \left( 3/c'\epsilon^2\right)^{\binom{k-1}{d-1}} 
\leq \nu_n \left( 3/c'\epsilon^2\right)^{\nu_n^{d-1}}
\end{split}
\end{equation}
and choosing $\nu_n=\floor{ (n\epsilon^2)^{1/(d-1)}c''}$, with $c''=(\log(3/c'\epsilon^{2}))^{-1/(d-1)}2^{-1}$, we have
\begin{equation}
\begin{split}
\log \mathcal{N}(2\epsilon, {\mathcal{G}}_{\bone,n,1}, \dist_H)  \leq \log \nu_n +\nu_n^d \log(3/c'\epsilon^{2})  \leq n \epsilon^2.
\end{split}
\end{equation}
The first condition of Theorem 6.23 in
%{\color{magenta}Ghosal and van der Vaart (2017)}
\cite{r10} 
is therefore satisfied. 
The second condition therein is satisfied by assumption \ref{cond: mvt th2} in Condition \ref{cond:angularprior}. 

\vspace{0.5em}
\noindent
\textit{Conclusion.} The result at point (a) in the statement of Theorem \ref{theo:post_consistency_mvt} follows from the above considerations. For the general case $d\geq 2$, the result at point (b) is a direct consequence of the one at point (a) and Proposition \ref{lem: basic metric}\ref{res: basic 4}.
For the specific case $d=2$, the result follows from an application of Propositions \ref{lem: basic metric}\ref{res: basic 1} and \ref{prop: basic metric bvt}\ref{res: new pick}.

As for the claim on the full support of $\Pi_{\Gset_\bone}$, by Corollary \ref{lem:1to1}  we have that for any $\epsilon>0$ and $g_\bone(\cdot|\widetilde{H}) \in \Gset_\bone$ there exists $g_\bone(\cdot|\widetilde{H}_k)\in \{g_\bone(\cdot|H): \,H \in \cup_{i=d+1}^\infty\Hset_i \}$ such that $\dist_H(g_\bone(\cdot|\widetilde{H}),g_\bone(\cdot|\widetilde{H}_k)) < \epsilon/4$. Thus, by Lemma \ref{lem: L1}, we have that 
\begin{equation*}
\begin{split}
&\Pi_{\mathcal{G}_\bone}(g \in \mathcal{G}_\bone:\dist_H(g,g_\bone(\cdot|\widetilde{H}))<\epsilon)\\
& \quad \geq 		\Pi_{\mathcal{G}_\bone}(g \in \mathcal{G}_\bone:\dist_H(g,g_\bone(\cdot|\widetilde{H}_k))<\epsilon/4)\\
&\quad \geq 		\Pi_{\mathcal{G}_\bone}(g\in \{g_\bone(\cdot|H_k), \, H_k \in \Hset_k\}:\dist_H(g,g_\bone(\cdot|\widetilde{H}_k))<\epsilon/4)\\
&  \quad \geq \Pi(\bphi \in \Phi_k: \Vert \bphi- \widetilde{\bphi} \Vert_1 < \epsilon/4c )
\end{split}
\end{equation*}
where $\widetilde{\bphi}$ is the vector of linear coefficients corresponding to $\widetilde{h}_{k-d}$ and $c$ is a positive constants.	
In passing, observe that the set on the second line is $\mathscr{B}(\mathcal{G}_\bone)$-measurable, since it is the intersection of the measurable set $\{g_\bone(\cdot|H_k), \, H_k \in \Hset_k\}$ and the Hellinger ball of radius $\epsilon/4$ around $g_\bone(\cdot|\widetilde{H}_k)$.  
By Condition \ref{cond:angularprior}, the term on the third line of the above display is positive, hence the conclusion.

\subsection{Proofs of the results in Section \ref{sec:binf_general_max}}\label{appendix:semi}

\subsubsection{A general consistency result}\label{sec:gencons}

We herein provide a general consistency result for the three max-stable model classes in Section \ref{sec:general_theory} of the main paper. 
In the following subsections, the proofs of Theorems \ref{th:alpha_frec}, \ref{theo: cons_weibull} and \ref{cor:Gumbel_cons} are developed by verifying its conditions for the specific cases of $\brho$-Fr\'echet, $\bomega$-Weibull and Gumbel multivariate max-stable models, respectively. 
The proof of Proposition \ref{prop:gen_cons} adapts arguments from the proofs of Theorems 6.17 and 6.23 in 
%{\color{magenta}Ghosal and van der Vaart (2017)}
\cite{r10}. Some key derivations yielding consistency at an exponential rate - equations \eqref{eq:mainbound}-\eqref{eq:expobound} of the main paper; see also \cite[Definition 3.1]{r555} -
%{\color{magenta}Choi and Ramamoorthi 2008, Definition 3.1}) 
are discussed more in detail, as they also play a crucial role in establishing the results of Section \ref{sec:binf_sample_max}.

\begin{prop}\label{prop:gen_cons}
	Let $\bX_1,\ldots,\bX_n$ be iid rv with distribution $G_{\bvartheta_0}(\cdot|H_0)$, where $H_0 \in \Hset_0$
	%, with Lipschitz continuous $h_0$, 
	and $\bvartheta_0  \in \bvarTheta$. Let $\Pi_{\Hset \times {\bvarTheta}}:=\Pi_\Hset \times \Pi_{\bvarTheta}$, where $\Pi_\Hset$ and $\Pi_{\bvarTheta}$ 
	are Borel priors on $\Hset$ and $\bvarTheta$.
	%	satisfy Conditions \ref{cond:angularprior} and \ref{cond:prior a}, respectively. 
	Denote by $\Pi_{\mathcal{G}_{\bvarTheta}}$ the prior induced  by $\Pi_{\Hset\times\bvarTheta}$ on $\Gset_{\bvarTheta}:=\{ g_{\bvartheta}(\cdot|H):\,(H,\bvartheta)\in (\Hset,\bvarTheta) \}$  and assume that Conditions \ref{cond:angularprior} and \ref{cond: newcond} are satisfied.
	Then, $\iprodGtHtruealt-\text{as}$
	\begin{itemize}
		\item[(a)] $\lim_{n \to \infty}\tilde{\Pi}_n(\tilde{\mathcal{U}}^\complement)=0$, for every $\dist_H$-neighbourhood $\tilde{\mathcal{U}}$ of $g_{\bvartheta_0}(\cdot|H_0)$;
		\item[(b)] $
		\lim_{n\to \infty}\dist_H(\hat{g}_n,g_{\bvartheta_0}(\cdot|H_0))=0$; 
		\item[(c)] $\lim_{n\to\infty}\Pi_n((\mathcal{U}_1\times\mathcal{U}_2)^\complement)=0
		$, for every $\dist_W$-neighborhood (if $d\geq 2$) or $\dist_{KS}$-neighborhood (if d=2) $\mathcal{U}_1$  of $H_0$ and $L_1$-neighborhood $\mathcal{U}_2$ of $\bvartheta_0$;  
	\end{itemize}
	where $\Pi_n(\cdot)=\Pi_{\Hset \times \bvarTheta}(\cdot|\bX_{1:n} )$, $\tilde{\Pi}_n(\cdot)=\Pi_{\Gset_{\bvarTheta}}(\cdot|\bX_{1:n})$ and $\hat{g}_n(\bx)=\int_{\Gset_{\bvarTheta}}g(\bx) \, \diff \tilde{\Pi}_n(g)$.
\end{prop} 

\begin{proof}
	The proof is organized in three parts: set theoretical preliminaries, main body of the proof and conclusions. 
	\\

	\noindent
	\textit{Set theoretical preliminaries.}
	Denote
	\begin{equation}\label{eq:genmap}
	\phi_{\Hset\times\bvarTheta}:(\Hset\times \bvarTheta, \sigBor_{\Hset}\otimes\sigBor_{\bvarTheta})\mapsto (\Gset_{\bvarTheta}, \sigBor_{\Gset_{\bvarTheta}}):(H,\bvartheta)\mapsto g_{\bvartheta}(\cdot|H),
	\end{equation}
	where $\sigBor_{\Hset},\sigBor_{\bvarTheta}$ and $\sigBor_{\Gset_{\bvarTheta}}$ are the Borel $\sigma$-field induced on $\Hset$, $\bvarTheta$ and $\Gset_{\bvarTheta}$ by the metrics $\dist_W$, $L_1$ and $\dist_H$, respectively.
	Let $\tilde{\mathcal{U}}$ be any $\dist_H$-neighbourhood of $g_{\bvartheta_0}(\cdot|H_0)$, $\mathcal{U}_1$ be any
	$\dist_W$-neighborhood (if $d\geq 2$) or $\dist_K$-neighborhood (if d=2) of $H_0$
	and $\mathcal{U}_2$ be any
	$L_1$-neighborhood of $\bvartheta_0$.
	We next derive a preliminary upper bound for the term
	$$
	\max
	\left\lbrace
	\tilde{\Pi}_n(\tilde{\mathcal{U}}^\complement),
	\Pi_n((\mathcal{U}_1\times\mathcal{U}_2)^\complement)
	\right\rbrace
	$$
	using some simple set theoretical arguments.

	Observe that, for some sufficiently small $\epsilon',\delta>0$, $B_{\epsilon',\dist}(H_0):=\{H \in \Hset: \, \dist(H,H_0)< \epsilon'\}\subset \mathcal{U}_1$, where $\dist=\dist_{KS}$ if $d=2$ and $\dist=\dist_W$ if $d>2$, and $B_{\delta,1}(\bvartheta_0) \subset \mathcal{U}_2$. Therefore,
	\begin{equation}\label{eq:newupboundgen1}
	\begin{split}
	\Pi_n(\{\mathcal{U}_1\times\mathcal{U}_2\}^\complement)
	&\leq 
	\Pi_n\left(
	B^\complement_{\epsilon',\dist}(H_0)
	\times B_{\delta,{1}}(\bvartheta_0)
	\right)
	+
	\Pi_n\left(
	\Hset\times  B^\complement_{{\delta/d},\infty}(\bvartheta_0)
	\right).
	%\\
	%&=: T_{1,n}+T_{2,n}.
	\end{split}
	\end{equation}
	Moreover, notice that by Propositions \ref{lem: basic metric}-\ref{prop: basic metric bvt}, for some $\epsilon''>0$,
	$$
	B_{\epsilon''}:=\{H \in \Hset: \, \dist_H(g_\bone(\cdot |H),g_\bone(\cdot |H_0))< 2\epsilon''\}\subset B_{\epsilon',\dist}(H_0).
	$$
	Furthermore, by continuity of the map $\bvartheta\mapsto g_{\bvartheta}(\cdot|H_0)$ with respect to the $L_1$ and the Hellinger metrics on a neighbourhood of $\bvartheta_0$, we can choose $\delta$ small enough to guarantee that for all $\bvartheta \in B_{\delta, {1}}(\bvartheta_0)$
	$$
	\dist_H(g_{\bvartheta}(\cdot|H_0),g_{\bvartheta_0}(\cdot|H_0))<\epsilon''.
	$$
	Hence, reverse triangular inequality entails that for all $(H,\bvartheta) \in B_{\epsilon''}^\complement\times B_{\delta,{\color{red}1}}(\bvartheta_0)$
	\begin{equation*}
	\begin{split}
	\dist_H(g_{\bvartheta}(\cdot|H),g_{\bvartheta_0}(\cdot|H_0))&\geq 
	\dist_H(g_{\bvartheta}(\cdot|H),g_{\bvartheta}(\cdot|H_0))
	-
	\dist_H(g_{\bvartheta}(\cdot|H_0),g_{\bvartheta_0}(\cdot|H_0))\\
	&=\dist_H(g_\bone(\cdot |H),g_\bone(\cdot |H_0))
	-
	\dist_H(g_{\bvartheta}(\cdot|H_0),g_{\bvartheta_0}(\cdot|H_0))\\
	&>\epsilon'', 
	\end{split}
	\end{equation*}
	from which we conclude that 
	$$
	B_{\epsilon',\dist}^\complement(H_0)\times B_{\delta, {1}}(\bvartheta_0)
	\subset
	B_{\epsilon''}^\complement\times B_{\delta, {1}}(\bvartheta_0)\subset \phi_{\Hset\times\bvarTheta}^{-1}(\widetilde{\mathcal{U}}_{\epsilon}^\complement),
	$$
	where $\epsilon=\epsilon''/4$ and
	%$\phi_{\Hset\times\bvarTheta}$ is as in \eqref{eq:genmap} and 
	%$\widetilde{\mathcal{U}}_{\epsilon}$ is defined via
	%\begin{equation*}
	%\label{eq:genset}
	$\widetilde{\mathcal{U}}_\epsilon:= 
	\{g \in \Gset_{\bvarTheta}: \dist_{H}(g,g_{\bvartheta_0}(\cdot|H_0))\leq 4\epsilon\}$.
	%\subset
	%\widetilde{\mathcal{U}}.
	%\end{equation*} 
	%in \eqref{eq:genset}, with $\epsilon=\epsilon''/4$. 
	As a result,
	%$
	%T_{1,n}\leq \widetilde{\Pi}_n(\widetilde{\mathcal{U}}_{\epsilon}^\complement).
	%$
	\begin{equation}\label{eq:newupboundgen2}
	\Pi_n\left(
	B^\complement_{\epsilon',\dist}(H_0)
	\times B_{\delta,{1}}(\bvartheta_0)
	\right) \leq \widetilde{\Pi}_n(\widetilde{\mathcal{U}}_{\epsilon}^\complement).
	\end{equation}
	%Without loss of generality
	In particular, we can choose $\epsilon''$, and thus $\epsilon$, small enough to also guarantee that $\widetilde{\mathcal{U}}_\epsilon \subset \widetilde{\mathcal{U}}$. Therefore,
	\begin{equation}\label{eq:newupboundgen3}
	\begin{split}
	\widetilde{\Pi}_n(\widetilde{\mathcal{U}}^\complement)&\leq \widetilde{\Pi}_n(\widetilde{\mathcal{U}}_\epsilon^\complement)\\
	&\leq \Pi_n\left(\phi_{\Hset\times\bvarTheta}^{-1}(\widetilde{\mathcal{U}}_\epsilon^\complement)\cap \{
	\Hset \times B_{\delta, {1}}(\bvartheta_0)
	\}\right) + \Pi_n\left(
	\Hset \times  B^\complement_{{\delta/d}, \infty}(\bvartheta_0)
	\right).
	\end{split}
	\end{equation}
	By combining \eqref{eq:newupboundgen1}-\eqref{eq:newupboundgen3} we finally deduce that
	\begin{equation}\label{eq: preliminaryineq}
	\begin{split}
	\max
	\left\lbrace
	\tilde{\Pi}_n(\tilde{\mathcal{U}}^\complement),
	\Pi_n((\mathcal{U}_1\times\mathcal{U}_2)^\complement)
	\right\rbrace & \leq \Pi_n\left(\phi_{\Hset\times\bvarTheta}^{-1}(\widetilde{\mathcal{U}}_\epsilon^\complement)\cap \{
	\Hset \times B_{\delta, {1}}(\bvartheta_0)
	\}\right)\\
	& \quad + 2\Pi_n\left(
	\Hset \times  B^\complement_{{\delta/d}, \infty}(\bvartheta_0)
	\right).
	\end{split}
	\end{equation}
	\\
	
	\noindent
	\textit{Main body of the proof.} 
	%Denote
	%\begin{equation}\label{eq:genmap}
	%\phi_{\Hset\times\bvarTheta}:(\Hset\times \bvarTheta, \sigBor_{\Hset}\otimes\sigBor_{\bvarTheta})\mapsto (\Gset_{\bvarTheta}, \sigBor_{\Gset_{\bvarTheta}}):(H,\bvartheta)\mapsto g_{\bvartheta}(\cdot|H),
	%\end{equation}
	%where $\sigBor_{\Hset},\sigBor_{\bvarTheta}$ and $\sigBor_{\Gset_{\bvarTheta}}$ are the Borel $\sigma$-field induced on $\Hset$, $\bvarTheta$ and $\Gset_{\bvarTheta}$ by the metrics $\dist_W$, $L_1$ and $\dist_H$, respectively.
	%%
	%Notice that, for some small $\epsilon>0$, 
	%\begin{equation}\label{eq:genset}
	%\widetilde{\mathcal{U}}_\epsilon:= 
	%\{g \in \Gset_{\bvarTheta}: \dist_{H}(g,g_{\bvartheta_0}(\cdot|H_0))\leq 4\epsilon\}
	%\subset
	%\widetilde{\mathcal{U}}.
	%\end{equation}
	%Then, for any $\delta>0$,
	%\begin{equation}\label{eq:upgen1}
	%\begin{split}
	%\widetilde{\Pi}_n(\widetilde{\mathcal{U}}^\complement)&\leq \Pi_n\left(\phi_{\Hset\times\bvarTheta}^{-1}(\widetilde{\mathcal{U}}\epsilon^\complement)\cap \{
	%\Hset \times B_{\delta, \infty}(\bvartheta_0)
	%\}\right)\\
	%& \quad + \Pi_n\left(
	%\Hset \times  B^\complement_{\delta, \infty}(\bvartheta_0)
	%\right)\\
	%&=: T_{n,1}+T_{n,2}.
	%\end{split}
	%\end{equation}
	%%
	We now analyse the terms on the right-hand side of \eqref{eq: preliminaryineq}.
	%,exploiting Condition \ref{cond: newcond}. 
	The first term can be decomposed as
	\begin{equation}\label{eq:eqgen2}
	\begin{split}
	\Pi_n\left(\phi_{\Hset\times\bvarTheta}^{-1}(\widetilde{\mathcal{U}}_\epsilon^\complement)\cap \{
	\Hset \times B_{\delta, {1}}(\bvartheta_0)
	\}\right) &=\widetilde{\Pi}_n(\Gset_{\bvarTheta,n})+\widetilde{\Pi}_n(\Gset_{\bvarTheta,n}^\complement),
	\end{split}
	\end{equation}
	where $\mathcal{G}_{\bvarTheta,n}$ is any sequence of measurable subsets of
	\begin{equation}\label{eq:subsetof}
	\widetilde{\mathcal{U}}_\epsilon^\complement\cap \phi_{\Hset\times\bvarTheta}\left(\{
	\Hset \times B_{\delta, {1}}(\bvartheta_0)
	\}\right)
	\end{equation}
	and $\Gset_{\bvarTheta,n}^\complement$ is the relative complement of $\Gset_{\bvarTheta,n}$ in the set above.
	Denote by $p(\bX_{1:n})$ the ratio between the marginal likelihood and the likelihood at $(H_0,\bvartheta_0)$, i.e.
	$$
	p(\bX_{1:n}):=\int_{\Gset_{\bvarTheta}}\prod_{i=1}^n\frac{g(\bX_i)}{g_{\bvartheta_0}(\bX_i|H_0)}\diff \Pi_{\Gset_{\bvarTheta}}(g),
	$$
	and let $\tests=(s_n, t_{n,1}, \ldots,t_{n,d})$ be any collection of measurable functions from $\text{supp}(\prodGtHtrue(\cdot|H_0))$ to $[0,1]$.
	Then, on one hand, we have the equality
	\begin{equation}\label{eq:upgen3}
	\begin{split}
	\widetilde{\Pi}_n(\Gset_{\bvarTheta,n})&= s_n(\bX_{1:n}) +(1-s_n(\bX_{1:n}))\widetilde{\Pi}_n(\Gset_{\bvarTheta,n})\\
	& = s_n(\bX_{1:n}) + \frac{1-s_n(\bX_{1:n})}{p(\bX_{1:n})}\int_{\Gset_{\bvarTheta,n}}\prod_{i=1}^n\frac{g(\bX_i)}{g_{\bvartheta_0}(\bX_i|H_0)}\diff \Pi_{\Gset_{\bvarTheta}}(g),
	\end{split}
	\end{equation}
	%with $s_n$ a test functional satisfying Condition \ref{cond: newcond}\ref{dentest}, 
	on the other hand, we have the identity
	\begin{equation}\label{eq:upgen4}
	\begin{split}
	\widetilde{\Pi}_n(\Gset_{\bvarTheta,n}^\complement)= \frac{1}{p(\bX_{1:n})}\int_{\Gset_{\bvarTheta,n}^\complement}\prod_{i=1}^n\frac{g(\bX_i|H)}{g_{\bvartheta_0}(\bX_i|H_0)}\diff \Pi_{\Gset_{\bvarTheta}}(g).
	\end{split}
	\end{equation}
	Moreover, denoting for $j=1, \ldots,d$,
	$$
	D_{n,j}:= \Hset\times\{
	\bvartheta\in \bvarTheta: \,
	\Vert \bvartheta_j - \bvartheta_{0,j}\Vert_\infty
	>{\delta/d}
	\},
	$$
	we can observe that
	\begin{equation}\label{eq:upgen5}
	\begin{split}
	\Pi_n&\left(
	\Hset \times  B^\complement_{{\delta/d}, \infty}(\bvartheta_0)
	\right)\\
	&\leq \sum_{j=1}^d \Pi_n(D_{n,j})\\
	&=\sum_{j=1}^dt_{n,j}(\bX_{1:n})+ \sum_{j=1}^d(1-t_{n,j}(\bX_{1:n})) \Pi_n(D_{n,j})\\
	&= \sum_{j=1}^dt_{n,j}(\bX_{1:n})+ \sum_{j=1}^d\frac{1-t_{n,j}(\bX_{1:n})}{p(\bX_{1:n})} \int_{D_{n,j}}\prod_{i=1}^n\frac{g_{\bvartheta}(\bX_i|H)}{g_{\bvartheta_0}(\bX_i|H_0)}\diff \Pi_{\Hset\times \bvarTheta}(H,\bvartheta).
	\end{split}
	\end{equation}
	%where $t_{n,1}, \ldots, t_{n,d}$ are test functionals satisfying Condition \ref{cond: newcond}\ref{margintest}. 
	%
	Combining \eqref{eq: preliminaryineq}-\eqref{eq:upgen5}, we can now conclude that 
	\begin{equation}\label{eq: finalbound}
	\begin{split}
	%	\widetilde{\Pi}_n(\widetilde{\mathcal{U}}^\complement)
	%	
	\max
	\left\lbrace
	\tilde{\Pi}_n(\tilde{\mathcal{U}}^\complement),
	\Pi_n((\mathcal{U}_1\times\mathcal{U}_2)^\complement)
	\right\rbrace 
	&\leq s_n(\bX_{1:n})
	+2\sum_{j=1}^d t_{n,j}(\bX_{1:n}) 	 + \frac{\Xi_n(\bX_{1:n},\tests,\Pi_{\Hset\times \bvarTheta})}{p(\bX_{1:n})}
	%	\\
	%	& \leq s_n(\bX_{1:n})
	%	+\sum_{j=1}^d t_{n,j}(\bX_{1:n}) + e^{cn}\Xi_n(\bX_{1:n},\Pi_{\Hset\times \bvarTheta})
	\end{split}
	\end{equation}
	where 
	%$\tests=(s_n, t_{n,1}, \ldots, t_{n,d})$ and 
	the functional $\Xi_n(\cdot,\cdot,\cdot)$ is defined via
	\begin{equation}\label{eq:xi_n}
	\begin{split}
	\Xi_n(\bX_{1:n},\boldsymbol{\tau},\Psi)&:=
	(1-s(\bX_{1:n}))\int_{\phi_{\Hset\times\bvarTheta}^{-1}(\Gset_{\bvarTheta,n})}\prod_{i=1}^n\frac{g_\bvartheta(\bX_i|H)}{g_{\bvartheta_0}(\bX_i|H_0)}\diff \Psi(H, \bvartheta)\\
	&\qquad + \int_{\phi_{\Hset\times\bvarTheta}^{-1}(\Gset_{\bvarTheta,n}^\complement)}\prod_{i=1}^n\frac{g_\bvartheta(\bX_i|H)}{g_{\bvartheta_0}(\bX_i|H_0)}\diff \Psi(H, \bvartheta)\\
	&\qquad +
	2\sum_{j=1}^d(1-t_{j}(\bX_{1:n}))\int_{D_{n,j}}\prod_{i=1}^n\frac{g_{\bvartheta}(\bX_i|H)}{g_{\bvartheta_0}(\bX_i|H_0)}\diff \Psi(H,\bvartheta),
	\end{split}
	\end{equation}
	for all choices of measurable functions $\boldsymbol{\tau}=(s, t_1, \ldots,t_d)$ from $\text{supp}(\prodGtHtrue(\cdot|H_0))$ to $[0,1]$ and Borel pm $\Psi$ on $\Hset \times \bvarTheta$.
	Moreover, by the Kullback leibler property,
	for any $c>0$ we have that 
	$
	p(\bX_{1:n}) \geq e^{-nc}
	$
	eventually almost surely as $n \to \infty$,
	hence 
	\begin{equation}\label{eq: applykulb}
	\frac{\Xi_n(\bX_{1:n},\tests,\Pi_{\Hset\times \bvarTheta})}{p(\bX_{1:n})} \leq  e^{nc}\Xi_n(\bX_{1:n},\tests,\Pi_{\Hset\times \bvarTheta}).
	\end{equation}
	Without loss of generality,
	we can assume that $\delta<\delta_*$, with $\delta_*$ as in Condition \ref{cond: newcond}, and select $\Gset_{\bvarTheta,n}$ and $\tests$ satisfying the properties therein.
	Consequently, choosing  $c$ such that
	\begin{equation}\label{eq:choosec}
	2c < \min\left\lbrace
	\epsilon^2, r, c_1(\delta), \ldots,c_d(\delta)
	\right\rbrace=: \overline{c},
	\end{equation}
	an application of Markov's inequality and  Fubini's theorem
	yields that as $n \to \infty$
	\begin{equation}\label{eq:tailprob}
	\prodGtHtrue \left(
	2\Vert \tests(\bX_{1:n}) \Vert_1+ e^{cn}\Xi_n(\bX_{1:n}, \tests, \Pi_{\Hset \times \bvarTheta})> \varepsilon \big{|}H_0
	\right) \lesssim e^{-n\overline{c}/2}.
	\end{equation}
	for all $\varepsilon>0$.
	\\
	
	\noindent
	\textit{Conclusions.} 
	In view of the bounds in \eqref{eq: finalbound},  \eqref{eq: applykulb} and \eqref{eq:tailprob}, the results at points (a) and (c) of the statement follows from an application of Borel-Cantelli lemma.
	Since the map $g\mapsto \dist_H^2(g, g_{\bvartheta_0}(\cdot|H_0))$ is convex, the result at point (b) follows directly from that at point (a) together with Theorem 6.8 in \cite{r10}. 
%	{\color{magenta}Ghosal and van der Vaart (2017)}.
\end{proof}

\begin{rem}\label{rem:proof_highlight}
	We stress that, once $\epsilon$, $\delta$ and an arbitrary measurable subset $\mathcal{G}_{\bvarTheta,n}$ are fixed, a functional $\Xi_n$ can be defined as in \eqref{eq:xi_n}
	%	are according to Condition Condition \ref{cond: newcond} 
	and the inequalities in \eqref{eq:upgen3} and \eqref{eq:upgen5}  
	are satisfied 
	for all choices of measurable functions $\boldsymbol{\tau}_n=(s_n, t_{n,1}, \ldots,t_{n,d})$ from 
	$$
	\text{supp}(\prodGtHtrue(\cdot|H_0))
	$$
	to $[0,1]$.
	Therefore, the inequality in \eqref{eq: finalbound} is valid also for posterior distributions $\Pi_n$ and $\tilde{\Pi}_n$ arising from priors which do not comply with Condition \ref{cond: newcond}.
	This fact is used in the proofs of Theorems \ref{th:rem_cont_frec}, \ref{th:rem_cont_gumb} and \ref{theo:rem_weib}.
	We also point out that $\Xi_n$
	tacitly depends on $\epsilon$ and $\delta$ via the sets 
	$\mathcal{G}_{\bvarTheta,n}$ (depending on $\epsilon$ and $\delta$ via the set in \eqref{eq:subsetof}) and $D_{n,1}, \ldots, D_{n,d}$ (depending on $\delta$).
	We recall that $\epsilon$ and $\delta$ are determined in the first place using some set theoretical arguments which move from the choices of $\tilde{\mathcal{U}}$, $\mathcal{U}_1$ and $\mathcal{U}_2$.
	Thus, the functional $\Xi_n$ ultimately depends on $\tilde{\mathcal{U}}$, $\mathcal{U}_1$ and $\mathcal{U}_2$. This consideration makes precise the remark following inequality \eqref{eq:mainbound} in the main paper.
\end{rem}

\subsubsection{Preliminary auxiliary results for the proof of Theorem \ref{th:alpha_frec}}\label{app:aux_frec}
For $l \in \nat_+$, we introduce the following symbol
\begin{eqnarray}
\label{eq:upsilon1}
\upsilon_l:= \Gamma(2l+1)\int_1^\infty \{\log(s)\}^l \text{Gamma}(s;2l+1,1)\diff s,
\end{eqnarray}
where $\text{Gamma}(s;a,b)$, $s>0$, denotes the Gamma probability density function with shape parameter $a$ and scale parameter $b$. Of particular interest for this work are the cases $l=1,\ldots,4$, where we have
\begin{equation}\label{eq:upsilons}
\begin{split}
&\upsilon_1= \frac{4}{e}-2\text{Ei}(-1), \hspace{4em} \upsilon_2=
2
G^{4,0}_{3,4}\left(1\bigg{|} \begin{smallmatrix}1,1,1 \\ 0,0,0,5\end{smallmatrix} \right),
\\
&\upsilon_3=6
G^{5,0}_{4,5}\left(1\bigg{|} \begin{smallmatrix}1,1,1,1 \\ 0,0,0,0,7\end{smallmatrix} \right), \hspace{1.2em}
\upsilon_4=24
G^{6,0}_{5,6}\left(1\bigg{|} \begin{smallmatrix}1,1,1,1,1 \\ 0,0,0,0,0,9\end{smallmatrix} \right),
\end{split}
\end{equation}
with $\text{Ei}(\cdot)$ denoting the exponential integral function and $G^{p,q}_{r,s}(\cdot|\begin{smallmatrix}t_1,\ldots,t_p \\ v_1,\ldots,v_q\end{smallmatrix})$ denoting the Meijer $G$-function. See Corollary \ref{lem: KLalpha} and Lemma \ref{lem:pseudokulb}. Moreover, we use the same notational convention for vectors raised to powers and rectangles in $[-\infty,\infty]^d$ introduced before Lemma \ref{aux1}.
Finally, for $l \in \nat_+$, $H_*,H, \widetilde{H}\in \Hset$ and $\brho_*,\brho, \widetilde{\brho},\bsigma_*,\bsigma, \widetilde{\bsigma}\in (\bzero, \binf)$, we define
$$
\mathscr{V}_{H_*, \brho_*, \bsigma_*}^{(l)}(H, \brho, \bsigma; \widetilde{H}, \widetilde{\brho}, \widetilde{\bsigma}):=
\left( 
\int_{(\bzero, \binf)} \left[
\log^+ \left\lbrace
\frac{g_{\brho,\bsigma}(\by|H)}{g_{\widetilde{\brho},\widetilde{\bsigma}}(\by|\widetilde{H})}
\right\rbrace
\right]^lg_{\brho_*, \bsigma_* }(\by|H_*)\diff \by
\right)^{1/l}.
$$	
Notice that the above functional is analogous to that introduced in Section \ref{sec:signed_kulb_simple} in the context of simple max-stable distributions. In particular, 
$$
\mathscr{V}_{H, \brho, \bsigma}^{(l)}(H, \brho, \bsigma; \widetilde{H}, \widetilde{\brho}, \widetilde{\bsigma})=\kulb_+^{(l)}(g_{{\brho}, {\bsigma}}(\cdot|{H}), g_{\widetilde{\brho}, \widetilde{\bsigma}}(\cdot|\widetilde{H}))
$$
where the term on the right-hand side denotes the $l$-th order positive Kullback-Leibler divergence from $ g_{\widetilde{\brho}, \widetilde{\bsigma}}(\cdot|\widetilde{H})$
to $ g_{{\brho}, {\bsigma}}(\cdot|{H})$. 
\begin{lemma}\label{lem: aux}
	Let $H\in \Hset$, with exponent function $V(\cdot|H)$. Then, for any  $l \in \nat_+$, $\widetilde{H} \in \Hset$, $\brho\in (\bzero, \binf)$, $\widetilde{\brho}  \in B_{\delta_1,\infty}(\brho)$, $\bsigma \in (\bzero, \binf)$,
	$\widetilde{\bsigma} \in B_{\delta_2,\infty}(\bsigma)$, with 
	\begin{equation}\label{eq:deltas}
	\delta_1 \in \left(0, \min_{1\leq j \leq d} \rho_j \varepsilon
	\right), \quad 	\delta_2 \in \left(0, \min_{1\leq j \leq d} \sigma_j \varepsilon
	\right), \quad \varepsilon \in (0,1/2],
	\end{equation} 
	we have that
	$$
	\left(
	\int_{(\bzero,\binf)}
	\left|
	V(\by^{\widetilde{\brho}/\brho} \{\bsigma/\widetilde{\bsigma}\}^{\widetilde{\brho}}|H)-V(\by|H)
	\right| ^l g_\bone(\by|\widetilde{H})\diff \by
	\right)^{1/l}
	$$
	is bounded from above by
	$$
	\frac{\left\lbrace \upsilon_l + \Gamma(1+l)
		\right\rbrace^{1/l} }{\min_{1\leq j \leq d}\rho_j}
	\Vert \brho -\widetilde{\brho} \Vert_1+
	\left\lbrace
	\Gamma(1+2l)
	\right\rbrace
	^{1/l}
	\max_{1\leq j \leq d} \left\lbrace \frac{3\rho_j}{\sigma_j}
	\left(\frac{3}{2}\right)^{\frac{3}{2}\rho_j}
	\right\rbrace
	\Vert \bsigma -\widetilde{\bsigma}\Vert_1.
	$$
	%	where, for $a,b, \widetilde{a},\widetilde{b}>0$, 
	%	\begin{equation}\label{eq:upper_bound_fun}
	%	{\color{magenta} 
	%	\Upsilon_\varepsilon(a, b; \widetilde{a}, \widetilde{b}):=
	%	\frac{\left\lbrace \upsilon(l)+\Gamma(1+l \varepsilon )\right\rbrace^{1/l}}{a}
	%	|a-\widetilde{a}|+\left\lbrace
	%	\Gamma(1+l(1+\varepsilon))
	%	\right\rbrace
	%	^{1/l}t_\varepsilon(a,b)|b-\widetilde{b}|,
	%	}
	%	\end{equation}
	%	and $\upsilon_l$ and $t_\varepsilon(a,b)$ are defined in \eqref{eq:upsilon1} and \eqref{eq:t_epsilon}, respectively.
\end{lemma}
\begin{proof}
	For $\bx> \bzero$, the function $V(1/\bx|H)$ defines a $D$-norm, see \cite{r2}.
%	{\color{magenta}Falk et al. (2010)}.
	%	\shortciteN{falk2010}. 
	Property (4.37) therein together with reverse triangle inequality  yield
	$$
	|V(\by^{\widetilde{\brho}/\brho} \{\bsigma/\widetilde{\bsigma}\}^\brho|H)-V(\by|H)|\leq 
	\Vert 	\by^{-\widetilde{\brho}/\brho} \{\bsigma/\widetilde{\bsigma}\}^{-\widetilde{\brho}}- \by^{-1}\Vert_1.
	$$ 
	The above inequality, the fact that $g_\bone(\cdot|\widetilde{H})$ has unit Fr\'{e}chet margins and a few manipulations lead to conclude that
	\begin{equation*}
	\begin{split}
	&\left(
	\int_{(\bzero,\binf)}
	\left|
	V(\by^{\widetilde{\brho}/\brho} \{\bsigma/\widetilde{\bsigma}\}^\brho|H)-V(\by|H)
	\right|^l g_\bone(\by|\widetilde{H})\diff \by
	\right)^{1/l}\\
	&\quad		 \leq \sum_{j=1}^d 
	\left( \int_0^\infty\left|{y_j^{-\widetilde{\rho}_j/\rho_j}} - {y_j^{-1}} \right|^ly_j^{-2}e^{-y_j^{-1}}\diff y_j
	\right)^{1/l}\\
	& \qquad
	+ \sum_{j=1}^d 
	\left(
	\int_0^\infty \left|1 - \left(\frac{\widetilde{\sigma}_j}{\sigma_j}\right)^{\widetilde{\rho}_j} \right|^l{y_j^{-2-l \widetilde{\rho}_j/\rho_j}}e^{-y_j^{-1}}\diff y_j
	\right)^{1/l}\\
	&\quad =: T_1+T_2.
	\end{split}
	\end{equation*} 
	On one hand, an application of the mean-value theorem and the inequality $\log(x)\leq x$, $x\geq 1$, yield
	\begin{equation*}
	\begin{split}
	T_1
	&\leq 
	\sum_{j=1}^d 
	\left|1-\frac{\widetilde{\rho}_j}{\rho_j} \right|
	\left( \int_0^1\left[
	\frac{-\log y_j}{y_j^{1+\varepsilon}}
	\right]^ly_j^{-2}e^{-y_j^{-1}}\diff y_j
	+\int_1^\infty\left[
	\frac{\log y_j}{y_j^{1-\varepsilon}}
	\right]^ly_j^{-2}e^{-y_j^{-1}}\diff y_j
	\right)^{1/l}\\
	& \leq
	\left\lbrace \upsilon_l+\Gamma(1+l \varepsilon )\right\rbrace^{1/l}
	\sum_{j=1}^d 
	\left|1-\frac{\widetilde{\rho}_j}{\rho_j} \right|\\
	& \leq \left\lbrace \upsilon_l + \Gamma(1+l)
	\right\rbrace^{1/l}
	\sum_{j=1}^d 
	\left|1-\frac{\widetilde{\rho}_j}{\rho_j} \right|.
	\end{split}
	\end{equation*}
	On the other hand, using sequentially the integral representation of the Gamma function, the mean-value theorem, the fact that $|\log(x)|\leq \max(|x-1|,|x-1|/x )$, $x>0$, together with the bounds in \eqref{eq:deltas}, we obtain  
	\begin{equation}\label{eq:similar}
	\begin{split}
	T_2 &\leq \left\lbrace
	\Gamma(1+l(1+\varepsilon))
	\right\rbrace
	^{1/l} 
	\sum_{j=1}^d\left|1 - \left(\frac{\widetilde{\sigma}_j}{\sigma_j}\right)^{\widetilde{\rho}_j} \right|\\
	& \leq 
	\left\lbrace
	\Gamma(1+2l)
	\right\rbrace
	^{1/l}
	\sum_{j=1}^d \widetilde{\rho}_j \left|
	\log (\widetilde{\sigma}_j/\sigma_j)\right| \max\{
	(\widetilde{\sigma}_j/\sigma_j)^{\widetilde{\rho}_j},1
	\}
	\\
	&\leq \left\lbrace
	\Gamma(1+2l)
	\right\rbrace
	^{1/l}
	\sum_{j=1}^d \frac{3\rho_j}{\sigma_j}
	\left(\frac{3}{2}\right)^{\frac{3}{2}\rho_j}
	|\sigma_j -\widetilde{\sigma_j}|.
	\end{split}
	\end{equation}
	%	where
	%	\begin{eqnarray}
	%	\label{eq:upsilon2}
	%	\label{eq:t_epsilon}
	%	t_{\varepsilon}(a,b):=
	%	\left(\sup_{x \in a(1\pm\varepsilon)}b^{-x}\right) 
	%	\left( 
	%	\sup_{x \in a(1\pm\varepsilon),
	%	y \in b(1\pm\varepsilon)
	%	}y^{-x-1}
	%	\right).
	%	\end{eqnarray}
	The result now follows.
\end{proof}

The next result follows from the the previous one and Lemma \ref{aux1}.

\begin{lemma}\label{cor: alpha_kulb}
	Let  $H_0\in \Hset_0$ satisfy property \ref{cond: finite_new} in Definition \ref{cond:mvt_angular}. Let $\brho_0\in (\bzero,\binf)$, $\bsigma_0\in (\bzero,\binf)$. Then, for all $l \in \nat_+$ and $\epsilon>0$ there exist $H^* \in \Hset'$, $\delta_1,\delta_2, \delta_3>0$ such that
	\begin{equation}\label{eq:wehave}
	\kulb_+^{(l)}(g_{\brho_0, \bsigma_0}(\cdot|H_0), g_{\brho_0, \bsigma_0}(\cdot|H))<\epsilon,
	\end{equation}
	for all $\brho \in B_{\delta_1,1}(\brho_0)$, $\bsigma \in B_{\delta_2,1}(\bsigma_0)$ and $H \in B_{\delta_3,\infty}(H^*)$.
	In the particular case of $l=1$, it also holds that
	$
	\kulb({g_{\brho_0,\bsigma_0}(\cdot|H_0),g_{\brho,\bsigma}(\cdot|H)})<\epsilon.
	$
\end{lemma}
\begin{proof} Observe that, for any $H \in \Hset$, $\brho \in (0,\binf)$, $\bsigma \in (0,\binf)$, Minkowski's inequality	yields
	\begin{equation}\label{eq:thetwo}
	\begin{split}
	\kulb_+^{(l)}(g_{\brho_0, \bsigma_0}(\cdot|H_0), g_{\brho_0, \bsigma_0}(\cdot|H))&=
	\mathscr{V}_{H_0, \brho_0, \bsigma_0}^{(l)}(H_0, \brho_0, \bsigma_0; H, \brho, \bsigma)\\ &\leq 
	\mathscr{V}_{H_0, \brho_0, \bsigma_0}^{(l)}(H_0, \brho_0, \bsigma_0; H, \brho_0, \bsigma_0)\\
	&\quad +\mathscr{V}_{H_0, \brho_0, \bsigma_0}^{(l)}(H, \brho_0, \bsigma_0; H, \brho, \bsigma).
	\end{split}
	\end{equation}
	%	On one hand, 
	%	$$
	%	T_1\leq \frac{\gamma}{\min_{j=1,\ldots,d} a_{0,j}} \Vert \ba_0-\ba\Vert_1.
	%	$$ 
	%	Thus, $T_1$ can be made arbitrarily small by choosing $\ba\in B_{\delta_1,1}(\ba_0)$, with small radius $\delta_1$. On the other hand, 
	A change of variables allows to re-express the first term on the right-hand side of \eqref{eq:thetwo} as follows
	$$
	\kulb_+^{(l)}(g_{\brho_0, \bsigma_0}(\cdot|H_0), g_{\brho_0, \bsigma_0}(\cdot|H))=	\kulb_+^{(l)}(g_{\bone}(\cdot|H_0), g_{\bone}(\cdot|H)).
	$$ 
	Thus, by Proposition \ref{cor: New_KL_cor}, such a term can be made arbitrarily small by choosing $H \in B_{\delta_3,\infty}(H^*)$, for suitable $H^* \in \Hset'$ and $\delta_3>0$. As for the second term on the right-hand side of \eqref{eq:thetwo}, a change of variables and Minkowski's inequality yield
	\begin{equation*}
	\begin{split}
	&\mathscr{V}_{H_0, \brho_0, \bsigma_0}^{(l)}(H, \brho_0, \bsigma_0; H, \brho, \bsigma)\\
	&\quad \leq 
	\sum_{j=1}^d\left(
	\int_{(\bzero,\binf)}
	\left[
	\log^+
	\left\lbrace
	\frac{\rho_{0,j}}{\rho_{j}}
	y_j^{1-\rho_j/\rho_{0,j}}
	\frac{\sigma_j^{\rho_j}}{\sigma_{0,j}^{\rho_{j}}}
	\right\rbrace
	\right]^l g_\bone(\by|H_0)\diff \by
	\right)^{1/l}\\
	& \qquad + 	\left(
	\int_{(\bzero,\binf)}
	\left[
	\max
	\left\lbrace
	0, V(\by^{\brho/\brho_0} \{\bsigma_0/\bsigma\}^\brho|H)-V(\by|H)
	\right\rbrace
	\right]^l g_\bone(\by|H_0)\diff \by
	\right)^{1/l}\\
	& \qquad + 	\left(
	\int_{(\bzero,\binf)}
	\left[
	\log^+
	\left\lbrace
	\frac{\sum_{\part \in \allpart_d}
		\prod_{i=1}^{m}\{-V_{I_i}(\by|H)\}}{\sum_{\part \in \allpart_d}
		\prod_{i=1}^{m}\{-V_{I_i}(\by^{\brho/\brho_0} \{\bsigma_0/\bsigma\}^\brho|H)\}}
	\right\rbrace
	\right]^l g_\bone(\by|H_0)\diff \by
	\right)^{1/l}\\
	& \quad=:T_1+T_2+T_3.
	\end{split}
	\end{equation*}
	Applying once more Minkowski's inequality and exploiting the fact that $g_\bone(\cdot|H_0)$ has unit-Fr\'{e}chet margins we obtain
	\begin{equation*}
	\begin{split}
	T_1 &\leq \sum_{j=1}^d 
	\left[
	\log^+\left\lbrace
	\frac{\rho_{0,j}}{\rho_{j}}
	\frac{\sigma_j^{\rho_j}}{\sigma_{0,j}^{\rho_{j}}}
	\right\rbrace
	+\left(
	1-
	\frac{\rho_j}{\rho_{0,j}}\right) \indic(\rho_{0,j}>\rho_j)\gamma^{1/l}_{(l,+)} \right.\\
	& \qquad \qquad
	\left.
	+\left(
	\frac{\rho_j}{\rho_{0,j}}-1
	\right) \indic(\rho_{0,j}>\rho_j)\gamma_{(l,-)}^{1/l}
	\right],
	\end{split}
	\end{equation*}
	with $\gamma_{(+,l)},\gamma_{(l,-)}$ as in \eqref{eq:gammas}. Consequently, $T_1$ can be made arbitrarily small by choosing $\delta_1,\delta_2$ small enough.
	If $\delta_1, \delta_2$ also comply with \eqref{eq:deltas} for some $\varepsilon \in (0,1/2]$, with $(\brho_0,\bsigma_0)$ and $(\brho,\bsigma)$ in place of $(\brho,\bsigma)$ and $(\widetilde{\brho},\widetilde{\bsigma})$, respectively, the facts that
	$B_{\delta_1,1}(\brho_0)\subset B_{\delta_1,\infty}(\brho_0)$ and $B_{\delta_1,1}(\bsigma_0)\subset B_{\delta_1,\infty}(\bsigma_0)$ together with
	Lemma \ref{lem: aux}  entail that 
	\begin{equation*}
	\begin{split}
	T_2 \leq 
	\sum_{j=1}^d  \leq c_1 \delta_1 +c_2 \delta_2,
	\end{split}
	\end{equation*}
	where  $c_1, \, c_2$ are positive constant depending on $l$, $\brho_0$, $\bsigma_0$. 
	Furthermore, selecting $\delta_3$ such that
	$$
	0<\delta_3 < \min\left\lbrace 1,
	\frac{\inf_{\bv \in \mathring{\resimp}}h^*(\bv)}{2}, \min_{1 \leq j \leq d} \frac{p_j^*}{cd^{-1}}
	\right\rbrace^2
	$$
	where $c=1/\Gamma(d)$ and $p_j^* =H^*(\{ \be_j \})$, 	$j=1, \ldots,d$, by Lemma \ref{lem: approx} there exists $k>d$ and $H_k\in \Hset_k \cap B_{\delta_3, \infty}(H^*)\subset B_{2\delta_3, \infty}(H)$ with coefficients $\varphi_{\bkappa_1},\ldots,\varphi_{\bkappa_d},\varphi_\balpha,\balpha\in \Gamma_k$, satisfying  
	\begin{equation*}
	\begin{split}
	&\Vert h-h_{k-d} \Vert_\infty  \leq  \frac{2\delta}{\inf_{\bv \in \mathring{\resimp}}h^*(\bv) -\delta_3}(\inf_{\bv \in \mathring{\resimp}}h^*(\bv) -\delta_3) \leq \sqrt{\delta_3} \min\{h,h_{k-d}\},\\
	&\max
	\left(
	\frac{p_j - \varphi_{\bkappa_j}}{\varphi_{\bkappa_j}},
	\frac{\varphi_{\bkappa_j} - p_j}{p_j} 
	\right)
	\leq  \frac{2\delta_3}{p_j^*-\delta_3 cd^{-1}} \leq  \sqrt{\delta_3}, \quad j=1, \ldots,d,
	\end{split}
	\end{equation*}	
	wherefrom we deduce that facts analogous to (i)-(ii) in the proof of Lemma \ref{lem: ratio 1} hold true and we conclude that for all $\bz \in (\bzero, \binf)$ 
	\begin{equation}\label{eq:delta_bound}
	\max_{\part \in \allpart_d}\max_{I_i \in \part} \max \left( \log^+
	\left\lbrace 
	\frac{-V_{I_i}(\bz|H_k)}
	{-V_{I_i}(\bz|H)}
	\right\rbrace,
	\log^+
	\left\lbrace 
	\frac{-V_{I_i}(\bz|H)}
	{-V_{I_i}(\bz|H_k)}
	\right\rbrace
	\right)
	\leq \log(1+\sqrt{\delta_3}).
	\end{equation}
	Now, following steps like those in \eqref{eq:jensen} and applying Minkowski's inequality, we obtain
	\begin{equation*}
	\begin{split}
	T_3&\leq 
	\left(
	\int_{(\bzero,\binf)}
	\left[\max_{\part \in \allpart_d} \log^+
	\left\lbrace \prod_{i=1}^m
	\frac{-V_{I_i}(\by|H)}{-V_{I_i}(\by^{\brho/\brho_0} \{\bsigma_0/\bsigma\}^\brho|H)}
	\right\rbrace
	\right]^l g_\bone(\by|H_0)\diff \by
	\right)^{1/l}
	\\
	&\leq
	\left(
	\int_{(\bzero,\binf)}
	\left[d \max_{\part \in \allpart_d}\max_{I_i \in \part} \log^+
	\left\lbrace 
	\frac{-V_{I_i}(\by|H)}{-V_{I_i}(\by^{\brho/\brho_0} \{\bsigma_0/\bsigma\}^\brho|H)}
	\right\rbrace
	\right]^l g_\bone(\by|H_0)\diff \by
	\right)^{1/l}
	\\
	&\leq 
	\left(
	\int_{(\bzero,\binf)}
	\left[d \max_{\part \in \allpart_d}\max_{I_i \in \part} \log^+
	\left\lbrace 
	\frac{-V_{I_i}(\by|H_k)}{-V_{I_i}(\by^{\brho/\brho_0} \{\bsigma_0/\bsigma\}^\brho|H_k)}
	\right\rbrace
	\right]^l g_\bone(\by|H_0)\diff \by
	\right)^{1/l}
	\\
	&\qquad + 
	\left(
	\int_{(\bzero,\binf)}
	\left[d \max_{\part \in \allpart_d}\max_{I_i \in \part} \log^+
	\left\lbrace 
	\frac{-V_{I_i}(\by|H)}{-V_{I_i}(\by|H_k)}
	\right\rbrace
	\right]^l g_\bone(\by|H_0)\diff \by
	\right)^{1/l}
	\\
	& \qquad +\left(
	\int_{(\bzero,\binf)}
	\left[d \max_{\part \in \allpart_d}\max_{I_i \in \part} \log^+
	\left\lbrace 
	\frac{-V_{I_i}(\by^{\brho/\brho_0} \{\bsigma_0/\bsigma\}^\brho|H_k)}
	{-V_{I_i}(\by^{\brho/\brho_0} \{\bsigma_0/\bsigma\}^\brho|H)}
	\right\rbrace
	\right]^l g_\bone(\by|H_0)\diff \by
	\right)^{1/l}\\
	&=: T_3^{(1)} + T_3^{(2)} + T_3^{(3)}.
	\end{split}
	\end{equation*}
	The term $T_3^{(1)}$ can be made arbitrarily small by resorting to Lemma \ref{aux1} and choosing $\delta_1,\delta_2$ sufficiently small. 
	By the bound in \eqref{eq:delta_bound}, the same is true for $T_3^{(2)}$ and $T_3^{(3)}$. The proof is now complete.
\end{proof}
If $d=2$, a similar result can be established also when the angular density pertaining to $H_0$ is unbounded.
\begin{lemma}\label{lem:newkulblem2}
	Let $d=2$ and  $H_0\in \Hset_0$ satisfy property \ref{cond: infinite_new} in Definition \ref{cond:mvt_angular}. Let $\brho_0\in (\bzero,\binf)$, $\bsigma_0\in (\bzero,\binf)$. Then, for all $l \in \nat_+$ and $\epsilon>0$ there exist $H^* \in \Hset'$, $\delta_1,\delta_2, \delta_3>0$ such that
	$$
	\kulb(g_{\brho_0, \bsigma_0}(\cdot|H_0), g_{\brho, \bsigma}(\cdot|H))<\epsilon,
	$$
	for all $\brho \in B_{\delta_1,1}(\brho_0)$, $\bsigma \in B_{\delta_2,1}(\bsigma_0)$ and $H \in B_{\delta_3,\infty}(H^*)$. Further assuming that, for some $s>0$,
	$$
	\int_0^1 (h_0(t))^{1+s}\diff t<\infty,
	$$
	then, for any $l \in \nat_+$, $H^*$ and $\delta_j$, $j=1, \ldots,3$, can be chosen in such a way that \eqref{eq:wehave} is also satisfied.
\end{lemma}
\begin{proof}
	To prove the first claim, observe that
	\begin{equation*}
		\begin{split}
			\kulb(g_{\brho_0, \bsigma_0}(\cdot|H_0), g_{\brho, \bsigma}(\cdot|H))&\leq \kulb(g_{\brho_0, \bsigma_0}(\cdot|H_0), g_{\brho_0, \bsigma_0}(\cdot|H))\\
		&\quad + \mathscr{V}_{H_0, \brho_0, \bsigma_0}^{(1)}(H, \brho_0, \bsigma_0; H, \brho, \bsigma).
		\end{split}
	\end{equation*}
The first term on the right-hand side equals $\kulb(g_{\bone}(\cdot|H_0), g_{\bone}(\cdot|H))$ and can be made arbitrarily small by resorting to the first result in Proposition \ref{prop:newprop_KL}. As for the second term, it can be bounded from above by proceeding as in the proof of Lemma \ref{cor: alpha_kulb}, hence the result.

Similarly, to prove the second claim, note that inequality \eqref{eq:thetwo} holds true and that the first term on the right-hand side therein equals
$$
\kulb_{+}^{(l)}(g_\bone(\cdot|H_0),g_\bone(\cdot|H)).
$$
Therefore, it can be made arbitrarily small appealing to the second result in Proposition \ref{prop:newprop_KL}. Once more, the second term can be bounded from above by proceeding as in the proof of Lemma \ref{cor: alpha_kulb}, yielding the conclusion.
\end{proof}

For given $H_0\in \Hset_0$, $\brho_0\in (\bzero,\binf)$, $\bsigma_0\in (\bzero,\binf)$ and for all $\epsilon>0$ define
$$
\mathcal{V}_\epsilon:= 
\{(H, \brho,\bsigma) \in \Hset \times(\bzero,\binf)	\times(\bzero,\binf): \,
\kulb(g_{\brho_0, \bsigma_0}(\cdot|H_0), g_{\brho, \bsigma}(\cdot|H))<\epsilon
\}.
$$
moreover, for all $l \in \nat_+$, define 
\begin{equation}\label{eq:vsetfrec}
\mathcal{V}_\epsilon^{(l)}:= 
\{(H, \brho,\bsigma) \in \Hset \times(\bzero,\binf)	\times(\bzero,\binf): \,
\kulb_+^{(l)}(g_{\brho_0, \bsigma_0}(\cdot|H_0), g_{\brho, \bsigma}(\cdot|H)))<\epsilon
\}.
\end{equation}
We are now in the position to establish the following result, which is used in the proof of both Theorem \ref{th:alpha_frec} and Lemma \ref{lem:pseudokulb} (auxiliary to Theorem \ref{th:rem_cont_frec}).

\begin{cor}\label{lem: KLalpha}
	Under Conditions \ref{cond:genprior}\ref{cond:indep}--\ref{cond:angularpmprior} and \ref{cond:genprior}\ref{cond:compactprior} , for all $\epsilon>0$ we have that
	$$
	\Pi_{\Hset \times \bvarTheta}\left( \mathcal{V}_\epsilon \right)>0.
	$$   
	Thus, the induced prior $\Pi_{\Gset_{\bvarTheta}}$ on the Borel sets of
	$
	\Gset_{\bvarTheta}=\{g_{\brho,\bsigma}(\cdot|H): \, H \in\Hset, (\brho,\bsigma)\in (\bzero,\binf)\times(\bzero,\binf) \}
	$
	possesses the Kullback-Leibler property. Further assuming that Condition  \ref{cond:frecextend}\ref{cond:strongertruedens} is satisfied, we also have 
	$$
	\Pi_{\Hset \times \bvarTheta}\left(\cap_{l=1}^4 \mathcal{V}_\epsilon^{(l)}\right)>0.
	$$
\end{cor}
\begin{proof}
	Preliminarily observe that, since  $(\bx, H,\brho,\bsigma) \mapsto
	g_{\brho,\bsigma}(\bx|H)$ is tacitly chosen as a Borel measurable map between $(\bzero,\binf)\times\{\Hset\times(\bzero,\binf)\times(\bzero,\binf)\}$ and $[0,\infty]$, arguments similar to those in Appendix A of \cite{r48}
%	{\color{magenta}Petrone and Wasserman (2002)} 
	guarantee that $\mathcal{V}_\epsilon\in \sigBor_\Hset\otimes \sigBor_{\bvarTheta}$ and $\mathcal{V}_\epsilon^{(l)}\in \sigBor_\Hset\otimes \sigBor_{\bvarTheta}$, $l=1,\ldots,4$, where $\sigBor_\Hset$ is the Borel $\sigma$-algebra of $(\Hset,\dist_W)$ and $\sigBor_{\bvarTheta}$ is the Borel $\sigma$-algebra induced by the $L_1$-topology on $(\bzero,\binf) \times (\bzero,\binf)$.

	By Lemmas \ref{cor: alpha_kulb}-\ref{lem:newkulblem2}, whenever $H \in \Hset_0$, the divergences
	$$
	\kulb(g_{\brho_{0},\bsigma_{0}}(\cdot|H_0),g_{\brho,\bsigma}(\cdot|H))
	$$
	can be made arbitrarily small over a parameter subset $B_{\delta_3,\infty}(H^*)\times B_{\delta_1,1}(\brho_0)\times B_{\delta_2,1}(\bsigma_0)$, for suitable $H^* \in \Hset'$ and suitably small radii $\delta_1$, $\delta_2$, $\delta_3$. 
	Under the additional assumption in Condition  \ref{cond:frecextend}\ref{cond:strongertruedens}, the same is true for the higher-order positive divergences
	$$
	\kulb_+^{(l)}(g_{\brho_{0},\bsigma_{0}}(\cdot|H_0),g_{\brho,\bsigma}(\cdot|H)), \quad l=1,\ldots,4.
	$$
	On one hand,  Condition \ref{cond:genprior}\ref{cond:angularpmprior} guarantees the existence of a $\sigBor_\Hset$-measurable subset of $ B_{\delta_3,\infty}(H^*)$ with positive $\Pi_\Hset$-probability, 
	see the Kullback-Leibler property section in the proof of Theorem \ref{theo:post_consistency_mvt} for details. On the other hand,  $B_{\delta_1,1}(\brho_0)\times B_{\delta_2,1}(\bsigma_0)$ is an open set with respect to $L_1$-topology on $(\bzero,\binf) \times (\bzero,\binf)\equiv(0,\infty)^{2d}$, hence it is $\sigBor_{\bvarTheta}$-measurable and, by Condition \ref{cond:genprior}\ref{cond:shapescale},  $\Pi_{\bvarTheta}(B_{\delta_1,1}(\brho_0)\times B_{\delta_2,1}(\bsigma_0))>0$. The result now follows.
\end{proof}

Before stating the last lemma, we recall that for $\rho,\sigma>0$, $x>0$,  $G_{\rho,\sigma}(x):=e^{-(x/\sigma)^{-\rho}}$ denotes the (univariate) two-parameter Fr\'echet cdf.

\begin{lemma}\label{lem: tests}
	For each  $\epsilon>0$, there exist Borel-measurable functions $t_{n,j}:  (0, \infty)^n\mapsto\{0,1\}$, $j=1, \ldots,d$, $n \in \mathbb{N}_+$, and positive constants $c_j(\epsilon)$, $j=1, \ldots,d$, such that
	$$
	G_{\rho_{0,j},\sigma_{0,j}}^{(n)} t_{n,j} \leq 2e^{-nc_j(\epsilon)}, \quad \sup_{(\rho_j,\sigma_j) \in B^\complement_{\epsilon,\infty}((\rho_{0,j}, \sigma_{0,j}))} G_{\rho_{j},\sigma_j}^{(n)}(1- t_{n,j}) \leq 2e^{-nc_j(\epsilon)},
	$$
	where, for all $\rho,\sigma>0$, 
	$$
	G_{\rho,\sigma}^{(n)} t_{n,j} =\int_{(0, \infty)^n}t_{n,j}(x_1,\ldots,x_n) \prod_{i=1}^n \frac{\rho e^{-(x_i/\sigma)^{-\rho}}}{\sigma(x_i/\sigma)^{\rho+1}}\diff x_i. 
	$$
\end{lemma}
\begin{proof}
	First, assume $(\rho,\sigma)\in B^\complement_{\epsilon,\infty}((\rho_{0,j},\sigma_{0,j}))$ and, in particular, $\sigma \in(\sigma_{0,j}-\epsilon,\sigma_{0,j}+\epsilon)^\complement$. Then, it is not difficult to check that
	\begin{equation*}
	\begin{split}
	\dist_{KS}(G_{\rho,\sigma}, G_{\rho_{0,j},\sigma_{0,j}})
	%	&\geq
	%	\begin{cases}
	%	\exp\left\lbrace
	%	-\left(\frac{\sigma_{0,j}}{\sigma_{0,j}+\epsilon}\right)^{\rho_0}
	%	\right\rbrace
	%	-e^{-1}, \,\quad \sigma>\sigma_{0,j}+\epsilon \\
	%	e^{-1}-\exp\left\lbrace
	%	-\left(\frac{\sigma_{0,j}}{\sigma_{0,j}-\epsilon}\right)^{\rho_0}
	%	\right\rbrace,  \quad \sigma< \sigma_{0,j}-\epsilon 
	%	\end{cases}\\
	&\geq  \min\left\lbrace 
	\exp\left\lbrace
	-\left(\frac{\sigma_{0,j}}{\sigma_{0,j}+\epsilon}\right)^{\rho_0}
	\right\rbrace
	-e^{-1},
	e^{-1}-\exp\left\lbrace
	-\left(\frac{\sigma_{0,j}}{\sigma_{0,j}-\epsilon}\right)^{\rho_0}
	\right\rbrace
	\right\rbrace.
	\end{split}
	\end{equation*}
	Next, assume $(\rho,\sigma)\in B^\complement_{\epsilon,\infty}((\rho_{0,j},\sigma_{0,j}))$ and $\sigma \in(\sigma_{0,j}-\epsilon,\sigma_{0,j}+\epsilon)$, thus it must be that  $\rho \in(\rho_{0,j}-\epsilon,\rho_{0,j}+\epsilon)^\complement$. In this case, it holds that
	\begin{equation*}
	\begin{split}
	&\dist_{KS}(G_{\rho,\sigma}, G_{\rho_{0,j},\sigma_{0,j}})\\
	%	&\geq
	%	\begin{cases}
	%	e^{-1/2}-e^{-2^{-a/a_{0,j}}}, \quad a<a_{0,j}-\epsilon\\
	%	e^{-1/2}-e^{-2^{-a_{0,j}/a}},  \quad a>a_{0,j}+\epsilon 
	%	\end{cases}\\
	&\,\geq 
	%	 e^{-1/2}- \max\left\lbrace 
	%	e^{-(a_{0,j}-\epsilon)/a_{0,j}},
	%	e^{-a_{0,j}/(a_{0,j}+\epsilon)}
	%	\right\rbrace,
	\min \left\lbrace
	e^{-1/2}-e^{-2^{-(\rho_{0,j}-\epsilon)/\rho_{0,j}}}, e^{-1/2}-e^{-2^{-\rho_{0,j}/(\rho_{0,j}+\epsilon)}},
	e^{-2^{-(\rho_{0,j}+\epsilon)/\rho_{0,j}}}-e^{-1/2},\right.\\
	& \hspace{3.8em} \left. e^{-2^{-\rho_{0,j}/(\rho_{0,j}-\epsilon)}}-e^{-1/2}
	\right\rbrace.	\end{split}
	\end{equation*}
	Therefore, we deduce	
	$$
	\epsilon_j:=\inf_{(\rho,\sigma)\in B^\complement_{\epsilon,\infty}((\rho_{0,j},\sigma_{0,j}))}\dist_{KS}(G_{\rho,\sigma}, G_{\rho_{0,j},\sigma_{0,j}})>0.
	$$
	Define the empirical distribution function map
	$$
	\widehat{G}_{n,j}: (x_1, \ldots,x_n)\mapsto\left\lbrace \frac{1}{n}\sum_{i=1}^n \indic(x_i \leq x)\right\rbrace_{x\in (0, \infty)}, 
	$$	
	and the associated test functions
	$$
	t_{n,j}:(x_1, \ldots,x_n)\mapsto\indic \left( \dist_{KS}(\widehat{G}_{n,j}, G_{\rho_{0,j}, \sigma_{0,j}})>\epsilon_j/2 \right), \quad j=1,\ldots,d.
	$$
	For each $j=1, \ldots,d$, $(\rho_j,\sigma_j) \in B^\complement_{\epsilon,\infty}((\rho_{0,j}, \sigma_{0,j}))$, reverse triangular inequality and the first formula on page 253 of \cite{r289}
%	{\color{magenta}Wellner (1992)}
	%	 \citeN[p. 253]{wellner92} 
	yield
	\begin{equation*}
	\begin{split}
	G_{\rho_{j},\sigma_j}^{(n)}(1- t_{n,j})&=G_{\rho_{j},\sigma_j}^{(n)}\{\dist_{KS}(\widehat{G}_{n,j}, G_{\rho_{0,j},\sigma_{0,j}})\leq \epsilon_j/2\}\\
	& \leq G_{\rho_{j},\sigma_j}^{(n)}\left\lbrace\dist_{KS}(G_{\rho_{j},\sigma_j}, G_{\rho_{0,j},\sigma_{0,j}})\leq \epsilon_j/2+\dist_{KS}(\widehat{G}_{n,j}, G_{\rho_{j},\sigma_j})\right\rbrace\\
	& \leq G_{\rho_{j},\sigma_j}^{(n)}\left\lbrace \epsilon_j/2\leq \dist_{KS}(\widehat{G}_{n,j}, G_{\rho_{j},\sigma_j})\right\rbrace \\
	&=G_{\rho_{j},\sigma_j}^{(n)}\left\lbrace \sqrt{n}\epsilon_j/2 \leq \sqrt{n}\dist_{KS}(\widehat{G}_{n,j}, G_{\rho_{j},\sigma_j})\right\rbrace
	\\
	& \leq 2e^{-ne_j^2/2}.
	\end{split}
	\end{equation*}
	Similarly, we also obtain $G_{\rho_{0,j},\sigma_{0,j}}^{(n)} t_{n,j} \leq 2e^{-ne_j^2/2}$. The result now follows. 
\end{proof}

\subsubsection{Proof of Theorem \ref{th:alpha_frec}}

The proof is developed by verifying Condition \ref{cond: newcond} for the semiparametric $\brho$-Fr\'echet model under study and by applying Proposition \ref{prop:gen_cons}.
%organized in three parts, corresponding to the results at points (a)-(c) in the statement. The proof of the result at point (a) adapts arguments from the proofs of Theorems 6.17 and 6.23 in {\color{magenta}Ghosal and van der Vaart (2017)}. Some key derivations yielding consistency at an exponential rate (e.g., {\color{magenta}Choi and Ramamoorthi 2008, Definition 3.1}) are discussed more in detail, as they also 
Such derivations also play a crucial role in establishing the results of Theorem \ref{th:alpha_frec}.

Preliminary observe that by Corollary \ref{lem: KLalpha}, the Kullback-Leibler support of the prior $\Pi_{\Gset_{\bvarTheta}}$ given in Theorem \ref{th:alpha_frec} contains $g_{\brho_0,\bsigma_{0}}(\cdot|H_0)$. Hence, the first requirement of \ref{cond: newcond} is met.
Next, fixing $\epsilon>0$ and $\delta<\delta_*$, with $\delta_*$ as in  \eqref{eq:deltastar}, define 
\begin{equation}\label{eq:theset}
\begin{split}
{\mathcal{U}_\delta}&:=\{(\brho,\bsigma)\in (0,\infty)^{2d}: \, \Vert (\brho,\bsigma)
-(\brho_0,\bsigma_0)
\Vert_1<\delta\},\\
\widetilde{\mathcal{U}}_\epsilon&:= 
\{g \in \Gset_{\bvarTheta}: \dist_{H}(g,g_{\brho_0,\bsigma_0}(\cdot|H_0))\leq 4\epsilon\}
\subset
\widetilde{\mathcal{U}},
\end{split}
\end{equation}
and
\begin{equation*}
\begin{split}
\Gset_{\bvarTheta,\epsilon,\delta}^{(k)}&:= \phi_{\Hset\times\bvarTheta}\left( \Hset_k \times 
%B_{\delta,1}(\brho_0)\times B_{\delta,1}(\bsigma_0)
{\mathcal{U}_\delta}
\right)\cap \widetilde{\mathcal{U}}_\epsilon^\complement, \quad k=d+1,\ldots,\nu_n,\\
\Gset_{\bvarTheta,n}&:=\cup_{k=d+1}^{\nu_n}
\Gset_{\bvarTheta,\epsilon,\delta}^{(k)},
\end{split}
\end{equation*}
where $\nu_n$ is a sequence of integers increasing with $n$.
The latter is next selected by means of a metric entropy argument, which allows to derive a sequence of test functionals $s_n$ complying with Condition \ref{cond: newcond}\ref{dentest}. 
By Lemma \ref{aux: n2} and Proposition C.2 in 
\cite{r10},
%{\color{magenta}Ghosal and van der Vaart (2017)}, 
we have that, for $k=d+1,\ldots,\nu_n$,
\begin{equation*}
\begin{split}
\mathcal{N}\left(2\epsilon, \Gset_{\bvarTheta,\epsilon,\delta}^{(k)}, \dist_H\right)&
\leq \mathcal{N}\left(
c_1\epsilon^2,
\{
\bx \in \real^{|\Gamma_k|}: \Vert\bx \Vert_1 \leq 1
\}, L_1
\right)\\
&\quad \times  \mathcal{N}\left(
c_2\epsilon^2/k,
B_{\delta,1}(\brho_0), L_1
\right)\mathcal{N}\left(
c_3\epsilon^2/k,
B_{\delta,1}(\bsigma_0), L_1
\right)\\
&\leq 	\left(
\frac{c_4}{\epsilon^2}
\right)^{\binom{k-1}{d-1}}  \mathcal{N}\left(
c_2\epsilon^2/k,
B_{\delta_*,1}(\brho_0), L_1
\right)\mathcal{N}\left(
c_3\epsilon^2/k,
B_{\delta_*,1}(\bsigma_0), L_1
\right)\\
&\leq 
\left(
\frac{c_4}{\epsilon^2}
\right)^{\binom{k-1}{d-1}}  
\left(
\frac{c_5k}{\epsilon^2}
\right)^{2d}
\end{split}
\end{equation*}
where $c_1,c_2,c_3,c_4$ are positive constants depending only on $d,\brho_0,\bsigma_0$, $c_5$ is given by
$$
c_5=\max \left(
\sup_{\bx \in 	B_{\delta_*,1}(\brho_0)}\Vert \bx \Vert_1 3c_2^{-1},
\sup_{\bx \in 	B_{\delta_*,1}(\bsigma_0)}\Vert \bx \Vert_13c_3^{-1}
\right)
$$ 
and, without loss of generality, we assume that $\epsilon$ satisfies $c_1\epsilon^2<1$ and 
$$
c_2\epsilon^2/k < \sup_{\bx \in 	B_{\delta_*,1}(\brho_0)}\Vert \bx \Vert_1, \quad c_3\epsilon^2/k < \sup_{\bx \in 	B_{\delta_*,1}(\bsigma_0)}\Vert \bx \Vert_1.
$$
Therefore, 
\begin{equation*}
\begin{split}
\mathcal{N}\left(2\epsilon, \Gset_{\bvarTheta,n}, \dist_H\right)&\leq \sum_{k=d+1}^{\nu_n}
\left(
\frac{c_4}{\epsilon^2}
\right)^{\binom{k-1}{d-1}}
\left(
\frac{c_5k}{\epsilon^2}
\right)^{2d}\\
&\leq \nu_n
\left(
\frac{c_4}{\epsilon^2}
\right)^{\nu_n^{d-1}}\left(
\frac{c_5\nu_n}{\epsilon^2}
\right)^{2d}
\end{split}
\end{equation*}
and, for $\nu_n$ sufficiently large,
\begin{equation*}
\begin{split}
\log \mathcal{N}(2\epsilon, {\mathcal{G}}_{n,1}, \dist_H)  \leq c_6 \log \nu_n + c_7\nu_n^{d-1} \leq c_8\nu_n^{d-1},
\end{split}
\end{equation*}
for some positive $c_6$, $c_7$, $c_8$. Thus, for large enough $n$, choosing 
$$
\nu_n =\floor{(n/c_8)^{1/(d-1)}},
$$ 
it holds that 
$
\log \mathcal{N}(2\epsilon, {\mathcal{G}}_{n,1}, \dist_H)  \leq n \epsilon^2.
$
As a result, arguments in the proof of Theorem 6.23 in \cite{r10} 
%{	\color{magenta}Ghosal and van der Vaart (2017)}
guarantee the existence of a sequence of measurable test functions $s_n:(\bzero,\binf)^n\mapsto [0,1]$ satisfying 
%\begin{equation}\label{eq:expotest}
%	{\color{magenta}\prodGrhoProbalt} s_n \leq e^{n\epsilon^2}e^{-2n\epsilon^2}, \quad \sup_{G: g \in \Gset_{\bvarTheta,n}}\prodGProbaltgen(1-s_n)e^{-2\epsilon^2n},
%\end{equation}
%where, for any Borel pm $G$ on $(\bzero,\binf)$ with Lebesgue density $g$,
\begin{equation*}
\begin{split}
%	\prodGProbaltgen s_n&=
\int_{(\bzero,\binf)^n}s_n(\bx_1,\ldots,\bx_n)\prod_{i=1}^n g_{\brho_0,\bsigma_0}(\bx_i|H_0) \diff \bx_i &\leq e^{n\epsilon^2}e^{-2n\epsilon^2} \\
\sup_{G: g \in \Gset_{\bvarTheta,n}}\int_{(\bzero,\binf)^n}\{1-s_n(\bx_1,\ldots,\bx_n)\}\prod_{i=1}^n g(\bx_i) \diff \bx_i&\leq e^{-2\epsilon^2n}.
\end{split}
\end{equation*}
%Using Corollary \ref{lem: KLalpha}, we can now conclude that for any $c>0$, eventually almost surely as $n\to \infty$,
%\begin{equation}\label{eq:up3}
%	\begin{split}
%	{\color{blue}\widetilde{\Pi}_n(\Gset_{\bvarTheta,n})}&\leq s_n(\bX_{1:n}) +(1-s_n(\bX_{1:n}))\widetilde{\Pi}_n(\Gset_{\bvarTheta,n})\\
%	& {\color{blue}\leq s_n(\bX_{1:n}) + e^{cn}(1-s_n(\bX_{1:n}))\int_{\Gset_{\bvarTheta,n}}\prod_{i=1}^n\frac{g(\bX_i)}{g_{\brho_o,\bsigma_0}(\bX_i|H_0)}\diff \Pi_{\Gset_{\bvarTheta}}(g)}\\
%	&=: U_{n,1}^{(1)}(\bX_{1:n}),
%	\end{split}
%\end{equation}
%where, by Fubini's theorem and \eqref{eq:expotest}, for any $c< \epsilon^2$ the upper bound $U_{n,1}^{(1)}(\bX_{1:n})$ satisfies 
%\begin{equation}\label{eq:expo1}
%	{\color{magenta}\prodGrhoProbalt} U_{n,1}^{(1)}\leq 2 e^{-\epsilon^2n}
%\end{equation}
%and ${\color{magenta}\prodGrhoProbalt} U_{n,1}^{(1)}$ is a short notation for the expected value of $U_{n,1}^{(1)}(\bX_{1:n})$ with respect to ${\color{magenta}\prodGrhoProbalt} $. Using once more Corollary \ref{lem: KLalpha}, we also have that,  eventually almost surely as $n\to \infty$,
Furthermore, for $\Gset_{\bvarTheta,n}^\complement$ denoting the relative complement of $\Gset_{\bvarTheta,n}$ in $\phi_{\Hset\times\bvarTheta}\left( \Hset \times 
{\mathcal{U}_\delta}
\right)\cap \widetilde{\mathcal{U}}_\epsilon^\complement$ and $q$ is as in Condition \ref{cond:angularprior}\ref{cond: mvt th2}, as $n\to \infty$ we have
\begin{equation*}\label{eq:up4}
\begin{split}
\Pi_{\Gset_{\bvarTheta}}(\Gset_{\bvarTheta,n}^\complement)&\leq \Pi_{\Hset\times \bvarTheta}\left(
\{
\Hset \setminus \cup_{k = d+1}^{\nu_n} \Hset_k\} \times (0,\infty)^{2d}
\right)\\
%& {\color{blue}\leq e^{cn}\int_{\{
%	\Hset \setminus \cup_{k = d+1}^{\nu_n} \Hset_k\} \times (0,\infty)^{2d}}\prod_{i=1}^n\frac{g_{\brho,\bsigma}(\bX_i|H)}{g_{\brho_o,\bsigma_0}(\bX_i|H_0)}\diff \Pi_{\Hset\times \bvarTheta}(H,\brho,\bsigma)}\\
%&=: U_{n,1}^{(2)}(\bX_{1:n})
& \leq  e^{-q \nu_n^{d-1}}\sim e^{-nq/c_8 }.
\end{split}
\end{equation*}
%and, by Fubini's theorem and Condition \ref{cond:angularprior}, for any $c< q/c_8$ 
%\begin{equation}\label{eq:expo2}
%{\color{magenta}\prodGrhoProbalt} U_{n,1}^{(2)}\leq  e^{cn-q \nu_n^{d-1}}\sim e^{-n(q/c_8-c) }\to 0.
%\end{equation}
%Herein, $q$ is as in assumption (ii) of Theorems \ref{theo:post_consistency} and \ref{theo:post_consistency_mvt}, for the cases $d=2$ and $d>2$, respectively.
%
Finally, the existence of tests complying with Condition \ref{cond: newcond}\ref{margintest} is guaranteed by Lemma \ref{lem: tests}. Condition \ref{cond: newcond} is now verified and the final results follows from an application of Proposition \ref{prop:gen_cons}.

\subsubsection{Preliminary auxiliary results for the proof of Theorem \ref{theo: cons_weibull}}\label{sec:preli_weib}
Similarly to Section \ref{app:aux_frec}, we start by defining the functional
\begin{equation*}
\begin{split}
&\mathscr{V}_{H_*, \bomega_*, \bsigma_*, \bmu_*}^{(l)}(H, \bomega, \bsigma, \bmu; \widetilde{H}, \widetilde{\bomega}, \widetilde{\bsigma},\widetilde{\bmu})\\
&\quad :=
\left( 
\int_{(- \binf, \bmu_*)} \left[
\log^+ \left\lbrace
\frac{g_{\bomega,\bsigma,\bmu}(\bx|H)}{g_{\widetilde{\bomega},\widetilde{\bsigma},\widetilde{\bmu}}(\bx|\widetilde{H})}
\right\rbrace
\right]^lg_{\brho_*, \bsigma_* ,\bmu_*}(\bx|H_*)\diff \bx
\right)^{1/l}.
\end{split}
\end{equation*}
for $l \in \nat_+$, $H_*,H, \widetilde{H}\in \Hset$ and $\bomega_*,\bomega, \widetilde{\bomega},\bsigma_*,\bsigma, \widetilde{\bsigma}\in (\bzero, \binf)$, $\bmu_*,\bmu, \widetilde{\bmu}\in \reald$. Notice that, if for some $j \in \{1, \ldots,d\}$ it holds that $\widetilde{\mu}_j< \min(\mu_j, \mu_{*,j})$, then the above functional equals $+\infty$. More generally, the semiparametric model 
$$
\{
g_{\bomega,\bsigma,\bmu}(\by|H): H \in \Hset, (\bomega, \bsigma, \bmu)\in \bvarTheta
\},
$$
with $\bvarTheta=(1,\infty)^d\times (0,\infty)^d \times \reald$, is less regular than the other ones considered in Section \ref{sec:binf_general_max} of the main article.  Therefore, the analysis of positive Kullback-Leibler divergences is herein confined to the case of $l=1$. Control over the first-order positive divergence is sufficient to establish the strong consistency results of Theorem \ref{theo: cons_weibull}, while yielding weaker results when consistency is extended to the case of data of the form of sample maxima. See Sections \ref{sec:weibulldom} of the main paper and Section \ref{sec:cont_weib} of the present manuscript for details.
Once more, we resort to the notational convention for vectors raised to powers and rectangles in $[-\infty,\infty]^d$ introduced before Lemma \ref{aux1}. 

\begin{lemma}\label{lem:vweib}
	Let $H\in \Hset$, with exponent functions $V(\cdot|H)$. Then, for any  $\widetilde{H} \in \Hset$, $\widetilde{\bomega} \in (\bone, \binf)$, ${\bomega}  \in B_{\delta_1,\infty}(\widetilde{\bomega} )$, 
	$\widetilde{\bsigma} \in (\bzero, \binf)$,
	${\bsigma} \in B_{\delta_2,\infty}(\widetilde{\bsigma})$,
	$\widetilde{\bmu}\in \reald$, ${\bmu} \in B^+_{\delta_3,1}(\widetilde{\bmu}):=\{\bmu' > \widetilde{\bmu}: \, \Vert \bmu' - \widetilde{\bmu} \Vert_1 < \delta_3\}$ with 
	$\varepsilon\in(0,1/2)$ and
	\begin{equation}\label{eq:deltas_weib}
	\delta_1 \in \left(0, \min_{1\leq j \leq d}\{\widetilde{\omega}_j-1\} \varepsilon \wedge 1\right), \quad 	\delta_2 \in \left(0, \min_{1\leq j \leq d}\widetilde{\sigma}_j \varepsilon\right), \quad \delta_3 \in (0, \varepsilon).
	\end{equation} 
	Then, defining 
	$b_j \equiv b_j(\widetilde{\omega}_j, \widetilde{\sigma}_j)$,  $j=1,\ldots,d$, as in \eqref{eq:c_j} and
	$$
	c(\widetilde{\bomega}, \widetilde{\bsigma})
	:= 	\frac{2\max_{1 \leq j \leq d} \widetilde{\omega}_j}{
		\min_{1 \leq j \leq d}	b_j
	}
	\max_{1\leq j \leq d}
	\left[
	1+ 2^{\widetilde{\omega}_{j}}
	+2^{\widetilde{\omega}_{j}}
	\frac{\Gamma(2\widetilde{\omega}_j+1)}{2\widetilde{\sigma}_j^{\widetilde{\omega}_j(1-2\widetilde{\omega}_j)}}
	\left\lbrace
	1-\frac{\gamma(2\widetilde{\omega}_j, \widetilde{\sigma}_j^{-\widetilde{\omega}_j})}{\Gamma(2\widetilde{\omega}_j)}
	\right\rbrace
	\right], 
	$$
	the term
	$$
	\int_{(-\binf, \widetilde{\bmu})}
	\left[ V\left(
	\left\lbrace
	\frac{\bmu -\bx}{\bsigma}
	\right\rbrace^{-\bomega}
	\bigg{|}H
	\right) 
	-V\left(
	\left\lbrace
	\frac{\widetilde{\bmu} -\bx}{\bsigma}
	\right\rbrace^{-\bomega}
	\bigg{|}H
	\right) 
	\right]_+g_{\widetilde{\bomega}, \widetilde{\bsigma}, \widetilde{\bmu}}(\bx|\widetilde{H})\diff \bx
	$$
	is bounded from above by
	$
	c(\widetilde{\bomega}, \widetilde{\bsigma}) \Vert \bmu  - \widetilde{\bmu} \Vert_1
	$.
\end{lemma}
\begin{proof}
	By rguments similar to those in the proof of Lemma \ref{lem: aux} we obtain
	$$
	\left[ V\left(
	\left\lbrace
	\frac{\bmu -\bx}{\bsigma}
	\right\rbrace^{-\bomega}
	\bigg{|}H
	\right) 
	-V\left(
	\left\lbrace
	\frac{\widetilde{\bmu} -\bx}{\bsigma}
	\right\rbrace^{-\bomega}
	\bigg{|}H
	\right) 
	\right]_+ \leq \bigg{\Vert} 
	\left\lbrace
	\frac{\bmu -\bx}{\bsigma}
	\right\rbrace^{\bomega}
	-
	\left\lbrace
	\frac{\widetilde{\bmu} -\bx}{\bsigma}
	\right\rbrace^{\bomega}
	\bigg{\Vert}_1,
	$$
	where, by the multivariate mean-value theorem and the bounds in \eqref{eq:deltas_weib}, the term on the right-hand side is bounded from above by 
	$$
	\sum_{j=1}^d \frac{\omega_j}{\sigma_j^{\omega_j}}(\mu_j - \widetilde{\mu}_j) |\mu_j - x_j|^{\omega_j-1} 
	\leq 
	\frac{2\max_{1 \leq j \leq d} \widetilde{\omega}_j}{
		\min_{1 \leq j \leq d}	b_j
	}
	\sum_{j=1}^d (\mu_j - \widetilde{\mu}_j) |\mu_j - x_j|^{\omega_j-1},
	$$
	where 
	\begin{equation}\label{eq:c_j}
	b_j=\inf_{\sigma_j\in(\widetilde{\sigma}_j/2 , 3\widetilde{\sigma}_j/2), 
		\omega_j\in (\widetilde{\omega}_j/2 , 3\widetilde{\omega}_j/2 )}
	\sigma_j^{\omega_j}.
	\end{equation}
	Moreover, for $j=1, \ldots,d$, it holds that
	\begin{equation*}
	\begin{split}
	&\int_{(-\binf, \widetilde{\bmu})} |\mu_j - x_j|^{\omega_j-1}
	g_{\widetilde{\bomega}, \widetilde{\bsigma}, \widetilde{\bmu}}(\bx|\widetilde{H})\diff \bx \\
	&\quad \leq 
	1 + \int_{-\infty}^{\mu_j-\widetilde{\mu}_j-1} |\mu_j-\widetilde{\mu}_j-x_j|^{\omega_j-1} g_{\widetilde{\omega}_j, \widetilde{\sigma}_j}(x_j)\diff x_j
	\\
	&\quad  \leq 1+ 2^{\omega_j-1}
	\int_{-\infty}^{\mu_j-\widetilde{\mu}_j-1} (-x_j)^{\omega_j-1} g_{\widetilde{\omega}_j, \widetilde{\sigma}_j}(x_j)\diff x_j
	\\
	&\quad  \leq 1+ 2^{\widetilde{\omega}_{j}}
	+2^{\widetilde{\omega}_{j}}
	\int_{-\infty}^{-1} (-x_j)^{\widetilde{\omega}_j} g_{\widetilde{\omega}_j, \widetilde{\sigma}_j}(x_j)\diff x_j
	\end{split}
	\end{equation*}
	where, for $x_j<0$,
	\begin{equation}\label{eq:sing_weib}
	g_{\widetilde{\omega}_j,\widetilde{\sigma}_j}(x_j):=\exp(-(-x_j/\widetilde{\sigma}_j)^{\widetilde{\omega}_j})(-x_j/\widetilde{\sigma}_j)^{\widetilde{\omega}_j-1}(\widetilde{\omega}_j/\widetilde{\sigma}_j)
	\end{equation}
	and
	\begin{equation*}
	\begin{split}
	\int_{-\infty}^{-1} (-x_j)^{\widetilde{\omega}_j} g_{\widetilde{\omega}_j, \widetilde{\sigma}_j}(x_j)\diff x_j
	&\leq \widetilde{\omega}_j\int_1^\infty \widetilde{\sigma}_j^{-\widetilde{\omega}_j} x^{2\widetilde{\omega}_j-1}\exp\left(-x_j \widetilde{\sigma}_j^{-\widetilde{\omega}_j}\right)\diff x_j\\
	&=\frac{\Gamma(2\widetilde{\omega}_j+1)}{2\widetilde{\sigma}_j^{\widetilde{\omega}_j(1-2\widetilde{\omega}_j)}}
	\left\lbrace
	1-\frac{\gamma(2\widetilde{\omega}_j, \widetilde{\sigma}_j^{-\widetilde{\omega}_j})}{\Gamma(2\widetilde{\omega}_j)}
	\right\rbrace
	\end{split}
	\end{equation*}
	with $\Gamma(\cdot)$ and $\gamma(\cdot,\cdot)$ denoting the Gamma and the lower incomplete Gamma function. The result now follows.
\end{proof}

\begin{lemma}\label{lem:divweib}
	Let $H_0\in\Hset$ satisfy property \ref{cond: finite_new} in Definition \ref{cond:mvt_angular}. Let $(\bomega_0, \bsigma_0,\bmu_0) \in \bvarTheta$. Then, for all $\epsilon>0$ there exists $H^*\in \Hset'$ and $\delta_1,\delta_2,\delta_3,\delta_4>0$ such that
	\begin{equation}\label{eq:weibkulb1}
			\kulb_+^{(1)}(g_{\bomega_0, \bsigma_0, \bmu_0}(\cdot|H_0), g_{\bomega, \bsigma,\bmu}(\cdot|H)) < \epsilon
	\end{equation}
	for all $\bomega\in B_{\delta_1,1}(\bomega_0)$, $\bsigma\in B_{\delta_2,1}(\bsigma_0)$, $\bmu \in B^+_{\delta_3,1}(\bmu_0)$ and $H \in B_{\delta_4, \infty}(H^*)$.
	Thus, in particular, $\kulb(g_{\bomega_0, \bsigma_0,\bmu_0}(\cdot|H_0),g_{\bomega,\bsigma,\bmu}(\cdot|H))<\epsilon$.
\end{lemma}
\begin{proof}
	Observe that 
	\begin{equation}\label{eq:weibbound}
	\begin{split}
	\kulb_+^{(1)}(g_{\bomega_0, \bsigma_0, \bmu_0}(\cdot|H_0), g_{\bomega, \bsigma,\bmu}(\cdot|H))&=
	\mathscr{V}_{H_0, \bomega_0, \bsigma_0, \bmu_0}^{(1)}(H_0, \bomega_0, \bsigma_0, \bmu_0; H, \bomega, \bsigma,\bmu) \\
	& \leq 	\mathscr{V}_{H_0, \bomega_0, \bsigma_0, \bmu_0}^{(1)}(H_0, \bomega_0, \bsigma_0, \bmu_0; H,  \bomega_0, \bsigma_0, \bmu_0)\\
	& \quad + \mathscr{V}_{H_0, \bomega_0, \bsigma_0, \bmu_0}^{(1)}(H, \bomega_0, \bsigma_0, \bmu_0; H, \bomega, \bsigma,\bmu_0)\\
	& \quad +
	\mathscr{V}_{H_0, \bomega_0, \bsigma_0, \bmu_0}^{(1)}(H, \bomega, \bsigma, \bmu_0; H, \bomega, \bsigma,\bmu)\\
	&=: S_1+S_2+S_3.
	\end{split}
	\end{equation}
	A change of variables yields
	$
	S_1=\kulb_+^{(1)}(g_\bone(\cdot|H_0),g_\bone(\cdot|H)),
	$
	where the term on the right-hand side is defined as in Section \ref{sec:signed_kulb_simple}. Thus, by Proposition \ref{cor: New_KL_cor}, $S_1$ can be made striclty smaller than $\epsilon/3$ by choosing $H \in B_{\delta_4, \infty}(H^*)$, for a suitable $H^* \in \Hset'$ and a sufficiently small $\delta_4$. Similarly, we have
	$$
	S_2=\mathscr{V}_{H_0, \brho_0, \bsigma_0}^{(1)}(H, \brho_0, \bsigma_0^{-1}; H, \brho, \bsigma^{-1}),
	$$
	where the term on the right-hand side is defined as in Section \ref{app:aux_frec}, with $\brho_0=\bomega_0$ and $\brho=\bomega$. Thus, $S_2$ can can be made striclty smaller than $\epsilon/3$ by reasoning as in the proof of Lemma \ref{cor: alpha_kulb}. We are now left with the analysis of $S_3$.

	Using the inequality $\log^+(xyz) \leq \log^+(x)+\log^+(y)+\log^+(z)$, $x,y,z>0$, we obtain
	\begin{equation*}
	\begin{split}
	S_3 &\leq \int_{(-\binf,\bmu_0)} \left[ V\left(
	\left\lbrace
	\frac{\bmu -\bx}{\bsigma}
	\right\rbrace^{-\bomega}
	\bigg{|}H
	\right) 
	-V\left(
	\left\lbrace
	\frac{\bmu_0 -\bx}{\bsigma}
	\right\rbrace^{-\bomega}
	\bigg{|}H
	\right) 
	\right]_+ g_{\bomega_0,\bsigma_0,\bmu_0}(\bx|H_0) \diff \bx
	\\
	& \quad + \sum_{j=1}^d \int_{(-\binf,\bmu_0)} \log^+\left\lbrace
	\frac{\mu_j-x_j}{\mu_{0,j}-x_j} 
	\right\rbrace
	g_{\bomega_0,\bsigma_0,\bmu_0}(\bx|H_0) \diff \bx\\
	&\quad +
	\int_{(-\binf,\bmu_0)}
	\log^+
	\left\lbrace
	\frac{\sum_{\part \in \allpart_d}
		\prod_{i=1}^{m}\left[-V_{I_i}\left(
		\left\lbrace
		\frac{\bmu_0 -\bx}{\bsigma}
		\right\rbrace^{-\bomega}
		\bigg{|}H
		\right) \right]}{\sum_{\part \in \allpart_d}
		\prod_{i=1}^{m}\left[-V_{I_i}\left(
		\left\lbrace
		\frac{\bmu -\bx}{\bsigma}
		\right\rbrace^{-\bomega}
		\bigg{|}H
		\right) \right]}
	\right\rbrace
	g_{\bomega_0,\bsigma_0,\bmu_0}(\bx|H_0) \diff \bx\\
	&=: T_1+ T_2 + T_3.
	\end{split}
	\end{equation*}
	By Lemma \ref{lem:vweib}, 
%	(applied with $H_0,\bomega_0,\bsigma_0, \bmu_0$ in place of $\widetilde{H},\widetilde{\bomega},\widetilde{\bsigma}, \widetilde{\bmu}$), 
	the term $T_1$ can be made arbitrarily small by choosing $\delta_1, \delta_2, \delta_3$ small enough. Furthermore, we have that
	\begin{equation}\label{eq:T2}
	\begin{split}
	T_2&\leq \sum_{j=1}^d\int_{(-\binf,\bmu_0)} 
	\frac{\mu_j-\mu_{0,j}}{\mu_{0,j}-x_j} 
	g_{\bomega_0,\bsigma_0,\bmu_0}(\bx|H_0) \diff \bx\\
	& = \sum_{j=1}^d \frac{\mu_j-\mu_{0,j}}{\sigma_{0,j}}\Gamma(1-1/\omega_{0,j})
	\end{split}
	\end{equation}
	where the term on the right hand side can be made arbitrarily small by choosing $\delta_3$ small enough. Reasoning as in the proof of Lemma \ref{cor: alpha_kulb} and using Lemma \ref{lem:weibkulb}, we also conclude that
	$T_3$ can be made small by choosing $\delta_1, \delta_2, \delta_3, \delta_4$ small enough. Hence, there exists suitable choices of the radii yielding $S_3 < \epsilon/3$. The proof is now complete.  
\end{proof}

In the bivariate case ($d=2$), we have the following analogous result for angular pm's $H_0$ with unbounded angular density.
\begin{lemma}\label{lem:divweibunb}
	Let $H_0\in\Hset$ satisfy property \ref{cond: infinite_new} in Definition \ref{cond:mvt_angular}. Let $(\bomega_0, \bsigma_0,\bmu_0) \in \bvarTheta$. Then, for all $\epsilon>0$ there exists $H^*\in \Hset'$ and $\delta_1,\delta_2,\delta_3,\delta_4>0$ such that
	$$
	\kulb(g_{\bomega_0, \bsigma_0, \bmu_0}(\cdot|H_0), g_{\bomega, \bsigma,\bmu}(\cdot|H)) < \epsilon
	$$
	for all $\bomega\in B_{\delta_1,1}(\bomega_0)$, $\bsigma\in B_{\delta_2,1}(\bsigma_0)$, $\bmu \in B^+_{\delta_3,1}(\bmu_0)$ and $H \in B_{\delta_4, \infty}(H^*)$.
	Further assuming that, for some $s>0$,
	$$
	\int_0^1 (h_0(t))^{1+s}\diff t<\infty,
	$$
	then $H^*$ and $\delta_j$, $j=1, \ldots,4$, can be chosen in such a way that \eqref{eq:weibkulb1} is also satisfied.
\end{lemma}
\begin{proof}
	The proofs of the first and second results obtain by exploting the inequalities
	\begin{equation*}
	\begin{split}
	\kulb(g_{\bomega_0, \bsigma_0, \bmu_0}(\cdot|H_0), g_{\bomega, \bsigma,\bmu}(\cdot|H))
	& \leq 	\kulb(g_{\bomega_0, \bsigma_0, \bmu_0}(\cdot|H_0),  g_{\bomega_0, \bsigma_0, \bmu_0}(\cdot|H))\\
	& \quad + \mathscr{V}_{H_0, \bomega_0, \bsigma_0, \bmu_0}^{(1)}(H, \bomega_0, \bsigma_0, \bmu_0; H, \bomega, \bsigma,\bmu_0)\\
	& \quad +
	\mathscr{V}_{H_0, \bomega_0, \bsigma_0, \bmu_0}^{(1)}(H, \bomega, \bsigma, \bmu_0; H, \bomega, \bsigma,\bmu)
	\end{split}
	\end{equation*}
	and \eqref{eq:weibbound}, respectively, and proceeding as in the proof of Lemma \ref{lem:divweib}, but now appealing to Proposition \ref{prop:newprop_KL} in place of Proposition \ref{cor: New_KL_cor}.
\end{proof}
For given $H_0\in \Hset_0$, $\omega_0\in (\bone,\binf)$, $\bsigma_0\in (\bzero,\binf)$, $\bmu_0 \in \reald$ and for all $\epsilon>0$, define 
\begin{eqnarray}
\mathcal{V}_\epsilon&:=& 
\{(H, \bomega,\bsigma, \bmu) \in \Hset \times\bvarTheta: \,
	\kulb(g_{\bomega_0, \bsigma_0, \bmu_0}(\cdot|H_0), g_{\bomega, \bsigma,\bmu}(\cdot|H)) < \epsilon
\},
\\
\label{eq:v1setweib}
\mathcal{V}_\epsilon^{(1)}&:=& 
\{(H, \bomega,\bsigma, \bmu) \in \Hset \times\bvarTheta: \,
	\kulb_+^{(1)}(g_{\bomega_0, \bsigma_0, \bmu_0}(\cdot|H_0), g_{\bomega, \bsigma,\bmu}(\cdot|H)) < \epsilon
\}.
\end{eqnarray}
%We recall that herein $\bvarTheta=(1,\infty)^d\times(0,\infty)^d\times \reald$. 
The following result is an immediate consequence of Lemma \ref{lem:divweib}. Its proof is analogous to that of Corollary \ref{lem: KLalpha} and is therefore omitted.

\begin{cor}\label{lem: omega}
	Under Conditions \ref{cond:genprior}\ref{cond:indep}--\ref{cond:angularpmprior} and \ref{cond:genprior}\ref{cond:compactprior}, for all $\epsilon>0$ we have that
$$
\Pi_{\Hset \times \bvarTheta}\left( \mathcal{V}_\epsilon \right)>0.
$$   
Thus, the induced prior $\Pi_{\Gset_{\bvarTheta}}$ on the Borel sets of
$
\Gset_{\bvarTheta}=\{g_{\bomega,\bsigma,\bmu}(\cdot|H): \, H \in\Hset, (\bomega,\bsigma,\bmu)\in \bvarTheta\},
$ where
$\bvarTheta=(1,\infty)^d\times(0,\infty)^d\times \reald$,
possesses the Kullback-Leibler property. Further assuming that Condition  \ref{cond:frecextend}\ref{cond:strongertruedens} is satisfied, we also have 
$$
\Pi_{\Hset \times \bvarTheta}\left( \mathcal{V}_\epsilon^{(1)}\right)>0.
$$
\end{cor}
We now state the last techical result of this subsection, concerning tests on the marginal shape, scale and location parameters. 
We recall that, for all $\omega,\sigma>0$, $\mu \in \mathbb{R}$, the cdf of a univariate location-scale (reverse) $\omega$-Weibul distribution is 
$$
G_{\omega,\sigma, \mu}(x)= \exp	\left\lbrace-\left((\mu-x)/\sigma\right)^\omega \right\rbrace, \quad x < \mu.
$$
In what follows, as a convention, real valued terms of the form of suprema over an empty set are interpreted as equal to zero.

\begin{lemma}\label{lem: tests_weib}
	For $j =1, \ldots,d$, let $K_j \subset (1,\infty)\times(0,\infty)\times \real$ be a compact neighbourhood of $(\omega_{0,j},\sigma_{0,j}, \mu_{0,j})$.
	Then, for each  $\epsilon>0$ and $j=1, \ldots,d$, there exist Borel-measurable functions $t_{n,j}:  (0, \infty)^n\mapsto\{0,1\}$, $n \in \mathbb{N}_+$, and a positive constant $c_j(\epsilon)$  such that
	$$
	G_{\omega_{0,j},\sigma_{0,j}, \mu_{0,j}}^{(n)} t_{n,j} \leq 2e^{-nc_j(\epsilon)}, \quad \sup_{(\rho_j,\sigma_j, \mu_j) \in \Theta_j^{\emph{\text{(A)}}} } G_{\rho_{j},\sigma_j, \mu_j}^{(n)}(1- t_{n,j}) \leq 2e^{-nc_j(\epsilon)},
	$$
	where $\Theta_j^{\emph{\text{(A)}}} =\{(\omega,\sigma,\mu) \in K_j:
	\Vert (\omega,\sigma,\mu) - (\omega_{0,j},\sigma_{0,j},\mu_{0,j}) \Vert_\infty > \epsilon
	\}$ and, for all $\omega,\sigma>0$, $\mu \in \mathbb{R}$, 
	\begin{equation*}
	\begin{split}
	G_{\omega,\sigma, \mu}^{(n)} t_{n,j} =\int_{(-\infty, \mu)^n}t_{n,j}(x_1,\ldots,x_n) \prod_{i=1}^n 
	\frac{\omega e^{-((\mu-x_i)/\sigma)^{\omega	}}}{\sigma((\mu-x_i)/\sigma)^{1+\omega}}\diff x_i. 
	\end{split}
	\end{equation*}
\end{lemma}

\begin{proof}
	Wtihout loss of generality, we assume $K_j \cap \Theta_j^{\text{(A)}}\neq \emptyset$, for $j=1, \ldots,d$.
	We start by verifying the conditions of Lemma 10.6 in \cite{r999}, 
%	{\color{magenta}van der Vaart (2000)}, 
	which establishes the existence of a sequence of estimators $\bT_{n,j}$ that are uniformly consistent on $K_j$, for $j=1, \ldots,d$. 
	By assumption, 
%	$\bvarTheta$ is compact, thus so is $\bvarTheta_j$, 
	the latter are compact sets. Consider the subfamilies of univariate three-parameter Weibull distributions
	$$
	\{
	G_{\omega,\sigma, \mu }: \, (\omega,\sigma,\mu) \in K_j
	\}, \quad j	=1, \ldots, d.
	$$
	Clearly, each of such class is identifiable, i.e. $G_{\omega, \sigma,\mu}\neq G_{\widetilde{\omega}, \widetilde{\sigma},\widetilde{\mu}}$ if $(\omega, \sigma,\mu),(\widetilde{\omega}, \widetilde{\sigma},\widetilde{\mu})\in K_j$ and $(\omega, \sigma,\mu) \neq(\widetilde{\omega}, \widetilde{\sigma},\widetilde{\mu})$. Furthermore, for $j=1, \ldots,d$, the map 
	\begin{equation}\label{eq:themap}
	(\omega,\sigma,\mu)\mapsto G_{\omega, \sigma,\mu}
	\end{equation}
	on $K_j$ is continuous under the total variation norm. To see this, let $(\omega_n, \sigma_n,\mu_n)$ be a sequence in $K_j$ converging to $(\omega,\sigma,\mu)$ and denote by $\overline{G}_{\omega_n, \sigma_n,\mu_n}$ and $\overline{G}_{\omega, \sigma,\mu}$ the renormalised restrictions of $G_{\omega_n, \sigma_n,\mu_n}$ and $G_{\omega, \sigma,\mu}$ on $(-\infty,\mu-\varepsilon)$, for some $\varepsilon>0$. 
	There exists $n_\varepsilon\in \nat_+$ such that, for all $n \geq n_\varepsilon$, $\mu_n \in (\mu\pm \varepsilon)$  and
	%	 for some $n_\varepsilon\in \nat_+$, $\mu_n \in (\mu\pm \varepsilon)$ for all $n \geq n_\varepsilon$
	\begin{equation*}
		\begin{split}
			\dist_T(G_{\omega_n, \sigma_n,\mu_n},G_{\omega, \sigma,\mu})
		&\leq \dist_T(\overline{G}_{\omega_n, \sigma_n,\mu_n},G_{\omega, \sigma,\mu})+ 
		2 G_{\omega, \sigma,\mu}((-\infty,\mu-\varepsilon)^\complement)\\
		&\quad +
		2 G_{\omega_n, \sigma_n,\mu_n}((-\infty,\mu-\varepsilon)^\complement).
		\end{split}
	\end{equation*}
	The first term on the right-hand side converges to zero by Scheff\'e's lemma. The second term equals $1- \exp(-(\varepsilon/\sigma)^\omega)$ and can be made arbitrarily small by choosing $\varepsilon$ small enough. As for the third term, we have
	\begin{equation*}
	\begin{split}
	\limsup_{n \to \infty}G_{\omega_n, \sigma_n,\mu_n}((-\infty,\mu-\varepsilon)^\complement) & \leq \limsup_{n \to \infty} G_{\omega_n, \sigma_n,\mu_n}((\mu - \varepsilon, \mu_n)) \\
	& =\limsup_{n \to \infty} \left[
	1- \exp\left\lbrace
	-((\mu_n-\mu+\varepsilon)/\sigma_n)^{\omega_n}
	\right\rbrace
	\right]
	\\
	& \leq  1-\exp(-(2\varepsilon/\sigma)^\omega),
	\end{split}
	\end{equation*}
	where the term on the third line can also be made arbitrarily small by choosing $\varepsilon$ small enough. It follows that the map  in \eqref{eq:themap} is sequentially continuous and, therefore, continuous.  

	At this stage, Lemma 10.6 in \cite{r999} 
%	{\color{magenta}van der Vaart 2000} 
	can be applied and, letting $\bX_{1:n}=(\bX_1, \ldots,\bX_n)$ be a sample of iid draws from $G_{\bomega_0, \bsigma_0,\bmu_0}$, the test functions 
	$$
	\tilde{t}_{n,j}(X_{j,1:n}):=\indic(\Vert \bT_{n,j}-(\omega_{0,j}, \sigma_{0,j}, \mu_{0,j})\Vert_\infty \geq \epsilon/2), \quad j=1, \ldots,d,
	$$
	satisfy 
	$$
	G_{\omega_{0,j}, \sigma_{0,j}, \mu_{0,j}}^{(n)}\tilde{t}_{n,j} \to 0, \quad  \sup_{ (\rho_j,\sigma_j, \mu_j) \in \bvarTheta_j^{\text{(A)}} } G_{\rho_{j},\sigma_j, \mu_j}^{(n)}(1- \tilde{t}_{n,j})\to 0,
	$$
	as $n \to 0$, see also \citep[ p. 144, lines 24--29]{r999}.
%	{\color{magenta}van der Vaart (2000, p. 144, lines 24--29)}. 
	The final result now obtains using the construction in the proof of Lemma 10.3 in \cite{r999},
%	{\color{magenta}van der Vaart (2000)},
	 see in particular the last three lines on page 143 and the first nine lines on page 144.  
\end{proof}

\subsubsection{Proof of Theorem \ref{theo: cons_weibull}}

The proof proceeds analogously to that of Theorem \ref{th:alpha_frec},  now using  Corollary \ref{lem: omega} in place of Corollary \ref{th:alpha_frec}, Lemma \ref{lem:entropyweib} in place of \ref{aux: n2} and Lemma \ref{lem: tests_weib} in place of Lemma \ref{lem: tests}. Full derivations are therefore omitted.

%, it can be shown that eventually almost surely as $n 	\to \infty$
%$$
%\widetilde{\Pi}_n(\widetilde{\mathcal{U}}^\complement) \leq S_{1,n}(\bX_{1:n}), \quad
%\Pi_n((\mathcal{U}_1\times\mathcal{U}_2)^\complement) \leq S_{2,n}(\bX_{1:n}),
%$$
%where $S_{n,j}$, $j=1,2$ are measurable functions satisfying for any $\varepsilon>0$, as $n$ grows larger, 
%$$
%\prodweibtrue(S_{j,n}(\bX_{1:n})>\varepsilon|H_0)\lesssim e^{-nc_j}, \quad j=1,2,
%$$
%with $c_1,c_2>0$ positive constants. This yields the results at points (a) and (c). Once again, the result at point (b) follows from the one at point (a) and Theorem 6.8 in {\color{magenta}Ghosal and van der Vaart (2017)}.

%%%%%%%%%%%%%%%%%%%%%%%%%%%%%%%%%%%%%%%%%%%%%%%%%%%%%%%%%%%%%%%%%%%

\subsubsection{Proof of Theorem \ref{cor:Gumbel_cons}}

%%%%%%%%%%%%%%%%%%%%%%%%%%%%%%%%%%%%%%%%%%%%%%%%%%%%%%%%%%%%%%%%%%%
We have the following facts:
\\

\noindent
\textbf{Fact 1.} For all $l \in \nat_+$, $H\in \Hset$, $\bsigma> \bzero$ and $\bmu \in \reald$, the $l$-th order positive Kullback-Leibler divergence from $g_{ \bsigma, \bmu}(\cdot|H)$ to the true data generating density $g_{ \bsigma_0, \bmu_0}(\cdot|H_0)$, i.e.
\begin{equation}\label{eq:div_gumb}
\int_{\reald} 
\left[
\log^+
\left\lbrace
\frac{
	g_{\bsigma_0, \bmu_0}(\bx|H_0)
}{
	g_{\bsigma, \bmu}(\bx|H)
}
\right\rbrace
\right]^lg_{\bsigma_0, \bmu_0}(\bx|H_0) \diff \bx,
\end{equation}
equals
$$
\mathscr{V}_{H_0, \bone, \bsigma_0^*}^{(l)}(H_0, \brho^*_0, \bsigma_0^*; H, \brho^*, \bsigma^*),
$$
which is defined as in Section \ref{app:aux_frec}, with
\begin{equation}\label{eq:newpar}
\begin{split}
\brho_0^*&=\left(
\sigma_{0,1}^{-1}, \ldots, \sigma_{0,d}^{-1}
\right), \quad
\bsigma_{0}^*=\left(e^{\mu_{0,1}},
\ldots, e^{\mu_{0,d}}
\right),\\
\brho^*&=(\sigma_1^{-1}, \ldots,\sigma_d^{-1}),\hspace{1.5em}
\bsigma^*=\left(
e^{\mu_1}, \ldots,e^{\mu_d}
\right).
\end{split}
\end{equation}
Analogously,
$$
\kulb(g_{ \bsigma_{0},\bmu_0}(\cdot|H_0),g_{ \bsigma,\bmu}(\cdot|H))=\kulb(g_{\brho^*_0, \bsigma_{0}^*}(\cdot|H_0),g_{\brho^*, \bsigma^*}(\cdot|H)).
$$
Clearly, for any $\delta>0$ there exists $\varepsilon>0$ such that, 
\begin{equation}\label{eq:implies}
\bsigma \in B_{\varepsilon,1}(\bsigma_0)
\text{ and } 
\bmu \in B_{\varepsilon,1}(\bmu_0)
\implies \brho^* \in B_{\delta,1}(\brho_0^*) \text{ and } \bsigma^* \in B_{\delta,1}(\bsigma^*_0).
\end{equation}
Therefore, results analogous to Lemmas  \ref{cor: alpha_kulb}--\ref{lem:newkulblem2} and Corollary \ref{lem: KLalpha} hold true also in the present setting. In particular, for any $\epsilon>0$, the prior $\Pi_{\Hset \times \bvarTheta}$ resulting from Conditions \ref{cond:genprior}\ref{cond:indep}--\ref{cond:angularpmprior} and \ref{cond:genprior}\ref{cond:scaleloc} possesses the Kullback-Leibler property. 
%{\color{magenta}and puts positive mass on a set of $(H,\bsigma,\bmu)$ such that the divergences in \eqref{eq:div_gumb}
%are smaller than $\epsilon$, for $l=1, \ldots,4$.}
\\

\noindent
\textbf{Fact 2}. For $j=1, \ldots,d$, $\sigma_j >0$ and $\mu_j \in \real$, recalling that $G_{ \mu_j, \sigma_j}(x)=\exp(-\exp((\mu_j-x)/\sigma_j))$, $x\in \real$, we have that 
\begin{equation*}
\begin{split}
\dist_{KS}(G_{\sigma_{0,j},\mu_{0,j}},
G_{\sigma, \mu} 
)
= 
\dist_{KS}\left(
G_{\rho^*_{0,j}, \sigma^*_{0,j}},
G_{\rho^*_{j}, \sigma^*_{j}}
\right)
\end{split}
\end{equation*}
where the parameters $\rho^*_{0,j}, \sigma^*_{0,j}, \rho^*_{j}, \sigma^*_{j}$ are defined as in \eqref{eq:newpar} and, for $x>0$,
\begin{equation*}
\begin{split}
G_{\rho^*_{0,j}, \sigma^*_{0,j}}(x)&=\exp\left\lbrace
-\left(
x/\sigma_{0,j}^*
\right)^{-\rho_{0,j}^*}
\right\rbrace,\\
G_{\rho^*_{j}, \sigma^*_{j}}(x)&=\exp\left\lbrace
-\left(
x/\sigma_{j}^*
\right)^{-\rho_{j}^*}
\right\rbrace.
\end{split}
\end{equation*}
Furthermore, if $(\sigma_j,\mu_j)\in B_{\delta,\infty}^\complement((\sigma_{0,j}, \mu_{0,j}))$, for some $\delta>0$, then there exists $\epsilon>0$ such that
$
(\rho_j^*,\sigma_j^*)\in B_{\epsilon,\infty}^\complement((\rho_{0,j}^*,\sigma_{0,j}^*))$.
Consequently, arguments analogous to those in the proof of Lemma \ref{lem: tests} apply also in the present setting and a similar result obtains.
\\

\noindent
\textbf{Fact 3}.  For all $k \geq d+1$, $H_k, \widetilde{H}_k \in \Hset_k$, $\bsigma, \widetilde{\bsigma} \in B_{\varepsilon,1}(\bsigma_{0})$ and $\bmu, \widetilde{\bmu} \in B_{\varepsilon,1}({\bmu}_0)$, with $\varepsilon>0$, a change of variables yields
\begin{equation}\label{eq:hellingumbel}
\dist_H \left( g_{\bsigma, \bmu}(\cdot|H_k), 
g_{ \widetilde{\bsigma}, \widetilde{\bmu}}(\cdot|\widetilde{H}_k)
\right)=\dist_H \left( g_{\brho^*, \bsigma^*}(\cdot|H_k), 
g_{\widetilde{\brho}^*, \widetilde{\bsigma}^*}(\cdot|\widetilde{H}_k)
\right),
\end{equation}
where $\brho^*, \bsigma^*$ are given in \eqref{eq:newpar} and $\widetilde{\brho}^*, \widetilde{\bsigma}^*$ are defined in an analogous fashion. Let $\brho_0^*, \bsigma_0^*$ be as in \eqref{eq:newpar} and choose $\varepsilon>0$ such that the implication in \eqref{eq:implies} is satisfied for a positive $\delta>0$ complying with 	\eqref{eq:deltastar}, where $\rho_{0,j}$ and $\sigma_{0,j}$ are now replaced by $\rho_{0,j}^*$ and $\sigma_{0,j}^*$, respectively, for $j=1, \ldots,d$. Hence, assuming without loss of generality that
$$
\varepsilon <\frac{1}{2}\min
\left\lbrace
\min_{1\leq j\leq d} \sigma_{0,j}, \min_{1\leq j\leq d} \mu_{0,j}
\right\rbrace,
$$
by Lemma \ref{aux: n2}, the term on the right-hand side of \eqref{eq:hellingumbel} is bounded from above by
\begin{equation*}
\begin{split}
&  \sqrt{ c_0 \Vert \bphi_\circ - \widetilde{\bphi}_\circ\Vert_1} + \sqrt{c_0k \Vert \brho^*-\widetilde{\brho}^* \Vert_1 + c_0k \Vert \bsigma^*-\widetilde{\bsigma}^* \Vert_1}\\
& \leq \sqrt{ c_0 \Vert \bphi_\circ - \widetilde{\bphi}_\circ\Vert_1}+
\sqrt{c_0^*k \Vert \bsigma-\widetilde{\bsigma} \Vert_1 + c_0^*k \Vert \bmu-\widetilde{\bmu} \Vert_1},
\end{split}
\end{equation*}
where $c_0$ is a positive constant depending on $d$, $\brho_0^*$ and $\bsigma_0^*$, 
$$
c_0^*:= c_0 \max\left\lbrace \frac{4}{\min_{1\leq j} \sigma_{0,j}^2}, \max_{1\leq j\leq d} e^{2\mu_{0,j}}
\right\rbrace,
$$
while $\Vert \bphi_\circ - \widetilde{\bphi}_\circ\Vert_1$ is as in Lemma \ref{lem: L1}.
\\

Combining Facts 1-3, the results in the statement can be established by mirroring the proof of Theorem \ref{th:alpha_frec}.

%\section{Proofs of the results in Section \ref{sec:binf_sample_max}}\label{appendix: proof_rc} 

\subsection{Proofs of the results in Section \ref{sec:repar}}

\subsubsection{Proof of Proposition \ref{prop:RC}}\label{sec:RCproof}
Preliminary observe that, under Condition \ref{cond:strong},
$g_{\bvartheta_0}(\cdot|H_0)$ is positive and continuous on $(\bzero,\binf)$, if $\bvartheta_0=(\brho_0, \bone)$, or $(-\binf, \bzero)$, if $\bvartheta_0=(\bomega_0, \bone, \bzero)$, or $\reald$, if $\bvartheta_0=( \bone, \bzero)$, respectively. Moreover, for the choices of $\ba_{m_n}$ and $\bb_{m_n}$ given in Section \ref{eq: EB} of the main article, we have that
$$
\{\bx \in \reald: f_{m_n}(\bx)>0\} \subset \{\bx \in \reald: g_{\bvartheta_0}(\bx|H_0)>0\}.
$$
In this view, let $E_n$ be a sequence of measurable events such that $\prodGtHtrue(E_n|H_0) =o(e^{-nc})$ as $n \to \infty$, for some $c>0$, and notice that for any $\epsilon \in(0,c)$
\begin{equation}\label{eq:westart}
Q_n(E_n) \leq e^{n \epsilon} \prodGtHtrue(E_n|H_0) + Q_n(E_n \cap \{R_n > \epsilon\})
\end{equation}
where $R_n$ denotes the sequence of rescaled of log likelihood ratios
$$
R_n := \frac{1}{n}\sum\log\left\lbrace
\frac{f_{m_n}(\bX_{n,i})}{g_{\bvartheta_0}(\bX_{n,i}|H_0)}
\right\rbrace
$$
and $\bX_{n,1},\ldots, \bX_{n,n}$ are iid rv's distributed according to $F^{m_n}(\ba_{m_n}\cdot\, +\bb_{m_n})$.
By assumpton, the first term on the right-hand side of \eqref{eq:westart} is of order $o(e^{-(c-\epsilon)n})$.
We next show that
%, if $\epsilon$ is chosen small enough, 
the second term is of order $O(1/n^2)$, wherefrom the final result follows.

Let $\varepsilon_n^2=\dist_H^2(f_{m_n}, g_{\bvartheta_0}(\cdot|H))$ and observe that, under Condition \ref{cond:strong}, Corollary 3.1 in 
%{\color{magenta}Falk et al. (2020)} 
\cite{r32}
together with Lemma B.1(ii) and equation (B.1) in \cite{r10}
%{\color{magenta}Ghoshal and van der Vaart (2017)} 
imply $\varepsilon_n \to 0$ as $n \to \infty$. This fact, together with Condition \ref{cond:densratio} and Theorem 5 in \cite{r434}
%{\color{magenta}Wong and Shen (1995)} 
allow to deduce that, for all sufficiently large $n \geq n_0$
\begin{equation}\label{eq: bound_1}
\mathbb{E}{R_n}=\kulb(f_{m_n},g_{\bvartheta_0}(\cdot|H_0))\leq
\left[
a+8 \max\left\lbrace
1,\log \left(\frac{J_0}{\varepsilon_n} \right)\right\rbrace 
\right] \varepsilon_n^2 < \frac{\epsilon}{2}
\end{equation}
where $a$ is a positive global constant. Moreover, arguments analogous to those in the proof of Lemma 4.1 in
% {\color{magenta}Falk et al. (2020)} 
\cite{r32}
allow to deduce that
\begin{equation}\label{eq: bound_2}
\max(\varsigma_2,\varsigma_3)<1+\varsigma_4<\infty,
\end{equation}
where, for $l=2,3,4$,
$$
\varsigma_l:= \sup_{n \geq n_0}(-1)^l\int\left[\log \left\lbrace
\frac{f_{m_n}(\bx)}{g_{\bvartheta_0}(\bx|H_0)}
\right\rbrace\right]^l f_{m_n}(\bx) \diff \bx.
$$
The bounds in \eqref{eq: bound_1}-\eqref{eq: bound_2}, Markov's inequality and a few simple manipulations now yield
\begin{equation*}
\begin{split}
Q_n(E_n \cap \{R_n > \epsilon\}) & \leq Q_n(\{R_n -\mathbb{E}R_n > \epsilon/2\}) 	\\
& \leq  Q_n(\{|R_n -\mathbb{E}R_n| > \epsilon/2\}) 	\\
& \leq (2/\epsilon)^4\left[
\frac{1+\varsigma_4}{n^3}(1+2\epsilon+3\epsilon^2/2)
+\frac{3}{n^2}\varsigma_2^2
\right]\\
&=O(1/n^2).
\end{split}
\end{equation*}
The proof is now complete.

\subsection{Proofs of the results in Section \ref{sec:frec}}

\subsubsection{Auxiliary results for the proof of Theorem \ref{th:rem_cont_frec}}

\begin{lemma}\label{lem:pseudokulb}
	Under the assumptions of Theorem \ref{th:rem_cont_frec}, for any $c>0$ eventually almost surely as $n \to \infty$
	$$
	\int_{\Hset \times \bvarTheta} \prod_{i=1}^n 
	\left\lbrace
	\frac{g_{\brho,\bsigma}(\overline{\bM}_{m_n,i}|H)}{g_{\brho_0, \bone}(\overline{\bM}_{m_n,i}|H_0)}
	\right\rbrace
	\diff(\Pi_\Hset \times \overline{\Psi}_n)(H, \brho, \bsigma) \geq e^{-nc}.
	$$
\end{lemma}
\begin{proof}
	
	Denote by $\underline{\Pi}_{\text{sc}}^{(d)}$ any pm with positive Lebesgue density $\underline{\pi}_{\text{sc}}^{(d)}$ on $(0,\infty)^d$ satisfying
	$$
	\sup_{\bx \in (1\pm \eta)^d} \underline{\pi}_{\text{sc}}^{(d)}(\bx)< M,
	$$
	where $M$ is defined via
	$$
	M^{1/d}:= \inf_{x \in (1\pm \eta)}\pisc(x)
	$$
	and is positive and bounded by Condition  \ref{cond:frecextend}\ref{cond:pisc}\ref{posit}.
	Define $\Pi_{\bvarTheta}=\Pi_{\text{sh}}\times \underline{\Pi}_{\text{sc}}^{(d)}$, where $\bvarTheta=(0,\infty)^{2d}$. Then, denote $\Pi_{\Hset\times \bvarTheta}=\Pi_{\Hset}\times \Pi_{\bvarTheta}$ and observe it satisfies the hypotheses of Corollary \ref{lem: KLalpha}. Consequently, for all $\epsilon>0$ we have that
	$$
	\Pi_{\Hset \times \bvarTheta}\left(\cap_{l=1}^4 \mathcal{V}_\epsilon^{(l)}\right)>0.
	$$   
	where $\mathcal{V}_\epsilon^{(l)}$, $l \in \nat_+$, is defined as in \eqref{eq:vsetfrec}. Letting $ 1-\eta<a<1$, $1<b<1+\eta$ and
	$$
	\mathcal{V}=\{(H,\brho, \bsigma) \in \Hset \times (0, \infty)^{2d}:\, 
	a<\sigma_j<b 
	\} \cap \left\lbrace
	\cap_{l=1}^4 \mathcal{V}_\epsilon^{(l)}
	\right\rbrace,
	$$
	we also have $\Pi_{\Hset \times \bvarTheta}\left( \mathcal{V}\right)>0$. Then, define $\overline{\Pi}_{\Hset \times \bvarTheta}(\cdot):=\Pi_{\Hset \times \bvarTheta}(\, \cdot \cap \mathcal{V})/ \Pi\left( \mathcal{V}\right)$.

	Next, observe that by assumption
	eventually almost surely as $n \to \infty$
	$$
	\frac{b}{1+\eta} < \widehat{\sigma}_{n,j}/a_{m_n,j} < \frac{a}{1-\eta}, \quad j =1, \ldots,d,
	$$
	thus, by construction, $\prod_{j=1}^d\pisc(x_j/\{\hat{\sigma}_{n,j}/a_{m_n,j}\})>\underline{\pi}_{\text{sc}}^{(d)}(\bx)$ for all $\bx \in (a,b)^{d}$ and
	\begin{equation}\label{eq:ineq_aKL}
	\begin{split}
	&\int_{\Hset \times \bvarTheta} \prod_{i=1}^n 
	\left\lbrace
	\frac{g_{\brho,\bsigma}(\overline{\bM}_{m_n,i}|H)}{g_{\brho_0, \bone}(\overline{\bM}_{m_n,i}|H_0)}
	\right\rbrace
	\diff(\Pi_\Hset \times \overline{\Psi}_n)(H, \brho, \bsigma)\\
	& \quad \geq \int_{\mathcal{V}} \prod_{i=1}^n 
	\left\lbrace
	\frac{g_{\brho,\bsigma}(\overline{\bM}_{m_n,i}|H)}{g_{\brho_0, \bone}(\overline{\bM}_{m_n,i}|H_0)}
	\right\rbrace
	\diff(\Pi_\Hset \times \overline{\Psi}_n)(H, \brho, \bsigma)\\
	& \quad \geq \int_{\mathcal{V}} \prod_{i=1}^n 
	\left\lbrace
	\frac{g_{\brho,\bsigma}(\overline{\bM}_{m_n,i}|H)}{g_{\brho_0, \bone}(\overline{\bM}_{m_n,i}|H_0)}
	\right\rbrace
	\diff\Pi_{\Hset \times \bvarTheta}(H, \brho, \bsigma)\\
	& \quad \geq \Pi_{\Hset \times \bvarTheta}(\mathcal{V}) \exp
	\left\lbrace
	-nI_n
	\right\rbrace,
	\end{split}
	\end{equation}
	where the last line follows from an application of Jensen's inequality and
	$$
	I_n:=\int_{\mathcal{V}} \frac{1}{n} \sum_{i=1}^n  \log^+
	\left\lbrace
	\frac{g_{\brho_0, \bone}(\overline{\bM}_{m_n,i}|H_0)}{g_{\brho,\bsigma}(\overline{\bM}_{m_n,i}|H)}
	\right\rbrace
	\diff\overline{\Pi}_{\Hset \times \bvarTheta}(H, \brho, \bsigma).
	$$
	By Fubini's theorem and Condition \ref{cond:densratio}, for all $n \geq n_0$,
	\begin{equation}\label{eq: expectfrec}
	\begin{split}
	\mathbb{E}(I_n)&=
	\int_{\mathcal{V}} \left[
	\int
	\log^+
	\left\lbrace
	\frac{g_{\brho_0, \bone}(\bx|H_0)}{g_{\brho,\bsigma}(\bx|H)}
	\right\rbrace f_{m_n}(\bx) \diff \bx
	\right]  
	\diff\overline{\Pi}_{\Hset \times \bvarTheta}(H, \brho, \bsigma)\\
	&
	\leq J_0
	\int_{\mathcal{V}} 
	\kulb_{+}^{(1)}(g_{\brho_0, \bone}(\cdot|H_0), g_{\brho, \bsigma}(\cdot|H))
	\diff\overline{\Pi}_{\Hset \times \bvarTheta}(H, \brho, \bsigma)	\\
	&  < J_0 \epsilon.
	\end{split}
	\end{equation}
	Consequenlty, for all $n\geq n_0$, the term on the right hand side of \eqref{eq:ineq_aKL} is bounded from below by
	\begin{equation}\label{eq:almostthere}
	e^{-nJ_0\epsilon}\Pi_{\Hset \times \bvarTheta}(\mathcal{V}) \exp
	\left\lbrace
	-n\{I_n- \mathbb{E}(I_n)\}
	\right\rbrace.
	\end{equation}
	Moreover, denoting
	\begin{eqnarray*}
		\zeta_n &:= & \mathbb{E} \left(
		\log^+
		\left\lbrace
		\frac{g_{\brho_0, \bone}(\overline{\bM}_{m_n,1}|H_0)}{g_{\brho,\bsigma}(\overline{\bM}_{m_n,1}|H)}
		\right\rbrace
		\right),
		\\
		\zeta_n^{(l)}&:=&	
		\mathbb{E}
		\left|\int_{\mathcal{V}}   
		\left[
		\log^+
		\left\lbrace
		\frac{g_{\brho_0, \bone}(\overline{\bM}_{m_n,1}|H_0)}{g_{\brho,\bsigma}(\overline{\bM}_{m_n,1}|H)}
		\right\rbrace-
		\zeta_n
		\right]
		\diff\overline{\Pi}_{\Hset \times \bvarTheta}(H, \brho, \bsigma)
		\right|^l, \quad l \in \nat_+,
		\\
	\end{eqnarray*}
	by Markov's inequality and equation (6.2) in 
%	{\color{magenta} Billinsgley (1995)}, 
	\cite{r879},
	for all $\varepsilon>0$
	\begin{equation}\label{eq:polytail}
	\begin{split}
	Q_n(|I_n - \mathbb{E}(I_n)|) \leq\varepsilon^4 \left\lbrace
	\frac{1}{n^3} \zeta_n^{(4)}
	+\frac{1}{n^2} \left( 
	\zeta_n^{(2)}
	\right)^2
	\right\rbrace,
	\end{split}
	\end{equation}
	where, by Jensen's inequality, Fubini's theorem and Minkowski's inequality, for all $n \geq n_0$ 
	\begin{equation*}
	\begin{split}
	\zeta_n^{(4)}
	&\leq 
	\mathbb{E}
	\int_{\mathcal{V}}  
	\left|
	\log^+
	\left\lbrace
	\frac{g_{\brho_0, \bone}(\overline{\bM}_{m_n,1}|H_0)}{g_{\brho,\bsigma}(\overline{\bM}_{m_n,1}|H)}
	\right\rbrace
	-\zeta_n
	%		\mathbb{E} \left(
	%		\log^+
	%		\left\lbrace
	%		\frac{g_{\brho_0, \bone}(\overline{\bM}_{m_n,1}|H_0)}{g_{\brho,\bsigma}(\overline{\bM}_{m_n,1}|H)}
	%		\right\rbrace
	%		\right)
	\right|^4
	\diff\overline{\Pi}_{\Hset \times \bvarTheta}(H, \brho, \bsigma)
	\\
	&=
	\int_{\mathcal{V}}  
	\mathbb{E} \left|
	\log^+
	\left\lbrace
	\frac{g_{\brho_0, \bone}(\overline{\bM}_{m_n,1}|H_0)}{g_{\brho,\bsigma}(\overline{\bM}_{m_n,1}|H)}
	\right\rbrace
	-\zeta_n
	%		\mathbb{E} \left(
	%		\log^+
	%		\left\lbrace
	%		\frac{g_{\brho_0, \bone}(\overline{\bM}_{m_n,1}|H_0)}{g_{\brho,\bsigma}(\overline{\bM}_{m_n,1}|H)}
	%		\right\rbrace
	%		\right)
	\right|^4
	\diff\overline{\Pi}_{\Hset \times \bvarTheta}(H, \brho, \bsigma)
	\\
	& \leq \int_{\mathcal{V}}
	\left[ J_0 
	%
%	\mathscr{V}_{H_0, \brho_0, \bone}^{(4,+)}(H_0, \brho_0, \bone; H, \brho, \bsigma)
	\kulb_{+}^{(4)}(g_{\brho_0, \bone}(\cdot|H_0), g_{\brho, \bsigma}(\cdot|H))
	+J_0^4 \left\lbrace
%	\mathscr{V}_{H_0, \brho_0, \bone}^{(1,+)}(H_0, \brho_0, \bone; H, \brho, \bsigma)
	\kulb_{+}^{(1)}(g_{\brho_0, \bone}(\cdot|H_0), g_{\brho, \bsigma}(\cdot|H))
	\right\rbrace^4
	\right. \\
	& \quad \qquad \left. + 6J_0^3 \left\lbrace
%	\mathscr{V}_{H_0, \brho_0, \bone}^{(1,+)}(H_0, \brho_0, \bone; H, \brho, \bsigma) 
	\kulb_{+}^{(1)}(g_{\brho_0, \bone}(\cdot|H_0), g_{\brho, \bsigma}(\cdot|H))
	\right\rbrace^2
	%
%	\mathscr{V}_{H_0, \brho_0, \bone}^{(2,+)}(H_0, \brho_0, \bone; H, \brho, \bsigma)
	%
	\kulb_{+}^{(2)}(g_{\brho_0, \bone}(\cdot|H_0), g_{\brho, \bsigma}(\cdot|H))
	\right]\diff\overline{\Pi}_{\Hset \times \bvarTheta}(H, \brho, \bsigma)\\
	& \leq M'
	\end{split}
	\end{equation*}
	with $M'$ a positive bounded constant, and
	\begin{equation*}
	\begin{split}
	\zeta_n^{(2)} \leq 1+\zeta_n^{(4)} \leq 1+M'.
	\end{split}
	\end{equation*}
	%	let $\Pi_{\Gset_{\bvarTheta}}$ be induced prior on $\Gset_{\bvarTheta}:=\{ g_{\brho,\bsigma}(\cdot|H):\,(H,\brho,\bsigma)\in (\Hset,\bvarTheta) \}$. 
	Thus the term on the right-hand side of \eqref{eq:polytail} is of order $O(1/n^2)$ and, by Borel-Cantelli lemma, we conclude that eventually almost surely 
	$$
	I_n - \mathbb{E}(I_n) \leq \varepsilon
	$$
	for all $\varepsilon>0$.  As a result, eventually almost surely as $n \to \infty$, the term in \eqref{eq:almostthere} is bounded from below by
	$$
	e^{-n(\varepsilon+J_0\epsilon)}\Pi_{\Hset \times \bvarTheta}(\mathcal{V}). 
	$$
	Since $\epsilon$ and $\varepsilon$ can be selected arbitrarily small, the final result now follows.
\end{proof}

\subsubsection{Proof of Theorem \ref{th:rem_cont_frec}}\label{appsec:proof_theo_frecdom}

Denote by $\overline{\Pi}_{\text{sc}}^{(d)}$ the pm on $(0,\infty)^d$ whose Lebesgue density equals
$ \prod_{j=1}^d
\left\lbrace
u_{\text{sc}}(x_j)/u,
\right\rbrace
$, $\bx=(x_1, \ldots, x_d) > \bzero$,
where
$$
u=\int_{(0,\infty)}u_{\text{sc}}(s)\diff s,
$$
and define $\Pi_{\bvarTheta}=\Pi_{\text{sh}}\times \overline{\Pi}_{\text{sc}}^{(d)}$, where $\bvarTheta=(0,\infty)^{2d}$. Then, denote $\Pi_{\Hset\times \bvarTheta}=\Pi_{\Hset}\times \Pi_{\bvarTheta}$ and observe it satisfies the assumptions of Theorem \ref{th:alpha_frec}. Consequently, the induced prior $\Pi_{\Gset_{\bvarTheta}}$ on $\Gset_{\bvarTheta}:=\{ g_{\brho,\bsigma}(\cdot|H):\,(H,\brho,\bsigma)\in (\Hset\times\bvarTheta) \}$ and the latter statistical model jointly satisfy Condition \ref{cond: newcond}.  

Reasoning as in the proof of Proposition \ref{prop:gen_cons}, we can conclude that there exist $\epsilon>0$ and $\delta>0$ such that, for any sequence  $\mathcal{G}_{\bvarTheta,n}$  of measurable subsets of
$$
\widetilde{\mathcal{U}}_\epsilon^\complement\cap \phi_{\Hset\times\bvarTheta}\left(\{
\Hset \times B_{\delta, 1}((\brho_0,\bone))
\}\right)
$$
and any collection $\tests=(s_n, t_{n,1}, \ldots,t_{n,d})$ of measurable functions from $\text{supp}(\prodGrhooneProb(\cdot|H_0))$ to $[0,1]$, we have
\begin{equation}\label{eq: wannaboundfrec}
\max
\left\lbrace
\tilde{\Pi}_n(\tilde{\mathcal{U}}^\complement),
\Pi_n((\mathcal{U}_1\times\mathcal{U}_2)^\complement)
\right\rbrace
\leq 2\Vert \tests(\overline{\bM}_{m_n, 1:n}) \Vert_1+ \frac{\Xi_n(\overline{\bM}_{m_n, 1:n},\tests, \Pi_\Hset \times \overline{\Psi}_n)}{p(\overline{\bM}_{m_n, 1:n})}
\end{equation}
where $\widetilde{\mathcal{U}}_\epsilon:= 
\{g \in \Gset_{\bvarTheta}: \dist_{H}(g,g_{\brho_0, \bone}(\cdot|H_0))\leq 4\epsilon\}$, $\overline{\Psi}_n ={\Psi}_n \circ \psi_n^{-1} $, with $\psi_n$ as in 
the first line of 	\eqref{eq:reparam}, 
$$
p(\overline{\bM}_{m_n, 1:n}):=\int_{\Hset \times \bvarTheta} \prod_{i=1}^n 
\left\lbrace
\frac{g_{\brho,\bsigma}(\overline{\bM}_{m_n,i}|H)}{g_{\brho_0, \bone}(\overline{\bM}_{m_n,i}|H_0)}
\right\rbrace
\diff(\Pi_\Hset \times \overline{\Psi}_n)(H, \brho, \bsigma)
$$
and the functional $\Xi_n$ is defined as in \eqref{eq:xi_n} (see also Remark \ref{rem:proof_highlight} for further details). 
By assumption, eventually almost surely as $n \to \infty$
$$
1-\eta < \widehat{\sigma}_{n,j}/a_{m_n,j} < 1+\eta, \quad j =1, \ldots,d,
$$
thus Condition \ref{cond:frecextend}\ref{cond:pisc}\ref{piscbound}
guarantees that $\pisc(x/\{\hat{\sigma}_{n,j}/a_{m_n,j}\})\leq u_{\text{sc}}(x)$, for all $x>0$, and
$$
\Xi_n(\overline{\bM}_{m_n, 1:n},\tests, \Pi_\Hset \times \overline{\Psi}_n) \leq 
\left(
\frac{u}{1-\eta}
\right)^d
\Xi_n(\overline{\bM}_{m_n, 1:n},\tests, \Pi_{\Hset \times\bvarTheta}). 
$$
Moreover, by Lemma \ref{lem:pseudokulb}, for any $c>0$, eventually almost surely as $n \to \infty$, $p(\overline{\bM}_{m_n, 1:n}) \geq  e^{-nc}$. Therefore, there exists a constant $M>0$ such that eventually almost surely as $n \to \infty$ the term on the right-hand side of 
\eqref{eq: wannaboundfrec} is bounded from above by
$$
M \left\lbrace
2\Vert \tests(\overline{\bM}_{m_n, 1:n}) \Vert_1+ e^{cn}\Xi_n(\overline{\bM}_{m_n, 1:n},\tests, \Pi_{\Hset \times\bvarTheta})
\right\rbrace.
$$
Without loss of generality,
we can assume that $\delta<\delta^*$, with $\delta^*$ as in Condition \ref{cond: newcond}, and select $\Gset_{\bvarTheta,n}$ and $\tests$ satisfying the properties therein. Hence, selecting $c$ as in \eqref{eq:choosec}, the exponential tail bound in \eqref{eq:tailprob} is satisfied for a $n$-dimensional iid sample from $G_{\brho_0,\bone}(\cdot|H_0)$. Applying 
Proposition \ref{prop:RC} we can thus conclude that 
\begin{equation}
Q_n \left(
2\Vert \tests(\overline{\bM}_{m_n, 1:n}) \Vert_1+ e^{cn}\Xi_n(\overline{\bM}_{m_n, 1:n},\tests, \Pi_{\Hset \times\bvarTheta})> \varepsilon
\right) \lesssim n^{-1-c'},
\end{equation}
for some $c' \in (0,1)$. The first part of the statement now follows from Borel-Cantelli lemma. By the definitions of $\Pi_n$, $\tilde{\Pi}_n$, $\postobs$ and $\postobsdens$,  the results at points (a$'$)-(c$'$) are immediate consequences.

\subsubsection{Proof of the inequality in Example \ref{ex:est_frec}}\label{sec:ex_frec_est_proof}
Recall that $\bZ_1, \ldots, \bZ_{nm_n}$ are iid rv's with absolutely continuous distributions $F_0$ and copula $C_0$, thus we have the representation
$$
Z_{i,j}=F_{0,j}^{\leftarrow}(U_{i,j}), \quad j=1, \ldots,d, \, i=1, \ldots nm_n,
$$
where $\bU_i=(U_{i,1}, \ldots, U_{i,d})$ are iid rv's with distribution $C_0$. In the sequel, we also denote
$$
\bar{U}_{i,j}=1-U_{i,j}, \quad j=1, \ldots,d, \, i=1, \ldots nm_n.
$$
Moreover, denoting by $U_{k,j}\equiv U_{k,nm_n,j}$ the $k$-th order statistic of the marginal sample $U_{1,j}, \ldots,U_{nm_n,j}$, observe that for $j=1, \ldots,d$,
$$
\widehat{\sigma}_{n,j}:=\widehat{F}_{nm_n,j}^{\leftarrow}(1-1/m_n)=F_{0,j}^{\leftarrow}(U_{n(m_n-1),j}). 
$$
We next show that, for all $\epsilon>0$, as $n \to \infty$, the estimator $\widehat{\sigma}_{n,j}$ of $F_{0,j}^{\leftarrow}(1-1/m_n)\equiv a_{m_n,j}$ satisfies 
\begin{equation}\label{eq:expineq_tau}
\prodFmalt(\vert \widehat{\sigma}_{n,j}/a_{m_n,j} -1  \vert > \epsilon)\leq 4 e^{-\tau_j \sqrt{n+1}},
\end{equation}
where $\tau_j \equiv \tau_j(\epsilon)$ is a positive constant.

Preliminary observe that
\begin{equation*}
\begin{split}
&\prodFmalt(\vert \widehat{\sigma}_{n,j}/a_{m_n,j} -1  \vert > \epsilon)\\
&\qquad=\prodFmalt( \widehat{\sigma}_{n,j}/a_{m_n,j} -1   > \epsilon)+\prodFmalt( \widehat{\sigma}_{n,j}/a_{m_n,j} -1  <- \epsilon)\\
&\qquad=: T_{n,j}^{(1)}+T_{n,j}^{(2)}. 
\end{split}
\end{equation*}
As for $T_{n,j}^{(1)}$,  since $F_{0,j}\in \mathcal{D}(G_{\rho_{0,j}})$, as $n \to \infty$ we have that
\begin{equation*}
\begin{split}
T_{n,j}^{(1)}&=\prodcop(1-U_{n(m_n-1),j}< 1-F_{0,j}(a_{m_n,j}(1+\epsilon)))\\
&=\prodcop(\bar{U}_{n+1,j}< m_n\left\{1-F_{0,j}(a_{m_n,j}(1+\epsilon))\right\}/m_n)\\
&\leq \prodcop(\bar{U}_{n+1,j}<(1+\epsilon')(1+\epsilon)^{-\rho_{0,j}}/m_n)
\end{split}
\end{equation*}
for an arbitrarily small $\epsilon'>0$ such that
$$
\upsilon_j:=1-(1+\epsilon')(1+\epsilon)^{-\rho_{0,j}}>0.
$$
Thus, the term on the right-hand side of the above inequality is bounded from above by
\begin{equation*}
\begin{split}
&\prodcop(\bar{U}_{n+1,j}<(1+\epsilon')(1+\epsilon)^{-\rho_{0,j}}(n+1)/(nm_n))\\
\quad &= \prodcop(\bar{U}_{n+1,j}-(n+1)/(nm_n)<-\upsilon_j(n+1)/(nm_n))\\
\quad & \leq  \prodcop(|\bar{U}_{n+1,j}-(n+1)/(nm_n)|>\upsilon_j(n+1)/(nm_n)).
\end{split}
\end{equation*}
By Fact 3 in \cite{r2454}
%{\color{magenta}Fact 3 in Csorgo et al. (1986)} 
we now conclude that, as $n \to \infty$,
$$
T_{n,j}^{(1)}\leq 2\exp(-\tau_{j,1} \sqrt{n+1}),
$$
where $\tau_{j,1}=\upsilon_j/10$. By a similar reasoning, for some $\tau_{j,2}\equiv \tau_{j,2}(\epsilon)>0$, as $n \to \infty$,
$$
T_{n,j}^{(2)}\leq 2\exp(-\tau_{j,2} \sqrt{n+1}).
$$
The inequality in \eqref{eq:expineq_tau} now follows by setting $\tau_j=\min(\tau_{j,1}, \tau_{j,2})$.

%Using the inequality in \eqref{eq:expineq_tau}, it is immediate to see that, as $n \to \infty$,
%\begin{equation*}
%\begin{split}
%\prodFmalt(\Vert \widehat{\boldsymbol{\sigma}}_n/\ba_{m_n} -\bone  \Vert_\infty > \epsilon)&\leq \sum_{j=1}^d\prodFmalt(\vert \widehat{\sigma}_{n,j}/a_{m_n,j} -1  \vert > \epsilon)\\
%&\leq \sum_{j=1}^d4d e^{-\tau_j \sqrt{n+1}}\\
%& \leq 4d e^{-\tau \sqrt{n+1}},
%\end{split}
%\end{equation*}
%where $\tau=\min_{1\leq j \leq d}\tau_j$, which is the result.

\subsubsection{Technical derivations for Example \ref{ex:exp_dep_frec}}\label{sec: first_ex_frec}
Note that, for $\bx \in (1,\infty)^2$, $F_0(\bx)$ allows the representation
$$
F_0(\bx)=C_0\left(1-1/x_1^{\rho_{0,1}},1-1/x_2^{\rho_{0,2}}
\right),
$$
where 
\begin{equation}\label{eq:cupula_ex}
C_0(\bu)=1-(1-u_1)-(1-u_2)+\left(\frac{1}{1-u_1}+\frac{1}{1-u_2} \right)^{-1}.
\end{equation}
As established in \cite[Example 2.2.]{r32}, $C_0$ satisfies Condition \ref{cond:strong}\ref{cond:copdiff}. The verification of Condition \ref{cond:strong}\ref{margins} is immediate for Pareto margins $F_{0,j}(x_j)=1-1/x_j^{\rho_{0,j}}$, $x_j>1$, $j=1,2$. Moreover, it is already known \citep[e.g.][Example 5.16 and p. 289]{r200} that $F_0 \in \mathcal{D}(G_{\brho_0, \bone}(\cdot|H_0))$. We racall that herein, for $\bx \in (0, \infty)^2$,
\begin{equation}\label{eq:lim_df}
\begin{split}
G_{\brho_0, \bone}(\bx|H_0)&=\exp\left\lbrace
-x_1^{-\rho_{0,1}}-x_2^{-\rho_{0,1}}
+\left(
x_1^{\rho_{0,1}}+x_2^{\rho_{0,2}}
\right)^{-1}
\right\rbrace\\
&=:\exp\left\lbrace-V(\bx^{\brho_0}) \right\rbrace,
\end{split}
\end{equation}
where $\bx^{\brho_0}=(x_1^{\rho_{0,1}},x_2^{\rho_{0,2}})$ and $V(\cdot)\equiv V(\cdot|H_0)$. A valid choice of the norming sequences, asymptotically equivalent to that in \eqref{eq:norming}, is 
$$
\ba_{m_n}=\left((m_n-1)^{1/\brho_{0,1}},(m_n-1)^{1/\brho_{0,2}} \right), \quad \bb_{m_n}=\bzero.
$$
Thus, the cdf associated to the rv's $\bM_{m_n,i}/\ba_{m-n}$ is given by
\begin{equation}\label{eq:resc_max_df}
\begin{split}
F_0^{m_n}(\ba_{m_n}\by)&=\left\lbrace
1-\frac{V^{(m_n)}(\by^{\brho_{0}})}{m_n-1}
\right\rbrace^{m_n}
\end{split}
\end{equation} 
where $\by \in \times_{j=1}^2(1/(m_n-1)^{1/\rho_{0,j}}, \infty)$ and, for all $\bz \in \times_{j=1}^2(1/(m_n-1), \infty)$
\begin{equation}\label{eq:V_m}
\begin{split}
V^{(m_n)}(\bz)&=\frac{1}{z_1}+\frac{1}{z_2}+\left( z_1 +z_2 -\frac{1}{m_n-1} \right)^{-1}.
\end{split}
\end{equation} 
In the present example,  for $\by \in \times_{j=1}^2(1/(m_n-1)^{1/\rho_{0,j}}, \infty)$,
\begin{equation}\label{eq:dens_max}
\begin{split}
f_{m_n}(\by)&=\frac{m_n}{m_n-1} \left\lbrace
1-\frac{V^{(m_n)}(\by^{\brho_{0}})}{m_n-1}
\right\rbrace^{m_n-1} \\
&\quad \times\left\lbrace
V^{(m_n)}_{\{1\}}(\by^{\brho_{0}})
V^{(m_n)}_{\{2\}}(\by^{\brho_{0}})
-V^{(m_n)}_{\{1,2\}}(\by^{\brho_{0}})
\right\rbrace\prod_{j=1}^2 \rho_{0,j}y_{j}^{\rho_{0,j}-1},	
\end{split}
\end{equation}
with 
\begin{equation*}
\begin{split}
-V^{(m_n)}_{\{j\}}(\bz)&=\frac{1}{z_j^2}\left\lbrace
1-\frac{z_j^2}{(z_1+z_2-1/(m_n-1))^2}
\right\rbrace, \quad j=1,2, \\
-V^{(m_n)}_{\{1,2\}}(\bz)&=2\left\lbrace
z_1+y_z-1/(m_n-1)
\right\rbrace^{-3}, 
\end{split}
\end{equation*}
for $\bz \in \times_{j=1}^2(1/(m_n-1), \infty)$, while, for $\by \in(0,\infty)^2$, 
\begin{equation}\label{eq:lim_dens}
\begin{split}
g_{\brho_0, \bone}(\by|H_0)=\exp\left\lbrace-V(\bx^{\brho_0}) \right\rbrace
\left\lbrace
V_{\{1\}}(\by^{\brho_{0}})
V_{\{2\}}(\by^{\brho_{0}})
-V_{\{1,2\}}(\by^{\brho_{0}})
\right\rbrace\prod_{j=1}^2 \rho_{0,j}y_{j}^{\rho_{0,j}-1},
\end{split}	
\end{equation}
with
\begin{equation}\label{eq:derV}
\begin{split}
-V_{\{j\}}(\by)&=\frac{1}{y_j^2}\left\lbrace
1-\frac{y_j^2}{(y_1+y_2)^2}
\right\rbrace, \quad j=1,2,  \quad
-V_{\{1,2\}}(\by)=2\left\lbrace
y_1+y_2
\right\rbrace^{-3}.
\end{split}
\end{equation}
Therefore, for $\by \in \times_{j=1}^2(1/(m_n-1)^{1/\rho_{0,j}}, \infty)$
\begin{equation}\label{eq:dens_ratio_eq}
\begin{split}
\frac{f_{m_n}(\by)}{g_{\brho_0, \bone}(\by|\theta_0)}&=\frac{m_n}{m_n-1}\left\lbrace
1-\frac{V^{(m_n)}(\by^{\brho_{0}})}{m_n-1}
\right\rbrace^{m_n-1}\exp\left\lbrace V(\by^{\brho_0}) \right\rbrace\\
&\quad \times \left\lbrace
\frac{	V_{\{1\}}(\by^{\brho_{0}})
	V_{\{2\}}(\by^{\brho_{0}})}{	V_{\{1\}}(\by^{\brho_{0}})
	V_{\{2\}}(\by^{\brho_{0}})
	-V_{\{1,2\}}(\by^{\brho_{0}})}
\right.\frac{V^{(m_n)}_{\{1\}}(\by^{\brho_{0}})
	V^{(m_n)}_{\{2\}}(\by^{\brho_{0}})}{	V_{\{1\}}(\by^{\brho_{0}})
	V_{\{2\}}(\by^{\brho_{0}})}
\\
&\qquad +
\left.
\frac{-V_{\{1,2\}}(\by^{\brho_{0}})}{	V_{\{1\}}(\by^{\brho_{0}})
	V_{\{2\}}(\by^{\brho_{0}})
	-V_{\{1,2\}}(\by^{\brho_{0}})}
\frac{-V^{(m_n)}_{\{1,2\}}(\by^{\brho_{0}})}{-V_{\{1,2\}}(\by^{\brho_{0}})}
\right\rbrace\\
&\quad =: \frac{m_n}{m_n-1}I_n^{(1)}(\by)\{ I_n^{(2)}(\by)+I_n^{(3)}(\by) \}.
\end{split}
\end{equation}
Since for $y_j>1/(m_n-1)^{1/\rho_j}$ we have $u_j:=1/(y_j^{\rho_{0,j}}(m_n-1))\leq 1$, $j=1,2$, then
\begin{equation*}
\begin{split}
I_n^{(1)}(\by) &=
\left\lbrace
1-\frac{V^{(m_n)}(\by^{\brho_{0}})}{m_n-1}
\right\rbrace^{m_n-1}\exp\left\lbrace \frac{1}{m_n-1}V(\by^{\brho_0}) \right\rbrace^{m_n-1}\\
&=\left[
\left\lbrace
1+\left( \frac{1}{u_1}+\frac{1}{u_2}-1\right)^{-1}-u_1-u_2
\right\rbrace\exp\left\lbrace
u_1+u_2 -\left(
\frac{1}{u_1}+\frac{1}{u_2}
\right)^{-1}
\right\rbrace
\right]^{m_n-1}\\
&\leq 1.
\end{split}
\end{equation*}
Moreover, for $j=1,2$, we have
\begin{equation*}
\begin{split}
\frac{-V^{(m_n)}_{\{j\}}(\by^{\brho_{0}})}{-V_{\{j\}}(\by^{\brho_{0}})}&= \left\lbrace
\frac{ (m_n-1) y_1^{\rho_{0,1}}+(m_n-1) y_2^{\rho_{0,2}} }{ (m_n-1) y_1^{\rho_{0,1}}+(m_n-1) y_2^{\rho_{0,2}} -1 }\right\rbrace^2\\
&\quad \times	\frac{\left( (m_n-1) y_1^{\rho_{0,1}}+(m_n-1) y_2^{\rho_{0,2}} -1 \right)^2-y_j^{2\rho_{0,j}}(m_n-1)^2}{\left( (m_n-1) y_1^{\rho_{0,1}}+(m_n-1) y_2^{\rho_{0,2}} \right)^2- y_j^{2\rho_{0,j}}(m_n-1)^2}\\
&\leq 4
\end{split}
\end{equation*}
and 
\begin{equation*}
\begin{split}
\frac{-V^{(m_n)}_{\{1,2\}}(\by^{\brho_{0}})}{-V_{\{1,2\}}(\by^{\brho_{0}})}=
\left\lbrace
\frac{ (m_n-1) y_1^{\rho_{0,1}}+(m_n-1) y_2^{\rho_{0,2}} }{ (m_n-1) y_1^{\rho_{0,1}}+(m_n-1) y_2^{\rho_{0,2}} -1 }\right\rbrace^3 \leq 8.
\end{split}
\end{equation*}
Thus
$
I_n^{(2)}(\by)+I_n^{(3)}(\by) \leq 16
$
and we deduce that for any $\overline{\epsilon}>0$ as $n \to \infty$
\begin{equation}\label{eq:ratiobound}
\Vert f_{m_n}/ g_{\brho_0, \bone}(\cdot|H_0) \Vert_\infty \leq (1+\overline{\epsilon})16,
\end{equation}
wherefrom the final result follows.

\subsubsection{Technical derivations for Example \ref{ex:log_dep_frec}}\label{sec:frecsecondex}
Note that also in this case, for $\bx \in (1,\infty)^2$, $F_0(\bx)$ allows the representation
$$
F_0(\bx)=C_0\left(1-1/x_1^{\rho_{0,1}},1-1/x_2^{\rho_{0,2}}
\right),
$$
where $C_0$ is now the Joe-B5 copula, i.e., for $\bu\in[0,1]^2$.
$$
C_0(\bu)=1-\left[\{1-(1-u_1)^3\}\{1-(1-u_2)^3\}\right]^{1/3}.
$$
On the other hand, for all $\bx \in(0,\infty)^2$, in this example $G_{\brho_0, \bone}(\cdot|H_0)$ allows the representation
$$
G_{\brho_0, \bone}(\bx|H_0)=\evc\left( e^{1/x_1^{\rho_{0,1}}}, e^{1/x_1^{\rho_{0,1}}} \bigg{|}H_0\right),
$$
where $\evc(\cdot|H_0)$ is
the logistic extreme-value copula with dependence parameter $3$
\begin{equation*}
\begin{split}
\evc(\bu|H_0)&=\exp\left\lbrace
-((-\log u_1)^3+ (-\log u_3)^3)^{1/3}
\right\rbrace
\\
&=:\exp\{-L(-\log u_1, -\log u_2)\},
\end{split}
\end{equation*}
with $L(\cdot)\equiv L(\cdot|H_0)$.
Standard multivariate calculus allows to show that 
$$
\lim_{m_n \to \infty} C_0^{m_n}(\bu^{1/m_n})=\evc(\bu|H_0), \quad \forall \bu \in [0,1],
$$
and that $C_0$ and $L(\cdot|H_0)$ satisfy Condition \ref{cond:strong}\ref{cond:copdiff}. In particular, the above equation and the fact that $F_{0,j}\in\mathcal{D}(G_{\rho_{0,j}})$, $j=1, \ldots,d$, allow to conclude that $F_0\in \mathcal{D}(G_{\brho_0, \bone}(\cdot|H_0))$. Also observe that, since $F_{0,j}$ are one-parameter Pareto, Condition \ref{cond:strong}\ref{margins} is still satisfied.
To establish the property in Condition \ref{cond:densratio}, it is possible to follow the lines of Section \ref{sec: first_ex_frec}. Thus, we herein highlight only the main changes. 

Observe that, for all $\by \in \times_{j=1}^2(1/(m_n-1)^{1/\rho_{0,j}}, \infty)$, $F_0^{m_n}(\ba_{m_n}\by\,)$ and its density $f_{m_n}(\by)$ are still of the form in \eqref{eq:resc_max_df} and \eqref{eq:dens_max}, with 
$$
V^{(m_n)}(\bz)=\left\lbrace
\frac{1}{z_1^3}+\frac{1}{z_2^3}-\left( \frac{1}{m_n-1}\frac{1}{z_1z_2} \right)^{3}
\right\rbrace^{1/3}
$$
for all $\bz \in \times_{j=1}^2(1/(m_n-1), \infty)$, whose derivatives are
\begin{equation*}
\begin{split}
-V^{(m_n)}_{\{j\}}(\bz)&=\left\lbrace
\frac{1}{z_1^3}+\frac{1}{z_1^3}-
\left( \frac{1}{m_n-1}\frac{1}{z_1z_2} \right)^{3}\right\rbrace^{-2/3}\frac{1}{z_j^4}
\left\lbrace1-
\left(\frac{1}{m_n-1}\frac{1}{z_{-j}} \right)^3\right\rbrace,
\quad j=1,2,
\end{split}
\end{equation*}
\begin{equation*}
\begin{split}
-V^{(m_n)}_{\{1,2\}}(\bz)&=2\left\lbrace
\frac{1}{z_1^3}+\frac{1}{z_1^3}-
\left( \frac{1}{m_n-1}\frac{1}{z_1z_2} \right)^{3}\right\rbrace^{-5/3}\prod_{j=1}^2\frac{1}{z_j^4}
\left\lbrace1-
\left(\frac{1}{m_n-1}\frac{1}{z_{-j}} \right)^3\right\rbrace\\
&\quad +3
\left\lbrace
\frac{1}{z_1^3}+\frac{1}{z_1^3}-
\left( \frac{1}{m_n-1}\frac{1}{z_1z_2} \right)^{3}\right\rbrace^{-2/3}\frac{1}{(m_n-1)^3}\prod_{j=1}^2\frac{1}{z_j^4},
\end{split}
\end{equation*}
with $z_{-j}=z_2$ if $j=1$ and $z_{-j}=z_1$ if $j=2$.
Whereas, for $\by \in (0, \infty)^2$, $g_{\brho_{0}, \bone}(\by|H_0)$ and $G_{\brho_{0}, \bone}(\by|H_0)$ are still of the form in \eqref{eq:lim_dens} and in the second line of \eqref{eq:lim_df}, respectively, but now $V(\by)\equiv V(\by|H_0)=(1/y_1^3 +1/y_2^3)^{1/3}$ and
\begin{equation*}
\begin{split}
-V_{\{j\}}(\by)&=\frac{1}{y_j^4}\left(
\frac{1}{y_1^3}+\frac{1}{y_1^3}
\right)^{-2/3}, \quad j=1,2, \\
-V_{\{1,2\}}(\by)&=2\frac{1}{y_1^4}\frac{1}{y_2^4}\left(
\frac{1}{y_1^3}+\frac{1}{y_1^3}
\right)^{-5/3}.
\end{split}
\end{equation*}
Notice that the equality in \eqref{eq:dens_ratio_eq} is still valid. Since for $y_j>1/(m_n-1)^{1/\rho_j}$ we have $u_j:=1/(y_j^{\rho_{0,j}}(m_n-1))^3\leq 1$, $j=1,2$, then 
$$
s(u_1,u_2):=\left\lbrace
1-(u_1+u_2-u_1 u_2)^{1/3}
\right\rbrace \exp \left\lbrace
(u_1+u_2)^{1/3}
\right\rbrace \leq 1,
$$
from which we deduce that
\begin{equation*}
\begin{split}
\left\lbrace
1-\frac{V^{(m_n)}(\by^{\brho_{0}})}{m_n-1}
\right\rbrace^{m_n-1}\exp\left\lbrace V(\by^{\brho_0}) \right\rbrace=\left\lbrace s(u_1,u_2) \right\rbrace^{m_n-1} \leq 1.
\end{split}
\end{equation*}
Moreover, we have that
$$
1 \leq q(u_1,u_2):=\frac{u_1+u_2}{u_1+u_2-u_1u_2} \leq 2,
$$
which implies that 
\begin{equation*}
\begin{split}
\frac{V^{(m_n)}_{\{1\}}(\by^{\brho_{0}})
	V^{(m_n)}_{\{2\}}(\by^{\brho_{0}})}{	V_{\{1\}}(\by^{\brho_{0}})
	V_{\{2\}}(\by^{\brho_{0}})}&=\{q(u_1,u_2)\}^{4/3}\prod_{j=1}^2(1-u_j)\\
& \leq 2^{4/3}
\end{split}
\end{equation*}
and
\begin{equation*}
\begin{split}
\frac{-V^{(m_n)}_{\{1,2\}}(\by^{\brho_{0}})}{-V_{\{1,2\}}(\by^{\brho_{0}})}&=\{q(u_1,u_2)\}^{5/3}\prod_{j=1}^2(1-u_j)
+3\{q(u_1,u_2)\}^{5/3}(u_1+u_2-u_1u_2)
\\
& \leq 2^{10/3}.
\end{split}
\end{equation*}
Consequently, for all $\overline{\epsilon}>0$ and $n$ sufficiently large,
$$
\Vert f_{m_n}/ g_{\brho_0, \bone}(\cdot|H_0) \Vert_\infty \leq (1+\overline{\epsilon})2^{10/3},
$$
whence the final result.

\subsection{Proofs of the results in Section \ref{sec:gumb}}

\subsubsection{Proof of the claim in Example \ref{ex:gumb_est}}\label{sec:suppgumb}
Without loss of generality, assume $\epsilon<1$. Observe that,
%
%\begin{equation}\label{eq:initial_step}
%\begin{split}
%&\prodFmalt(\Vert (\widehat{\boldsymbol{\sigma}}_n/\ba_{m_n}, (\widehat{\boldsymbol{\mu}}_n-\bb_{m_n})/\ba_{m_n} )-(\bone,\bzero)  \Vert_\infty > \epsilon)\\
%& \quad \leq
%\prodFmalt(\Vert \widehat{\boldsymbol{\sigma}}_n/\ba_{m_n}-\bone  \Vert_\infty > \epsilon)
%+
%\prodFmalt(\Vert  (\widehat{\boldsymbol{\mu}}_n-\bb_{m_n})/\ba_{m_n}   \Vert_\infty > \epsilon).
%\end{split}
%\end{equation}
%In particular, 
%\begin{equation*}
%\begin{split}
%\prodFmalt(\Vert  (\widehat{\boldsymbol{\mu}}_n-\bb_{m_n})/\ba_{m_n}   \Vert_\infty > \epsilon) \leq \sum_{j=1}^d
%\prodFmalt(\vert  (\widehat{\mu}_{n,j}-b_{m_n,j})/a_{m_n,j}   \vert > \epsilon),
%\end{split}
%\end{equation*}
for $j=1, \ldots,d$,  
\begin{equation*}
\begin{split}
\prodFmalt(\vert  (\widehat{\mu}_{n,j}-b_{m_n,j})/a_{m_n,j}   \vert > \epsilon)&=\prodcop(1-U_{n(m_n-1),j}< 1-F_{0,j}(a_{m_n,j}\epsilon+b_{m_n,j}))\\
& \quad+\prodcop(1-U_{n(m_n-1),j}> 1-F_{0,j}(-a_{m_n,j}\epsilon+b_{m_n,j})),
\end{split}
\end{equation*}
where $\bU_i=(U_{i,1},\ldots,U_{i,d})$, $i=1,\ldots,nm_n$, are iid rv's with cdf $C_0(\bu)$, the copula function of $F_0$, and $U_{k,j}\equiv U_{k,nm_n,j}$ denotes the $k$-th order statistic of the marginal sample $U_{1,j}, \ldots,U_{nm_n,j}$.
Hence, a few adataptions to the arguments in Section \ref{sec:ex_frec_est_proof} yield that, for $j=1,\ldots,d$, as $n \to \infty$,
\begin{equation}\label{eq:jthbound}
\prodFmalt(\vert  (\widehat{\mu}_{n,j}-b_{m_n,j})/a_{m_n,j}   \vert > \epsilon)\leq 2 \exp\{-\tau \sqrt{n+1}\},
\end{equation}
where $\tau=\tau(\epsilon)$ is a positive constant. Therefore, by Borel-Cantelli lemma, eventually almost surely as $n \to \infty$,
$$
(\widehat{\mu}_{n,j}-b_{m_n,j})/a_{m_n,j}\to 0, \quad j=1, \ldots,d.
$$
%, and, therefore,
%\begin{equation}\label{eq:centering_bound}
%\prodFmalt(\Vert  (\widehat{\boldsymbol{\mu}}_n-\bb_{m_n})/\ba_{m_n}   \Vert_\infty > \epsilon) \leq 2d \exp\{-\tau \sqrt{n+1}\}.
%\end{equation}
%

Furthermore, we have that,
%\begin{equation*}
%\begin{split}
%\prodFmalt(\Vert \widehat{\boldsymbol{\sigma}}_n/\ba_{m_n}-\bone  \Vert_\infty > \epsilon)\leq \sum_{j=1}^d
%\prodFmalt(\vert \widehat{\sigma}_{n,j}/a_{m_n,j}-1  \vert > \epsilon),
%\end{split}
%\end{equation*}
%where, 
for $j=1, \ldots,d$,
\begin{equation*}
\begin{split}
&\prodFmalt(\vert \widehat{\sigma}_{n,j}/a_{m_n,j}-1  \vert > \epsilon)\\
&=\prodFmalt\left(
\int_{\widehat{\mu}_{n,j}}^{z_{0,j}} (1-\widehat{F}_{nm_n,j}(z)) \diff z
-\int_{b_{m_n,j}}^{z_{0,j}}(1-F_{0,j}(z))\diff z
> \epsilon \frac{a_{m_n,j}}{m_n}
\right)\\
& \quad +\prodFmalt\left(
\int_{\widehat{\mu}_{n,j}}^{z_{0,j}} (1-\widehat{F}_{nm_n,j}(z)) \diff z
-\int_{b_{m_n,j}}^{z_{0,j}}(1-F_{0,j}(z))\diff z
< -\epsilon \frac{a_{m_n,j}}{m_n}
\right)\\
& =: T_{1,n}^{(j)}+T_{2,n}^{(j)}.
\end{split}
\end{equation*}
An inequality analogous to that in \eqref{eq:jthbound}, for some $\epsilon'<\epsilon/2$ and $\tau'\equiv\tau(\epsilon')>0$, together with a few simple manipulations, yield that, for $j=1, \ldots,d$, as $n \to \infty$,
$$
\prodFmalt\left( 
\int_{\widehat{\mu}_{n,j}}^{\widehat{\mu}_{n,j}\vee b_{m_n,j}} (1-\widehat{F}_{nm_n,j}(z)) \diff z
> 	\frac{\epsilon}{2}\frac{a_{m_n,j}}{m_n}\right)\leq 2 \exp\{-\tau' \sqrt{n+1}\}.
$$
Therefore, as $n \to \infty$,
\begin{equation*}
\begin{split}
T_{1,n}^{(j)} &\leq
\prodFmalt\left(
\int_{\widehat{\mu}_{n,j}}^{\widehat{\mu}_{n,j}\vee b_{m_n,j}} (1-\widehat{F}_{nm_n,j}(z)) \diff z
+\int_{b_{m_n,j}}^{z_{0,j}}(F_{0,j}(z)-\widehat{F}_{nm_n,j}(z))\diff z
> \epsilon \frac{a_{m_n,j}}{m_n}
\right)\\
& \leq \prodFmalt\left(
\dist_{W,1}(F_{0,j},\widehat{F}_{nm_n,j})
> \frac{\epsilon}{2} \frac{a_{m_n,j}}{m_n}
\right)+\exp\{-\tau' \sqrt{n+1}\}
\end{split}
\end{equation*}
where $\dist_{W,1}$ is the $1$-Wasserstein distance. Since  $(\epsilon/2)a_{m_{n,j}}/m_n<1$ as $n\to \infty$, by Theorem 2 in \cite{r699} it holds that
\begin{equation*}
\begin{split}
\prodFmalt\left(
\dist_{W,1}(F_{0,j},\widehat{F}_{nm_n,j})
> \frac{\epsilon}{2} \frac{a_{m_n,j}}{m_n}
\right) &\leq k_j \exp\left\lbrace
-k_j' \frac{\epsilon^2}{4}\frac{a_{m_n,j}^2}{m_n^2}nm_n
\right\rbrace\\
& \quad + k_j\exp\left\lbrace
-k_j' \left(\frac{\epsilon}{2}\frac{a_{m_n,j}}{m_n}nm_n
\right)^{\alpha_j-k_j'' }\right\rbrace\indic(\alpha_j<1),
\end{split}
\end{equation*}
for some positive constants $k_j, k_j'$ and any $k_j''\in(0,\alpha_j)$. By the assumption in \eqref{eq:rate}, if $\alpha_j\geq 1$, the expression on the right hand-side of the above display simplifies to
\begin{equation*}
\begin{split}
k_j \exp\left\lbrace
-k_j' \frac{\epsilon^2}{4}\frac{a_{m_n,j}^2}{m_n}n
\right\rbrace &=
k_jn^{-1-c'} \exp\left\lbrace
-n \frac{a_{m_n,j}^2}{m_n}\left(
k_j' \frac{\epsilon^2}{4}
- (1+c')\frac{\log n}{n a_{m_n,j}^2} m_n\right)
\right\rbrace \\
&\leq k_jn^{-1-c'}
\end{split}
\end{equation*} 
for any $c'>0$ and $n$ large enough, while, if $\alpha_j < 1$ and $k_j''=\alpha_j-s$, it boils down to
\begin{equation*}
\begin{split}
&k_j \exp\left\lbrace
-k_j' \frac{\epsilon^2}{4}\frac{a_{m_n,j}^2}{m_n}n
\right\rbrace + k_j\exp\left\lbrace
-k_j' \left(\frac{\epsilon}{2}a_{m_n,j}n
\right)^{s }\right\rbrace\\
& \quad \leq 2k_j \exp
\left\lbrace
-k_j' \frac{\epsilon^2}{4}
\min\{n a_{m_n,j}^2 /m_n,
a_{m_n,j}^s n^s \}
\right\rbrace\\
&\quad =2k_j n^{-1-c'}
\exp
\left\lbrace
-\min\left[
\frac{a_{m_n,j}^2}{m_n}n\left(
k_j' \frac{\epsilon^2}{4}
- (1+c')\frac{\log n}{n a_{m_n,j}^2} m_n\right), \right. \right.\\
& \hspace{12em} \left. \left.
a_{m_n,j}^s n^s \left(
k_j' \frac{\epsilon^s}{2^s} 
-(1+c')\frac{\log n}{a_{m_n,j}^s n^s} \right)
\right]
\right\rbrace\\
& \quad \leq 2k_j n^{-1-c'}.
\end{split}
\end{equation*} 
Consequently, letting $k=\max_j k_j$, as $n \to \infty$,
$$
\max_jT_{1,n}^{(j)} \leq 2kn^{1+c'}+\exp\{-\tau' \sqrt{n+1}\}=O(n^{-1-c'}).
$$
A similar reasoning leads to conclude that,  as $n \to \infty$, 
$$
\max_jT_{2,n}^{(j)}=O(n^{-1-c'}).
$$ 
%We can now deduce that, as $n \to \infty$,
%$$
%\prodFmalt(\Vert \widehat{\boldsymbol{\sigma}}_n/\ba_{m_n}-\bone  \Vert_\infty > \epsilon)=O(n^{-1-c'}).
%$$
In particular, since $c'>0$ is arbitrary, in the two displays above we can replace $O$ with $o$. By Borel-Cantelli lemma we can now conclude the with probability $1$ 
$$
\widehat{\sigma}_{n,j}/a_{m_n,j}\to 1, \quad j=1, \ldots,d,
$$ 
%This fact, together with \eqref{eq:initial_step} and \eqref{eq:centering_bound}, allow to conclude that, for any $c'>0$, as $n \to \infty$,
%$$
%\prodFmalt(\Vert (\widehat{\boldsymbol{\sigma}}_n/\ba_{m_n}, (\widehat{\boldsymbol{\mu}}_n-\bb_{m_n})/\ba_{m_n} )-(\bone,\bzero)  \Vert_\infty > \epsilon)=o(n^{-1-c'}),
%$$
as $n \to \infty$, which completes the proof.

\subsubsection{Technical derivations for Example \ref{ex:mod_gumb}}
\label{appsec:gumbexamp}
Note that, for $\bx \in (0,\infty)^2$, $F_0(\bx)$ allows the representation
$$
F_0(\bx)=C_0\left(1-e^{-x_1},1-e^{-x_2}
\right),
$$
where $C_0$ is as in \eqref{eq:cupula_ex} and satisfies Condition \ref{cond:strong}\ref{cond:copdiff} (see Section \ref{sec: first_ex_frec}). The marginal distributions $F_{0,j}$, $j=1,2$, are exponential, thus satisfy Condition \ref{cond:strong}\ref{margins} 
\citep[e.g.,][p. 1311]{r3003}.
%{\color{magenta}(e.g., Falk and Marohn, 1993, p. 1311)}.
Moreover, $F_0 \in \mathcal{D}(G_{
	 \bone, \bzero}(\cdot|H_0))$, where, for $\bx \in \reald$,
$$
G_{\bone, \bzero}(\bx|H_0)=\exp\left\lbrace
-e^{-x_1}-e^{-x_2}
+\left(
e^{x_1}+ e^{x_2}
\right)^{-1}
\right\rbrace, 
$$ 
\citep[e.g.][Example 5.16]{r200}. Valid norming sequences, asymptotically equivalent to those in \eqref{eq:norming}, are 
$$
\ba_{m_n}=(1, 1), \quad \bb_{m_n}=(\log(m_n-1), \log(m_n-1)).
$$
Thus, the probability density pertaining to $F_0^{m_n}(\ba_{m_n}\bx + \bb_{m_n})$, $\bx \in (-\log(m_n-1),\infty)^2$, is given by
$$
f_{m_n}(\bx)=f_{m_n}^*(e^{x_1},e^{x_2})e^{x_1}e^{x_2},
$$
where $f_{m_n}^*$ is a probability density defined as in \eqref{eq:dens_max}, with $\brho_0=\bone$ and $V^{(m_n)}$ as in \eqref{eq:V_m}. Furthermore, for $\bx \in \real^2$,
$$
g_{\bone, \bzero}(\bx|H_0)=g_{ \bone}(e^{x_1},e^{x_2}|H_0)e^{x_1}e^{x_2},
$$
where the simple max-stable density $g_{\bone}(\cdot|H_0)$ equals the density in \eqref{eq:lim_dens}, with $\brho_0=\bone$ and $V\equiv V(\cdot|H_0)$ as in \eqref{eq:lim_df}. 
Hence, 
$$
\Vert f_{m_n}/g_{\bone, \bzero}(\cdot|H_0)\Vert_\infty \leq
\Vert f_{m_n}^*/g_{\bone}(\cdot|H_0)\Vert_\infty
$$
and, by \eqref{eq:ratiobound}, we can conclude that the requirement in Condition \ref{cond:densratio} is satisfied in the present example.

\subsection{Proofs of the results in Section \ref{sec:weibulldom}}\label{sec:cont_weib}

\subsubsection{Auxiliary results for the proof of Theorem \ref{theo:rem_weib}}

\begin{lemma}\label{lem:pseudokulbweib}
	Under the assumptions of Theorem \ref{theo:rem_weib}, for any $c>0$, 
%	with $\prodFm$-probability tending to $1$ as $n \to \infty$
	$$
	\lim_{n\to \infty}\prodFm \left(
	\int \prod_{i=1}^n 
	\left\lbrace
	\frac{g_{\bomega,\bsigma,\bmu}(\overline{\bM}_{m_n,i}|H)}{g_{\bomega_0, \bone, \bzero}(\overline{\bM}_{m_n,i}|H_0)}
	\right\rbrace
	\diff(\Pi_\Hset \times \overline{\Psi}_n)(H, \brho, \bsigma) \geq e^{-nc}
	\right)  = 0
	$$
\end{lemma}
\begin{proof}
	We follow the main lines of the proof of Lemma \ref{lem:pseudokulb}, with a few adaptations.
	Let $\underline{\Pi}_{\text{sc}}^{(d)}$ and $\underline{\Pi}_{\text{sc}}^{(d)}$ be pm with positive Lebesgue density $\underline{\pi}_{\text{sc}}^{(d)}$ on $I_{\text{sc}}^d$  and $\underline{\pi}_{\text{loc}}^{(d)}$ on $I_{\text{loc}}^d$, respectively, and satisfying
	$$
	\sup_{\bx \in (1\pm \eta)^d} \underline{\pi}_{\text{sc}}^{(d)}(\bx)< M_{\text{sc}},
	\quad \sup_{\bx \in (-\eta,+\eta)^d} \underline{\pi}_{\text{loc}}^{(d)}(\bx)< M_{\text{loc}},
	$$
	where $M_{\text{sc}}$ and  $M_{\text{loc}}$ are defined via
	$$
	M_{\text{sc}}^{1/d}:= \inf_{x \in (1\pm \eta)}\pisc(x), \quad
	M_{\text{loc}}^{1/d}:= 
	\inf_{x \in [-\eta, +\eta]}\piloc(x),
	$$
	and are positive and bounded by Conditions  \ref{cond:frecextend}\ref{cond:pisc}\ref{posit} and \ref{cond:gumbextend}\ref{cond:piloc}\ref{posit1}.
	Define 
	$$
	\Pi_{\bvarTheta}=\Pi_{\text{sh}}\times \underline{\Pi}_{\text{sc}}^{(d)} \times \underline{\Pi}_{\text{loc}}^{(d)}. 
	$$
	By construction,  $\Pi_{\bvarTheta}$ satisfies Condition \ref{cond:angularprior}\ref{cond:compactprior}.
	Then, denote $\Pi_{\Hset\times \bvarTheta}=\Pi_{\Hset}\times \Pi_{\bvarTheta}$ and observe that 
%	the assumption of a compact parametric space $\bvarTheta$ in Corollary \ref{lem: omega} is nonessential for the result in the statement to hold true, as long as the prior $\Pi_{\bvarTheta}$ is positive on a neighbourhood of $(\bomega_0, \bsigma_{0}, \bmu_0)$. Since $(\bomega_0, \bsigma_{0}, \bmu_0)=(\bomega_0,\bone, \bzero) \in K_{\text{sh}} \times I_{\text{sc}}^d\times I_{\text{loc}}^d$, this is herein the case by construction.
	%
%	As all 
	the other assumptions of Corollary \ref{lem: omega} are satisfied by hypothesis. We can thus conclude that for all $\epsilon>0$ 
	$$
	\Pi_{\Hset \times \bvarTheta}\left( \mathcal{V}_\epsilon^{(1)}\right)>0,
	$$   
	where $\mathcal{V}_\epsilon^{(1)}$ is defined as in \eqref{eq:v1setweib}. 

	Next, observe that by Condition \ref{cond:extendweib}\ref{cond:compsupp2} there exist neighborhoods $\mathcal{U}^\bone\subset I_{\text{sc}}^d$ of $\bone$ and $\mathcal{U}^\bzero \subset I_{\text{loc}}^d$ of $\bzero$ such that,
	with $\prodFmalt$-probability tending to $1$, as $n \to \infty$
	\begin{eqnarray*}
		\prod_{j=1}^d\pisc(x_j/\{\widehat{\sigma}_{n,j}/a_{m_n,j}\})&>&\underline{\pi}_{\text{sc}}^{(d)}(\bx), \quad \forall \bx \in  \mathcal{U}^\bone,\\
		\prod_{j=1}^d\piloc
		\left(
		\frac{x_j - \frac{\widehat{\mu}_{n,j}-b_{m_n,j}}{a_{m_n,j}}}{\widehat{\sigma}_{n,j}/a_{m_n,j}} 
		\right)&>&\underline{\pi}_{\text{loc}}^{(d)}(\bx), \quad \forall \bx \in  \mathcal{U}^\bzero.
	\end{eqnarray*}
	Defining $\mathcal{V}=\{\Hset \times K_{\text{sh}} \times \mathcal{U}^{\bone} \times \mathcal{U}^{\bzero}
	\} \cap  \mathcal{V}_\epsilon^{(1)}$, we also have $\Pi_{\Hset \times \bvarTheta}\left( \mathcal{V}\right)>0$. Accordingly, let $\overline{\Pi}_{\Hset \times \bvarTheta}(\cdot):=\Pi_{\Hset \times \bvarTheta}(\, \cdot \cap \mathcal{V})/ \Pi\left( \mathcal{V}\right)$ and observe that, analogously to \eqref{eq:ineq_aKL},
	\begin{equation}\label{eq:ineq_aKLweib}
	\begin{split}
	\int_{\Hset \times \bvarTheta} \prod_{i=1}^n 
	\left\lbrace
	\frac{g_{\bomega,\bsigma, \bmu}(\overline{\bM}_{m_n,i}|H)}{g_{\bomega_0,\bone, \bzero }(\overline{\bM}_{m_n,i}|H_0)}
	\right\rbrace
	\diff(\Pi_\Hset \times \overline{\Psi}_n)(H, \bomega, \bsigma, \bmu)\geq \Pi_{\Hset \times \bvarTheta}(\mathcal{V}) \exp
	\left\lbrace
	-nI_n
	\right\rbrace,
	\end{split}
	\end{equation}
	where 
	$$
	I_n:=\int_{\mathcal{V}} \frac{1}{n} \sum_{i=1}^n  \log^+
	\left\lbrace
	\frac{g_{\bomega,\bsigma, \bmu}(\overline{\bM}_{m_n,i}|H)}{g_{\bomega_0,\bone, \bzero }(\overline{\bM}_{m_n,i}|H_0)}
	\right\rbrace
	\diff\overline{\Pi}_{\Hset \times \bvarTheta}(H, \bomega, \bsigma, \bmu).
	$$
	Using Condition \ref{cond:densratio} and reasoning as in \eqref{eq: expectfrec}, we obtain once more the inequality $\mathbb{E}(I_n) \leq J_0 \epsilon$, for all $n \geq n_0$. Hence, by Markov's inequality, on a set of $Q_n$-probability larger than $1-\sqrt{\epsilon}$, the term on the right-hand side of \eqref{eq:ineq_aKLweib} is bounded from below by
	$
	\Pi_{\Hset \times \bvarTheta}(\mathcal{V}) \exp
	\left\lbrace
	-nJ_0 \sqrt{\epsilon}
	\right\rbrace,
	$
	for all $n \geq n_0$. Since $\epsilon$ can be chosen arbitrarily small, the result now follows.
\end{proof}

\subsubsection{Proof of Theorem \ref{theo:rem_weib}}\label{appsec:proof_weibdom}
By Conditions \ref{cond:extendweib}\ref{cond:compsupp1}--\ref{eq:estlocscaleweib}, the data-dependent prior $\overline{\Psi}_n$ obtained after the reparametrisation in the second line of 	\eqref{eq:reparam} have marginal densities on scale and location components satisfying
\begin{eqnarray*}
	\left\lbrace  x \in \real: \, \pisc(x_j/\{\widehat{\sigma}_{n,j}/a_{m_n,j}\})\{\widehat{\sigma}_{n,j}/a_{m_n,j}\}^{-1}>0 \right\rbrace &\subset &
	K_{\text{sc}}\\
	\left\lbrace  x \in \real: \, \piloc\left(
	\frac{x_j - \frac{\widehat{\mu}_{n,j}-b_{m_n,j}}{a_{m_n,j}}}{\widehat{\sigma}_{n,j}/a_{m_n,j}} 
	\right)
	\{\widehat{\sigma}_{n,j}/a_{m_n,j}\}^{-1}
	>0 \right\rbrace &\subset&
	K_{\text{loc}}
\end{eqnarray*}
with $\prodFmalt$-probability tending to $1$ as $n \to \infty$, for a compact subinterval of $(0, \infty)$, $K_{\text{sc}}$,  having $1$ as an interior point, and a compact interval in $\real$, $K_{\text{loc}}$, having $0$ as an interior point. Then, by Conditions \ref{cond:frecextend}\ref{cond:pisc}\ref{piscbound} and \ref{cond:gumbextend}\ref{cond:piloc}\ref{piscbound1} there exist positive constants $c_{\text{sc}}$ and $c_{\text{loc}}$ such that
\begin{eqnarray*}
	\pisc(x_j/\{\widehat{\sigma}_{n,j}/a_{m_n,j}\})\{\widehat{\sigma}_{n,j}/a_{m_n,j}\}^{-1} &\leq& c_{\text{sc}} \overline{u}_{\text{sc}}(x), \quad \forall x \in K_{\text{sc}}\\
	\piloc\left(
	\frac{x_j - \frac{\widehat{\mu}_{n,j}-b_{m_n,j}}{a_{m_n,j}}}{\widehat{\sigma}_{n,j}/a_{m_n,j}} 
	\right)
	\{\widehat{\sigma}_{n,j}/a_{m_n,j}\}^{-1} &\leq& c_{\text{loc}} \overline{u}_{\text{loc}}(x), \quad \forall x \in K_{\text{loc}}
\end{eqnarray*} 
with $\prodFmalt$-probability tending to $1$ as $n \to \infty$, where
\begin{eqnarray*}
	\overline{u}_{\bullet}(x)&:=& u_{\bullet}(x)/\int_{K_{\bullet}}u_{\text{sc}}(s) \diff s, \hspace{2em} x \in K_{\bullet}
	%	\\
	%	 \overline{u}_{\text{loc}}(x)&:=&u_{\text{loc}}(x)/\int_{K_{\text{loc}}}u_{\text{loc}}(s) \diff s, \quad  x \in K_{\text{loc}}.
\end{eqnarray*}
and ``$\bullet$" stands either for ``$\text{sc}$" or ``$\text{loc}$".
Denote by $\overline{\Pi}_{\bullet}^{(d)}$ the pm on $K_{\bullet}^d$ whose Lebesgue density equals
$$ 
\prod_{j=1}^d
\overline{u}_{\bullet}(x_j)
$$
and define $\Pi_{\bvarTheta}=\Pi_{\text{sh}}\times \overline{\Pi}_{\text{sc}}^{(d)}\times \overline{\Pi}_{\text{loc}}^{(d)}$. 
Notice that, by construction,  $\Pi_{\bvarTheta}$ satisfies Condition \ref{cond:angularprior}\ref{cond:compactprior} and, by assumption, $\Pi_\Hset$ satifies Condition \ref{cond:genprior}\ref{cond:angularpmprior}.
%the support of $\Pi_{\bvarTheta}$ is compact and contains the true limiting parameter $(\bomega_0, \bone, \bzero)$ as an interior point.
Therefore, the arguments in the proof of Theorem \ref{theo: cons_weibull} apply to $\Pi_{\Hset\times \bvarTheta}:=\Pi_{\Hset}\times \Pi_{\bvarTheta}$.
%also adapts to $\Pi_{\Hset\times \bvarTheta}$, allowing to claim that properties analogous to those in Condition \ref{cond: newcond} are satisfied once restricting to $\text{supp}(\Pi_{\bvarTheta})$, despite the parametric space $\bvarTheta$ not being defined herein as a compact set. 
%
This gives sufficient conditions to proceed analogously to the proof of Theorem \ref{th:rem_cont_frec} and conclude by applying Lemma \ref{lem:pseudokulbweib} and Proposition \ref{prop:RC}.

$\;$
\end{supplement}

\end{document}